\documentclass[11pt]{article}
\usepackage{amsmath,amssymb,amsfonts}
\usepackage{amsthm, txfonts}
\usepackage{todonotes}
\usepackage{enumitem}
\usepackage{tikz-cd, tikz} 
\usepackage{mathrsfs}
\usepackage{placeins}

\usetikzlibrary{decorations.pathmorphing}
\usepackage{transparent}
\usepackage{amsrefs}
\usepackage{color}
\usepackage{fullpage}
\usepackage{soul}
\usepackage{graphicx,caption}
\graphicspath{{figures/}}
\makeatletter
\renewcommand\paragraph{\@startsection{paragraph}{4}{\z@}%
            {-2.5ex\@plus -1ex \@minus -.25ex}%
            {-1ex}%
            {\normalfont\normalsize\bfseries}}
\makeatother
\newtheorem{theorem}{Theorem}[subsection]
\newtheorem{claim}[theorem]{Claim}
\newtheorem{fact}[theorem]{Fact}
\newtheorem{proposition}[theorem]{Proposition}
\newtheorem{corollary}[theorem]{Corollary}
\newtheorem{lemma}[theorem]{Lemma}

\newtheorem*{namedtheorem}{\theoremname}
\newcommand{\theoremname}{testing}
\newenvironment{named}[1]{\renewcommand{\theoremname}{#1}\begin{namedtheorem}}{\end{namedtheorem}}

\theoremstyle{definition}
\newtheorem{definition}[theorem]{Definition}
\newtheorem{example}[theorem]{Example}

\newtheorem{remark}[theorem]{Remark}
\numberwithin{equation}{subsection}

\DeclareMathOperator{\Aut}{Aut}

\DeclareMathOperator{\Sub}{Sub}

\DeclareMathOperator{\PSL}{PSL}
\DeclareMathOperator{\SO}{SO}
\DeclareMathOperator{\PSO}{PSO}
\DeclareMathOperator{\injrad}{injrad}

\DeclareMathOperator{\systole}{systole}
\DeclareMathOperator{\minsystole}{\sigma}

\DeclareMathOperator{\id}{id}
\DeclareMathOperator{\hyp}{hyp}
\DeclareMathOperator{\ext}{ext}

\DeclareMathOperator{\disc}{d}
\DeclareMathOperator{\nd}{nd}
\DeclareMathOperator{\Id}{Id}
\DeclareMathOperator{\elem}{elem}

\newcommand{\RC}{ \mathrm{RC}} 
 
\newcommand{\pbase}{p_{base}}
\newcommand{\modsim}{/ \mathord \sim} 

\newcommand{\inj}{\mathrm{inj}^\infty}
\newcommand\dblsetminus{\setminus\hspace{-1ex}\setminus}
\newcommand{\RR}{\mathbb{R}}
\newcommand{\Sec}{\mathrm{Sec}}
\newcommand{\NN}{\mathbb{N}}
\newcommand{\CC}{\mathbb{C}}
\newcommand{\ZZ}{\mathbb{Z}}
\newcommand{\HH}{\mathbb{H}}
\newcommand{\bbS}{\mathbb{S}}

\newcommand{\SSS}{\mathcal{S}}
\newcommand{\DDD}{\mathcal{D}}
\newcommand{\PPP}{\mathcal{P}}
\newcommand{\LLL}{\mathcal{L}}

\newcommand{\length}{\mathrm{length}}
\newcommand{\dsm}{\mathbin{{\setminus}\mspace{-5mu}{\setminus}} }
\newcommand{\I}{\mathbf{I}}
\newcommand{\II}{\mathbf{II}}

\newcommand{\cplx}{\nu}
\newcommand{\nhd}{\mathscr{N}}

\title{On the Chabauty space of $\PSL_2(\RR)$, I: lattices and grafting}
\author{Ian Biringer, Nir Lazarovich, and Arielle Leitner} 

\begin{document}

\maketitle

\begin{abstract}
This is the first of two papers on the global topology of the space $\Sub(G)$ of all closed subgroups of $G=\PSL_2(\RR)$, equipped with the Chabauty topology. In this paper, we study the spaces of lattices and elementary subgroups of $G$, and prove a continuity result for conformal grafting of (possibly infinite type) vectored orbifolds that will be useful in both papers. 

More specifically, we first identify the homotopy type of the space of elementary subgroups of $G$, following Baik--Clavier. Then for a fixed finite type hyperbolizable $2$-orbifold $S$, we show that the space $\Sub_S(G)$ of all lattices $\Gamma < G$ with $\Gamma \backslash \HH^2 \cong S$ is a fiber orbibundle over the moduli space $\mathcal M(S)$. We describe the closure $\overline{\Sub_S(G)}$ in $\Sub(G)$ and show that $\partial \Sub_S(G)$ has a neighborhood deformation retract within $\overline{\Sub_S(G)}$. When $S$ is not one of finitely many low complexity orbifolds, we show that $\overline{\Sub_S(G)}$ is simply connected. In the simplest exceptional case, when $S$ is a sphere with three total cusps and cone points, we show that $\overline{\Sub_S(G)}$ is a (usually nontrivial) lens space. Finally, we show that when $(X_i,v_i) \to (X_\infty,v_\infty)$ is a (possibly infinite type) smoothly converging sequence of vectored hyperbolic $2$-orbifolds, and we graft in Euclidean annuli along suitable collections of simple closed curves in the $X_i$, then after uniformization, the resulting vectored hyperbolic $2$-orbifolds converge smoothly to the expected limit. As part of the proof, we give a new lower bound on the hyperbolic distance between points in a grafted orbifold in terms of their original distance.
\end{abstract}

\section {Introduction}

Let $G$ be a locally compact second countable group and let $\Sub(G)$ be the set of all closed subgroups of $G$. We endow $\Sub(G)$ with the \emph{Chabauty topology}, wherein $H_n\to H$ if 
\begin{enumerate} 
\item each $h \in H$ is the limit of a sequence of elements $h_n \in H_n$, and 
\item any accumulation point of a sequence of elements $h_n \in H_n$ lies in $H$.
\end{enumerate} 
With this topology, $\Sub(G)$ is compact and metrizable \cite[Lemma E.1.1]{Benedettilectures}. The Chabauty topology was introduced in several places, among them \cites{chab, Fell}. It is essential in hyperbolic geometry, where the Chabauty topology on $\mathrm{Isom}(\mathbb H^n)$ is  usually called the \emph{geometric topology}, see \cite{CEG} for a survey; it can be used to compactify the symmetric spaces and Bruhat-Tits buildings of semisimple $G$, see \cite{Haettel3}, and the moduli spaces of genus g  \cite{harvey2016chabauty}; it gives a natural way to transition between different geometric structures on manifolds, e.g.\ the eight $3$-dimensional geometries in Thurston's geometrization conjecture, see \cite{cooper}; and it is the natural topology underlying the theory of \emph{invariant random subgroups} of $G$, see e.g.\ \cites{7sam,abertkesten,gelander,krifka2020irs}.  For an overview of the Chabauty  topology in a general setting we recommend \cite{harpe}.

Only in very special cases can the entire topology of $\Sub(G)$ be understood.  For example, it is easy to see that $\Sub(\RR) \cong [0,\infty]$ via the map that sends $\{e\} \mapsto \infty$, $\RR \mapsto 0$ and $\langle x \rangle \mapsto x$ for $x>0$. The topology becomes more interesting in dimension two:  Hubbard--Pourezza \cite{HP} showed that $\Sub(\RR^2) \cong \bbS^4$ via a homeomorphism that takes the non-lattices in $\Sub(G)$ to the suspension $S(K) \subset S(\bbS^3) = \bbS^4$ of the trefoil knot $K \subset \bbS^3$. In higher dimensions, though, a full description of $\Sub(\RR^n)$ is out of reach.  Most of what is known is due to Kloeckner~\cite{kloeckner}, who showed that while $\Sub(\RR^n)$ is not a manifold for $n>2$, it is a stratified space in the sense of Goresky-MacPherson, and is simply connected.   There are a few non-abelian groups $G$  for which $\Sub(G)$ is reasonably well understood, e.g.\ the Heisenberg group  and some other low dimensional examples \cites{BHK,Haettel2}, but for most $G$ the topology of $\Sub(G)$ is quite complicated.

Various authors have also made progress in understanding limits of certain categories of subgroups inside some larger ambient group $G$, thus shedding light on part of $\Sub(G)$:
for example the categories of abelian subgroups \cites{Haettel,leitner2016limits,leitner2016conjugacy, ciobotaru2017chabauty, baik2016space}, connected subgroups \cite{lazarovich2019local}, and lattices \cite{wang1968limit}. 

\medskip

From now on, we study the topology of $\Sub(G)$ when $G = \PSL_2(\RR)$.  The group $G$ acts by  fractional linear transformations on the upper half plane $\HH^2$. This action is isometric and identifies $G \cong \mathrm{Isom}^+(\HH^2)$, and we can then view subgroups of $G$ geometrically through their actions on $\HH^2$.  The best understood subgroups are those that are \emph{elementary}, i.e.\  that have a global fixed point in $\HH^2 \cup \partial \HH^2$ or stabilize a geodesic line in $\HH^2$. Elementary subgroups include all closed, non-discrete proper subgroups, as well as all virtually cyclic discrete groups, and they can be completely classified. We give a complete list in Proposition \ref{prop: classification of elem subgroups}, and we describe the topology of the space of elementary subgroups in Proposition \ref{prop: subspace of elem subgroups}. As a more concrete corollary, one can prove the following, which slightly generalizes work of Baik--Clavier \cite{baikclavier}.
\begin{named}{Corollary \ref{cor: homotopy type of sub elem}}
The space of elementary subgroups of $G$ is homotopy equivalent to the subset of $\RR^3$ that is the union of the $2$-spheres $\bbS_n$, $n=2,3,\ldots,\infty$, where $\bbS_n$ has radius $1-1/n$ and is centered at $(1-1/n,0,0).$ In particular, $\Sub_{elem}(G)$ is simply connected.
\end{named}

All nonelementary proper subgroups in $\Sub(G)$ are discrete groups, which are best viewed as follows in terms of their quotient orbifolds. Fix some unit vector $v_{base}$ in the tangent space at some point in $\HH^2$. If $\Gamma < G$  is discrete, then the quotient $X := \Gamma \backslash \HH^2$ is a hyperbolic $2$-orbifold, and the vector $v_{base}$  projects  under the quotient map to a vector $v $ in the tangent bundle of $X$.   We call the pair $(X,v)$ a \emph{vectored hyperbolic $2$-orbifold}. See \S \ref{sec: orbifolds} for more information about orbifolds and their tangent bundles. 

It turns out that the map $\Gamma \longmapsto (X,v)$ gives a bijection 
\begin{equation*}
    \Sub_d(G) : = \{ \text { discrete subgroups of } G \ \} \longrightarrow \{ \text { vectored orientable hyperbolic } 2\text{-orbifolds } (X,v) \ \} \modsim,
\end{equation*}
where here two vectored orbifolds  $(X,v)$ and $(X',v')$ are equivalent if there is an orientation preserving isometry $f :X \to X'$ where the derivative satisfies $d f(v) = v'$. We denote the inverse of the map by $$(X,v) \longmapsto \Gamma(X,v).$$  Under this bijection, the Chabauty topology on $\Sub_d(G)$ corresponds to the \emph{vectored smooth topology} on the space of vectored orbifolds, in which $(X,v)$ is close to $(X',v')$ if large compact sets around the base points of these two orbifolds are diffeomorphic and almost isometric. See \S \ref{sec: convergence section}  for details, or \cite{CEG} \S I.3.1, I.4.1, for the classical case of surfaces.

There are certain subsets of $\Sub_d(G)$ that are vectored versions of the classical moduli spaces of hyperbolic orbifolds with fixed topology. Namely, let $S$ be an orientable, hyperbolizable $2$-orbifold that has \emph{finite type}, meaning that it is obtained from a compact orbifold by deleting a finite set of points. Let
$$\Sub_S(G) := \{ \ \Gamma \in \Sub_d(G) \ | \ \Gamma \backslash \HH^2 \text{ has finite volume and is homeomorphic to } S \}.$$
With respect to the group/orbifold correspondence above, elements of $\Sub_S(G)$ correspond to vectored, finite volume hyperbolic $2$-orbifolds homeomorphic to $S$, up to vectored isometry. In \S \ref{sec: vecmodsec}, we see that $\Sub_S(G)$ is a bundle over the classical \emph{moduli space} $\mathcal M(S)$, whose elements are hyperbolic structures on $S$ up to orientation preserving isometry. In particular, we show the following.

\begin{named}{Proposition \ref{prop: orbibundle}}
The natural map $\Sub_S(G) \longrightarrow \mathcal M(S),$ where $ \Gamma(X,v) \longmapsto [X],$
is a fiber orbibundle whose regular fibers are homeomorphic to $T^1 S$, and $\Sub_S(G)$ is a $6g+2(k+l)-3$ dimensional manifold, where $g$ is the genus of $S$, $k$ is the number of cusps, and $l$ is the number of cone points.
\end{named}

% 

% \begin{theoremA}
%     Suppose that $S$ is a hyperbolizable $2$-orbifold with finite topological type. Then the boundary $\partial \Sub_S(G)$ is the union of all $\Sub_T(G)$, where $T$ is a hyperbolizable $2$-orbifold that embeds as an essential suborbifold in $S$, together with a subset of 
%     $\Sub_{elem}(G)$.
% \end{theoremA}

% \begin{theoremA}
% $\partial \Sub_S(G)$ has a neighborhood retract in $\overline{\Sub_S(G)}$.
% \end{theoremA}

% \begin{theoremA}
% Suppose that $S$ is a hyperbolizable 2-orbifold with sufficient genus, cone points and cusps. Then, the closure $\overline {\Sub_S(G)}$ is simply connected.
% \end{theoremA}

% \begin{theoremA}
% Suppose that $S$ is a sphere with a total of $3$ cusps and cone points, not all of which are cone points, then $\overline{\Sub_S(G)}$ is homeomorphic to a lens space.
% \end{theoremA}

% {\bf
% Is there a way to summarize points 1-3 below to one theorem about how grafting behaves wrt Chabauty topology? 

% ************************
% }

\subsection{Chabauty boundaries of vectored moduli spaces} The main point of this paper is to describe the spaces $\Sub_S(G)$ and how they limit onto each other. In a forthcoming sequel, we will combine this work with an analysis of the space $\Sub^\infty(G)$ consisting of all discrete subgroups of $G$ with infinite covolume. 

Fix a finite type, orientable, hyperbolizable $2$-orbifold $S$ and let $$\partial \Sub_S(G) := \overline{\Sub_S(G)} \setminus \Sub_S(G)$$ be the topological boundary of $\Sub_S(G)$. 

\begin{named}{Proposition \ref{prop: boundarybewhat}}[Informal version]
$\partial \Sub_S(G)$ is the union of some elementary groups with all the spaces $\Sub_T(G)$, where $T$ is an essential proper suborbifold of $S$.
\end{named}
For instance, if $S$ is a surface other than the thrice punctured sphere, the relevant elementary groups are the set of all conjugates of the one parameter subgroups $A,N \leq G$, where $A$ is the diagonal subgroup and $N$ is the upper uni-triangular subgroup. A formal statement of the proposition is given in \S \ref{boundaries and ends}. Here, vectored $2$-orbifolds in $\Sub_S(G)$ converge to the boundary when curves are pinched to cusps, or when the base vector goes out a cusp. Given a sequence in $\Sub_S(G)$ converging into the boundary, if the injectivity radius at the base vector stays bounded away from zero, the limit is in some $\Sub_T(G)$, while if the injectivity radius goes to zero, we get convergence to a conjugate of $A$ or $N$, depending on whether there is a geodesic at bounded distance from the base vector or not. This proposition will not be surprising to experts, and some similar ideas are present even in a 1971 paper of Harvey \cite{harvey2016chabauty}. We mention the proposition here since it is the first step to understanding the following theorem, the proof of which takes up a large portion of the paper.

\begin{named}{Theorem \ref{thm: deformation retract thm}}[Informal version]
If $S$ is a finite type, hyperbolizable $2$-orbifold, then the boundary $\partial \Sub_S(G)$ is a neighborhood deformation retract within the closure $\overline {\Sub_S(G)}$.
\end{named}
More precisely, let $\epsilon>0$ be less than the Margulis constant and define the \emph{$\epsilon$-end  neighborhood} of $\Sub_S(G)$  to be the subset $\Sub_S^\epsilon(G)$ consisting of all vectored  orbifolds $(X,v) \in \Sub_S(G)$ such that either 
\begin{enumerate}
	\item the \emph{systole} of $X$, i.e. the length of the shortest closed geodesic, is less than $\epsilon$, or 
\item the vector $v$ lies in a component of the $\epsilon$-thin part of $S$ that is a horoball neighborhood of the cusp.
\end{enumerate}
Then $\Sub_S^\epsilon(G) \cup \partial \Sub_S(G)$ is a neighborhood of $ \partial \Sub_S(G)$ within $\overline { \Sub_S(G)}$, and we show that this neighborhood deformation retracts onto $ \partial \Sub_S(G)$.

Theorem \ref{thm: deformation retract thm} is thematically similar to the existence of a deformation retract from moduli space to its $\epsilon$-thick part, an incomplete result of Harer \cite{harer1988cohomology} that was completed by Ji-Wolpert \cite{jiwolpert}, see also Ji \cite{ji2014well} for a better version of the proof. In those results, the idea is to flow in the direction of the Weil-Petersson gradient of the systole function until the systole increases to $\epsilon$. In Theorem \ref{thm: deformation retract thm}, we instead want to \emph{decrease} the systole to zero, but it is hard to use Weil-Petersson gradients to do this in a way that produces a deformation retract onto $\partial \Sub_S(G)$, mostly because we have to control the base vectors of our orbifolds.

Briefly, the idea behind our proof is as follows. If $\systole(X)<\epsilon$, we push $(X,v)$  to the boundary $\partial \Sub_S(G)$ by pinching all short curves, while if $v$ lies in a cusp neighborhood, we leave $X$ alone and just flow $v$ straight  out the cusp, in which case the groups $\Gamma(X,v)$ limit to  a conjugate of $N$. One needs to do this in a canonical way that does not depend on any marking for $X$ in order for the resulting deformation retraction to be continuous; this is why Harer and Ji-Wolpert use Weil Petersson gradients above. The tool we use is \emph{conformal grafting}, see e.g.\ \cites{mcmullen,bourque}, in which length $L$ Euclidean annuli are inserted along the short curves in $X$, and then the resulting piecewise hyperbolic/Euclidean orbifold $X_L$ is uniformized to give a hyperbolic $2$-orbifold $X_L^{hyp}$, see \S \ref{sec: graftingsec}. It is well known that as $L\to \infty$, the lengths of the corresponding curves on $X_L^{hyp}$ tend to zero. The advantage of grafting for us is that the operation of inserting a Euclidean annulus is explicit enough that one can easily convert a parametrization for $X$ into a parametrization for $X_L$. The disadvantage is that uniformization is somewhat opaque, so it takes some work to identify the appropriate limits in $\partial \Sub_S(G)$, which depend on the hyperbolic geometry of $X_L^{hyp}$, not just on $X_L$.  

\medskip

Our main motivation in proving Theorem \ref{thm: deformation retract thm} above was to better understand the topology of the closure of $\Sub_S(G) \subset \Sub(G)$. In particular, we use it to prove the following.

\begin{named}{Theorem \ref{thm: vankampen}}[Informal version] Suppose that $S$ is not one of finitely many low-complexity exceptional orbifolds. Then the closure $\overline {\Sub_S(G)}$ is simply connected.
\end{named}

Since $\Sub_S(G)$ is a fiber orbibundle over $\mathcal M(S)$ with regular fiber $T^1 S$, the fundamental group of $\Sub_S(G)$ is an extension of $\pi_1 T^1 S$ by the mapping class group $\mathrm{Mod}(S)$, see Remark \ref{rem: repsremark}. Essentially, the reason that all loops in $\Sub_S(G)$ die in the closure is that the mapping class group is generated by Dehn twists and half Dehn twists, and by pinching the associated twisting curves one can homotope the corresponding loops in $\Sub_S(G)$ to points in $\partial \Sub_S(G)$. 

Formally, the proof of Theorem \ref{thm: vankampen} is by induction. The space $M := \overline{\Sub_S(G)}$  is the union of certain elementary groups and certain spaces $\Sub_T(G)$, where $T$ embeds in $S$. For each $\cplx$, let $M_\cplx \subset M$ be the subset of all elementary groups and all $\Sub_T(G)$ where $T$ has complexity at most $\cplx$. Using induction, we \emph{essentially}\footnote{This sketch contains lies. It is not a priori obvious that all the $M_\cplx$ are simply connected, and indeed $M_0$ may not be, so the proof is a little bit more roundabout than what is sketched here. But what is written above does capture the main ideas of the argument.} prove that for every $\cplx$ between $0$ and the complexity of $S$, the space $M_\cplx$ is simply connected. In the induction step, for every $T$ with complexity $\cplx$, we attach  $\Sub_T(G)$ to $M_{\cplx-1}$ along $\partial \Sub_T(G)$. By Theorem~\ref{thm: deformation retract thm}, $\partial \Sub_T(G)$ has a neighborhood deformation retract, so Van Kampen's Theorem applies and we can calculate $\pi_1 M_\cplx$ in terms of the groups $\pi_1 \Sub_T(G)$ and $\pi_1 M_{\cplx-1}=1$. The observation about Dehn twists above implies that all loops in $\Sub_T(G)$ are nullhomotopic in its closure, so $\pi_1 M_\cplx =1$ as well.

In the exceptional cases excluded in Theorem \ref{thm: vankampen}, the fundamental group of $\overline {\Sub_S(G)} $ may be nontrivial. For instance, in Example \ref{ex: sphere with 3 points example} we see that when $S$ is a sphere with three cone points, the space $\Sub_S(G)$ is compact, and hence is its own closure. Moreover, $$\Sub_S(G)\cong T^1 (S)/\mathrm{Isom}(S),$$ so $\pi_1 \overline{\Sub_S(G)}$ is a cyclic extension of the orbifold fundamental group $\pi_1 S$. When some of the cone points are replaced by cusps, $\Sub_S(G)$ becomes noncompact, and taking the closure in $\Sub(G)$ amounts to performing a certain Dehn filling. Analyzing the filling slope carefully, we prove:

\begin{named}{Proposition \ref{prop: lensprop}}[Informal version] If $S$ is a sphere with a total of $3$ cusps and cone points, not all of which are cone points, then $\overline{\Sub_S(G)}$ is homeomorphic to a (usually nontrivial) lens space.
\end{named}

Finally, when $S$ is a sphere with four cone points of distinct orders, the group $\pi_1 \Sub_S(G)$ contains a nonabelian free group. This is because $\Sub_S(G)$ has 3 ends (see Proposition \ref{prop: oneend}) while $\partial \Sub_S(G)$ is connected. It is not quite as easy to calculate $\pi_1$ completely in this case. Indeed, our proof of Theorem \ref{thm: vankampen} only gives a normal generating set for $\pi_1$. But we believe that with sufficient effort one could probably calculate $\pi_1 \overline {\Sub_S(G)}$ for all finite type hyperbolizable orbifolds $S$ without using arguments radically different from those in this paper.

\subsection{Continuity of conformal grafting in the smooth topology}
 
As mentioned above, our proof of Theorem \ref{thm: deformation retract thm} involves pinching short curves via  \emph{conformal grafting}, wherein a length $L$ Euclidean annulus is inserted along a geodesic $\gamma$ in a vectored hyperbolic orbifold $(X,v)$, and then the resulting conformal structure is uniformized to give a hyperbolic orbifold $(X_L^{hyp},v^{hyp})$, see see e.g.\ \cites{mcmullen,bourque}.  It is essential in our arguments that the resulting vectored orbifold varies continuously in the smooth topology with respect to $L$ and $(X,v)$. 

Now, Tan \cites{tanigawa1997grafting} proved in 1997 that length $L$ grafting along a simple closed curve in a closed surface $S$ gives a homeomorphism $T(S)\longrightarrow T(S)$ of Teichm\"uller space. The type of continuity statement we need is a bit stronger than this, since we need to keep track of base vector and (more importantly) understand what happens to $X_L^{hyp}$ in the smooth topology when $L\to \infty$. Borque \cite[Proposition 1.1]{bourque} proved something close to what we need, showing that when a surface $X$ is fixed and $L\to \infty$, the surfaces $X_L^{hyp}$ converge to the cusped surface obtained from $X$ by `infinite length' conformal grafting. With some awkwardness, in the current paper it would be sufficient to use a continuity result for grafting that is roughly similar in scope to Borque's work. However, in the sequel to this paper we will need a stronger result. So, we prove:

\begin{named}{Theorem \ref{thm: contgraft}}[Weaker informal version]
Suppose that $(X^i,v^i)\to (X^\infty,v^\infty)$ is a smoothly convergent sequence of vectored hyperbolic $2$-orbifolds, $L^i \to L^\infty$, and we create orbifolds $((X_{L^i}^i)^{hyp}, (v^i)^{hyp})$ by doing length $L^i$-grafting along a sequence $\gamma^i\subset X^i$ of curves that converge to a curve  $\gamma^\infty \subset X^\infty$. Then $$((X_{L^i}^i)^{hyp}, (v^i)^{hyp}) \to ((X_{L^\infty})^{hyp}, (v^\infty)^{hyp})$$ in the smooth topology.
\end{named}

Here, one novelty of our theorem is that the orbifolds $X^i$ may have infinite type. The statement of Theorem \ref{thm: contgraft} given in the body of the paper is stronger than the one above: it allows $\gamma^i$ to be a multicurve and $L^i$ to depend on the particular component, it allows there to be additional grafting along bounded length curves in the $X^i$ whose distances to the base vector go to infinity, and it also characterizes the limits if the $v^i$ are inserted in the Euclidean annuli. When $L^i\to \infty$ and the $v^i$ are taken deeper and deeper in such annuli, the limits are taken in the Chabauty topology and are certain (precisely identified) elementary subgroups of $G$, see also Proposition \ref{prop: limits in collars}.

There are two technical results that are used in the proof of Theorem \ref{thm: contgraft}. Both of them have to do with the fact that grafting produces an orbifold with a piecewise hyperbolic/Euclidean structure, which is then uniformized to be hyperbolic, and the last uniformization step is geometrically opaque. First, we construct uniform lower bounds for the hyperbolic distance \emph{(after uniformization!)} across any grafting cylinder.

\begin{named}{Proposition \ref{prop: modelprop}}[Weaker informal version]
Suppose that $X$ is a hyperbolic $2$-orbifold, $\gamma \subset X$ is a simple closed curve, and a Euclidean cylinder $C$ of length $L$ is grafted along $\gamma$ to create a hyperbolic $2$-orbifold $X_L^{hyp}$. Then the hyperbolic distance across $C \subset X_L^{hyp}$ is at least some explicit function of $L$ and $\ell(\gamma)$. 
\end{named}

This is a uniform version of a lemma of Bourque \cite[Lemma 4.1]{bourque}, whose estimates depend on $X$, not just on $L$ and $\ell(\gamma)$. See also the proof of Proposition 4.1 in Hensel \cite{hensel}, which includes some related estimates in the case that $\gamma$ is short. Our proof of Proposition \ref{prop: modelprop} is simpler than both of their arguments, and the full statement given in the body of the paper is much stronger. First, the statement we give does not only concern the total hyperbolic distance across $C$; rather, it gives a lower bound for the hyperbolic distance from a point in $C$ to $\partial C$, and even allows one to build a good bilipschitz model for the hyperbolic metric on $C \subset X_L^{hyp}$. Also, the way the statement is worded above, the estimates get worse as $\ell(\gamma)\to 0$, but in the real statement of the proposition we replace $C$ with its union with the `standard collar' around the geodesic $\gamma \subset X$ and give estimates that depend only on $L$ and an upper bound for $\ell(\gamma)$. 

Our second technical result on grafting produces lower bounds on  distances in $X_L^{hyp}$.

\begin{named}{Proposition \ref{prop: globalesstimates}}[Weaker informal version]
Given $\LLL$, there is a homeomorphism
$\eta : [0,\infty] \longrightarrow [0,\infty], $
 with $\eta(\infty)=\infty$ as follows. Suppose that $X$  is a hyperbolic $2$-orbifold, $\gamma$ is a simple closed geodesic on $X$ with $\ell(c) \le \LLL$, and $L \geq 0$. Let $X_L^{hyp}$ be the orbifold obtained from $X$ via  length $L$ grafting along $\gamma$. Then for all $x,y\in X$ that lie outside the standard collar of $\gamma \subset X$, we have
$$d_{(X_L)^{hyp}}\big (\iota(x),\iota(y)\big ) \geq \eta\big (d_X(x,y)\big ).$$
\end{named}

Here, the precise statement of Proposition \ref{prop: globalesstimates} mainly differs from the above in that we are allowed to graft simultaneously along a possibly infinite collection of curves $\gamma$. Note that the hyperbolic metric on  $X_L^{hyp}$ is bounded above by the piecewise hyperbolic/Euclidean metric one obtains by gluing in the Euclidean grafting cylinder to $X$, see for instance Tanigawa \cite[\S 2.1]{tanigawa1997grafting}, so it is easy to produce \emph{upper bounds} on hyperbolic distances in $X_L^{hyp}$. To produce the lower bounds in Proposition \ref{prop: globalesstimates}, we construct an explicit quasiconformal model
$Y$ for $X_L^{hyp}$ by using Fenchel-Nielsen coordinates to shrink the length of $\gamma$ by an appropriate amount. A $K$-quasiconformal map $f$ distorts hyperbolic distances at most according to an inequality like $$d(f(x),f(y))\leq \eta(d(x,y)),$$ where $\eta$ is some function depending on $K$, so the main point of the proposition is to build the quasiconformal model $Y$ in such a way that distances in it can be easily related to distances in $X$. To do this, we extend $\gamma$ to a pants decomposition and define an explicit map $X \longrightarrow Y$ on each pants. One has to be careful doing this, since the distortion of the constructed map cannot depend on the length of the pants curves, which may be extremely large. Indeed, part of this argument involves a detailed and technical manipulation of right angled hexagons, which we relegate to an appendix.

Given the two propositions, here is the main idea behind the proof of Theorem \ref{thm: contgraft}. If $p^\infty \in X^\infty$ is the basepoint of $v^\infty$, and some huge $R>>0$ is fixed, then for large $i$ the smooth convergence  $(X^i,v^i) \longrightarrow (X^\infty,v^\infty)$ gives an almost isometric embedding $$\psi: B(p^\infty,R) \hookrightarrow X^i.$$ This $\psi $ can be extended across the grafting annuli in $X^i$ and $X^\infty$ to give a map $\tilde \psi $ from a large compact subset of $(X^\infty_L)^{hyp}$ into $(X^i_L)^{hyp}$. It is not obvious that $\tilde \psi$ is almost isometric, since the relevant hyperbolic metrics on the domain and codomain are produced using uniformization. However, $\tilde \psi$ is almost conformal, and using the distance estimates from the two propositions one can show that the image of $\psi$ covers a large ball around the base vector in $(X^i_L)^{hyp}$. (Here, Proposition \ref{prop: globalesstimates} applies except when considering points outside the standard collar around $\gamma$, and for points inside the collar one can use the stronger version of Proposition \ref{prop: modelprop}\footnote{Proposition \ref{prop: modelprop} also allows one to control the Chabauty limits one gets when the basevectors $v^i$ are taken deep in the grafting cylinders. See Proposition \ref{prop: limits in collars} and the second half of the statement of Theorem \ref{thm: contgraft} given in the body of the paper.} given in the body of the paper.) But by standard compactness theorems about quasiconformal maps on the disc, see \S \ref{sec: convergence section}, smooth convergence of a sequence of vectored hyperbolic orbifolds can be verified by producing a sequence of almost conformal maps from large compact subsets of the limit to large compact subsets of the approximates. So, Theorem \ref{thm: contgraft} follows.

\subsection{Outline}
In \S \ref{sec: sub elementary}, we describe the space of elementary subgroups of $G$. Section \ref{sec: orbifolds} contains the necessary background on hyperbolic $2$-orbifolds, including information on orbifold basics, uniformization, geodesics, pants decompositions, Teichm\"uller space, Fenchel-Nielsen coordinates, moduli space and its ends, the Thick-Thin decomposition, the correspondences between discrete groups and vectored orbifolds, and between the Chabauty topology and the smooth topology, as well as some technical results describing smooth convergence in terms of quasi-conformal maps, and on obtaining well behaved maps between smoothly-nearby vectored orbifolds. In section \ref{sec: finitevolumesec}, we discuss the spaces $\Sub_S(G)$ defined above, and study the boundary $\partial \Sub_S(G)$ and the topology of the closure, culminating in the proof of Theorem \ref{thm: vankampen}. Section \ref{sec: graftingsec} contains all our results on grafting, and section \ref{sec: retractsec} is devoted entirely to the proof of Theorem \ref{thm: deformation retract thm}. Finally, we include an appendix containing some work on hyperbolic right angled hexagons that is used in \S \ref{sec: graftingsec}.

\subsection{Acknowledgements}

 The authors would like to thank Jonah Gaster, Juan Souto and Martin Bridgeman for helpful conversations. I.B. was partially supported by NSF grant DMS-1611851 and CAREER Award DMS-1654114.
N.L was supported by the Israel Science Foundation (grant no. 1562/19), and by the German-Israeli Foundation for Scientific Research and Development. A.L. was partially supported by ISF grant 704/08.

\section{The space of elementary  subgroups}\label{sec: sub elementary}

Let $G=\PSL_2(\RR)$, and let $\Sub(G)$ be its Chabauty space. In this brief expository section, we describe the topology of the subspace of $\Sub(G)$ consisting of all closed, elementary subgroups of $G$. 

The group $G$ acts by fractional linear transformations on the upper half plane $\HH^2 \subset \CC$. When $\HH^2$ is equipped with the hyperbolic metric $ds^2 = 1/y^2 (dx^2 + dy^2)$, this  action induces an isomorphism of $G$ with the group of orientation preserving isometries of $\HH^2$.  The action of $G$ extends continuously to the ideal boundary $\partial \HH^2 := \RR \cup \infty$.  Nontrivial elements of $G$ can then be classified by their fixed points in $\HH^2 \cup \partial \HH^2$: a \emph{rotation} fixes a point $p\in \HH^2$, a \emph{hyperbolic type} isometry  has two fixed points in $\partial \HH^2$ and translates along the geodesic axis that connects them, and a \emph{parabolic} isometry has unique fixed point $\xi \in \partial \HH^2$ and leaves invariant all the horocycles tangent to $\xi$. See \cite{Benedettilectures} for details.

A subgroup $\Gamma < G$  is called \emph {elementary}  it has a finite orbit in $\HH^2 \cup \partial \HH^2.$ The following proposition is a well known consequence of the classification of isometries.
 
\begin{proposition}\label{prop: classification of elem subgroups} The closed, elementary subgroups of $G$ are as follows.
	\begin{enumerate}
		\item  the trivial group.
		\item The compact group $K(p)$ of all hyperbolic rotations around some $p \in \HH^2$, and the finite cyclic subgroups ${\bf k}(p,2\pi/n)$ generated by a $2\pi/n$-rotation around $p\in \HH^2 $ for $2\le n\in\mathbb N$.  Any such subgroup is conjugate to either $\SO(2) \subset \PSL_2(\RR)$ or one of its finite subgroups.
		\item  The group $A(\alpha)$ of all  hyperbolic type isometries that translate along some geodesic axis $\alpha \subset \HH^2$, and the infinite cyclic subgroups ${\bf a}(\alpha,t) \subset A(\alpha)$ consisting of isometries that translate by multiples of $t \in \RR$. When $\alpha$ is the imaginary axis in $\HH^2$,  the subgroup $A(\alpha)$ is the subgroup of $G$ consisting of diagonal matrices with determinant $1$.
		\item  The group $A'(\alpha) \cong A(\alpha) \rtimes \ZZ/2\ZZ$ of all orientation preserving isometries that  leave invariant an axis $\alpha \subset \HH^2$, and the infinite dihedral subgroups ${\bf a'}(\alpha,t,p) \subset A'(\alpha)$ generated by a $\pi$-rotation around a point $p \in \alpha$ and a hyperbolic type isometry that translates $\alpha $ by $t$. Note that in the latter case, $p$ is only well defined up to translation by a multiple of $t$. When defining $A'(\alpha)$, one can take  $\ZZ /2\ZZ$ to be generated by any $\pi$-rotation around a point $p$ on $\alpha$.
		\item The group $N(\xi)$ of all parabolic type isometries fixing some $\xi \in \partial \HH^2$ and its infinite cyclic subgroups.  Such subgroups are conjugate to either the upper uni-triangular subgroup of $G$ or one of its infinite cyclic subgroups.
		\item The group $B(\xi)$  of all  orientation preserving isometries fixing some $\xi \in \partial \HH^2$. Here, $B(\xi) \cong N(\xi) \rtimes \RR$ where the $\RR$ factor can be identified with $A(\alpha)$ for any $\alpha$ that has an endpoint at $\xi$. Any such group is conjugate to the Borel subgroup of upper triangular matrices of $G$.
		\item The subgroup ${\bf b}(\xi,t) \subset B(\xi)$ of all parabolic isometries fixing $\xi$ and all hyperbolic type isometries fixing $\xi$ that have translation length a multiple of $t$. Here, ${\bf b}(\xi,t) \cong N(\xi) \rtimes \ZZ$, where the $\ZZ$ can be identified with the infinite cyclic group generated by any hyperbolic type isometry fixing $\xi$ with translation length $t$.
	\end{enumerate} 
\end{proposition}

 \begin{proof}
 For discrete groups, this classification is given in Theorem 2.4.3 of Katok \cite{katok1992fuchsian}. Here is the more general proof, for completeness.
 
Suppose that $H $ is elementary. If all of the elements of $H $ are elliptic, we are in case 2 by \cite[Theorem  2.4.1]{katok1992fuchsian}. So, we may assume that there is at least one element $h\in H$ that is not elliptic. If our finite $H$-orbit has at least three points in it, it cannot be that all three of these points are fixed points of $h$, which is a contradiction since $h\in H$ and all $h$-orbits are infinite except for the fixed points. A similar argument shows that $H$ cannot have an orbit of two points inside of $\HH^2$. So, either $H$ has a global fixed point in $\partial \HH^2$, or it leaves invariant a pair of points in $\partial \HH^2$.

  Suppose $H$ fixes $\xi \in \partial \HH^2$. Here, one can see that the full stabilizer $B(\xi) \cong N(\xi) \rtimes \RR$ algebraically by conjugating $\xi$ to $\infty,$ in which case $B(\xi)$ becomes the upper triangular matrices, $N(\xi)$ is the upper unitriangular matrices, and the $\RR$ is the diagonal matrices. If $H$ is contained in $N(\xi)$, it is as in case 5 above. If $H \cap N(\xi) = id$, then $H$ is purely hyperbolic and isomorphic to a closed subgroup of $\RR$, so it is as in case 3 above. So, we can assume that $H$ intersects $N(\xi)$ nontrivially, but is not contained in it.  In this case, there is some $h \in H \setminus N(\xi)$ that is a hyperbolic type isometry whose attracting fixed point is $\xi$, and there is some $g \in H \cap N(\xi) \setminus 1$. The sequence of conjugates $h^{-1}gh$ converges nontrivially to $id$ in $N(\xi)$, so $H \cap N(\xi) $ is nonempty, non-discrete and closed. Since $N(\xi) \cong \RR$, it follows that $H \supseteq N(\xi)$, and we are in cases 7 or 6 depending on whether the projection of $H$ to $B(\xi)/N(\xi)$ is infinite cyclic or the entire group.
  
The only remaining case is when $H$ leaves invariant $\{a,b\} \subset \partial \HH^2$. The geodesic $\alpha$ joining $a,b$ is then invariant under $H$, so $H < A'(\alpha)$ from case 4, and we are done.
  \end{proof}

Let $\Sub_{elem}(G) \subset \Sub(G)$ be the subspace consisting of all closed, elementary subgroups. We now describe all possible Chabauty limits in $\Sub_{elem}(G)$, following work of Baik-Clavier \cite{baikclavier}, who studied the closure of the space of cyclic subgroups.

% \begin{tikzcd} 
% & & B(\nu)  && \\
% \textcolor{red}{ {\bf a'}(\alpha,t,p)} \arrow[rrddd, red, bend right=40,swap, "t \to \infty \ and \ p \to \partial \HH^2"] \arrow[rdddd, dotted, bend right = 30, "(1)",swap] \arrow[r, red,  bend left=40, "t \to 0 ", "\alpha \to \beta"'] 
% &  \textcolor{red}{A'(\beta)} \arrow[rd ,red, "\beta \to \nu",swap]
% & {\bf b}(\xi,t) \arrow[u, "t \to 0","\xi \to \nu"',swap] \arrow[d, "t\to \infty", "\xi \to \nu"'] 
% & \textcolor{red}{A(\beta)}\arrow[ld,red,"\beta \to \nu"]
% & \textcolor{red}{{\bf a}(\alpha,t)} \arrow[l,red,  bend right=40,swap,"t\to 0", "\alpha \to \beta"'] \arrow[llddd, bend left =40,red,"\alpha \to \partial \HH^2 \ or \ t\to \infty"]\\
% & & \textcolor{red}{N(\nu)} \\
% & 
% & \langle \gamma \rangle \subset N(\xi)  \arrow[u,red,swap,dotted,"(4)"] \arrow[d,red, dotted,"(5)"]&& \\
% & & \textcolor{red}{1}&& \\
% \\
% & {\bf k}(p',\theta) \arrow[r, dotted,"(2)"] \arrow[u,"\theta \to 0", "p'\to q"'] \arrow[rd,dotted,"(3)"] & & K(q) \arrow[r,"q \to \xi",swap] \\
% \end{tikzcd}
% \end{center}
\begin{figure}[ht]
    \centering
    \includegraphics{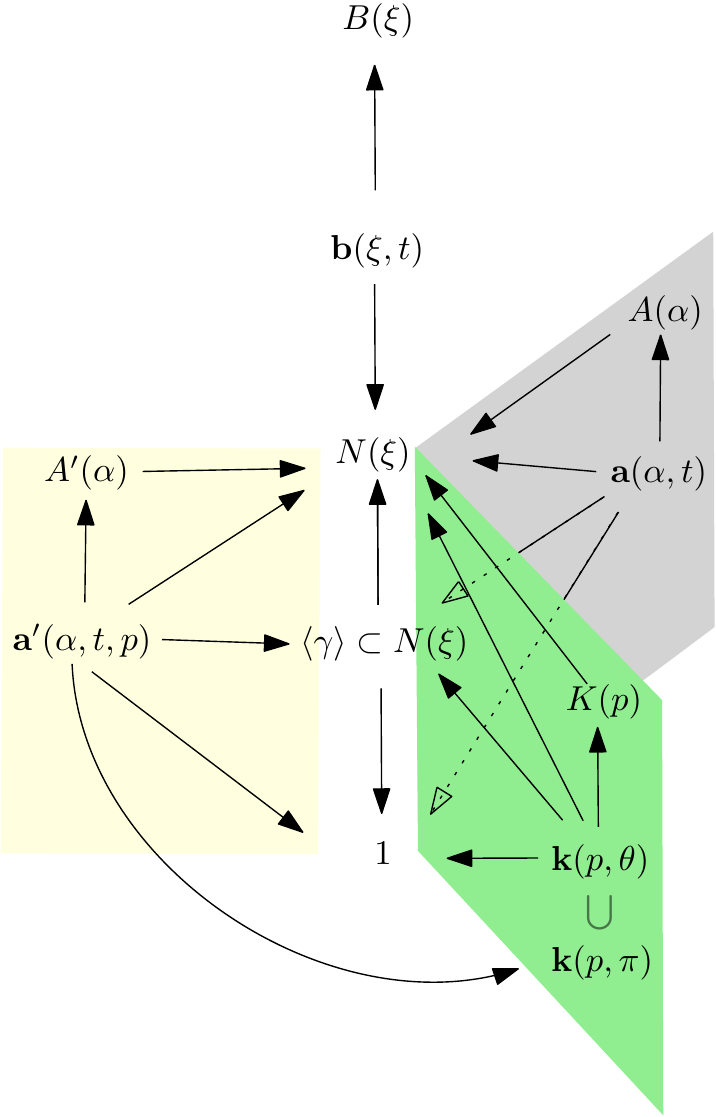}
    \caption{A schematic indicating the stratification of $\Sub_{elem}(G)$ by spaces of subgroups of the types listed in Proposition \ref{prop: classification of elem subgroups}. An arrow $A\longrightarrow B$ indicates that there is a sequence of subgroups of the form $A$ that limits onto a (in fact, any) subgroup of the form $B$.}
    \label{fig: heartsurgery}
\end{figure}

\begin{proposition}[The space of elementary subgroups]\label{prop: subspace of elem subgroups}
	The space $\Sub_{elem}(G)$ of closed, elementary subgroups of $G$ is a union of the spaces of subgroups of the types listed in Proposition \ref{prop: classification of elem subgroups}, which have the following topological types, and are glued together in the following ways.
	\begin{enumerate}
	    \item The spaces of all subgroups $B(\xi)$ and all subgroups $N(\xi)$, where $\xi \in \partial \HH^2$, are disjoint embedded circles in $\Sub(G).$
	    \item The space of all $\mathbf{b}(\xi,t)$, where $\xi\in \partial \HH^2, t\in \RR_+$, is an open cylinder. As $t \to 0$, the cylinder limits onto the circle of all $B(\xi)$. As $t \to \infty$, the cylinder limits onto the circle of all $N(\xi)$.
	    \item The space of all infinite cyclic parabolic subgroups, i.e.\ all the $\langle \gamma \rangle  \subset N(\xi)$ for $\xi \in \partial \HH^2$, is an open cylinder. On one end, it limits onto the circle of all $N(\xi)$, and the other limits to $1 \in \Sub(G)$. 
	    \item The space of all $A(\alpha)$ is homomorphic to the space $\mathcal G$ of all geodesics $\alpha \subset \HH^2$, which is homeomorphic to the symmetric product $\big ((\partial \HH^2)^2 \setminus \Delta \big )/(\xi,\xi') \sim (\xi',\xi)$, an open M\"obius band. The space of all $A'(\alpha)$ is  homeomorphic to $\mathcal G$ in the same way.
	    \item The space of all $K(p)$, where $p\in \HH^2$, is an open disc with boundary the circle $N(\xi)$. 
	    \item The space of all finite cyclic groups ${\bf k}(p,\theta)$ is homeomorphic to $\HH^2 \times \NN$ via the map $$(p,n) \mapsto {\bf k}(p,2\pi/n).$$
	    Moreover, if we abusively set ${\bf k}(p,2\pi/\infty) := K(p)$, the corresponding map $$\HH^2 \times (\NN \cup \{\infty\}) \longrightarrow \Sub(G))$$ is an embedding. This infinite union of discs in $\Sub(G)$ limits onto the set of all $N(\xi)$, the set of all infinite cyclic parabolic subgroups, and the identity. Specifically, if $p\in \HH^2$ is any fixed point, a sequence $\langle \gamma_i \rangle$ of finite cyclic groups converges to $N(\xi),$ some cyclic group, or $1$ if the translation distance $d(p,\gamma_i(p))$ converges to $0$, some finite number, or infinity, respectively. See \S 6 of \cite{baikclavier} for a translation of this into the notation above: as $(p_i,n_i)$ varies, the limit of ${\bf k}(p_i,2\pi/n_i)$ depends on the comparison between the rate at which $n_i\to \infty$ and the rate at which $p_i \to \partial \HH^2$.
	    \item The space of all infinite cyclic hyperbolic type groups ${\bf a}(\alpha,t)$ is parameterized by $\mathcal G \times (0,\infty)$, where  $\mathcal G$ is the space of all geodesics in $\HH^2$, which is homomorphic to an open M\"obius band. Setting $${\bf a}(\alpha,0):=A(\alpha),$$ we get an embedding $\mathcal G \times [0,\infty) \hookrightarrow \Sub(G)$. The image of this embedding limits onto the set of all $N(\xi)$, the set of all cyclic parabolic groups, and the identity. The limit of a sequence of infinite cyclic groups $\langle \gamma_i \rangle$ depends on the translation length $d(\gamma_i(p),p)$ of a fixed point in $\HH^2$, in the same way as in the previous case. See again \S 6 in \cite{baikclavier} to understand how the limit of ${\bf a}(\alpha_i,t_i)$ depends on the rates at which $\alpha_i \to \partial \HH^2$ and $t_i\to 0.$
	    \item The space of all infinite dihedral groups ${\bf a'}(\alpha,t,p)$ is homeomorphic to the order $2$ quotient \[\overrightarrow{ \mathcal G}   \times (0,\infty) \times (\RR/\ZZ) \modsim,\]  where here $\overrightarrow{ \mathcal G} $ is the space of oriented geodesics in $\HH^2$, which is an open cylinder, and the equivalence relation is generated by the involution $(\alpha,t,s) \mapsto (-\alpha,t,-s)$ that reverses the orientation of the geodesic and negates $s$. Here, after fixing a point $q\in \HH^2$, the parameterization takes $(\alpha,t,s) \mapsto {\bf a'}(\alpha,t,p_s)$, where $p_s$ is the point on $\alpha$ that lies a distance of $ts$ from the projection of $q$ onto $\alpha$, in the direction of the chosen orientation. Setting ${\bf a'}(\alpha,0,p) :=  A'(\alpha)$ for every $s$ and thinking of $(0,\infty) \times (\RR/\ZZ) \cong \RR^2 \setminus 0$ via the polar coordinates $(r,\theta)\mapsto (r\cos(2\pi\theta),r\sin(2\pi\theta))$, the definition of the parameterization above extends to $t=0$, giving an embedding
	    \[\overrightarrow{ \mathcal G}   \times \RR^2 \modsim \  \hookrightarrow \Sub(G),\]
	    where the involution now reverses the orientation of $\alpha \in \overrightarrow{\mathcal G}$ and reflects $\RR^2$ through the $x$-axis. The image of this embedding limits onto $\mathcal N$ almost exactly as in the case 7. The only difference is that if $t_i\to \infty$, the groups ${\bf a'}(\alpha_i,t_i,p_i)$ will converge to the order $2$ rotation group $K(p,\pi)$ if the $p_i$ can be chosen\footnote{Recall that ${\bf a'}(\alpha,t,p) = {\bf a'}(\alpha,t,p')$ if $p'$ and $p$ differ by a translation along $\alpha $ by a multiple of $t$.} to converge to $p,$ and while ${\bf a'}(\alpha_i,t_i,p_i)\to 1$ if it is not possible to choose $p_i$ to lie in a compact subset of $\HH^2$. If the geodesics $\alpha_i$ exit $\partial \HH^2$, however, the limiting trichotomy for ${\bf a'}(\alpha_i,t_i,p_i)$ is exactly the same as it is for ${\bf a}(\alpha_i,t_i)$. In particular, the description of the limit in terms of the parameters $\alpha_i,t_i$ is exactly the same as that given in \S 6 of \cite{baikclavier}.
 	\end{enumerate}
\end{proposition}
\begin{proof}
We will leave any remaining details to the reader, as basically every sentence in the lengthy statement of the proposition is trivial to prove using the definition of the Chabauty topology. In that sense, the statement is the proof. The basic fact that one uses over and over is that limits of group elements in $G$ are simple to understand, and that the elements of the elementary groups above are explicitly defined. In particular, the fixed points in $\HH^2 \cup \partial \HH^2$ of elements of $G$ vary continuously with the element, and if $(g_i)$ is a sequence of hyperbolic or elliptic elements whose fixed points converge to $\xi \in \partial \HH^2$, then $(g_i)$ converges to a parabolic element $g\in N(\xi)$ exactly when the translation distances $d(g_i(p),p)\to d(g(p),p)$, for some (any) fixed $p\in \HH^2$. (Convergence of translation distance implies a subsequential limit of $(g_i)$ by Arzela-Ascoli, and any such limit must have $\xi$ as a fixed point and must translate $p$ by the limiting translation distance, so must be $g$.) This is the only thing that is needed to understand the way that the hyperbolic and elliptic groups limit onto the parabolic groups in 6-8 above. The more complicated mathematics surrounding this proposition is what is done in Baik-Clavier~\cite{baikclavier}, where they understand a bit better in terms of the given parametrizations how the pieces above fit together.
\end{proof}

%\tinytodo{There are a few issues with this second figure. $\partial a'$ and $\partial a$ should include $N$. The $a'$ groups also limit onto the order two finite subgroups, which isn't shown, and the picture makes it look like the $a'$ groups have the same topology as the $a$ groups, when in fact you have this extra dimension coming from the fact that you get to pick the point of the $\pi$-rotation. Also, some of the notation is different from what is used above.}

%\begin{figure}[htbp!]
%	\begin{center}
	%	\def\svgwidth{0.7\textwidth}
	%	\input{./figures/heart.pdf_tex}
	%\end{center}
%\end{figure}

Figure \ref{fig: sub_elem} attempts to illustrate the global topology of $\Sub_{\elem}(G),$ as described above.

\begin{figure}[h!]
    \centering
    \includegraphics[scale=.5]{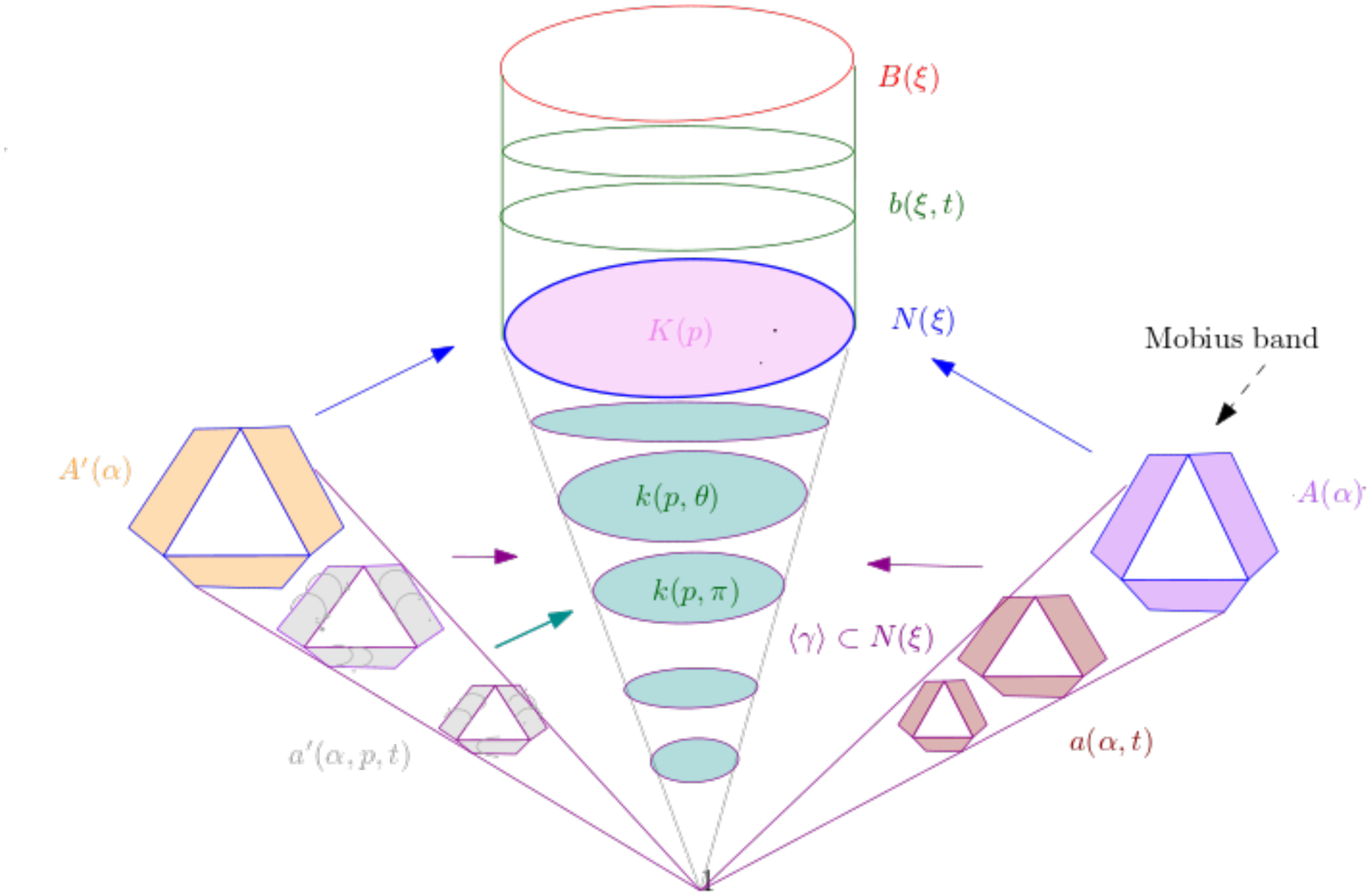}
    \caption{A schematic drawing of the topology of $\Sub_{elem}(G)$ by spaces of subgroups of the types listed in Proposition \ref{prop: subspace of elem subgroups}. For example, a blue arrow indicates the blue circle at the boundary of a Mobius band is glued to the blue $N(\xi)$.}
    \label{fig: sub_elem}
\end{figure}

The following is a quick consequence of the work above.

\begin{corollary}[Homotopy type of the space of elementary groups]\label{cor: homotopy type of sub elem}
$\Sub_{elem}(G)$ is homotopy equivalent to the subset of $\RR^3$ that is the union of the $2$-spheres $\bbS_n$, $n=2,3,\ldots,\infty$, where $\bbS_n$ has radius $1-1/n$ and is centered at $(1-1/n,0,0).$ In particular, $\Sub_{elem}(G)$ is simply connected.
\end{corollary}

Note that the subset of $\RR^3$ above is like a wedge sum of the spheres $\bbS_n$, since they all intersect at the origin, but as $n\to \infty$ the spheres $\bbS_n$ converge onto $\bbS_{\infty}$. The reader can compare this Corollary to the very similar Corollary 9.2 in Baik--Clavier \cite{baikclavier}\footnote{Note however that their assertion in the proof of Corollary 9.2 that $\mathbf C$ is homotopy equivalent to a wedge of spheres is not quite correct. They mean to talk about the space in the statement of Corollary \ref{cor: homotopy type of sub elem} instead.}.

\begin{proof}
The union of all the subgroups $B(\xi), {\bf b}(\xi,t), N(\xi), id$ is homeomorphic to a disk. Collapse this disk to the point $id$. The union $\mathcal K$ of all the groups $K(p)$, ${\bf k}(p,\theta)$, $id$ then collapses to a space homeomorphic to the subset $\cup_n \bbS_n\subset \RR^3$ mentioned in the  statement of the Corollary: the boundary of each disk $\HH^2 \times \{n\}$ collapses to a point, giving the sphere $\bbS_n$. The union $\mathcal A$ of all $A(\alpha), {\bf a}(\alpha,t), id$ is contractible (send $t\to \infty$ in the parametrization given in the proposition above), and this contraction extends to the closure in $\Sub_{elem}(G)$ after the collapse above. There is a similar deformation retraction of the union $\mathcal A'$ of all $A'(\alpha), {\bf a}'(\alpha,t,p), id$ onto the sphere ${\bf k}(p,\pi) \cup \{id\}$, which also extends to the closure after the collapse above. So, the whole space $\Sub_{elem}(G)$ is homotopy equivalent onto the collapsed $\mathcal K$, which is homotopy equivalent to $\cup_n \bbS_n$.
\end{proof}
\section{Background on Orbifolds}
\label{sec: orbifolds}
When $\Gamma < \PSL_2 (\RR)$  is discrete, the quotient $\Gamma \backslash \HH^2$ is most naturally interpreted as an \emph{orbifold}.   Since almost all of this paper is written using the language of orbifolds, we start by reviewing the definitions, both to orient the reader that mostly cares about manifolds, and to set our conventions. We mostly follow the approach in \cite{caramello2019introduction}, see also Thuston's notes \cite[Chapter 13]{Thurston} and Kleiner--Lott \cite{kleinerlott}. 

Let $S$ be a Hausdorff second countable topological space and let $\mathcal G$ be a pseudo-group of homeomorphisms of open subsets of some manifold $X$. An \emph{orbifold $(\mathcal G,X)$-atlas} on $S$ is a collection of open sets $\{U_i\}$ that covers $S$ and is closed under finite intersections, a collection of open sets $V_i \subset X$, each equipped with a $\mathcal G$-action $\Gamma_i \curvearrowright V_i$ %\todo{I don't think this is the arrow you mean. \\ IB: if you prefer a different arrow for action, feel free to change it everywhere.} 
by a finite group $\Gamma_i$, and a collection of maps $\phi_i$, called \emph{charts}, of the form
$$\begin{tikzcd}
V_i \arrow{rr} \arrow[bend left=30]{rrrr}{\phi_i} & & \Gamma_i \backslash V_i \arrow[rr] & & U_i,
\end{tikzcd}$$
where  the first map is the projection and the latter map is a homeomorphism,
 such that whenever $U_i \subset U_j$,  there is an injective homomorphism $h_{ij} : \Gamma_i \longrightarrow \Gamma_j$  and a $h_{ij}$-equivariant $C^r$-embedding $\psi_{ij} : V_i \longrightarrow V_j$, $\psi_{ij} \in \mathcal G$, such that the following diagram commutes:
$$
\begin{tikzcd}
	V_i \arrow{d}{\psi_{ij}} \arrow{rr}{\phi_i} & & U_i \arrow[hookrightarrow]{d} \\ V_j \arrow{rr}{\phi_j} & & U_j 
\end{tikzcd}
$$
We call any map $\psi_{ij}$  as above a \emph{transition map}  between the two charts. A \emph{$(\mathcal G,X)$-orbifold} with underlying space $S$ is obtained when we equip $S$ with a maximal such atlas. 

As an example, if $X$ is a $C^r$-manifold we can let $\mathcal H^r(X)$ be the pseudogroup of homeomorphisms between open subsets of $X$ that are $C^r$ with $C^r$ inverses. Any $(\mathcal H^r(X),X)$-orbifold is naturally a $(\mathcal H^r(\RR^n),\RR^n)$-orbifold; we call such orbifolds $C^r$-\emph{orbifolds} for short. Also, when $X = \RR^n$ or $\HH^n$, and $\mathcal G$ is the pseudogroup of homeomorphisms between open sets that are local isometries, we call $(\mathcal G,X)$-orbifolds \emph{Euclidean $n$-orbifolds} and \emph{hyperbolic $n$-orbifolds}, respectively.  An \emph{orbifold with boundary} is obtained just as above, except that $X$ is allowed to be a manifold with boundary. 
 In general, most orbifolds that one encounters in real life arise as \emph{global quotients} of manifolds: if $\Gamma$ acts  properly discontinuously by $\mathcal G$-diffeomorphisms on $X$, then $S := \Gamma \backslash X$ is naturally a $(\mathcal G,X)$-orbifold, see e.g.\ \cite[Proposition 13.2.1]{Thurston}.

If $S$ is an orbifold and $p\in S$, the \emph{local group} $G_p$ of $p$ is the stabilizer in $\Gamma_i$ of any point of $\phi_i^{-1}(p) \subset V_i$. Here, the isomorphism class of $G_p$ is well-defined, independent of the choice of point or chart. One says that $p$ is \emph{singular} if the local group $G_p$ is nontrivial, and $p$ is \emph{regular} otherwise. If $S,S'$ are two $(\mathcal G,X)$-orbifolds, a function $f : S \longrightarrow S'$ is \emph{represented by $\psi$ in local coordinates around $p\in S$} if there are charts $$\phi : V\longrightarrow U, \ \ \phi ' : V' \longrightarrow U',$$
 with $p\in U $ and $f(p) \in U'$ and associated group actions $\Gamma \curvearrowright V$ and $\Gamma' \curvearrowright V'$, a homomorphism $h : \Gamma  \longrightarrow \Gamma '$ and a $h$-equivariant map $\psi : V  \longrightarrow V'$ such that $\phi' \circ \psi = f \circ \phi$. We say $f$ is a $\mathcal G$-map if around every element of $S$ it is represented  in local coordinates by an element of $\mathcal G$. Note that $\mathcal G$-maps are  by definition local homeomorphisms. A \emph{$\mathcal G$-isomorphism} is just a $\mathcal G$-map with a $\mathcal G$-inverse. If $\mathcal G=C^0$, we call a $\mathcal G$-isomorphism a \emph{homeomorphism}, when $G = C^r, \ r>0$ we say \emph{$C^r$-diffeomorphism}, and when $\mathcal G$ is a pseudo-group of local isometries, we say \emph{isometry}. Note that all such maps preserve the orbifold structure; e.g.\ homeomorphisms of orbifolds are not just homeomorphisms of the underlying topological spaces.  Finally, a map between orbifolds is \emph{continuous}, or $C^r$, or \emph{holomorphic}, or \emph{isometric}, or \emph{Marxist}\footnote{Editor's note: Marxism is generally considered to be a global property, so should be deleted from this list.} if it is when represented in local coordinates.

Almost following \S 3.3 of \cite{caramello2019introduction}, fix $r\geq 0$ and recall that a surjective $C^r$ map $\pi : E \longrightarrow B$ of $C^r$-orbifolds is called a \emph {fiber orbibundle} if there is a third $C^r$-orbifold $F$ such that $B$ is covered by charts $\phi_i : V_i \longrightarrow U_i \subset B$ with associated $C^r$ actions $\Gamma_i \curvearrowright V_i$, and for each $i $ there is an action of $\Gamma_i$ on $F \times V_i$ that agrees with the given action on the second coordinate\footnote{In \cite{caramello2019introduction}, the action is required to be the product of an action on $F$ and the given one on $V_i$, but the more general version we give here is more natural, and is perhaps what the author meant to say.}, and a $C^r$-diffeomorphism $\Phi$ as in the commutative diagram below:
\[\begin{tikzcd}
\Gamma_i \backslash (F \times V_i) \arrow[rr,"\Phi_i"] \arrow[d] & & \pi^{-1}(U_i) \subset E \arrow[d] \\
\Gamma_i \backslash V_i \arrow[rr] & &  U_i \subset B
\end{tikzcd}
\]
Now any  $d$-orbifold $S$ of class $C^r$, where $r \geq 1$, has a \emph{tangent (orbi)bundle} $$\rho: TS \longrightarrow S,$$  where $TS$ is a $C^{r-1}$-orbifold with dimension $2d$ and the map $p$ is $C^{r-1}$. Each $S$-chart $\phi_i: V_i \longrightarrow U_i$ with associated group action $\Gamma_i \curvearrowright V_i$  determines a chart $d\phi_i : TV_i \longrightarrow TU_i$, where  the associated group action is $\Gamma_i \curvearrowright TV_i$, acting via the derivatives of the action on $V_i$.  Here, the fiber over $p\in U_i \subset S$ is $$TS_p:=\rho^{-1}(p) \cong G_p \backslash (TV_i)_{\phi_i^{-1}(p)},$$ called the \emph{tangent cone} at $p$. Here, $G_p \subset \Gamma_i$  is the local group at $p$, i.e.\ the $\Gamma_i$-stabilizer of $\phi^{-1}(p)$, and   $TS_p$ is not a vector space, but is the quotient of the vector space by a finite linear group. Sometimes, authors will refer to $(TV_i)_{\phi^{-1}(p)}$  equipped with the action of the local group as the \emph{tangent space} at $p$, in  contrast to the tangent cone defined above. (Note that we are using different notation here than in \cite{caramello2019introduction}.)
Finally, if $v : S \longrightarrow S'$  is a $C^1$-map of orbifolds, there is an associated derivative map $$df_p : TS_p \longrightarrow TS'_{v(p)}$$  defined  as usual by differentiating the map $f$ in coordinates.

An  \emph{orientation} on an orbifold $S$ is a choice of orientation for the domain of every chart $\phi : V \longrightarrow U \subset S$, such that all group actions and transition maps are orientation preserving. Similarly, a \emph{Riemannian metric} on $S$ is a choice of Riemannian metric on the domain of every chart, such that all group actions and transition maps preserve the metric. So, any hyperbolic or Euclidean orbifold comes equipped with a natural Riemannian metric.  A Riemannian metric on $S$ can be scaled by a function $S \longrightarrow \RR^+$, and an equivalence class of metrics up to scale is called a \emph{conformal class} of metrics, or a \emph{conformal structure} on $S$. See Chapter~4 of \cite{caramello2019introduction} for  an introduction to the Riemannian geometry of orbifolds.

If $S,S'$ are $(\mathcal G,X)$-orbifolds, a surjection $f: S' \longrightarrow S$ is a \emph{covering map} if for every $p\in S $ there is a chart $\phi : V \longrightarrow U$, with $V \subset X$ and $p\in U$, such that $f^{-1}(U)$ is a disjoint union of open sets $U' \subset S'$, and where there are charts $\phi' : V' \longrightarrow U' $ and an element $g : V \longrightarrow V'$ in $\mathcal G$ such that $\phi \circ g = f \circ \phi'$. For example, if $\Gamma$ acts properly discontinuously on $X$ then the quotient map $X \longrightarrow \Gamma \backslash X$ is a covering map. An orbifold is called \emph{good} if it is a quotient of a manifold by a properly discontinuous group, and \emph{bad} otherwise. Equivalently, good orbifolds are those whose \emph{orbifold universal cover} is a manifold. See e.g.\ \cite[\S 2.2]{kleinerlott} and \cite[\S 2.3]{caramello2019introduction} for details.

\subsection{Orientable $2$-orbifolds, and uniformization}
\label{sec: 2orbifolds}

Let $S$ be an orientable $2$-orbifold.  The local group of each singular point of $S$ is then cyclic. We call such singular points \emph{cone points}, and the \emph{order} of a cone point is the size of its local group.  Note that the underlying space of an orientable $2$-orbifold is always an  orientable surface. We sometimes write $T(m_1,\ldots,m_n)$ for an orbifold whose underlying space is homeomorphic to $T$, and which has cone points with orders $m_1,\ldots,m_n$.  However,  in general the orbifolds in this paper may be noncompact, with infinitely many cone points. We say $S$ has \emph{finite (topological) type} if it has the form $T(m_1,\ldots,m_n)$ for some finite type surface $T$, i.e.\ a surface $T$  obtained by removing finitely many points from a compact surface. Otherwise, $S$ has infinite type. Equivalently, $S$ has finite type exactly when its orbifold fundamental group $\pi_1 S$ is finitely generated. See \S 2.2 of \cite{caramello2019introduction} for more on orbifold fundamental groups. 

Almost  all orientable $2$-orbifolds $S$ are good: the only bad ones are the spheres with cone points $\bbS^2(m)$ and $\bbS^2(m,n), \ m\neq n$.  Good $S$ can be classified geometrically. Each is homeomorphic to a global quotient of exactly one of $\bbS^2$, $\RR^2$ or $\HH^2$, depending on its orbifold Euler characteristic, see e.g.\ \cite[\S 2.2]{kleinerlott} and \cite[\S 7.2.2]{primer}. Moreover, we have the following generalization of the geometric uniformization theorem for Riemann orbifolds.

\begin{theorem}[Orbifold Uniformization]\label{thm: uniformization}
Suppose $S$ is a good, smooth $2$-orbifold that is not covered by the $2$-sphere and fix a $C^0$-Riemannian metric $d$ on $S$ that is locally Lipschitz.  Then there is a  $C^0$-Riemannian metric $d'$ on $S$ in the conformal class of $d$ and a $C^1$-isometry from $(S,d')$ to some smooth Riemannian $2$-orbifold with constant curvature.  Moreover, this $d'$ is unique up to global scale.
\end{theorem}

 It will be important when talking about grafting below to have a version of the Uniformization Theorem that applies, like above, to  metrics that are only Lipschitz continuous.   Note that ideally one would like to just say that  the metric $d'$  produced in the theorem itself has `constant curvature'. But since $d'$ is only $C^0$,  one cannot take derivatives of it with respect to the original smooth structure.

\begin{proof}
 The Measurable Riemann Mapping Theorem, c.f.\ \cite[Proposition 4.8.12]{hubbard}, implies that there is a Riemann surface $R$ and a quasiconformal map $f: S \longrightarrow R$ that is conformal  when regarded as a map $(S,d) \longrightarrow R$. A result of Lehto \cite{lehto}  implies that a quasi-conformal map with locally Lipschitz dilatation is $C^1$, so by our assumption on $d$,  the map $f$ above is $C^1$. 
The universal cover $\tilde R$ is conformally equivalent to either  $\hat \CC$, $\CC$ or $\HH^2$ by the classical uniformization theorem, see e.g.\ \cite[Ch 1]{hubbard}, and by the assumption in the statement of the theorem, it is conformally equivalent to one of the latter two. In the hyperbolic case, the deck group $\Gamma$  acts conformally, and the only conformal automorphisms $\HH^2$ are isometries, c.f.\ \cite[Theorem A.3.3]{Benedettilectures}, so taking the quotient we get a conformal hyperbolic metric on $R$ in the conformal class of $d$, and we can let $d'$  be the pullback of this metric to $S$. The Euclidean case is similar, except that conformal automorphisms of $\CC$ are Euclidean similarities instead of isometries, see again \cite[Theorem A.3.3]{Benedettilectures}. However, all deck transformations of the orbifold cover $\tilde R \longrightarrow R$ are either finite order or fixed point free, and any such similarity of $\CC$ is an isometry, so the argument continues as before. In both cases, uniqueness of $d'$ up to scale follows from the fact that conformal automorphisms of $\CC$ and $\HH^2$ are similarities. 
\end{proof}

\subsection{Hyperbolic $2$-orbifolds}
We now restrict ourselves to the case of complete orientable hyperbolic $2$-orbifolds $X$, sometimes with geodesic boundary.  In the latter case, $X$ is an orientable $(\mathcal G, H)$-orbifold, where $H\subset \HH^2$  is a half plane and $\mathcal G$  is the pseudogroup of  homeomorphisms between open subsets of $H$ that are local isometries.  Below, we assume  a basic familiarity with hyperbolic geometry. For the uninitiated, \cite{Benedettilectures} is a good reference.

\subsubsection{Convex cores, ends, cusps, flares} Suppose that $X$  is as above, possibly with geodesic boundary. Then $X = \Gamma \backslash C$, for some group $\Gamma$ acting by isometries on a convex subset $C \subset \HH^2$. The \emph{limit set} $\Lambda(\Gamma) $ of $\Gamma$ is the set of accumulation points on $\partial \HH^2$ of any $\Gamma$-orbit in $C \subset \HH^2.$ Let $CH(\Gamma)\subset \HH^2$  be the convex hull of $\Lambda(\Gamma)$, i.e.\  the  intersection of all convex subsets of $C \subset \HH^2$ that limit onto $\Lambda(\Gamma) \subset \partial H^2$. The \emph{convex core\footnote{The reader may be unfamiliar with convex cores of orbifolds with boundary. Another way to think about this is to complete $X$ to an orbifold without boundary by  gluing flares to all geodesic boundary components and then taking the convex core of the result.}} $CC(X)$ of $X$ is the quotient $\Gamma \backslash CH(\Gamma)\subset X$. Note: 

The orbifold $X$ is called \emph{elementary} if the corresponding group $\Gamma$ is, see \S \ref{sec: sub elementary} for a discussion of elementary groups. Discrete, elementary groups are either trivial, a finite cyclic group of rotations, an infinite cyclic group of parabolic isometries, or an infinite cyclic or infinite dihedral group stabilizing a geodesic axis in $\HH^2$. So, an elementary orbifold $X$ is either $\HH^2$, a cone, an annulus, or a disc with two order $2$ cone points\footnote{These `generalized annuli' are discussed in \S \ref{sec: conformalannuli}.}. When $X$  is elementary, the convex core $CC(X)$ is either empty or is an admissible geodesic in $X$, in the sense defined after Lemma \ref{lem: geodesic representatives} below. Otherwise, $CC(X)$ is a $2$-suborbifold of $X$ with geodesic boundary. 

The components of $X \setminus CC(X)$  have two possible types: \emph{half-planes} and \emph{flares}. A half-plane is isometric to a half plane in $\HH^2$. A flare is isometric to the quotient of a half-plane $H$ by a hyperbolic type isometry with axis $\partial H$.  See \cite[\S 8]{Thurston} and \cite{Benedettilectures} for details  in the surface case. All the arguments easily generalize. 
 
 When $X$ has finite topological type, $CC(X)$ has finite volume and all its ends are cusps. Here, a \emph{cusp} is an end of $X$ that has a neighborhood isometric to an annulus $B / \langle \gamma\rangle $, where $B \subset \HH^2$ is an open horoball centered at $\xi \in \partial \HH^2$ and $\gamma$ is a parabolic isometry stabilizing $\xi$. So, an orbifold with finite topological type  has finitely many ends, all of which are either flares or cusps. In general, though, the space of ends of a hyperbolic $2$-orbifold can be a rather nasty Cantor set.

 \subsubsection{Curves, admissible homotopies and geodesics}
\label{sec: curves}
Suppose that $X$ is a hyperbolic orientable $2$-orbifold, possibly with geodesic boundary, and that $X_{reg} \subset X$ is the set of all regular points. A free homotopy $h_t : \bbS^1  \longrightarrow X,$ $t\in [0,1]$ is called \emph{admissible} if $h_t(\bbS^1) \subset X_{reg}$ for all $t\in (0,1).$ The following is Theorem 5.1 in \cite{tan2006generalizations}.

\begin{lemma}[Geodesic representatives]\label{lem: geodesic representatives}
Let $\gamma \subset X_{reg}$ be a simple closed curve. Then either
\begin{enumerate}
\item $\gamma$  bounds a disc in $X$ with a single cone point,
\item $\gamma$ bounds a cusp neighborhood of $X$ that is contained in $X_{reg}$,
    \item $\gamma$ is admissibly homotopic to a simple, closed geodesic in $X_{reg}$,
    \item $\gamma$ bounds a disc in $X$ with two order $2$ cone points, and is admissibly homotopic to a closed geodesic that twice traverses an arc connecting the two cone points.
\end{enumerate}
\end{lemma}

Geodesics on $X$ arising from cases 3 and 4 above will be called \emph{admissible}. We call a simple closed geodesic in $X_{reg}$, as in 3, a \emph{regular geodesic}, and the geodesics in 4 are called \emph{degenerate geodesics}.  Importantly, we will always consider  admissible geodesics as maps $\bbS^1 \longrightarrow X$. So, a degenerate geodesic $\gamma$ is a loop whose image is an arc, and the \emph{length} of $\gamma$, usually denoted $\ell(\gamma)$, is twice the length of the arc. Often, we will parameterize our geodesics via length preserving maps
$$\bbS^1_{\ell(\gamma)} \longrightarrow X,$$
 where $\bbS^1_{\ell(\gamma)}$ is a circle with length $\ell(\gamma)$.  Admissible geodesics  always minimize intersection numbers in their admissible homotopy classes, c.f.\ Proposition 2.3 in \cite{Fanoni}.

\subsubsection{Pants decompositions}\label{sec: pants sec}
Hyperbolic $2$-orbifolds admit (generalized) pants decompositions, which we recall here for the reader's convenience, following Fanoni \cite{Fanoni} and Basmajian-Saric \cite{basmajian2019geodesically}.  The main difference from the surface case is that we allow  boundary geodesics (or cusps) in a pair of pants to be replaced by cone points.

\begin{definition}[Pairs of pants] A \emph{generalized pair of pants} is a genus zero orientable hyperbolic $2$-orbifold $P$ with geodesic boundary such that the number of boundary components, the number of cusps, and the number of cone points all sum to $3$. 
\end{definition} 

 Generalized pairs of pants are exactly the compact orientable hyperbolic $2$-orbifolds with geodesic boundary in which all admissible geodesics are peripheral. So, one could hope to construct a pants decomposition of a general $X$ by splitting along a maximal collection of admissible geodesics. There are two subtleties associated with this, which we now describe. 
 
 First, some elements of a maximal collection of admissible geodesics in a hyperbolic $2$-orbifold may be degenerate. In this case, there will be an adjacent pair of pants, but only its interior will be embedded in $X$.  To formalize this, we define a \emph{generalized embedded pair of pants} in $X$ to be a locally isometric map $P \longrightarrow X$  from a generalized pair of pants  that is an embedding on $int(P)$, and where the boundary components of $P$ map onto admissible geodesics in $X$. So, each boundary component is either embedded or maps two to one onto an arc connecting two order two cone points of $X$.
 
 Second, when $X$ has infinite type a collection of maximal admissible geodesics in $X$ may not be locally finite, so may accumulate somewhere in $X$.  As a result, there can be regions in the complement of such a collection that are noncompact.  Indeed, any generalized embedded pair of pants must lie in the convex core $CC(X)$, so the union of all such pants misses any half-plane component of $X\setminus CC(X).$ 

Here is a general statement that addresses both subtleties.
 
 \begin{theorem}[Pants decompositions for infinite type orbifolds]\label{thm: pantsdecomps}
  Suppose that $X$ is a nonelementary orientable hyperbolic $2$-orbifold, possibly with geodesic boundary, and let $\mathcal L$ be the union of all boundary components of the convex core $CC(X)$ that are biinfinite geodesics, rather than closed geodesics.  Then $CC(X) \setminus \mathcal L$ is a union of generalized embedded pairs of pants glued along their boundaries. Moreover, if $\mathcal S$ is any collection of pairwise disjoint, pairwise nonhomotopic admissible geodesics on $X$, then the boundaries of these pants can be taken to include all elements of $\mathcal S$.
 \end{theorem}
 
 Note that as $X$ is homeomorphic to the interior of its convex core, it is a formal consequence of the statement above that any such $X$ admits a topological  decomposition into generalized pairs of pants, as long as one allows some pants to have missing boundary components corresponding to compact components of $\partial CC(X)$, and one still allows pants to be nonembedded on the boundary.  The proof follows from essentially the same arguments as in the surface case. See Baik--Kim \cite[Lemma 6.1]{baik2021complete}.

\subsubsection{Teichm\"uller space and Fenchel Nielsen coordinates}
\label{sec: fenchelnielsen}
Let $S$ be an oriented hyperbolizable $2$-orbifold with finite topological type, possibly with boundary. The \emph{Teichm\"uller space} of $S$ is defined as the set of equivalence classes 
\[ \mathcal T(S)  := \left \{ \ (X,h) \ \Big | \ \substack{X \text{ is a finite volume oriented hyperbolic $2$-orbifold with} \\ \text{geodesic boundary, }  h : S \longrightarrow X \text{ is an \ o.p.\ homeomorphism }} \right \} / \ \sim\]
where here $(X,h)\sim (X',h')$ if there is an orientation preserving (o.p.) isometry $f : X \longrightarrow X'$ such that $(h')^{-1} \circ f \circ h$ is homotopic to the identity map. There is a natural topology on $\mathcal T(S)$ that can be defined in terms of pointwise convergence of holonomy representations, quasiconformal distortion, geodesic length functions, or smooth convergence of metric tensors. The references \cites{hubbard,primer} are standard when $S $ is a surface. For $2$-orbifolds, the reader can consult \cite[Chapter 13]{Thurston}, \cite[Chapter 7]{choi2012geometric},  or \cite{ohshika1987teichmuller}. The orbifold references deal only with \emph{compact} $2$-orbifolds with geodesic boundary, but all of the work there generalizes to the finite type setting if we allow boundary lengths to be zero. 

As for surfaces, there is a description of $\mathcal T(S)$ via Fenchel-Nielsen coordinates. Hyperbolize $S$ and use Theorem \ref{thm: pantsdecomps} to choose a pants decomposition, which we regard as a union $C$ of disjoint admissible geodesics that decomposes $S$ into generalized pairs of pants, and which includes all the boundary components of $S$. If $S$ has genus $g$, and $n$ punctures, and $k$ cone points, and $m$ boundary components, an Euler characteristic count shows that $C$ has $3g-3+n+m +k$ nonperipheral components and $m$ boundary components. For every component $\gamma \subset C$, we have a \emph{length coordinate} $\ell_\gamma : \mathcal T(S) \longrightarrow \RR_{>0},$
where $\ell_\gamma([(X,h)])$ is the length of the admissible geodesic homotopic to $h(\gamma) \subset X$. And for each nonperipheral $\gamma \subset C$, there is a \emph{twist coordinate} $\tau_\gamma : \mathcal T(S) \longrightarrow \RR $ defined by using a transversal curve to measure the amount of `twist' present in the way the two adjacent pants in $X$ are glued, exactly as in \cite[\S 10.6.1]{primer}. (Note that there is even a twist parameter associated to a pants curve that is a degenerate geodesic in $S$.) We then have:

\begin{theorem}[Fenchel-Nielsen coordinates]\label{thm: fenchelnielsen} The map 
\[\mathcal T(S) \longrightarrow \RR_{>0}^{3g-3+n+2m+k} \times \RR^{3g-3+n+m+k}\]
given by the length and twist coordinates is a homeomorphism. In particular, $\mathcal T(S)$ is homeomorphic to an open disk with dimension $6g+2n+2k+3m-6$.
\end{theorem}

As a corollary, if we look at a subspace of Teichm\"uller space consisting of hyperbolic structures in which the boundary lengths are fixed, we get an open disk with dimension $6g+2n+2k+2m-6.$ For a proof of Theorem \ref{thm: fenchelnielsen}, see \cite[Corollary 13.3.7]{Thurston} or \cite[Theorem 7.0.1]{choi2012geometric}, altered to allow cusps. See also \cite{maskit2001matrices} for an algebraic approach to the general case.

\begin{remark}
In the previous section we saw that generalized pants decompositions exist even for infinite type orbifolds. One can also develop Fenchel-Neilsen coordinates for Teichm\"uller spaces of infinite type orbifolds, but there the theory is more subtle: for example, an arbitrary gluing of infinitely many pairs of pants need not be metrically complete. See e.g.\  \cites{alessandrini2011fenchel, alessandrini2012various} for a description of Fenchel-Nielsen coordinates on the Teichm\"uller spaces of certain infinite type surfaces.
\end{remark}

\subsubsection{Moduli space and its end(s)}

If $S$ is a finite type hyperbolizable $2$-orbifold, the \emph{moduli space} of $S$ is defined as \[\mathcal M(S) := \big \{ \text{ oriented, finite volume hyperbolic $2$-orbifolds } X \text{ homeomorphic to } S \big \} / \text{o.p.\ isometry}.\]
Just as in the surface case, $\mathcal M(S)$ can be considered as an orbifold that is the global quotient of the Teichm\"uller space $\mathcal T(S)$, which we discussed in  \S \ref{sec: fenchelnielsen}, by the \emph{mapping class group} \[\mathrm{Mod}(S) = \big \{\ \text{o.p.\ homeomorphisms } f: S \longrightarrow S \ \big \} / \text{ isotopy},\]
which acts properly discontinuously on $\mathcal T(S)$ from the right via $[(X,h)] \cdot [f] = [(X,h \circ f)]$. Here, we defined a point in Teichm\"uller space as a pair $(X,h: S \longrightarrow X)$, where $X$ is a priori an unoriented orbifold, but if we always equip $X$ with the pushforward of some fixed orientation on $S$, then the map from Teichm\"uller space to moduli space is just the projection $[(X,h)] \mapsto [X]$. See \cite{primer} for a reference in the surface case; the orbifold case is discussed in \cite{choi2012geometric} and \cite{Thurston}, for instance.

Given $\epsilon>0$, the \emph {$\epsilon$-thin part} of moduli space is the subset $\mathcal M^\epsilon(S)$ consisting of all $X$ that have an admissible closed geodesic with length less than $\epsilon$. The subsets $\mathcal M^\epsilon(S) \subset \mathcal M(S)$ have compact complements, by Mumford's compactness theorem (see \cite{primer} in the surface case, and the proof obviously generalizes), and we have $\cap_{\epsilon>0} \mathcal M^\epsilon(S)=\emptyset.$ The following proposition is probably well-known. We couldn't find a reference for it in the case of orbifolds, so although the proof is basically the same as for surfaces, we review it here to reassure the reader that all the curve complex arguments take place on the regular part of the orbifold. In the statement, a cusp is considered as a cone point of infinite order.

\begin{proposition}\label{prop: endofmoduli}
Suppose that $S$ is not a sphere with 3 or 4 cone points, possibly of infinite order. Then for small $\epsilon$, the subset $\mathcal M^\epsilon(S)$ is nonempty and path connected. Consequently, $\mathcal M(S)$ has one end.

If $S$ is a sphere with 3 cone points, then $\mathcal M(S)$ is a point and $\mathcal M^\epsilon(S)$ is empty for small $\epsilon$.

If $S$ is a sphere with four cone points, pick a partition $P$ of the orders of the cone points into two sets of two. For instance, if the orders are $2,2,3,\infty$ then there are two possible partitions, $\{\{2,2\},\{3,\infty\}\}$ and $\{\{2,3\},\{2,\infty\}\}.$ Let $\mathcal C^\epsilon_P \subset \mathcal M(S)$ be the subset consisting of all $[X]$ such that such that there is a closed geodesic on $X$ of length less than $\epsilon $ that separates $X $ into two pieces that induce the partition $P$ of cone point orders. Then for small $\epsilon $, these subsets $\mathcal C^\epsilon_P$ are the path components of $\mathcal M^\epsilon(S)$, and the number of ends of $\mathcal M(S)$ is the number of such partitions. In particular, if in the following table distinct letters represent distinct orders of cone points, we have
    \begin{center}
\begin{tabular}{ c|c } 
Orders & \# Ends  \\ 
 \hline
 a,a,a,a & 1  \\ 
 a,a,a,b & 1  \\ 
 a,a,b,b & 2 \\ 
 a,a,b,c & 2 \\
 a,b,c,d & 3
\end{tabular}
\end{center}

\end{proposition}

\begin{proof}
Assume $S$ is not a sphere with three or four cone points, and also let's assume for the moment that $S$ is not a torus with one (possibly infinite order) cone point. Pick markings $h: S \longrightarrow X$ and $h' : S ' \longrightarrow X'$, let $S_{reg}$ be the regular part of $S$ and suppose that $\alpha,\alpha' \subset S_{reg}$ are essential, nonperipheral simple closed curves whose geodesic lengths\footnote{Recall that the geodesic length of $\alpha$ is the length of the admissible geodesic in $X$ homotopic to $h(\alpha)$, and the existence of such an admissible geodesic is discussed in Lemma \ref{lem: geodesic representatives}.} in $X,X'$ are less than $\epsilon$. By our assumption on $S$, the curve complex of  $S_{reg}$ is connected, c.f.\ \cite[Theorem 4.3]{primer}, so there is a sequence $\alpha=\alpha_0,\alpha_1,\ldots,\alpha_n = \alpha'$ of essential, nonperipheral simple closed curves on $S_{reg}$ where each $\alpha_i$ is disjoint from $\alpha_{i+1}$. Using Fenchel Nielsen coordinates, one can then join $(X,h)$ to $(X',h') $ with a path through $\epsilon$-thin surfaces by pinching $\alpha_1$ while keeping $\alpha_0$ short, then pinching $\alpha_2$ while keeping $\alpha_1$ short, etc... 

If $S$ is a torus with one cone point, let $\mathcal T^\epsilon(S) \subset \mathcal T(S)$ be the subset consisting of hyperbolic structures that have an admissible closed geodesic of less than $\epsilon$. Letting $\gamma$ range over all simple closed, nonperipheral essential curves in $S_{reg}$, we have that $\mathcal T^\epsilon(S)$ is a disjoint union $$\mathcal T^\epsilon(S) = \sqcup_\gamma \mathcal T^\epsilon_\gamma(S),$$ where each $\mathcal T^\epsilon_\gamma(S)$  consists of structures where $\ell(\gamma)<\epsilon$. By Fenchel-Nielsen coordinates, each $\mathcal T^\epsilon_\gamma(S) \cong (0,\epsilon) \times \RR$, and so is path connected, and the mapping class group $\mathrm{Mod}(S)$ acts transitively on the set of all $\gamma$. (On the punctured torus $S_{reg}$, all simple closed curves are nonseparating.) So while $\mathcal T^\epsilon(S)$ is disconnected, after taking the projection $\mathcal T(S) \longrightarrow \mathcal M(S)$ all the path components become identified, so the image $\mathcal M^\epsilon(S)$ in moduli space is path connected. 

If $S $ is a sphere with 3 cone points, then Fenchel-Nielsen coordinates (see Theorem \ref{thm: fenchelnielsen}) imply that $\mathcal M(S)$ is a point.

The case of a sphere with four cone points is the same as the case of a torus with one cone point, except that the $\mathrm{Mod}(S)$-orbits of essential, nonperipheral simple closed curves on $S_{reg}$ are exactly the sets of curves determining a particular partition of cone point orders.
\end{proof}

\subsubsection{Resolving singularities in finite covers}
 
 By Selberg's Lemma \cite{selberg1962discontinuous} any  finite type hyperbolic $2$-orbifold has a finite cover that is a surface.  In general, one cannot expect a $2$-orbifold to have a finite surface cover, since there can be a sequence of cone points in $X$ with orders going to infinity, and a cone point $p \in X$ will lift to a cone point in any finite cover whose degree is less than the order of $p$.  However, it is possible to resolve any \emph{finite} set of singularities of $X$  in a cover of sufficiently high degree.
 
 \begin{proposition}\label{prop: finitecovers}
 Suppose that $X $ is a hyperbolic orientable $2$-orbifold and that $F \subset X$  is a finite set.  Then there is a finite cover $\pi: \hat X \longrightarrow X$ such that each element of $\pi^{-1}(F)$ is a regular point of $\hat X.$
 \end{proposition}
 
  In particular, this implies that any $X$ with finitely many cone points has a finite  cover that is a surface. Before proving  Proposition \ref{prop: finitecovers}, we record the following  standard fact about fundamental groups of noncompact surfaces with boundary.
  
  \begin{lemma}\label{lem: pi1surfaces}
    Suppose that $X$ is a noncompact orientable $2$-orbifold with boundary. Then the orbifold fundamental group $\pi_1 S$ is a free product of cyclic groups. Moreover, this decomposition can be taken so the conjugacy class of each boundary component of $X$ is represented by a generator for some infinite cyclic factor, and the conjugacy class of each cone point is represented by a generator for some finite cyclic factor.
  \end{lemma}
 \begin{proof}
  Let $S$  be the surface obtained from $X$ by deleting discs around all its cone points. We will show that $\pi_1 S$ is free and that representatives of the conjugacy classes corresponding to the boundary components of $S $ can be taken as part of the free basis. The lemma will then follow from this and Van Kampen's theorem. The approach we follow in the surface case is an argument of Whitehead \cite[Lemma 2.1]{whitehead1961immersion}, which is commonly cited as proving that fundamental groups of noncompact surfaces are free. It applies verbatim to give the stronger result we want above, but it is worth going through the argument briefly in order to see why.  We follow the excellent exposition written by Putman \cite{putmanspines}.
  
  Fix a triangulation of $S$, and let $D$ be the dual graph. As in \cite{putmanspines}, there is a subgraph $F$ of $D$ that includes all vertices of $D$, and where every component of $F$ is an infinite ray with finite trees attached at certain vertices. Set $Y$ to be the embedded graph in $S$ that is the union of all the edges of our triangulation that do not cross $F$, and note that $Y$ contains all of the boundary components of $S$ as embedded cycles. Then as in \cite{putmanspines}, $S$ is a union of regular neighborhoods $\nhd (F)$ and $\nhd (Y)$. The regular neighborhoods of the components of $F$ are discs with once-punctured boundaries, so retracting these discs onto their boundaries, we see that $S$ deformation retracts onto $\nhd(Y)$, and hence onto $Y$. Since $Y$ is a graph in which all boundary components of $S$ appear as embedded cycles, $\pi_1 S \cong \pi_1 Y$ is free and representatives of the boundary components can be taken as part of a free basis.
\end{proof}

  We are now ready for the proof of Proposition \ref{prop: finitecovers}.
 
\begin{proof} As remarked above, if $X$ is compact this follows from Selberg's lemma. So assume $X$ is noncompact.   Take $X$, and delete open disks around the cone points $\{p_1, ... p_k \}$ to obtain a new orbifold $N$.  By Lemma \ref{lem: pi1surfaces}, $$\pi_1 (N) \cong \langle c_1\rangle  \ast \cdots \ast \langle c_k\rangle  \ast G, $$ where the $c_i$ are the boundaries of the deleted disks and $G$ is a free product of cyclic groups. Define $$f: \langle c_1\rangle  \ast \cdots \ast \langle c_k\rangle  \ast G \to \ZZ/ n_1\ZZ \times \cdots \times \ZZ/ n_k\ZZ, \ \ \ f(c_i) = 1 \in \ZZ / n_i\ZZ, \ \ f(G) = 0,$$ where $n_i$ is the order of $c_i$ in $\pi_1 (X)$, i.e.\ the order of the cone point $p_i$.  The kernel of $f$ is a finite index subgroup in $\pi_1(N)$ and so it gives a finite index orbifold cover $\pi : \hat N \longrightarrow N$.  Glue a disk along each component of the preimage in $\hat N$ of each $c_i$ to give a new orbifold $\hat X$. On each component $\hat c_i \subset \pi^{-1}(c_i)$, the covering map $\pi$ restricts to an $n_i$-to-1 cover $\hat c_i \to c_i$. This can be extended to a $n_i$-to-1 orbifold cover from the new disk in $\hat X$ to the singular disk we deleted from $X$ to form $N$.  Piecing together these maps, we get an extended covering map $\hat X \longrightarrow N$. Since $\pi$ had finite degree, this map does too, and the preimages of all the $p_i$ consist completely of smooth points, the centers of the disks added in above.
\end{proof}

% Let $\Sigma \to  S(m_1,\ldots,m_n)$ be a cover by a surface of degree $m_1 \cdots m_n$.  Then the length of a regular geodesic is $\frac{1}{m_1 \cdots m_n}$ the length of its lift in $\Sigma$.  The length of a degenerate geodesic is $\frac{2}{m_1 \cdots m_n}$ the length of its lift in $\Sigma$. \\
% %
% \indent We denote by $X \dsm c$ the orbifold with geodesic boundary obtained by cutting $X$ along an admissible geodesic. Any curve on a cone surface $S(m_1,\ldots,m_n)$ is homotopic to a unique admissible geodesic \cite{Fanoni}, Proposition 2.2.  If $\mathcal{P}$ is a pants decomposition of $S=S(m_1,\ldots,m_n)$ with genus $g$ and $b$ boundary components, then $\mathcal{P}$ contains $3g-3 +n +b$ curves and $S \dsm \mathcal{P}$ is a collection of open pairs of pants, \cite{Fanoni} Proposition 2.4.  An orbifold is a cone surface without boundary components. 

% \textcolor{blue}{ Following \cite{CHK} example 2.39, we remind the reader of the notion of a tangent space to a 2-orbifold $S= \Gamma \backslash X$, where $X$ is a modeling space.  Notice $\Gamma$ acts naturally on the unit tangent bundles $UT( X)$ and the quotient orbifold $UT( X)/ \Gamma$ is the \emph{unit tangent bundle} of $S$. \\
% %
% \indent The notion of triangulating a manifold extends to triangulating a a differential orbifold.  Every 2-orbifold is the quotient of a surface by a finite group, so we simply require the triangulation of the orbifold to lift to a triangulation on the surface.  See \cite{Bonahon} discussion between Theorem 2.5 and Proposition 2.6. 
% }

\subsection{Injectivity radius and the thick-thin decomposition}
\label{sec: injrad and thick-thin}
Let $\Gamma < \PSL_2(\RR)$  be a discrete group, and let $X := \Gamma \backslash \HH^2$  be the quotient orbifold. Let $\Gamma^{\infty} \subset \Gamma$  be the subset consisting of all  infinite order elements, i.e.\ all elements of hyperbolic or parabolic type.  Given $q\in \HH^2$, we define the \emph{infinite order translation distance} of $q$ to be the minimum
$$\tau^\infty(q) := \min \Big \{ \, d(\gamma(q),q) \ \big | \ \gamma \in \Gamma^\infty \Big \},$$
 In the quotient, we define the \emph {infinite order injectivity radius} of $X$ at a point $p\in X$ to be $$\inj(X,p) := \frac 12
 \tau^\infty(\tilde p),  $$
 where $\tilde p \in \HH^2$ is any lift of $p$. Then given $\epsilon>0$, the \emph{$\epsilon$-thick} and \emph{$\epsilon$-thin parts} of $X$  are defined as
$$X_{\geq \epsilon} := \{ p\in X \ | \ \inj(X,p) \geq \epsilon \}, \ \ \ X_{< \epsilon} := \{ p\in X \ | \ \inj(X,p) < \epsilon \},$$
 respectively.  

\begin{theorem}[Thick Thin Decomposition] \label{thm: thickthin}  There is a constant $\epsilon_0>0$, called \emph{the Margulis constant for $\HH^2$},  with the following property. If $\epsilon<\epsilon_0$, the thin part $X_{< \epsilon}$ is a disjoint union of its connected components. For each component $U \subset X_{< \epsilon}$, there is a subgroup $\Lambda \subset \Gamma$ and a subset $W \subset \HH^2$ that is precisely invariant under $\Lambda < \Gamma$, such that $U = \Lambda \backslash W $ and where we are in one of the following three situations:
\begin {enumerate}
	\item[(C)] $U$  is homeomorphic to $\bbS^1 \times (0,\infty)$  and is a neighborhood of a single end $\xi$ of $X$.  The subset $W \subset \HH^2$ is a horoball, and $\Lambda $ is a cyclic  parabolic subgroup of $\Gamma$.  We call $U$ a \emph{cusp neighborhood} and $\xi$ a \emph{cusp}.
\item[(R)] $U$ is a metric neighborhood  of a regular geodesic $c \subset X$ such that $\length(c)<2\epsilon$, and  there is a homeomorphism $U \longrightarrow \bbS^1 \times (-1,1)$ taking $c$ to $\bbS^1 \times 0$. The subset $W \subset \HH^2$  is a metric neighborhood of a geodesic $\tilde c$  that is a lift of $c$, and $\Lambda < \Gamma$  is a  cyclic hyperbolic subgroup  consisting of all elements of $\Gamma$  that leave $\tilde c$  invariant. We call $U$ a \emph{regular Margulis tube}.
\item[(D)] $U$ is a metric neighborhood  of a  degenerate geodesic $c \subset X$ such that $\length(c)<2\epsilon$. If $p_c : \bbS^1 \longrightarrow X$ is a  parametrization of $c$ and $i : \bbS^1 \longrightarrow \bbS^1$  is a reflection such that if $p_c(i(x))=p_c(x)$ for all $x\in \bbS^1$, then there is a homeomorphism $$U \longrightarrow \bbS^1 \times [0,1]/ (x,0) \sim (i(x),0)$$  such that $p_c(x) \mapsto [(x,0)]$ for all $x \in \bbS^1$. The subset $W \subset \HH^2$  is a metric neighborhood of a geodesic $\tilde c$  that is a lift of $c$, and $\Lambda < \Gamma$  is an infinite dihedral subgroup  consisting of all elements of $\Gamma$  that leave $\tilde c$  invariant. We call $U$ a \emph{degenerate Margulis tube}.
\end {enumerate}
\end{theorem}  

 \noindent In the theorem above, recall that $W \subset \HH^2$  is \emph {precisely invariant}  under a subgroup $\Lambda < \Gamma$ if $\gamma(W)=W$  for all $\gamma \in \Lambda$, and $\gamma(W) \cap W = \emptyset$  for all $\gamma \in \Gamma - \Lambda$.  Note that in this case,  the quotient $\Lambda \backslash W $ embeds in $X=\Gamma \backslash \HH^2$. Notice that since this thick-thin decomposition is only according to the infinite order elements, the cone points do not contribute.

 There are a number of statements of thick thin decompositions in the literature.  For manifolds, the reader can consult \cites{Benedettilectures,CEG} in the hyperbolic setting and \cite[\S 10]{BGS} in the setting of nonpositive curvature.  For orbifolds, there is a  less detailed statement in \cite{AX}, and a $3$-dimensional  statement in \cite[Theorem 7.16]{CHK}. Note that there is a choice  to be made when deciding how to handle  orbifolds: here we just disregard elliptics, but \cite{AX} and \cite{CHK} have slightly different conventions. However, all versions of the thick-thin decomposition follow readily from the Margulis Lemma, see pg 101 of \cite{BGS}, Theorem D.1.1 in \cite{Benedettilectures}, and the proof below. We would like to stress that while the statement of Theorem \ref{thm: thickthin} will not surprise experts in hyperbolic geometry, we were not able to find any equivalent statement in the literature.  In particular, we have never seen a treatment of the thick-thin decomposition for orbifolds that uses infinite order translation distance, nor a statement that explicitly identifies the types of thin parts in $2$-dimensions as above.

%\tinytodo{IB: after the paper is written, perhaps we can talk about how much of this proof to include. I'm inclined to just leave it like it is, since I don't know any reference that actually states it like above.  In particular, I don't know of any reference that uses `infinite order injectivity radius', so I wasn't completely confident that it would work until I wrote it down. However, we could probably abbreviate the proof if you both want that.}

\begin {proof}
We adapt the argument on pages 110--112 of \cite{BGS} to the case of orbifolds.

Let $W \subset \HH^2$ be a connected component of the preimage of $U \subset X$ in $\HH^2$. Since the subset $\Gamma^\infty \subset \Gamma$ is conjugation invariant, the function $\tau^\infty$ is $\Gamma$-invariant, so $\Gamma$   permutes the components of $\{ p \in \HH^2 \ | \ \tau^\infty(p) < \epsilon \}$. But $W$ is one of these components, so $W$ is precisely invariant under some subgroup $\Lambda < \Gamma$.  Note that 
\begin{enumerate} \item[(*)] if $\gamma \in \Gamma^\infty$ and  there is a point $p\in W$ with $d(\gamma(p),p) < \epsilon$, then $\gamma \in \Lambda$.
\end{enumerate} 
For in this situation, $\gamma$ also translates $\gamma(p)$ by less than $\epsilon$, so by the convexity of the distance function, $\gamma$  translates every point on the segment from $p $ to $\gamma(p)$ by less than $\epsilon$.  From this, we see that $\gamma(p)\in W$, implying that $\gamma \in \Lambda$.

 Now given $p\in W$, let $\Gamma_p  < \Gamma$  be the subgroup  generated by the set of all $\gamma \in \Gamma$ with $d(\gamma(p),p)<\epsilon$, and let  $\Gamma_p  ^\infty = \Gamma_p \cap \Gamma^ \infty$. By the Margulis lemma, c.f.\ pg 101 of \cite{BGS}, $\Gamma_p $ has a finite index nilpotent subgroup $\Delta_p < \Gamma_p$. If two nontrivial elements $\alpha,\beta \in \PSL_2 (\RR)$ commute, they have the same fixed points in $\HH^2 \cup \partial \HH^2$; as any nilpotent group has nontrivial center, it follows that the fixed points in $\HH^2 \cup \partial \HH^2$ of all  nontrivial elements of $\Delta_p$ are the same.  Consequently, \emph{all  elements of $\Gamma_p^\infty$ have the same fixed points in $\partial \HH^2$}, since any such  element has  a power that is a nontrivial element of $\Delta_p$. 

Now let $p,q\in W$.  We claim that the fixed points in $\partial \HH^2$ of  $\Gamma_p^\infty$ and $\Gamma_q^\infty$ are the same. Indeed, since $p\in W$, there is some $\gamma \in \Gamma^\infty $ with $d(\gamma(p),p)<\epsilon$. If $q$ is sufficiently close to $p$, then $d(\gamma(q),q)<\epsilon$ as well.   In this case, $\gamma \in \Gamma_p \cap \Gamma_q$, so from the previous paragraph, the fixed points of  infinite order elements in $\Gamma_p$ and $\Gamma_q$ are the same. The claim follows for arbitrary pairs of points $p,q \in W$ by connectedness.  So, by the classification of isometries of $\HH^2$,  we are in one of the following two situations.
\begin{enumerate}
	\item all  elements $\gamma \in \cup_{p\in W} \Gamma_p^\infty$ are hyperbolic type with some common axis $\tilde c \subset\HH^2$, or
\item all  elements $\gamma \in \cup_{p\in W} \Gamma_p^\infty$ are parabolic with a common fixed point $z \in \partial \HH^2$.
\end{enumerate}

Suppose now that $\gamma \in \Lambda$, the $\Gamma$-stabilizer of $W$. If $p \in W$, then $\gamma(p)\in W$ as well, so there is some $\alpha \in \Gamma^\infty$ with $d(\alpha(\gamma(p)),\gamma(p))<\epsilon$.  Setting $\beta =\gamma^{-1} \circ \alpha \circ \gamma$,  it follows that $d(\beta(p),p)<\epsilon$, so $\beta \in \Gamma_p^\infty$. Since $\alpha \in \Gamma_{\gamma(p)}^\infty$, the previous paragraph implies that $\alpha,\beta$  have the same fixed points in $\partial \HH^2$. But $\gamma$  conjugates one to the other, so $\gamma $ permutes these fixed points.  In case 1 above,  this means that either $\gamma $ is  a hyperbolic type isometries with axis $\tilde c$, or a rotation by $\pi$ around some point on $\tilde c$. In case 2, $\gamma$  must be a parabolic isometry fixing $z$, for it is easy to check that if $\gamma$ is hyperbolic type, then the group $\langle \gamma, \alpha \rangle < \Gamma$  is not discrete. (See pg 112 of \cite{BGS}.)  Translating this to a statement about the entire group $\Lambda$, it follows that either 
\begin{enumerate}
\item[(C)] $\Lambda$ is a  cyclic group of parabolic isometries fixing $z\in \HH^2$,
\item[(R)] $\Lambda$ is a cyclic group of hyperbolic type isometries with axis $\tilde c,$ or
\item[(D)] $\Lambda$ is  an infinite dihedral group of isometries stabilizing $\tilde c$.
\end{enumerate} Indeed, note that by $(*)$ above, $\Lambda$ contains any $\Gamma_p^\infty$, so cannot consist only of elliptic elements.

We now describe $W$ explicitly.  First, consider the set
$$V := \{p\in \HH^2 \ | \ d(\gamma(p),p)<\epsilon \text{ for some } \gamma \in \Lambda \cap \Gamma^\infty\}.$$
 In case (C), $V$ is a horoball centered at $z$, while in cases (R), (D), $V$ is a metric neighborhood of $\tilde c$. By $(*)$, $$W \subset V \subset \{p \in \HH^2 \ | \ \tau^\infty(p)<\epsilon\}.$$
But in all three cases, $V$ is connected and $W$  is a component of $ \{p \in \HH^2 \ | \ \tau^\infty(p)<\epsilon\}$, so $W=V$.  This finishes the proof, as the descriptions of the homeomorphism types of $U = \Lambda \backslash W$ follow immediately.
  \end {proof}

\subsection{Discrete groups and  vectored $2$-orbifolds}
\label{sec: convergence section}

A \emph {vectored, oriented hyperbolic $2$-orbifold} is a pair $(X,v)$, where $X $ is an oriented hyperbolic $2$-orbifold and $v \in T^1X$.  We say that $(X,v)$ and $(X',v')$ are \emph {isometric}  if there is a  (differentiable) orientation preserving isometry $f :X \to X'$ such that $d f(v) = v'$.  
As in the introduction, there is a natural bijection
\begin{equation} \label{eq: bijection} \Sub_d(G) : = \{ \text { discrete } \Gamma < G \ \} \longrightarrow \{ \text { vectored, oriented hyperbolic } 2\text{-orbifolds } (X,v) \ \} \ / \ \text {isom},\end{equation}
 defined by fixing a base vector $v_{base} \in T\HH^2$ and mapping $\Gamma < G$ to the pair $(X,v)$, where $X := \Gamma \backslash \HH^2$ and $v$ is the projection of $v_{base}$  under the derivative of the quotient map. See e.g.\ \cite[Chapter E]{Benedettilectures}, noting that in our context of oriented hyperbolic $2$-orbifolds, a choice of base vector $ v$ determines an orthonormal frame $(v,w)$, where $w$ is  the vector obtained by rotating $v$ by $\pi/2$ counterclockwise. We will write
$$(X,v) \longmapsto \Gamma(X,v)$$
 for the inverse of this map.

 There is a natural \emph{smooth topology}  on the set of isometry classes of vectored oriented Riemannian orbifolds, where $(X_i,v_i) \to (X_\infty,v_\infty)$ if there is a sequence $R_i \to \infty$ and smooth embeddings
\begin{equation} \label{eq: smooth convergence maps}
\psi_i : B_{X_\infty}(p_\infty,R_i) \longrightarrow X_i,
\end{equation}
 where $p_\infty$ is the base point of $v_\infty$, such that each $\psi_i$ is orientation preserving, $d\psi_i^{-1}(v_i) = v_\infty\in TX$ for all $i$, and where if $g_i$ is the Riemannian metric on $X_i$, we have $\psi_i^* g_i \to g_\infty$ in the $C^\infty$ topology on  $X_\infty$.  We refer to the reader to \cite[Ch 10]{petersen2006riemannian} and the appendix in \cite{urms} for more information about the smooth topology in the manifold setting.  We will often refer to the maps $\psi_i$  as the \emph{almost isometric maps} coming from smooth convergence. Note that in order to say that the $\phi_i$ are more and more isometric as $i\to \infty$, we pull back all the metrics to the domain, which is a fixed space.  If both the domain and the co-domain are fixed,  though, this is essentially equivalent to the following lemma saying that the maps converge smoothly to an isometry: 
 
 \begin{lemma}[Metric convergence implies convergence of maps]\label{lem: convergingmaps}
 Suppose that $X,Y$ are orbifolds, $U \subset X$ is open and connected, and $\phi_i : U \longrightarrow Y$ is a sequence of embeddings such that $\phi_i^* g_Y \to g_X$, where $g_X,g_Y$ are the Riemannian metrics on $X,Y$. If for some $p \in X$, the sequence $(\phi_i(p))$ is precompact, then $(\phi_i)$ has a subsequence  that smoothly converges to an isometric embedding.
 \end{lemma}
 
  There are probably a number of proofs of this fact.  Here is one phrased in terms of the machinery in Abert-Biringer \cite{urms}.
 
 \begin{proof}
  As described in the appendix to \cite{urms}, fix $k$ and consider the $k^{th}$ iterated tangent bundles $T^kX,T^kY$  equipped with their iterated Sasaki metrics. Then each $\phi_i$ induces a map $\phi_i^k : T^k U \longrightarrow T^kY$. Since $\phi_i^* g_Y \to g_X$ smoothly, so in particular in the $C^k$-topology, Corollary A.7 in \cite{urms} implies that $(\phi_i^k)^* g_Y^k \to g_X^k$ in the $C^0$-topology. In other words, the maps $\phi_i^k$ are uniformly bi-Lipschitz. As $\phi_i(p)$ is precompact in $Y$, if $p^k$ is the element of the zero section of $T^kX$ that lies over $p$, then $\phi_i^k(p^k)$ is precompact in $T^kY$. Arzela-Ascoli then implies that after passing to a subsequence,  we may assume that $(\phi_i^k)$ converges uniformly on compact subsets of $T^kX.$   In any fixed coordinate chart, all $k^{th}$  partial derivatives  of $\phi^i$ appear as components of $\phi_i^k$, so it follows that $\phi_i$ converges in the $C^k$ topology. A diagonal argument produces a smoothly convergent subsequence, and the limit is clearly an isometry.
 \end{proof}

\begin{remark}
	 In the definition of the smooth topology, it is equivalent to require that $d\psi_i^{-1}(v_i) \to v_\infty\in TX$ instead of $d\psi_i^{-1}(v_i) = v_\infty\in TX$.  The reason is that if $d\psi_i^{-1}(v_i) \to v_\infty$, there are diffeomorphisms $\phi_i : X_\infty \longrightarrow X_\infty$ that are the identity outside of a small neighborhood of $p_\infty$, that $C^\infty$-converge to the identity on $X_\infty$, and where $d\phi_i(v_\infty) = d\psi_i^{-1}(v_i)$. (Such $\phi_i$ can be constructed explicitly in a local chart.) Then $\psi_i \circ \phi_i$ are maps as required in the definition of smooth convergence. \label{rem: points can converge}
\end{remark}

 Our first result in this section is that when the left side of \eqref{eq: bijection} is endowed with the  Chabauty topology and the right side the smooth topology,  the map in \eqref{eq: bijection} is a homeomorphism:

\begin {proposition}[Chabauty vs smooth convergence]
	\label{prop: conv thm} Suppose that $(X_i,v_i),$ $i=1,\ldots,\infty$ are vectored, oriented hyperbolic $2$-orbifolds, $\Gamma_i := \Gamma(X_i,v_i) $, and $p_i$ is the basepoint of $v_i$.  Then $(X_i,v_i)\to (X_\infty,v_\infty)$  in the smooth topology if and only if $\Gamma_i \to \Gamma_\infty$  in the Chabauty topology.

\end {proposition}

The reader may be more familiar with the equivalence between the Chabauty topology on  the space of discrete, torsion-free subgroups of $\mathrm{Isom}(\HH^n)$ and the \emph{Gromov-Hausdorff topology} on the space of framed hyperbolic $n$-manifolds, see e.g.\ \cite{CEG}, or alternatively \cite[\S 3]{7sam} for a statement in the context of general symmetric spaces of noncompact type. For hyperbolic {manifolds}, the Gromov-Hausdorff topology is the same as the (framed) smooth topology---this is part of the argument in \cite[\S 3]{7sam}, but see also \cite[Ch E]{Benedettilectures}.    It follows from the proof below that these two topologies also agree on orientable hyperbolic $2$-orbifolds, or more generally hyperbolic $n$-manifolds with isolated singularities, but as we do not need the Gromov-Hausdorff topology in this paper, we do not discuss it further.  %However, we  would not be surprised if the Gromov-Hausdorff topology and the smooth topology coincided in general for framed orbifolds modeled on a fixed symmetric space of noncompact type. it's

Before starting the proof of Proposition \ref{prop: conv thm}, we record the following lemma, which gives a useful criterion for when two sequences of subgroups of $G$ converge to the same limit in the Chabauty topology.  It  will be applied in the proof of Proposition \ref{prop: conv thm} with one of the sequences being constant, turning it into a convergence criterion, but we will find the more general statement useful later.

\begin{lemma}[Sequences with the same limit]\label{lem: same limit}  Suppose we have two sequences   $\Gamma_i,\Delta_i$ of closed subgroups of $G$, a nested sequence of  compact sets $C_1 \subset C_2 \subset \cdots $ of $\HH^2$ with  $\cup_i C_i = \HH^2$, and a sequence of  embeddings $$f_i : C_i \longrightarrow \HH^2$$  such that $f_i \to id$  uniformly on compact subsets of $\HH^2$,  and such that the maps $f_i$ are \emph{$(\Gamma_i,\Delta_i)$-equivariant}, in the sense that if $B \subset C_i$ is an open ball, we have that
\begin{itemize}
	\item if $\gamma_i \in \Gamma_i$  satisfies $\gamma_i(B) \subset C_i$, there is some $\delta_i \in \Delta_i$ with $\delta_i \circ f_i = f_i \circ \gamma_i $ on $B$, 
\item if $\delta_i \in \Delta_i$  satisfies $\delta_i(f_i(B)) \subset f_i(C_i)$, there is some $\gamma_i \in \Gamma_i$ with $\delta_i \circ f_i = f_i \circ \gamma_i $ on $B$.
\end{itemize}
Then if $\Gamma_i \to \Gamma $ in the Chabauty topology, we have $\Delta_i \to \Gamma$ as well.
\end{lemma}
\begin{proof}
Let $\gamma \in \Gamma$. Pick a sequence $\gamma_i \in \Gamma_i $ with $\gamma_i \to \gamma$. If $B \subset \HH^2$ is an open ball, we know that $B, \gamma_i(B) \in C_i$  for all large $i$. By hypothesis, there is some $\delta_i \in \Delta_i$ with $\delta_i \circ f_i = f_i \circ \gamma_i $ on $B$. But $f_i\to id$ and $\gamma_i \to \gamma$, implying that $\delta_i \to \gamma$ on $B$. Since $B$ is an open ball and $\delta_i$ is an isometry of $\HH^2$, we have $\delta_i \to \gamma$ in $G$.

Next, suppose $\delta_i \in \Delta_i$  for all $i$, and that  some subsequence $\delta_{i_j} \to \delta \in G$.   Pick some open ball $B \subset \HH^2$. Since the sets $C_i$ exhaust $\HH^2$ and $f_i \to id$,  the images $f(C_i)$ also exhaust $\HH^2$. So, as $(\delta_{i_j})$ converges,  for large $j$ we can assume that
$f_{i_j}(B), \delta_{i_j}(f_{i_j}(B)) \subset f(C_{i_j}).$
 By hypothesis, there are $\gamma_{i_j} \in \Gamma_{i_j}$ with  $\delta_{i_j} \circ f_{i_j} = f_{i_j} \circ \gamma_{i_j} $ on $B$.  Since $\delta_{i_j}\to \delta$ and $f_{i_j} \to id$, we have $\gamma_{i_j} \to \delta$ on $B$, and hence globally. So, $\delta$ is an  accumulation point of elements of the $\Gamma_i$, and hence $\delta\in \Gamma$  as desired.\end{proof}

 We are now ready to prove Proposition \ref{prop: conv thm}.

\begin {proof}[Proof of  Proposition \ref{prop: conv thm}]
 We will show that the  bijection $(X,v) \longmapsto \Gamma(X,v)$ is proper and continuous.  Since the space $\Sub_d(G)$ of discrete subgroups is a metric space, it will follow that the bijection is a homeomorphism. 

 We first show the map is proper.  Let $\mathcal K \subset \Sub_d(G)$ be compact.  

\begin{claim}Then there is some open neighborhood $U \ni 1$ in $G$  such that $\Gamma \cap U = \{1\}$  for all $\Gamma \in \mathcal K$. 	
\end{claim}

\begin{proof}
Indeed, if there were not,  then there would be elements and groups $1 \neq \gamma_n \in \Gamma_n \in \mathcal K$ with $\gamma_n \to 1$ and $\Gamma_n \to \Gamma \in \mathcal K$. Since $\Gamma $ is  discrete, in exponential coordinates around $1\in G$  there is some $\epsilon >0$  such that $\exp^{-1}( \Gamma) \cap B(0,\epsilon)= \{0\}$. Since $\exp^{-1}(\gamma_n) \to 0$, there are integers $m_n$ such that $||m_n \exp^{-1}(\gamma_n)|| \in [\epsilon/4,\epsilon/2]$ for all $n$. But $m_n \exp^{-1}(\gamma_n)= \exp^{-1}(\gamma_n^{m_n}) \in \exp^{-1}(\Gamma_n)$, and a subsequence of $\exp^{-1}(\gamma_n^{m_n})$ limits to an element of $\exp^{-1}(\Gamma)$, contradicting that $\exp^{-1}( \Gamma) \cap B(0,\epsilon)= \{0\}$.
\end {proof}

Let $ \pbase \in \HH^2$  be the basepoint of the fixed $v_{base}\in T\HH^2$ chosen at the beginning of the section. The set of isometries $\gamma \in \PSL_2(\RR)$  such that $\gamma(B(\pbase,1)) \cap B(\pbase,1) \neq \emptyset$ is compact, so it follows from the claim that there is some $N < \infty$  such that if $\Gamma \in  \mathcal K$, then $$\# \big \{ \gamma \in \Gamma  \ | \ \gamma(B(\pbase,1)) \cap B(\pbase,1) \neq \emptyset \big \} \leq N.$$
 As a result, there is some $\delta>0$  such that if $\Gamma \in \mathcal K$ and $\Gamma = \Gamma(X,v)$, then $vol(B_X(p,1)) \geq \delta$, where $p \in X$  is the base point of $v$. So\footnote{Since each $X$  is always an  orientable hyperbolic $2$-orbifold, its singularities are isolated. Condition (i) of \cite[Theorem 3.5]{lu} says that the $p^{th}$-derivatives of the metric tensor should only depend on $p$, and not on the point or the orbifold. This is obvious for hyperbolic orbifolds, since these derivatives can just be calculated at any fixed point in $\HH^2$.}, by Theorem 3.5 of \cite{lu},  the set of all $(X,v)$ with $\Gamma(X,v) \in \mathcal K$ is compact in the smooth topology. Hence, the map $(X,v) \longmapsto \Gamma(X,v)$  is proper.

 We now show that $(X,v) \longmapsto \Gamma(X,v)$  is continuous. So, suppose that $(X_i,v_i)\to (X_\infty,v_\infty)$  in the smooth topology and set $\Gamma_i := \Gamma(X_i,v_i)$.  Let $\psi_i : B_{X_\infty}(p_\infty,R_i) \longrightarrow X_i,$  be maps  as in \eqref{eq: smooth convergence maps}  witnessing smooth convergence, choose universal covers $\pi_i : \HH^2 \longrightarrow X_i$ with $\pi_i(\pbase) = p_i$ and $d\pi(v_{base})=v_i$, and lift the $\psi_i$ to maps
$$\tilde \psi_i : B_{\HH^2}(\pbase,R_i) \longrightarrow \HH^2, \ \ \pi_i \circ \tilde \psi_i = \psi_i \circ \pi_\infty, \ \ d\tilde\psi_i(v_{base})=v_{base}.$$
As $i\to\infty$, the maps $\tilde \psi_i$ converge in the $C^\infty$ topology to an isometry of $\HH^2$. And since the limit must fix $v_{base}$,  we have $\tilde \psi_i \to id$. The embeddings $\tilde \psi_i$ are obviously $(\Gamma_\infty, \Gamma_i)$-equivariant in the sense of Lemma \ref{lem: same limit}, so $\Gamma_i \to \Gamma_\infty$  as desired.
\end{proof}

 We will also need the following criterion for smooth convergence in Theorem \ref{thm: contgraft}.

\begin{proposition}
[Convergence via quasi-conformal maps]\label{prop: conv via quasiconformal} Suppose that $(X_i,v_i),$ $i=1,\ldots,\infty$ are vectored, oriented hyperbolic $2$-orbifolds, and $p_i$ is the base point of $v_i$.  Suppose that there are $R_i,R_i' \to \infty$, $\delta>0$, and $K_i\to 1$, and for large $i$,  $K_i$-quasiconformal embeddings
$$\psi_i : B_{X_\infty}(p_\infty,R_i) \longrightarrow X_i,$$
such that $B_{X_i}(p_i,R_i') \subset \psi_i (B_{X_\infty}(p_\infty,R_i))$,  the restriction of $\psi_i $ to $B_{X_\infty}(p_\infty,\delta)$ is conformal, and $d\psi_i^{-1}(v_i)\to v_\infty \in TX_\infty$. Then $(X_i,v_i)\to (X_\infty,v_\infty)$  smoothly.
\end{proposition}

Its proof relies on the following lemma.

\begin{lemma}[Compactness for quasi-conformal maps of large subsets of $D$]\label{lem: compactquasi}
 Suppose that $B_1 \subset B_2 \subset \cdots$ is a nested sequence of compact subsets of $\HH^2$, with $\cup_i B_i = \HH^2$, and we have a sequence $$f_i : B_i \longrightarrow \HH^2$$ of $K$-quasiconformal embeddings such that $(f_i(0))$ is precompact in $D$ and $f_i(B_i) \supset B_i'$, where $B_i'$ is another nested sequence of compact sets exhausting $\HH^2$. Then some subsequence of $(f_i)$  converges uniformly on compact sets to a $K$-quasiconformal homeomorphism $f : \HH^2 \longrightarrow \HH^2$. 
\end{lemma}
\begin{proof}
	This is  similar to Corollary 4.4.3 in \cite{hubbard}, except that the $f_i$ are not homeomorphisms $\HH^2 \longrightarrow \HH^2$,  but are  embeddings defined on a sequence of balls that exhaust $\HH^2$, and where the images also exhaust $\HH^2$.  However, even in this setting, we can use the results of that section (\S 4.4) of \cite{hubbard} to prove the desired result. 

Let $d_{B_i}, d_{B_i'}, d_{f_i(B_i)}, d_{\HH^2}$ be the Poincar\'e metrics on the given domains in $\CC$.  Then
$$B_i' \subset f_i(B_i) \subset D \implies d_{B_i'} \geq d_{f_i(B_i)} \geq d_{\HH^2},$$
since the given metrics are also the Kobayashi metrics on their domains, see e.g.\ pg 287 of \cite{hubbard}.  Moreover, since the sets $B_i$ exhaust $\HH^2$, we have $d_{B_i} \to d_{\HH^2}$ uniformly on compact subsets of $D$. (One  way to see this is to work in the disk model, and to note that $d_{B_i}$ is sandwiched between $d_{\HH^2}$ and the Poincar\'e metric associated to a disk of Euclidean radius $r_i < 1$, where $r_i \to 1$ as $i\to \infty$. The latter Poincar\'e metric is the pullback of $d_{\HH^2}$ by the map $z \mapsto z/r_i$, so certainly converges to $d_{\HH^2}$.) Similarly, $d_{B_i'} \to d_{\HH^2}$ as well.

Applying Theorem 4.4.1 of \cite{hubbard} to $f_i$ and $f_i^{-1}$, there are universal homeomorphisms $\delta, \Delta : \RR_+ \longrightarrow \RR_+$  such that  for all $i$, we have $$\delta ( d_{B_i}(x,y)) \leq d_{f_i(B_i)}(f_i(x),f_i(y)) \leq \Delta ( d_{B_i}(x,y)), \ \ \forall x,y\in B_i.$$
  From this and the facts that  $d_{\HH^2} \leq d_{f_i(B_i)} \leq d_{B_i'}$ and $d_{B_i}\to d_{\HH^2}$, and $d_{B_i'}\to d_{\HH^2}$,  we have that for every compact subset $C \subset D$, the inequalities
\begin{equation}\label{eq: inequalities}
	\frac 12 \delta ( d_{D}(x,y)) \leq d_{D}(f_i(x),f_i(y)) \leq 2 \Delta ( d_{D}(x,y)) \ \  \forall x,y\in C,
\end{equation}
hold as long as $i$  is sufficiently large.  So, $(f_i)$ is uniformly equicontinuous on compact sets. As $(f_i(0))$ is precompact in $D$, Arzela-Ascoli then implies that  a subsequence of $(f_i)$ converges uniformly to a map $f : D \longrightarrow D$.  The lower bound in \eqref{eq: inequalities}, together with the assumption that the images of the $f_i$ exhaust $D$, implies that $f$ is a homeomorphism. \end{proof}

	\begin {proof} [Proof of Proposition \ref{prop: conv via quasiconformal}]
Fix  universal covering maps $$\pi_i : \HH^2 \longrightarrow X_i,$$
such that $d(\pi_i)(v_{base}) = v_i \in TX_i$ for all $i$, where $v_{base}$ is our usual fixed base vector in $\HH^2$, with basepoint $\pbase.$ Let $$B_i := B_{\HH^2}(\pbase,R_i) \subset \HH^2$$  be the ball with hyperbolic radius $R_i$ around $\pbase$.  Since $B_i$  is simply connected, there are lifts
$$\begin{tikzcd}
& & & \HH^2 \arrow{d}{\pi_i} \\ B_i \arrow{rr}{\pi_\infty} \arrow[rrru,dotted, "f_i", bend left=15] &  & B_{X_\infty}(p_\infty,R_i) \arrow{r}{\psi_i} & X_i	
\end{tikzcd}
$$
such that $f_i(\pbase) \to \pbase$.  Note that $f_i : B_i \hookrightarrow \HH^2$  is a $K_i$-quasiconformal embedding, and that
$$f_i(B_i) \supset B_i' := B_{\HH^2}(\pbase,R_i').$$
Lemma \ref{lem: compactquasi} implies that  after passing to a subsequence, the maps $f_i$ converge uniformly on compact subsets of $\HH^2$ to a homeomorphism $f : \HH^2 \longrightarrow \HH^2$. Since $K_i\to 1$, it follows that $f$ is a $1$-quasiconformal, i.e.\ a M\"obius transformation. (See the proof of Corollary 4.4.3 in \cite{hubbard}  for details.)  
Moreover, since $\psi_i$  is conformal on the $\delta$-ball around $p_\infty$, the lifts $f_i$ are conformal on a fixed ball around $\pbase$. Since the $f_i$ uniformly converge to $f$, the Cauchy integral formula implies $$df(v_{base}) = \lim_i df_i(v_{base}) = v_{base},$$ so $f=id$. If $\Gamma_i := \Gamma(X_i,v_i)$, $i=1,\ldots,\infty$ are the corresponding deck groups, the maps $f_i$ are  $(\Gamma_\infty,\Gamma_i)$-equivariant, so Lemma \ref{lem: same limit}  implies that $\Gamma_i \to \Gamma_\infty$.
\end{proof}

We end this section with two lemmas that help us adjust the almost isometric maps given by smooth convergence to have certain desired properties, i.e.\ sending a given geodesic to a geodesic, and being isometric in a small neighborhood of a basepoint.

\begin{lemma}[Sending geodesics to geodesic]\label{lem: geodesicspreserved}
Let $X_i$ be hyperbolic $2$-orbifolds, $i=1,2,\ldots,\infty$, let $c_\infty$ be a simple, closed geodesic in $X_\infty$, and let $U \supset c_\infty$ be any open neighborhood of $c_\infty$. Suppose $$\phi_i : U \longrightarrow X_i$$
are embeddings with $\phi_i^* g_i \to g_\infty$, where $g_i$  is the  Riemannian metric on $X_i$. Then there is a new sequence $$\psi_i : U \longrightarrow X_i$$  of embeddings  with $\psi_i^* g_i \to g_\infty$ such that $\phi_i = \psi_i$ except on a neighborhood of $c_\infty$ whose closure is contained in $U$, and where if $p : \bbS^1 \longrightarrow X_\infty$ is a constant speed parametrization of $c_\infty$, the composition $\psi_i \circ p$ is a  constant speed parametrization of a geodesic in $X_i$ for all large $i$.
\end{lemma}
\begin{proof}
Any small metric  neighborhood of $c_\infty$ in $X_\infty$  has strictly convex boundary. So, for large $i$ its image under $\phi_i$ is an  annulus with convex boundary, and hence contains a geodesic. It follows that there are simple, closed geodesics $c_i \subset X_i$ such that for large $i$, $\phi_i^{-1}(c_i)$ lies in an arbitrarily small neighborhood of $c_\infty$. So by Lemma \ref{lem: convergingmaps}, if we compose any sequence of unit speed parametrizations for the $c_i$ with the maps $\phi_i^{-1}$, we'll get loops $X_\infty$ that converge smoothly after passing to a subsequence, and any limit must be a parametrization of $c_\infty$. So reparametrizing, we can assume that there are constant speed parametrizations
$$p_i : \bbS^1 \longrightarrow c_i \subset X_i$$
 such that the maps $\phi_i^{-1} \circ p_i \to p$ smoothly. 
 
  Suppose for simplicity that $c$ is  a regular geodesic. Choose a smooth parametrization
 $$\bbS^1 \times (-1,1) \longrightarrow U$$
 such that $(x,0) \mapsto p(x)$.
 Then for large $i$, the functions $\phi_i^{-1} \circ p_i$ take the form $$\phi_i^{-1} \circ p_i: \bbS^1 \longrightarrow U \cong \bbS^1 \times (-1,1), \ \ \ \phi_i^{-1} \circ p_i(x) = (\sigma_i(x), h_i(\sigma_i(x)),$$
 where $\sigma_i : \bbS^1 \longrightarrow \bbS^1 $ are diffeomorphisms with $\sigma_i \to id$ smoothly, and $h_i : \bbS^1 \longrightarrow (-1,1)$ are smooth functions converging smoothly to zero.  Essentially, what we are saying here is that the image of a function that smoothly approximates $x \mapsto (x,0)$ is the graph of a smooth function approximating $x \mapsto 0$.
 
 Pick a smooth function $b : (-1,1) \longrightarrow \RR$  such that $b(t)=0$ if $|t|> 2/3$, and $b(t)=1$ if $|t|<1/3$. Then
$$A_i : \bbS^1 \times (-1,1) \longrightarrow \bbS^1 \times (-1,1), \ \ A_i(x,t) = \big ( x+b(t)(\sigma_i(x)-x), t + b(t) h_i(\sigma_i(x))\big )$$
is a sequence of smooth self-maps of $U$ converging smoothly to the identity. Moreover, $A_i=id$ outside of $\bbS^1 \times (-2/3,2/3) \subset U,$ and we have $$A_i(x,0) = (\sigma_i(x),h_i(\sigma_i(x))) =  \phi_i^{-1} \circ p_i(x).$$ 
So, if we set $\psi_i = \phi_i \circ A_i$, then we will have $$\psi_i \circ p (x) = \phi_i \circ A_i(x,0) = \phi_i \circ \phi_i ^{-1} p_i(x) = p_i(x),$$
which means that $\psi_i$ takes the constant speed parametrization  $p$ of $c_\infty$ to  the constant speed parametrizations $p_i$ of the geodesics $c_i$. As we still have $\psi^* g_i \to g$ smoothly on $U$, so we are done. 

The case when $c$ is a degenerate geodesic is  exactly the same, except that we parametrize $U$ as $\bbS^1 \times [0,1) / (x,0) \sim (i(x),0)$, where $i :  \bbS^1  \longrightarrow \bbS^1 $  is a reflection, instead of as $\bbS^1 \times (-1,1) $. 
\end{proof}

\begin{lemma}[Straightening near the basepoint]\label{lem: straightening}
Suppose that $X_i$ are hyperbolic $2$-orbifolds, $i=1,2,\ldots, \infty$ and $p \in X^\infty$ is a regular point. Let $U$ be a neighborhood of $p$, let $p_i\to p$ in $U$, and assume that we have maps 
$\phi_i : U \longrightarrow X_i$
such that $\phi_i^* g_i \to g_\infty$, where $g_i$ is the Riemannian metric on $X_i$. Then  there are maps $$\psi_i : U \longrightarrow X_i$$ such that $\psi_i^* g_i \to g_\infty$ and where the following conditions hold:
\begin{itemize}
\item we have $\psi_i(p_i)=\phi_i(p_i)$, and also $d\psi_i = d\phi_i$ at $p_i$,
\item $\psi_i = \phi_i$  except on some fixed neighborhood of $ p$ whose closure is contained in $U$,
    \item  there is some fixed neighborhood of $p$ on which each $\psi_i$ is an isometric embedding.
\end{itemize}
  Moreover, the same statement is true if all the $X_i$ are Euclidean orbifolds.
\end{lemma}
\begin{proof}
 Since everything here is local, it suffices to just set $X_i=\HH^2$ for all $i$.  Changing coordinates independently in the codomains of all the $\phi_i$, we can also assume that $\phi_i(p_i)=p_i$ for all $i$, and that $d\phi_i |_{p_i}=id$ for all $i$. Since $p_i\to p$, Lemma \ref{lem: convergingmaps} then implies that $\phi_i \to id$ smoothly. For each $i$, let $\phi_i^t$ be the straight line homotopy from $\phi_i^0= id$ to $\phi_i^1 = \phi_i$. Note that $\phi_i^t$ is smooth in both variables, and converges smoothly as $i\to \infty$ to the constant identity homotopy.
 
  Shrinking $U$ if necessary, let's assume $U = B(p,\epsilon)$ for some $\epsilon>0$. Choose a smooth bump function $b : [0,\epsilon] \longrightarrow \RR$ with $b(t)=0 $ for $|t|<1/3\epsilon$ and $b(t)=1$ for $|t|>2/3\epsilon$. For $q \in U$, set $$\psi_i(q) := \phi_i^{b(d(p,q))}(q).$$ 
   Then on $B(p,\epsilon/3)$ we have $\psi_i\equiv id$, while outside $B(p,2\epsilon/3)$ we have $\psi_i \equiv \phi_i$. Since $\psi_i\to id$ smoothly, we are done.
   The Euclidean case is exactly the same.
 \end{proof}

\section{Finite volume orbifolds}
\label{sec: finitevolumesec}
 
 In this section we study the subspace of $\Sub(G)$ corresponding to finite volume vectored hyperbolic $2$-orbifolds, which is a union of vectored versions of the classical moduli spaces of hyperbolic structures.
 
 \subsection{Vectored moduli spaces} \label{sec: vecmodsec}
Let $S$ be an oriented hyperbolizable $2$-orbifold with finite topological type. As in the introduction, we set
$$\Sub_S(G) := \{ \ \Gamma \in \Sub_d(G) \  \ | \ \ \Gamma \backslash \HH^2 \cong S \text{ and has finite volume} \ \}.$$
As in \S \ref{sec: convergence section}, after fixing a base vector $v_{base} \in T^1 \HH^2$ and an orientation on $\HH^2$, discrete groups can be viewed through their quotients as vectored oriented hyperbolic $2$-orbifolds:
$$\Sub_S(G) \longleftrightarrow \{ \ (X,v) \ \ | \ \ X \text{  oriented, finite volume and homeomorphic to } S, v \in T^1X \ \} \modsim,$$
where $(X,v)\sim (X',v')$ if there is an orientation preserving (hereafter, o.p.) isometry $f: X \longrightarrow X'$ with $df(v)=v'.$ Here,
$ \Gamma \in \Sub_S(G)$ corresponds to the vectored orbifold $(\Gamma \backslash \HH^2,v)$, where $v$ is the projection to $T^1(\Gamma \backslash \HH^2)$ of  $v_{base},$ and the group that maps to $(X,v)$ is denoted $\Gamma(X,v).$ As in \S \ref{sec: convergence section}, the Chabauty topology on $\Sub_S(G)$ corresponds to the smooth topology on the corresponding set of vectored orbifolds. 

When $S$ is compact, $\Sub_S(G)$ is open in $\Sub(G)$. To see this, note that the set of discrete groups is open and for any sequence $(X_i,v_i)$ limiting smoothly to a vectored orbifold of the form $(X,v)$, with $X$ compact, the almost isometric maps from $X$ to $X_i$ are diffeomorphisms as soon as their domains are balls with radius bigger than the (finite) diameter of $X$. When $S$ is noncompact, though, $\Sub_S(G)$ is neither open nor closed.

The space $\Sub_S(G)$ is a vectored version of the classical moduli space 
\[\mathcal M(S) := \big \{ \text{ oriented hyperbolic $2$-orbifolds } X \text{ homeomorphic to } S \big \} / \text{o.p.\ isometry}.\]
Just as in the surface case, $\mathcal M(S)$ can be considered as an orbifold that is the global quotient of the Teichm\"uller space $\mathcal T(S)$, which we discussed in  \S \ref{sec: fenchelnielsen}, by the \emph{mapping class group} \[\mathrm{Mod}(S) = \big \{\ \text{o.p.\ homeomorphisms } f: S \longrightarrow S \ \big \} / \text{ isotopy},\]
which acts properly discontinuously on $\mathcal T(S)$ from the right via $[(X,h)] \cdot [f] = [(X,h \circ f)]$. Here, we defined a point in Teichm\"uller space as a pair $(X,h: S \longrightarrow X)$, where $X$ is a priori an unoriented orbifold, but if we always equip $X$ with the pushforward of some fixed orientation on $S$, then the map from Teichm\"uller space to moduli space is just the projection $[(X,h)] \mapsto [X]$.  Let
\[\pi_{\Sub}: \Sub_S(G) \longrightarrow \mathcal M(S), \ \ \pi_{\Sub}(\Gamma) = [\Gamma \backslash \HH^2],\]
so that with respect to the vectored orbifold notation above, we have $\pi_{\Sub}\big (\Gamma(X,v) \big ) = [X]$. Note that because of the equivalence relation of pointed isometry, the preimage $\pi_{\mathrm{Sub}}^{-1}([X])$ of a point $[X]\in \mathcal M(S)$ is homeomorphic to the quotient  $\mathrm{Isom}^+(X) \backslash T^1 X$, which is a $3$-manifold. 

\begin{proposition}\label{prop: orbibundle}
For all $S$ as above, the natural map $\pi_{\Sub} : \Sub_S(G) \longrightarrow \mathcal M(S)$ is a fiber orbibundle with fiber $T^1S$, and $\Sub_S(G)$ is a $6g+2(k+l)-3$ dimensional manifold, where $g$ is the genus of $S$, and $k$ is the number of cusps, and $l$ is the number of cone points.
\end{proposition}

Recall that the definition of a fiber orbibundle requires that there are local trivializations of the form $$U \times T^1 S / \Delta \longrightarrow \Sub_S(G),$$ where $U$ is the domain of an orbifold chart for $ \mathcal M(S)$ and $\Delta$ is a finite group. In particular, the set theoretic fibers $\pi_{\mathrm{Sub}}^{-1}([X])$ are indeed allowed to be  homeomorphic to quotients of $T^1 S$ rather than to $T^1 S$ itself.

\begin{proof}
This is proved using an argument inspired by the Bers construction of the universal curve over Teichm\"uller space, compare with \cite[Proposition 6.9.3]{hubbard}. Intuitively, for every $[X] \in U$, the fiber $\pi_{\Sub}^{-1}([X])$ can be identified with $\mathrm{Isom}^+(X) \backslash T^1 X$. So, if we are going to show that $\pi_{\Sub}$ is a fiber orbibundle, we need to identify all the unit tangent bundles $T^1 X$ with a fixed space, $T^1 S$, in a way that varies continuously in the appropriate sense. To do this, we need to work in a canonical homeomorphism from $S$ to each $X$. While there are many ways to do this, we choose to take the `Douady-Earle map', defined as follows.

Fix a finite volume hyperbolic structure on $S$ for reference. Given a pair $(X,h)$ representing a point in $\mathcal T(S)$, choose any quasiconformal homeomorphism $f : S \longrightarrow X$ that is homotopic to the marking $h$. Lifting $f$ to universal covers we get a quasiconformal
$\tilde f : \tilde S \longrightarrow \tilde X,$
which extends continuously to a quasi-symmetric map $\partial \tilde f : \partial \tilde S \longrightarrow \partial \tilde X$ of Gromov boundaries, see e.g.\ \cite[Corollary 4.9.4]{hubbard}. Let $$\tilde F : \tilde S \longrightarrow \tilde X $$ be the Douady-Earle extension of $\tilde f$, see \cite{douady1986conformally} or \S 5.1 of \cite{hubbard}. This $\tilde F$ is a smooth quasiconformal map that is equivariant with respect to the actions of the deck groups of the two universal covers, so it descends to a smooth quasiconformal homeomorphism
\begin{equation} F_{(X,h)} : S \longrightarrow X. \label{eq: douadyearle}\end{equation}
In fact, $F_{(X,h)}$ depends only on $X$, the homotopy class of $h$, the fixed hyperbolic metric on $S$, and not on any of the additional choices above. To see this, note first that the lift $\tilde f$ is $\rho$-equivariant with respect to some isomorphism $\rho$ between the deck groups of $\tilde S,\tilde X$ that is in the outer automorphism class given by the marking $h$. So, $\partial \tilde f$ takes the attracting fixed point of any hyperbolic type deck transformation $\gamma$ of $\tilde S$ to the attracting fixed point of $\rho(\gamma)$. Since such fixed points are dense in $\partial \tilde S$, the continuous map $\tilde f$ is determined by $\rho$. But $\rho$ is uniquely defined up to conjugation by a deck transformation of $\tilde X$, so $\partial \tilde f$ is well-defined up to post-composition by a deck transformation. It follows that $\tilde F$ is well-defined up to a similar post-composition, so $F_{(X,h)}$ is well defined.

With the Douady-Earle map in hand, we can now define the fiber bundle structure on $\Sub_S(G)$. Let $\pi_{\mathcal T} : \mathcal T(S) \longrightarrow \mathcal{M}(S)$ be the quotient map and choose an orbifold chart for $\mathcal M(S)$ of the form
\begin {equation} \label{eq: chart} \pi_{\mathcal T}|_{U_{\mathcal T}} : U_{\mathcal T} \longrightarrow U_{\mathcal T}/\Delta \cong U \subset \mathcal M (S),\end{equation}
where $\Delta < \mathrm{Mod}(S)$ is a finite subgroup. Let $U_{\mathrm{Sub}}  = \pi_{\Sub}^{-1}(U)$ and define a map
$$\Psi : U_{\mathcal T} \times T^1S \longrightarrow U_{\mathrm{Sub}}, \ \ \ \Psi \big (\, [(X,h)],\, v\,  \big ) = \Gamma\big (\, X, \, dF_{(X,h)}(v)\, \big  ) ,$$
where as in \S \ref{sec: convergence section} the expression $\Gamma(\cdot)$ indicates the subgroup of $G$ corresponding to the given vectored hyperbolic $2$-orbifold. Note that $\Psi$ is well-defined, since if $f : X \longrightarrow Y$ is an isometry and $h: S \longrightarrow X$ is a marking, we have $F_{(X,f\circ h)} = f \circ F_{(X,h)}$. We also define a right group action of $\Delta < \mathrm{Mod}(S)$ on $ U_{\mathcal T} \times T^1 S$ via
\[ \big ([(X,h)],v \big ) \cdot [\delta] = \Big ( [(X,h\circ \delta)], \, dF_{(X,h \circ \delta)}^{-1} \circ dF_{(X,h)}(v) \Big ), \]
for $[\delta]\in \Delta$. Note that the map $\Psi$ above factors through the quotient by this action. Also, the action is free, since if $([(X,h)],v)$ is fixed by $[\delta]$ then there is an isometry $f : X \longrightarrow X$ with $h \circ \delta$ homotopic to $f \circ h$, so $F_{(X,h \circ \delta)} = F_{(X,f \circ h)} = f \circ F_{(X,h)}$. But then 
\[ F_{(X,h \circ \delta)}^{-1} \circ F_{(X,h)} = F_{(X,h)}^{-1} \circ f \circ F_{(X,h)} ,
\]
is a conjugate of an isometry by a diffeomorphism. Since the derivative of a nontrivial orientation preserving isometry of a hyperbolic $2$-orbifold cannot fix any vector, neither can the derivative of $F_{(X,h \circ \delta)}^{-1} \circ F_{(X,h)}$, which contradicts our assumption that $[\delta]$ fixes $([(X,h)],v)$.

The map $\Psi$ and the group action $\Delta \curvearrowright U_{\mathcal T} \times T^1S$ are both continuous. Essentially, the reason for this is that the Douady-Earle extension of a homeomorphism on $\partial \HH^2$ varies continuously with the map, as do all of its derivatives, see \cite[Proposition 2]{douady1986conformally}. More carefully, one can regard Teichm\"uller space as the quotient \[\mathcal T (S) \cong \mathrm{Isom}(\HH^2) \backslash \mathcal {DF}(\pi_1 S) ,\] where  $\mathcal {DF}(\pi_1 S)$ is the space of discrete, faithful representations of $\pi_1 S$ and $\mathrm{Isom}(\HH^2)$ acts by conjugation. Fix a representation $\iota : \pi_1 S \hookrightarrow G$ giving our fixed hyperbolic structure on $S$. Given any other discrete, faithful representation $\rho$, we can pick a quasiconformal map $\tilde f : \HH^2 \longrightarrow \HH^2 $ conjugating $\iota $ to $\rho$, and the boundary extension $\partial \tilde f  :  \partial \HH^2 \longrightarrow \partial\HH^2$ will be a quasi-symmetric homeomorphism that depends only on $\rho$, for reasons discussed earlier in the proof. Given a point $[(X,h)] \in \mathcal U_{\mathcal T} \times T^1 S$, there is a unique $\rho$ projecting to $[(X,h)]$ such that the associated map $\tilde f$ fixes $0,1,\infty \in \partial \HH^2$, and it is not hard to see that $\tilde f$ varies continuously with $[(X,h)]$, with respect to the compact open topology on maps of $\partial \HH^2$. The Douady-Earle extensions $\tilde F$ of these $\tilde f$ and their derivatives then also vary continuously with $[(X,h)]$, by \cite[Proposition 2]{douady1986conformally}. But now take $v\in T^1 S$ and lift $v$ to a point $v_{base} \in T^1\HH^2$ under the covering map $T^1\HH^2 \longrightarrow \iota(\pi_1 S) \backslash T^1\HH^2 \cong T^1 S$. The group $\Gamma([(X,h)],v)$ is then obtained by conjugating $\iota(\pi_1 S)$ first by the map $\tilde F$ above, and then by the unique element of $G$ taking $dF(\tilde v)$ to our fixed base vector $v_{base} \in T^1 \HH^2$. Since both $F$ and $dF(\tilde v)$ vary continuously with $([(X,h)],v)$, so does $\Gamma([(X,h)],v)$. The argument for continuity of the $\Delta$-action is similar. 

We now claim that the induced map 
\begin {equation}\label{eq: quotientmapp} U_{\mathcal T} \times T^1 S / \Delta \longrightarrow U_{\Sub}
\end{equation}
is a homeomorphism. To see that it is injective, suppose that \begin{equation} \label{eq: sammme} \Psi([(X,h)],v) = \Psi([(X',h')],v').\end{equation}
Since $[X]=[X']$ in $\mathcal M(S)$ by the definition of our chart in \eqref{eq: chart} it follows that $[(X',h')]=[(X,h\circ \delta)]$ for some $[\delta] \in \Delta$, so let's just assume $X=X'$ and $ h'=h \circ \delta$. Now by \eqref{eq: sammme}, the vectored orbifolds $(X,dF_{(X,h)}(v))$ and $(X,dF_{(X,h\circ \delta)}(v'))$ are isometric, so we can pick an isometry $f : X \longrightarrow X$ such that $$ dF_{(X,h)}(v) =  df \circ dF_{(X,h\circ \delta)}(v') = dF_{(X,f \circ h\circ \delta)}(v') = dF_{(X, h \circ (h^{-1} f h \delta)}(v').$$ But as the map $\phi := h^{-1} f h \delta$ fixes $[(X,h)] \in \mathcal U_{\mathcal T}$, it also lies in $\Delta$, and then we have $([(X,h)],v) \cdot [\phi]=([(X',h')],v')$ as required.  Surjectivity and properness are fairly immediate, so we leave them as an exercise, and $U_{\Sub}$ is metric, so it follows that the map in \eqref{eq: quotientmapp} is a homeomorphism. 

The maps $\Phi$ as above, together with the groups $\Delta$ and their actions, thus form a system of local trivialization for the fiber orbibundle $\pi_{\Sub} : \Sub_S(G) \longrightarrow \mathcal M(S)$. Each $U_{\mathcal T}$ is $6g + 2(k+l) - 6$ dimensional, c.f.\ Theorem \ref{thm: fenchelnielsen}, and $T^1 S$ is $3$-dimensional, so since the action of each group $\Delta$ on $U_{\mathcal T} \times T^1 S$ is free, $\Sub_S(G)$ is a $6g + 2(k+l) - 3$ dimensional manifold. \end{proof}

\begin{remark} \label{rem: repsremark} Another way to understand $\Sub_S(G)$ is to view it as a quotient of an appropriate representation space. Let $\mathcal R(S)$ be the \emph{representation variety} of all `type preserving' representations $\pi_1 S \longrightarrow G,$ i.e.\ representations such that peripheral elements of $\pi_1 S$ map to parabolic elements of $G $. Consider the subsets $$\mathcal {DF}^+(S) \subset \mathcal {DF}(S) \subset \mathcal R(S),$$ where $\mathcal {DF}(S) $ is the set of all discrete, faithful representations, and  $\mathcal {DF}^+(S) $ consists of those representations that are also orientation preserving. Here, $\rho \in \mathcal {DF}(S)$ is \emph{orientation preserving} if any homeomorphism $S \longrightarrow \rho(\pi_1 S) \backslash \HH^2$ in the homotopy class determined by $\rho$ is orientation preserving. Then $\mathcal {DF}(S)$ is an open subset of $\mathcal R(S)$, c.f.\ \cite[Theorem 1.1]{bergeron2004note}, where we consider the latter the topology of pointwise convergence. Work of Weil \cite{weil1964remarks} shows\footnote{It is a bit difficult to extract this result from Weil's paper \cite{weil1964remarks}. The point is to follow \S 7. While \S 7 is only written for compact $S$, this is just for simplicity of exposition, and everything written there works basically verbatim for noncompact $S$ of finite type as long as one works with type preserving representations. In fact, the crucial estimate is his equation (19) on pg 155 --- this calculates the dimension of the space of `parabolic cocycles', which can be identified with the tangent space of $\mathcal R(S)$ --- and this estimate is actually written to apply in the noncompact case. It takes a bit of decoding to understand that this estimate gives $6g+2k+2l-3$, but we leave this as an exciting puzzle for the reader.}  that each $\rho \in \mathcal{DF}(S)$ is a smooth point of the representation variety $\mathcal R(S)$, of dimension $6g+2(k+l)-3$, so it follows that $ \mathcal{DF}(S)$ is a manifold of that dimension. To understand $\Sub_S(G)$, then, let $\mathrm{Aut^+_{tp}}(\pi_1 S)$ be the group of orientation preserving\footnote{Recall that every type preserving automorphism of $\pi_1 S$ is induced by a homeomorphism of $S$, and an automorphism is \emph{orientation preserving} if any such homeomorphism is.} type preserving automorphisms of $\pi_1 S$ and consider the commutative diagram 
\[
\begin{tikzcd}
\mathcal {DF}^+(S) \arrow[rr] \arrow[d, ]& & \Sub_S(G) = \mathcal {DF}^+(S) / \Aut^+_{tp}(\pi_1 S)  \arrow[d,]  \\
\mathcal T(S) \cong G \backslash \mathcal {DF}^+(S)  \arrow[rr]  & & \mathcal M(S) \cong G \backslash \mathcal {DF}^+(S) / \Aut^+_{tp}(\pi_1 S).
\end{tikzcd}
\]
The action of $\Aut^+_{tp}(\pi_1 S) $ on $\mathcal {DF}^+(S)$ is properly discontinuous and free, so the top map is a covering map, which gives another proof that $\Sub_S(G)$ is a manifold of the dimension indicated in Proposition \ref{prop: orbibundle}. Note also that the action of $G$ on $\mathcal {DF}^+ (S) $ is proper and free, so the map on the left is a principal $G$-bundle. The content of Proposition \ref{prop: orbibundle} is the map on the right is a fiber orbibundle, and of course the map on the bottom is an orbifold covering map.

The commutative diagram above can be used to describe the fundamental group of $\Sub_S(G)$. Namely, since $\mathcal {DF}^+(S)$ is a $G$-principal bundle over $\mathcal T(S)$, which is contractible, the long exact sequence for homotopy groups of a fiber bundle gives that $\pi_1 {DF}^+(S) \cong \pi_1 G \cong \ZZ$, where the generators come from parametrizations of the subgroup $\bbS^1 \cong \PSO(2)\subset G$. Since the top map in the diagram is a covering map, we get
$$1 \longrightarrow \ZZ \longrightarrow \pi_1 \Sub_S(G) \longrightarrow \Aut^+_{tp}(\pi_1 S)  \longrightarrow 1,$$
and by the Dehn-Nielsen-Baer Theorem \cite{primer},  we have
$$1 \longrightarrow \pi_1 S \longrightarrow \Aut^+_{tp}(\pi_1 S)  \longrightarrow \mathrm{Mod}(S) \longrightarrow 1.$$
One could also derive the same result by using Proposition \ref{prop: orbibundle} and the long exact sequence for homotopy groups of fiber \emph{orbi}bundles, assuming one could find such a reference, as then one would have that $\pi_1 \Sub_S(G)$ is an extension of $\pi_1 T^1S$ by $\mathrm{Mod(S)}$, while $\pi_1 T^1S$ is an extension of $\ZZ$ by $\pi_1 S$.
\end{remark}

\begin{example}\label{ex: sphere with 3 points example}
Suppose that $S$ is a finite type hyperbolic $2$-orbifold with minimal complexity, i.e.\ a sphere with 3 cones points or cusps. Then $\mathcal M(S)$ is a point, so 
$ \Sub_S(G) $ is homeomorphic to the quotient $ \mathrm{Isom}^+(S) \backslash T^1 S \cong T^1 ( \mathrm{Isom}^+(S) \backslash S)$. Here, $Y=\mathrm{Isom}^+(S) \backslash S$ is another minimal complexity orbifold, so $\Sub_S(G)$ is a Seifert fibered space. We determine its homeomorphism type as a lens space depending on the orders of the cones points in Proposition \ref{prop: lensprop}. As a particular example, let $S$ be the thrice punctured sphere. The unique complete hyperbolic structure on $S$ is obtained by doubling an ideal hyperbolic triangle. Since the two copies of this triangle in $S$ are separated by the three geodesics joining the cusps of $S$, they are permuted by any isometry of $S$. From this, it is easy to see that there are $6$ orientation preserving isometries of $S$, and that the quotient is isometric to the modular orbifold. It is well-known, c.f.\ \cite[\S 3.1]{ghys2007knots}, that the unit tangent bundle of the modular orbifold is homeomorphic to the trefoil complement in $\bbS^3$.  
\end{example}

\subsection{Boundaries and ends}
\label{boundaries and ends}

Let $\partial \Sub_S(G) := \overline{\Sub_S(G)} \setminus \Sub_S(G)$ be the topological boundary of $ \Sub_S(G) \subset \Sub(G)$.  Our first result in this section is that the boundary can be explicitly characterized. To this end, let 
\begin{align*}
        \mathcal A &:= \{A(\alpha) \ | \ \alpha \subset \HH^2 \text{ a geodesic}\}\\
        \mathcal A' &:= \{A'(\alpha) \ | \ \alpha \subset \HH^2 \text{ a geodesic}\}\\
        \mathcal N &:= \{N(\xi) \ | \ \xi \in \partial \HH^2\},
    \end{align*}
    where the groups $A(\alpha) ,A'(\alpha), N(\xi)$ above are the elementary hyperbolic, hyperbolic and order $2$ elliptic, and parabolic subgroups of $G$ defined in \S \ref{sec: sub elementary}, respectively. We show:
    
\begin{proposition}\label{prop: boundarybewhat}
    Suppose that $S$ is a hyperbolizable $2$-orbifold with finite topological type. Then the boundary $\partial \Sub_S(G)$ is the union of all $\Sub_T(G)$, where $T$ is a hyperbolizable $2$-orbifold that embeds as an essential proper suborbifold in $S$, together with some of the sets
    $\mathcal A,\mathcal A', \mathcal N$ above. Moreover, \begin{enumerate}
        \item $\mathcal A \subset \partial \Sub_S(G)$ if and only if $S$ admits an essential non-peripheral simple closed curve that does not bound a disk with two order $2$ cone points,
        \item $\mathcal A' \subset \partial \Sub_S(G)$ if and only if $S$ has at least two order $2$ cone points,
        \item $\mathcal N \subset \partial \Sub_S(G)$ if and only if $S$ is not a sphere with three cone points.
    \end{enumerate} 
\end{proposition}

We will need the following lemma.

\begin{lemma}[Limits when $\inj\to 0$]
Suppose that $(X_i,v_i) $ is a sequence of vectored hyperbolic $2$-orbifolds, where $v_i \in (T^1X_i)_{p_i}$, and that $\inj(p_i) \to 0$. Let $\Gamma_i = \Gamma(X_i,v_i)$. After passing to a subsequence, one of the following statements applies.
\begin{enumerate}
\item For all $i$, the vector $v_i$ is contained in a component of the thin part of $X_i$ that is a neighborhood of a cusp. The CCW angle from the vector $v_i$ to the geodesic through $p_i$ that exits the cusp converges to some $\theta \in [0,2\pi)$. Here, the groups $\Gamma_i = \Gamma(X_i,v_i)$ converge in $\Sub(G)$ to the elementary parabolic group $N(\xi)$, where the CCW angle from the base vector $v_{base} \in T^1\HH^2$ to the geodesic that terminates in the parabolic fixed point $\xi$ is the limiting angle $\theta$.
  \item For all $i$, the vector $v_i$ is contained in a component of the thin part of $X_i$ that is a neighborhood of an admissible geodesic $\gamma$. The distance $d(p_i,\gamma)$ converges to some value $d \in [0,\infty]$, and the CCW angle from the vector $v_i$ to the vector at $p_i$ that points along the shortest path to $\gamma$ also converges to some $\theta \in [0,2\pi)$. Here, if $d=\infty$ then $\Gamma_i \to N(\xi)$, the group defined as in 1.\ for the current $\theta$. If $d<\infty$, the $\Gamma_i$ converge to either $A(\alpha)$ or $A'(\alpha)$, depending on whether $\gamma$ is regular or degenerate, where $\alpha$ is the geodesic in $\HH^2$ whose distance from the base point $p_0$ of our base vector $v_0\in T^1 \HH^2$ is $d$, and where the CCW angle from $v_0$ to the shortest path from $p_0$ to $\alpha$ is $\theta$. 
\end{enumerate}\label{lem: thinlimits}
\end{lemma}
\begin{proof}
In the first statement, the groups $\Gamma_i$ converge to the element $N(\xi) \in \mathcal N$ where the CCW angle from the base vector $v_0 \in T^1\HH^2$ to the geodesic that terminates in the parabolic fixed point $\xi$ is the limiting angle $\theta$. This is proved as follows. Since the distance from $p_i$ to the boundary of the thin part goes to infinity, Lemma \ref{lem: same limit} implies that the $\Gamma_i$ converge to the same limit  as the groups $\Gamma_i'=\Gamma(X_i',v_i')$, where $X_i' \longrightarrow X_i$ is the annular cover corresponding to the component of the thin part in which $v_i$ lies, and $v_i' \in T^1 X_i'$ is the lift of $v_i$ that lies in the thin part of the cover. But the $\Gamma_i'$ are just cyclic groups of parabolics, and since their generators translate the base point of the fixed vector $v_0\in T^1 \HH^2$ by an amount tending to zero, the $\Gamma_i'$ converge to a continuous group of parabolics, which is easily seen to be the one identified above. The proof of the second statement is similar, and we leave it as an exercise for the reader. 
\end{proof}

\begin{proof}[Proof of Proposition \ref{prop: boundarybewhat}]
We first show that every element of $\partial \Sub_S(G)$ lies in one of the sets $\mathcal N, \mathcal A, \mathcal A'$ or $\Sub_T(G)$ above. Suppose that $\Gamma_i = \Gamma(X_i,v_i)$ is a sequence in $\Sub_S(G)$ with $\Gamma_i \to \Gamma \in \partial \Sub_S(G)$. Let $\inj (p_i)$ be the infinite order injectivity radius at the basepoint $p_i$ of $v_i$, as discussed in \S \ref{sec: injrad and thick-thin}. After passing to subsequence, we can assume that either $\inj(p_i)\to 0$ or $\inf \inj(p_i) >0$. In the first case, Lemma \ref{lem: thinlimits} shows that $\Gamma$ lies in either $\mathcal N, \mathcal A, $ or $\mathcal A'$. So, we may assume the second case holds.

When $\inf \inj(p_i) \geq \epsilon >0$, the limit $\Gamma = \lim_i \Gamma_i$ is discrete. Indeed, the assumption says that the amount that the basepoint $p_0 \in \HH^2$ is translated by any infinite order element of $\Gamma_i$ is at least $\epsilon$, and since all the $\Gamma_i$ are isomorphic to $\pi_1 S$, there is a fixed upper bound to the order of any elliptic element in any $\Gamma_i$. Hence, every nontrivial element of $\Gamma_i$ must translate the base vector $v_0\in T^1 \HH^2$ by some amount that is bounded below independent of $i$, which implies that the limit group $\Gamma$ is discrete.

Write $\Gamma=\Gamma(X,v)$, so that $(X_i,v_i) \to (X,v)$ smoothly. Since the volumes of the $X_i$ are all the same, the orbifold $X$ has finite volume. (Indeed, if it had infinite volume it would admit compact sets of arbitrarily large volume, which could be embedded in the approximating orbifolds.) Take some compact core $Y \subset X$ whose boundary components have length less than $1$. Since $(X_i,v_i) \to (X,v)$ smoothly, there are almost isometric embeddings $\phi_i : \nhd _2(Y) \longrightarrow X_i$ for large $i$, where $\nhd_2(Y)$ is the $2$-neighborhood in $X$. The images $\phi_i(Y)$ are \emph {essential} suborbifolds of $X_i$. Indeed, if $\gamma \subset \partial Y$ is a boundary component, for large $i$ the curve $\phi_i(\gamma)$ has length less than $1$ and does not contract to a point within its $1$-neighborhood (which is contained in $\phi_i(\nhd_2(Y))$). This implies that $\phi_i(\gamma)$ is essential in $X_i$, for otherwise we could lift it to $\HH^2$ and contract it to a point within its $1$-neighborhood using a coning construction. So, $(X,v) \in \Sub_T(G)$, where $T$ is a hyperbolizable $2$-orbifold that embeds as an essential proper suborbifold in $S$.

\medskip

We now show that if $T$ is a hyperbolizable $2$-orbifold that embeds as an essential proper suborbifold in $S$, then \emph{every} element of $\Sub_T(G)$ lies in $\partial \Sub_S(G)$. The easiest way to do this is to fix an embedding $T \subset S$ and extend a generalized pants decomposition of $T$ to a generalized pants decomposition of $S$. (See \S \ref{sec: pants sec}.) Pick transversals in $T$ for all the curves in our pants decomposition of $T$, and pick transversals in $S$ for the remaining curves in the pants decomposition of $S$. Given a finite volume $2$-orbifold $Y$ marked with a homeomorphism $h: T \longrightarrow Y$, we can extend the Fenchel Nielsen coordinates (see \S \ref{sec: fenchelnielsen}) for $(Y,h)$ to a set of Fenchel Nielsen coordinates for $S$, while taking the lengths of the components of $\partial T$ to zero. This produces a sequence of marked hyperbolic $2$-orbifolds $(X_i, h_i : S \longrightarrow X_i)$. For each $i$, let $Y_i \subset X_i$ be the suborbifold with geodesic boundary in the homotopy class of $h_i(T)$, and pick a vector $v_i \in T^1 Y_i$ that lies in the thick part of $Y_i$. Then as in the previous direction of the proposition, after passing to a subsequence the sequence $(X_i,v_i)$ converges smoothly to some finite volume vectored hyperbolic $2$-orbifold $(X,v)$.

If $K \subset X$ is any compact core, then for large $i$ there are almost isometric embeddings $$\phi_i : K \longrightarrow X_i$$ sending $v$ to $v_i$. Since $\length(\partial Y_i) \to 0$ and $v_i$ lies in the thick part of $Y_i$, we have $d(v_i,\partial Y_i) \to \infty$. Hence, for large $i$ the image $\phi_i(K) \subset Y_i$. The components of $\phi_i(\partial K)$ are essential in $Y_i$, by the same argument as in the previous direction of the proposition. Moreover, there is a universal lower bound for the length of any nonperipheral geodesic in $Y_i$, since any such geodesic must either intersect or be one of the chosen generalized pants curves, which all have fixed lengths independent of $i$. So, if $K$ is chosen large enough so that its boundary components are very short, the components of $\phi_i(\partial K)$ must be peripheral in $Y_i$, implying that $\phi_i( K)$ is a compact core for $Y_i$. In particular, $X$ is homeomorphic to $T$ and there is a natural marking $f : T \longrightarrow X$ that one can produce as a composition $$T \overset{h}{\longrightarrow} h(T) \cong Y_i \cong \phi_i(K) \overset{\phi_i^{-1}}{\longrightarrow} K \cong X,$$
where the $\cong$ are given by isotopies or by homeomorphisms between an orbifold and a compact core. With respect to this marking, the Fenchel-Nielsen coordinates of $(X,f)$ are exactly those of $(Y,h)$, the two marked orbifolds are equivalent. Hence, we have shown that for any finite volume $Y$ homeomorphic to $T$, there is a sequence $(X_i,v_i) \in \Sub_S(G)$ converging smoothly to $(Y,v)$ for some $v\in T^1 Y$. To realize an \emph{arbitrary} vector $w\in T^1 Y$ as the basevector of a limit from $\Sub_S(G)$, just pull back $w$ under the  almost isometric maps given by the convergence $(X_i,v_i)\to (Y,v)$ to get a new sequence of base vectors for the $X_i$.

Finally, we verify the last three statements of the proposition. For 1, suppose $S$ admits an essential, nonperipheral simple closed curve $\gamma$ that does not bound a disk with two order $2$ cone points. Extend $\gamma$ to a generalized pants decomposition of $S$ and choose a sequence of Fenchel Nielsen coordinates where $\ell(\gamma) \to 0$. This produces a sequence of orbifolds $X_i$ homeomorphic to $S$ with regular geodesics $\gamma_i$ whose lengths tend to zero. By Lemma \ref{lem: thinlimits}, if we base the $X_i$ at vectors $v_i \in (T^1X_i)_{p_i}$ in the Margulis tubes around the $\gamma_i$ and choose the distances $d(p_i,\gamma_i)$ and the vectors $v_i$ appropriately, we can arrange that $(X_i,v_i)$ converges to any desired element of $\mathcal A$. It is also clear from Lemma \ref{lem: thinlimits} that the only way to have a sequence in $\Sub_S(G)$ limit to an element of $\mathcal A$ is if the base vectors are chosen inside the Margulis tubes around short regular geodesics. So, 1 follows.

Assertion 2 is almost the same. One just has to note that there is a degenerate geodesic in $S$ exactly when there are a pair of order $2$ cone points.

For 3, note that whenever $S$ has an essential nonperipheral simple closed curve $\gamma$, we can realize every element of $\mathcal N$ as a limit of a sequence $(X_i,v_i)$ in $\Sub_S(G)$. Namely, use Fenchel-Nielsen coordinates to create an orbifold $X_i$ in which $\gamma$ is homotopic to a geodesic $\gamma_i$ of length $\frac 1i$, and take the base vector $v_i\in (T^1X_i)_{p_i}$ to lie in the corresponding Margulis tube, making sure that $d(p_i,\gamma_i)\to \infty$ but $\inj (p_i) \to 0$. By Lemma \ref{lem: thinlimits}, a subsequential limit of such a sequence always exists, and is contained in $\mathcal N$, and by varying the angle the vector $v_i$ makes with the shortest path to $\gamma_i$, any element of $\mathcal N$ can be so realized. 

So, if $\mathcal N \not \subset \Sub_S(G)$, there can be no nonperipheral simple closed curves on $S$, implying that $S$ is a sphere with $3$ total cone points and punctures. If $S$ has punctures, one can realize every element of $\mathcal N$ by choosing some hyperbolic $X \cong S$ and then taking a sequence of base vectors exiting a cusp of $X$, using argument similar to those in the rest of this proof. So, the only way we can have $\mathcal N \not \subset \Sub_S(G)$ is if $S$ is a sphere with three cone points. And for such $S$, $\Sub_S(G)$ is compact (see Example \ref{ex: sphere with 3 points example}) so certainly $\mathcal N$ is not in the boundary of $\Sub_S(G)$. 
\end{proof}

 In {low complexity} cases one can completely describe the topology of $\overline {\Sub_S(G)}$. Thanks are due to Tali Pinksy for telling us about the case of the modular curve. Ideas that are somewhat similar to those we use in the proof appear in \S 3.3 of her paper \cite{pinsky2014templates}.  We denote lens spaces by $L(p,q)$.

\begin{proposition}\label{prop: lensprop}
Let $S $ be a sphere with three total cone points and cusps. Then $\overline {\Sub_S(G)} = \Sub_S(G) \cup \mathcal N$, and the possible homeomorphism types of the closure are as follows.
\begin{enumerate}
    \item If $S$ is compact, $\overline {\Sub_S(G)} = \Sub_S(G) \cong T^1 S / \mathrm{Isom}^+(S)$.
    \item If $S$ has one cusp and two cone points of orders $n\neq k$, then $$\overline {\Sub_S(G)} \cong L(nk-n-k,n-1).$$ Regarding $\Sub_S(G) \cong T^1 S$, the group $\pi_1 \overline{\Sub_S(G)}$ is generated by either of the two singular fibers of $T^1 S \longrightarrow S$. The loop $\mathcal N \subset \overline{\Sub_S(G)}$ is homotopic to a regular fiber, and therefore to the $n^{th}$ (or $k^{th}$) power of a generator for $\pi_1 $.
    \item If $S$ has one cusp and two cone points of order $k$, then $\overline {\Sub_S(G)} \cong L(k-2,1)$. Regarding $ {\Sub_S(G)} \cong T^1 S / \mathrm{Isom}^+(S)$, the group $\pi_1 \overline {\Sub_S(G)}$ is generated by the projection of the singular fiber of $T^1 S\longrightarrow S$ to the quotient by $\mathrm{Isom}^+(S)$. Again, $\mathcal N$ is homotopic to a regular fiber, and is the $k^{th}$-power (or the square) of a generator.
    \item If $S$ has two cusps and one cone point of order $k$, then $\overline {\Sub_S(G)} \cong L(2k-2,1)$. Regarding $ {\Sub_S(G)} \cong T^1 S / \mathrm{Isom}^+(S)$, the group $\pi_1 \overline {\Sub_S(G)}$ is generated by \emph{half} the singular fiber of $T^1 S\longrightarrow S$. Here, we mean the projection to the quotient of a shortest path in $T^1S_p$, where $p$ is the cone point, from some vector $v \in T^1S_p$ to its negative $-v$. Again, $\mathcal N$ is homotopic to a regular fiber, and is the $2k^{th}$-power (or the square) of a generator.
    \item If $S$ is a thrice punctured sphere, then $\overline {\Sub_S(G)} \cong \bbS^3$.
\end{enumerate}
\end{proposition}

In particular, if $S$ is the modular curve, where $(n,k)=(2,3)$, then case 2 implies that $$\overline {\Sub_S(G)} \cong L(1,1)\cong \bbS^3.$$ Also, note that $L(p,q)\cong L(-p,q)$ and $L(p,q) = L(p,\pm q^{\pm 1})$, where inversion is taken mod $p$. Since $$(n-1)(k-1) = (nk-k-n)+1,$$ the homeomorphism type of $L(nk-n-k,n-1)$ in case 2 is indeed symmetric in $n$ and $k$, as one would expect. Moreover, note that $k(n-1) = (nk-n-k) + n$, so multiplication by $(n-1)$ gives an automorphism of $\ZZ / (nk-n-k) \ZZ \cong \pi_1 L(nk-n-k,n-1)$ that takes the congruence class of $k$ to that of $n$, implying that the statement about $\mathcal N$ is also symmetric in $k,n$. 

Here, case 2 is in some sense the important one. As we will see during the proof, the closure has the same topological type when the orbifold is $S$ as it does when the orbifold is $S / \mathrm{Isom}^+(S)$. In cases 3--5, $S$ has nontrivial isometry group, and the quotient is an orbifold as in case 2. 

\begin{proof} As the compact case is trivial, let's first assume that $S$ is a sphere with two cone points $p_k,p_n$ of distinct orders $k\neq n$, and one cusp. Then $\mathrm{Isom}^+(S) = 1$, so the natural map $T^1 S \longrightarrow \Sub_S(G), \ \ v \mapsto \Gamma(X,v)$ is a homeomorphism, as discussed in Example \ref{ex: sphere with 3 points example}. So, Proposition \ref{prop: boundarybewhat} says that $$\overline \Sub_S(G) = \Sub_S(G) \cup \mathcal N \cong T^1 S \cup \mathcal N.$$
Analyzing the proof of that proposition, one also sees that a sequence $v_i \in T^1 S$ that exits the cusp converges to a point in $\mathcal N$ exactly when the angles that the $v_i$ make with the \emph{outward direction} converge. Here, the `outward direction' is the direction of the shortest geodesic from the base point of $v_i$ to the cusp, which is well-defined for large $i$. For simplicity, we will think of elements of $\mathcal N$ as these limiting angles. 

% Pick a conformal parameterization $D \setminus \{0\} \hookrightarrow U$ of a cusp neighborhood $U \subset S$, where $D\subset \CC$ is the unit disc, circles around the origin map to horocycles, and rays coming into the origin map to geodesic rays. Then we get an induced parameterization $D \setminus \{0\} \times S^1 \longrightarrow T^1 U,$
% where the last $S^1$ coordinate records the CCW angle from the outward direction to the given vector. Pick a parametrization $S^1 \longrightarrow \mathcal N$ where $\theta\in S^1$ maps to the one parameter unipotent group stabilizing the endpoint on $\partial \HH^2$ of the geodesic ray from the base point\footnote{Recall that we have a standard choice of base point and base vector that determines the dictionary between vectored $2$-orbifolds and discrete subgroups of $\mathrm{Isom}^+(\HH^2)$.} of $\HH^2$ that is at an angle of $\theta$ CCW from the base vector. From the discussion above, these parametrizations combine to give a total parameterization
% $$D \times S^1 \longrightarrow T^1U \cup \mathcal N,$$
% where $T^1 U$ is glued to $\mathcal N$ as it is when regarded as a subset of $\Sub_S(G) \cong T^1 S$. In other words, $\Sub_S(G) \cup \mathcal N$ is a Dehn filling of $\Sub_S(G)\cong T^1 S$ where the meridians are loops in $T^1 S$ that lie over horocycles and where the angle with the outward direction is fixed.

Let us analyze more carefully how $\mathcal N$ is glued to $T^1S$ above. Topologize $S \cup \{\infty\} $ as the one-point compactification, and extend the bundle projection continuously to a map $$\pi: T^1 S \cup \mathcal N  \longrightarrow S \cup \{\infty\}.$$ It turns out $\pi$ is a fiber bundle in a neighborhood of $\infty$, but it is not necessary to work this out now. Let $\gamma $ be the bi-infinite geodesic ray on $S$ that separates the two cone points, and consider the preimage $$T := \pi^{-1}(\gamma \cup \infty).$$ This $T$ is a torus: indeed, if we trivialize the bundle $T^1 S|_{\gamma}$ so that $\gamma'$ is a constant field, then the two ends are glued with a $\pi$-rotation when we add in $\mathcal N$. The loop $\gamma \cup \infty$ divides $S \cup \infty$ into two halves. Drawing $S$ as in Figure \ref{fig: heegaard}, let us call the two halves the \emph{$k$-side} and the \emph{$n$-side}, and write $D_k, D_n$. Let us also orient $S$ so that tangent spaces on the `front' of $S$ in Figure \ref{fig: heegaard} are oriented CCW.

\begin{figure}
    \centering
    \includegraphics{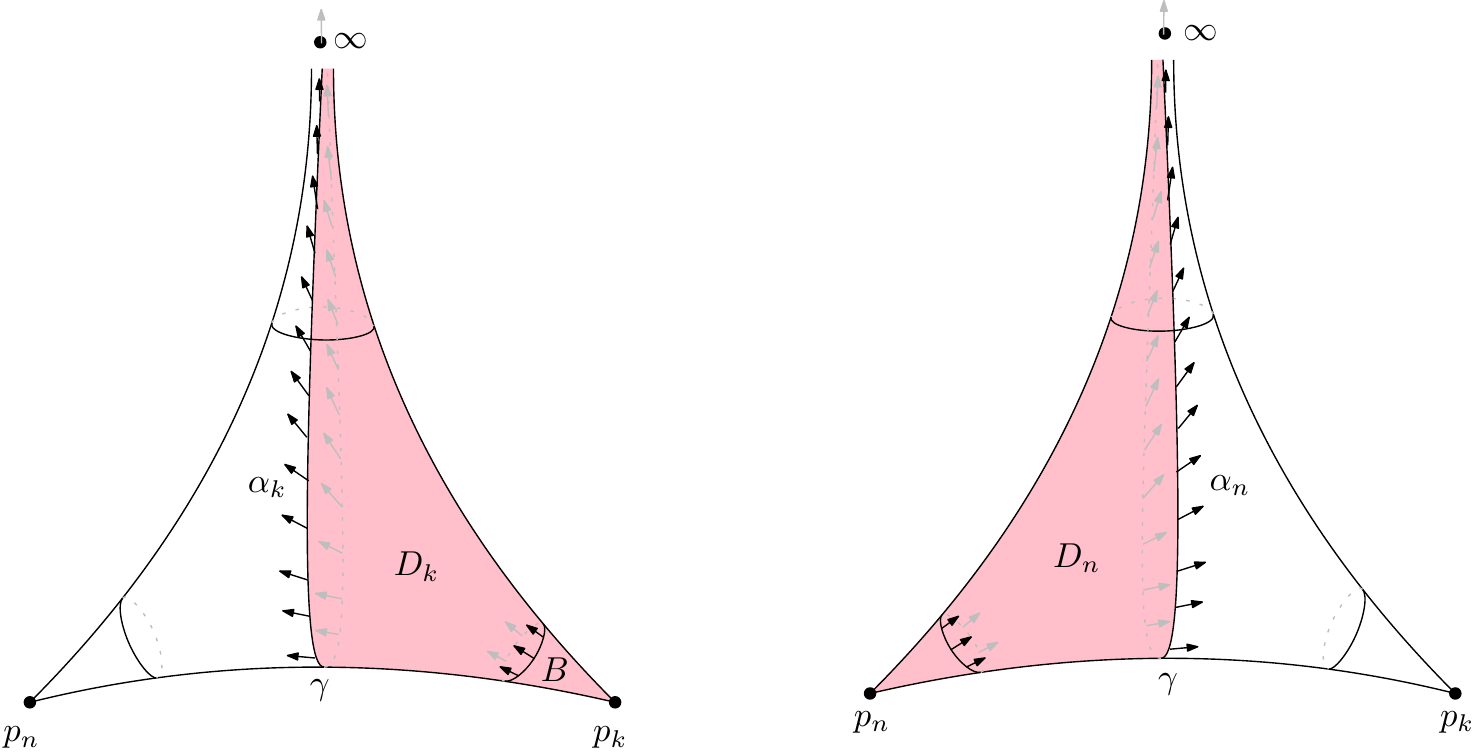}
    \caption{The subdivision of $S \cup \{\infty\}$ into $D_k$ and $D_n$ along $\gamma \cup \{\infty\}$, along with the sections $\alpha_k$ and $\alpha_n$ of $\pi$ along $\gamma \cup \{\infty\}$. Also drawn are the small metric neighborhoods $B$ around the cone points that we use in the proof, along with the outward normal fields along their boundaries, which are homotopic to $\alpha_k,\alpha_n$.}
    \label{fig: heegaard}
\end{figure}

The $k$-side $D_k \subset S$ is a closed disc with a single cone point $p_k$. For any $x\in D_k$ there is a unique geodesic from $p_k$ to $x$. The tangent fields of these geodesics combine to give a section $\sigma_k$ of the projection $\pi$ over $ {D_k} \setminus \{ p_k\}$. Here, we take the `unique geodesic from $p_k$ to $\infty$' to mean the geodesic from $p_k$ to the cusp, and we define $\sigma_k(\infty)$ to be the outward direction in $\mathcal N$. So defined, $\sigma_k$ is continuous over $D_k \setminus \{p_k\}$. It follows that $\pi^{-1}(D_k)$ is a solid torus: indeed, if $B$ is a small metric ball around $p_k$, the complement $D_k \setminus B$ is an annulus $A$, and $T^1 B$ is a solid torus that is attached along its boundary to $\pi^{-1}(A)$, which is homeomorphic to $A \times \bbS^1$ via the map that takes $(x,v)$ to $(x,\theta)$, where $\theta$ is the CCW angle from $\sigma_k(x)$ to $v$. Similarly, $\pi^{-1}(D_n)$ is a solid torus, so $T$ is a Heegaard splitting of $T^1S \cup \mathcal N$. 

To identify the homeomorphism type of the space $T^1 S \cup \mathcal N$, then, it suffices to concretely determine a pair of meridians $m_k, m_n$ on the two sides. Focusing on the $k$-side first, let us fix a convenient basis for $H_1(T)$. Give $\gamma \cup \infty = \partial D_k$ the boundary orientation, and take the first basis element to be $\alpha_k$, which is defined as the restriction of $\sigma_k$ to $\gamma \cup \infty$, oriented compatibly with $\gamma \cup \infty$. The second element is $\beta$, the $\pi$-fiber over a point of $\gamma$, with the orientation induced by that of $S$. To understand the meridian coming from the $k$-side, it is easier to replace the torus $T$ with unit tangent bundle over the boundary $\partial B$ of a small metric ball $B \subset D_k$ around the cone point $p_k$. Then within $\pi^{-1}(D_k \setminus p_k)$, the curve $\alpha_k$ is homotopic to the restriction of $\sigma_k$ to $\partial B$, i.e.\ the outward normal vector field over $\partial B$, considered as an oriented closed curve in a way compatible with the boundary orientation on $\partial B$. The ball $B$ is $k$-covered by a metric disk $\tilde B \subset \HH^2$; we equip the upper half plane $\HH^2$ with the `standard' CCW orientation and assume the covering map preserves orientation. The outward normal field over $\partial B$ is then $k$-covered by the outward normal field to $\partial \tilde B$, again oriented compatibly with the boundary orientation on $\partial \tilde B$. The meridian on $T^1 \partial B$ for the solid torus $T^1 B$ is $1$-covered (i.e., parametrized)  by a \emph{constant} vector field on $\partial \tilde B$, since the latter is homotopic to a point in $T^1 \tilde B$, and its projection is an essential curve on $T^1 \partial B$. Here, `constant' means in terms of any trivialization of $T^1 \HH^2$. But with respect to such a trivialization, the outward normal field over $\partial \tilde B $ rotates by $2\pi$ CCW, so the constant field (which covers the meridian) differs from it (which covers $k$ times the outward normal) by subtracting a copy of the CCW oriented fiber. Tracing all this back to $T$, we get that $$m_k := k \alpha_k - \beta \in H_1(T)$$ is a meridian for $D_k$. The same argument works verbatim for the $n$-side, giving a meridian 
$$m_n := n \alpha_n-\beta\in H_1(T)$$
for $D_n$, where $\alpha_n$ is defined in the same way as $\alpha_k$. To relate the curves $\alpha_k,\alpha_n$, note that the two are oriented in ways that are compatible with opposite orientations of $\gamma \cup \infty$, so it is easier to relate $\alpha_n$ and $-\alpha_k$, say. Moving along $-\alpha_k$, the vector rotates by $\pi$ clockwise in comparison to $\gamma$, while moving along $\alpha_n$ it rotates CCW by $\pi$ in comparison to $\gamma'$. So, $\alpha_n = -\alpha_k + \beta$. From this we get $$m_n = n(-\alpha_k+\beta) - \beta = -n\alpha_k + (n-1)\beta.$$
Finally, recall that $L(p,q)$ is constructed by gluing two solid tori with longitude-meridian framings together so that the meridian on one becomes the $(p,q)$-curve on the other. So, if we change bases on $T$, using instead the curves $\alpha_k$ and $m_k=k\alpha_k-\beta$, then writing
$$m_n = -n\alpha_k + (n-1)\beta = p \alpha_k + q (k\alpha_k - \beta)$$
we see that $q=1-n$ and $-n=p+kq \implies p=k(n-1)-n = kn-k-n$, so $$T^1S \cup \mathcal N \cong L( kn-k-n, 1-n) \cong L( kn-k-n, n-1)$$ as desired. Note also that the singular fiber $T^1S_{p_k}$ is homotopic to the outward normal field on $\partial B$ and hence to $\alpha_k$, which is a longitude for the solid torus $\pi^{-1}(D_n)$. So, the singular fiber $T^1S_{p_k}$ generates $\pi_1 (T^1 S \cup \mathcal N)$. Similarly, the singular fiber $T^1 S_{p_n}$ is a generator. Since $\mathcal N$ is homotopic to a regular fiber on $T^1 S$, it is the $n^{th}$ (or $k^{th}$) power of a generator.

We now reduce cases 3--5 to case 2. Suppose $S$ is an arbitrary sphere with $3$ total cone points or cusps.  As described in Example \ref{ex: sphere with 3 points example}, we can regard $\Sub_S(G) \cong T^1S / \mathrm{Isom}^+(S).$ The action of $\mathrm{Isom}^+(S)$ respects the gluing to $\mathcal N = \partial\Sub_S(G) $, so we have
$$\Sub_S(G) \cup \mathcal N \cong \Sub_{S'}(G) \cup \mathcal N,$$
where $S' = S / \mathrm{Isom}^+(S)$. Here, as $S'$ is a lower complexity orbifold with no orientation preserving isometries, it must be a sphere with two cone points of distinct orders  and one cusp, as in case 2. So, to prove cases 3--5 all one has to do is identify $S'$ for each given $S$. If $S$ has two cone points of the same order $k$ and one cusp, $\mathrm{Isom}^+(S) \cong \ZZ/2\ZZ$ and the nontrivial element fixes the cusp, exchanges the two cone points, and fixes another regular point on $S$. So, $S'$ has cone points of orders $k,2$ and one cusp, and case 3 follows from case 2, since the singular fibers of $T^1 S$ project to those of $T^1 S'$. If $S$ has two cusps and one cone point of order $k$, the only nontrivial orientation preserving isometry fixes the cone point, exchanges the two cusps, and fixes one additional regular point, so $S'$ has cone points of orders $2k,2$ and one cusp. So, case 2 proves case 4, noting that half the singular fiber of $T^1 S$ projects to the singular fiber of $T^1 S'$. Finally, if $S$ is a thrice punctured sphere, it is well known that $S'$ is the modular curve, so case 5 follows similarly. 
\end{proof}

 Let $\epsilon>0$ be smaller than the Margulis constant. If $S$ has finite volume, define the \emph{$\epsilon$-end  neighborhood} of $\Sub_S(G)$  to be the subset $\Sub_S^\epsilon(G)$ consisting of all groups $\Gamma=\Gamma(X,v) \in \Sub_S(G)$ such that for the corresponding vectored orbifolds $(X,v)$, either
\begin{enumerate}
	\item the \emph{systole} of $X$, i.e. the length of the shortest admissible geodesic, is less than $\epsilon$, or 
\item the vector $v$ lies in a component of the $\epsilon$-thin part of $X$ that is a neighborhood of a cusp.
\end{enumerate}
\noindent Note that these two cases are not mutually exclusive. Also, in 2 it is not really necessary to say that the component of the thin part is a cusp neighborhood, since if it is a Margulis tube then 1 holds. However, stating it as above reflects better the way that we think of membership in $\Sub_S^\epsilon(G)$. It is easy to verify that $\Sub_S^\epsilon(G)$ is open in $\Sub_S(G)$, that the complement $\Sub_S(G) \setminus \Sub_S^\epsilon(G)$ is compact, and that the complementary compact subsets exhaust $\Sub_S(G)$ as $\epsilon\to 0$. 

\begin{proposition}[Ends of vectored moduli spaces]\label{prop: oneend}
Suppose that $S$ is a finite type hyperbolizable $2$-orbifold and let $\epsilon>0$ be smaller than the Margulis constant. 
\begin{enumerate}
\item If $S$ is not a sphere with 3 or 4  cone points, the $\epsilon$-end neighborhood $\Sub_S^\epsilon(G)$ is path connected, and $\Sub_S(G)$ has one end.
    \item If $S$ is a sphere with three  cone points, then as described in Example \ref{ex: sphere with 3 points example}, $\Sub_S(G)\cong T^1 S / \mathrm{Isom}^+(S)$, which is compact, and then $\Sub_S^\epsilon(G)$ is empty for all small $\epsilon$. 
    \item If $S$ is a sphere with four cone points, pick a partition $P$ of the orders of the cone points into two sets of two. For instance, if the orders are $2,2,3,5$ then there are two possible partitions, $\{\{2,2\},\{3,5\}\}$ and $\{\{2,3\},\{2,5\}\}.$ Let $\mathcal C_P \subset \Sub_S^\epsilon(G)$ be the subset consisting of all $\Gamma(X,v)$ such that such that there is a closed geodesic on $X$ of length less than $\epsilon $ that separates $X $ into two pieces that induce the partition $P$ of cone point orders. Then for small $\epsilon $, these subsets $\mathcal C_P$ are the path components of $\Sub_S^\epsilon(G)$, and the number of ends of $\Sub_S(G)$ is the number of partitions, as in Proposition \ref{prop: endofmoduli}.
    % if in the following table distinct letters represent distinct orders of cone points, we have
%     \begin{center}
% \begin{tabular}{ c|c } 
% Orders & \# Ends  \\ 
%  \hline
%  a,a,a,a & 1  \\ 
%  a,a,a,b & 1  \\ 
%  a,a,b,b & 2 \\ 
%  a,a,b,c & 2 \\
%  a,b,c,d & 3
% \end{tabular}
% \end{center}
\end{enumerate}
\end{proposition}

\begin {proof}
First, assume that $S $ has a cusp, but is not a sphere with $3$ total cone points or cusps. We will show that $\Sub_S^\epsilon(G)$ is path connected. Since $S$ has a cusp, any vectored orbifold $(X,v)$ where $X$ has systole less than $\epsilon$ can be connected within $\Sub_S^\epsilon(G)$ to a vectored orbifold where the vector lies in a cuspidal component of the $\epsilon$-thin part, just by dragging $v$ into a cusp of $X$. So, it suffices to show how to connect two points of $\Sub_S^\epsilon(G)$ that are of the second type in its definition.

So, pick $\Gamma(X_0,v_0), \Gamma(X_1,v_1) \in \Sub_S^\epsilon(G)$ and assume that both $v_i$ lie in components of the $\epsilon $-thin parts that are neighborhoods of cusps $c_i$ of $X_i$. Choose markings 
$$h_i : S \longrightarrow X_i, \ \ i=0,1$$
such that $c_0,c_1$ are the $h_i$-images of a single cusp $c$ of $S$. (Here, we are using that the orientation preserving mapping class group acts transitively on the cusps of $S$.) Choose a path $(X_t,h_t) \in \mathcal T(S)$ from $(X_0,h_0)$ to $(X_1,h_1)$, say using Fenchel-Nielsen coordinates. Fixing a hyperbolic structure on $S$, let $$F_t : S \longrightarrow X_t$$ be the Douady-Earle map homotopic to $h_t$, as defined in \eqref{eq: douadyearle} during the proof of Proposition \ref{prop: orbibundle}. As in that proof, the Douady-Earle maps vary smoothly with $t$, and the map $$[0,1] \times T^1S \longrightarrow \Sub_S(G), \ \ (t,v) \mapsto \Gamma(X_t,dF_t(v))$$ is continuous. 

We claim that there is some neighborhood $U \subset S$ of the cusp $c$ such that for all $t\in [0,1]$ and all $p\in U$, we have $\inj(X_t, F_t(p)) < \epsilon.$ Indeed, let $T_t \subset X_t$ be the $\epsilon$-thin part. If our claim fails, we have $t_i\in [0,1]$ and a sequence $p_i\in S$ that exits the cusp $c$ such that for all $i$, we have $F_{t_i}(p_i) \not \in T_{t_i}$. Passing to a subsequence, suppose $t_i \to t$, and choose a neighborhood $V \subset S$ of $c$ such that $\overline V \subset F_t^{-1}(T_t)\subset S$. Since the $v_i$ exit $S$, after discarding finitely many $i$, we can assume that $p_i \in V$ for all $i$. But each $F_{t_i}^{-1}(T_{t_i})$ is an annular neighborhood of the cusp $c$ that does not contain $p_i$, so $F_{t_i}^{-1}(T_{t_i})$ cannot contain $\partial V$. Hence, we can pick points $q_i \in \partial V$ such that $F_{t_i}(q_i) \not \in T_{t_i}$ for all $i$. Passing to a subsequence, we have $q_i\to q \in \partial V$. But then by definition of $V$, we have $\inj (X_t,F_t(q)) < \epsilon$, so continuity of injectivity radius implies that $\inj_{X_{t_i}}(F_{t_i}(q_i)) < \epsilon$ for all large $i$, contradicting our assumption that $F_t(q_i) \not \in T_{t_i}$.

Returning to our original goal, pick some vector $w\in T^1 U$, where $U$ is as in the previous paragraph. The image $F_0(U)$ is an annular neighborhood of the cusp $c_0$ that lies in the $\epsilon$-thin part of $X_0$, so it is contained in the component of the thin part that is a neighborhood of $c$. So, there is a path in $T^1X_0$ from $v_0$ to $F_0(w)$ that lies over that component. Similarly, construct a path from $v_1$ to $F_1(w)$. Regarding these as paths of vectored orbifolds, where the underlying orbifold happens to be fixed, we get paths in $\Sub_S^\epsilon(G)$. And if we concatenate these paths on either side of the path $$t \longmapsto \Gamma(X_t,dF_t(w)),$$
we get a path in $\Sub_S^\epsilon(G)$ joining $
\Gamma(X_0,v_0)$ to $\Gamma(X_1,v_1)$ as desired.

\medskip

The case where $S$ is a sphere with three cone points is handled as in Example \ref{ex: sphere with 3 points example}. Namely, after equipping $S $ with a hyperbolic structure, we have $\Sub_S(G) \cong T^1 S / \mathrm{Isom}^+(S).$ So as $S$ is compact, so is $\Sub_S(G)$, and we are done. 

\medskip

Suppose now that $S $ is compact, but is not a sphere with three cone points. The projection \\
$\pi : \Sub_S(G) \longrightarrow \mathcal M(S)$ is a fiber orbibundle with connected fibers, and compactness of $S$ implies that $\Sub_S^\epsilon(G)$ is the preimage of $\mathcal M^\epsilon(S)$, since option 2 in the definition of $\Sub_S^\epsilon(G)$ is never relevant. So, the desired statement follows directly from Proposition \ref{prop: endofmoduli}.
\end{proof}

Finally, for the convenience of the reader we mention the following statement, which is one of the main results of this paper, and which is important for the calculations in the next section.

\begin{theorem}\label{thm: deformation retract thm}
			Suppose $S$ has finite topological type. Then  for small $\epsilon$, there  is a deformation retraction in $\overline{\Sub_S(G)}$
$$\Sub_S^\epsilon(G) \cup \partial \Sub_S(G) \longrightarrow \partial \Sub_S(G).$$ 
\end{theorem}

The proof of Theorem \ref{thm: deformation retract thm} is given in \S \ref{sec: retractsec}, and uses our work on grafting from \S \ref{sec: graftingsec}. The proof does not rely on anything from \S \ref{sec: trivialS}, so we make use of the result in the next section .

\subsection{Loops in the closure of $\Sub_S(G)$}

\label{sec: trivialS}

Let $S$ be a finite type orientable hyperbolizable $2$-orbifold, and fix a finite volume hyperbolic metric on $S$. In Proposition \ref{prop: boundarybewhat}, we saw that the closure of $\Sub_S(G) \subset \Sub(G)$ was obtained by attaching the vectored moduli spaces $\Sub_T(G)$ corresponding to lower complexity orbifolds $T$ that embed in $S$ with geodesic boundary, and some of the sets $\mathcal A, \mathcal A', \mathcal N$ of elementary subgroups. Here, we investigate more closely how all these spaces are glued together and calculate the fundamental group.

\begin{theorem}\label{thm: vankampen} Suppose that either a four-punctured sphere or a once-punctured torus embeds in $S$ as the interior of a surface with geodesic boundary. Then $\overline {\Sub_S(G)}$ is simply connected.
\end{theorem}

\begin{proof} 
Let $M := \overline{ \Sub_S(G). }$ By Proposition \ref{prop: boundarybewhat}, $M$ is the union of all $\Sub_T(G)$, where $T$ embeds as the interior of a suborbifold of $S$ with geodesic boundary, together with the sets $\mathcal A,\mathcal N$ of elementary groups, and possibly $\mathcal A'$ if $S$ has at least two cone points of order 2. 

For each finite type orbifold $T$, let  $\cplx(T)$ be the \emph {complexity} of $S$, defined as the complex dimension of the Teichmuller space of $T$, or alternatively $3g-3+n$, where $n$ is the total number of cone points and cusps. Let $M_\cplx \subset M$ be the union of all the elementary groups in $M$, together with all $\Sub_T(G)$ where $T$ has complexity at most $\cplx$. So, we have a filtration $ M_0 \subset \cdots \subset M_{\cplx(S)}=M.$
Note that $M_0$ is the union of $\mathcal N, \mathcal A,$ possibly $\mathcal A'$, and all $\Sub_T(G)$ where $T$ is a sphere with $3$ total cone points and cusps that embeds in $S$.

Our proof of Theorem \ref{thm: vankampen} will follow from three claims:
\begin{enumerate}
        \item[(a)] For any $T$ as above with complexity $\cplx \geq 1$, the inclusion $\Sub_T(G) \hookrightarrow M_\cplx$ is trivial on $\pi_1$.
    \item[(b)] For each $\cplx$, the group $\pi_1 M_\cplx$ is normally generated by the image of $\iota_*: \pi_1 M_{\cplx-1} \longrightarrow \pi_1 M_\cplx$, where $\iota : M_{\cplx-1} \longrightarrow M_\cplx$ is the inclusion.
    \item[(c)] There is a generating set for $\pi_1 M_0$ consisting of loops that are all freely homotopic in $M$ into $\Sub_T(G)$ for some $T$ as above with complexity $\xi \geq 1$.

\end{enumerate}
Assuming all three claims, the theorem is proved as follows: repeated applications of (b) imply that $\pi_1 M$ is normally generated by $\pi_1 M_0$. By (c) there is a generating set whose elements are freely homotopic into spaces $\Sub_T(G)$, and hence are all trivial by (a). Hence $\pi_1 M$ is trivial.

\medskip

Let us work on (a) first. It suffices to construct a generating set for $\pi_1 \Sub_T(G)$ consisting of loops that are null-homotopic in $M_\cplx$. As discussed in Remark \ref{rem: repsremark}, we have the horizontal short exact sequence, which fits together with another vertical short exact sequence 
\begin{center} 
\begin{tikzcd} 
& & & 1 \arrow[d] & \\
& &  & \pi_1( T) \arrow[d] & \\
1 \arrow[r] & \ZZ \arrow[r] & \pi _1 \Sub_T(G) \arrow[r] & \Aut^+_{tp}(\pi_1 T) \arrow[d] \arrow[r] & 1\\
& & & \mathrm{Mod} (T)  \arrow[d] & \\
& & & 1 & \\
\end{tikzcd} 
\end{center} 
%
%$$1 \longrightarrow \pi_1 T^1 T \longrightarrow \pi_1 \Sub_T(G) \longrightarrow \mathrm{Mod}(T) \longrightarrow 1,$$
%and there is a further short exact sequence
%$$1 \longrightarrow \ZZ \longrightarrow \pi_1 T^1 T \longrightarrow \pi_1 T \longrightarrow 1,$$
where $\pi_1 T$ is the orbifold fundamental group of $T$ and $\ZZ$ is generated by a regular fiber of $T^1 T$.  Our goal is to use these short exact sequences to give generators for $\pi_1 \Sub_T(G)$, and to show that these generators are all nullhomotopic in $\overline{\Sub_T(G)}$. Namely, we construct generators by starting with a generator for $\ZZ$, then adding in any elements of $\pi_1 \Sub_S(G)$ that project to generators for $\pi_1(T) \subset \Aut^+_{tp}(\pi_1 T)$, then adding in any elements of $\pi_1 \Sub_S(G)$ that project to generators for $\mathrm{Mod}(T)$. 
%\begin{itemize} 
%\item elements of $\pi_1(T^1_p T) \cong \ZZ$ are given by rotating the vector, we show this is homotopic to $\mathcal{N}$
%\item elements of $\pi_1(T)$ are given by following simple closed curves on $T$, we show these become nullhomotpic in $\overline{\Sub_T(G)}$
%\item elements of the mapping class group are given by Dehn twists and half Dehn twists around curves in $T$, we show these also become nullhomoptopic in $\overline{\Sub_T(G)}$.
%\end{itemize} 

First, a generator for the subgroup $\ZZ  \subset \pi_1 \Sub_T(G)$ is obtained by fixing a hyperbolic orbifold $X$ homeomorphic to $T$, a point $p\in X$, and considering the loop of all $\Gamma(X,v)$ with $v\in TX_p$. To contract this loop in $M_\cplx$, pinch some curve on $X$ to a cusp, homotoping our loop to a loop of the form $\Gamma(Y,v)$ with $v\in TY_q$, where $Y$ is an orbifold with a cusp that embeds in $T$. Moving the basepoint $q$ out the cusp homotopes the new loop to the loop $\mathcal N $ of one-parameter parabolic groups. By our assumptions on $S$, the thrice punctured sphere $S_{0,3}$ embeds in $S$ as the interior of a suborbifold with geodesic boundary. Hence, $\overline{\Sub_{S_{0,3}}(G)}$ is a subset of $M_\cplx$. But Proposition \ref{prop: lensprop} says that $\overline{\Sub_{S_{0,3}}(G)} \cong \bbS^3$ and that it contains $\mathcal N$, so $\mathcal N$ is homotopically trivial in $M_\cplx$. This implies that our original generator for $\Sub_T(G)$ is homotopically trivial in $M_\cplx$.\footnote{We have chosen to be less formal with our homotopies here to improve readability. Of course, if one wanted to write down all the details here, one would have to say (for instance) how the basepoint moves when we pinch to get a cusp. One way to approach this is to use the Douady-Earle maps from Proposition \ref{prop: orbibundle}, but we leave the details to the reader.}

Next, we describe loops in $\Sub_S(G)$ that project to generators of $\pi_1(T)$. Let $T'$ be the surface obtained by removing all the cone points of $T$, then $\pi_1 T' \twoheadrightarrow \pi_1 T$, so we get that $\pi_1 T$ is generated by simple, closed, essential, nonperipheral curves on $T'$. If $\gamma$ is such a curve, we can fix a hyperbolic orbifold $X\cong T$, and regard $\gamma$ as a parameterized geodesic $\gamma : [0,1] \longrightarrow X$, where $\gamma(0)=\gamma(1)$. Then the loop $$\Gamma_t = \Gamma(X,\gamma'(t)), \ \  t \in [0,1]$$ in $\Sub_T(G)$ projects to an associated generator for $\pi_1 \Sub_S(G)$. To contract this loop within $M_\cplx$, just pinch $\gamma$. If we parameterize the resulting loops appropriately, then we get a homotopy from the original loop to the \emph{element} of $\mathcal A$ (or $\mathcal A'$ if $\gamma$ is degenerate) that leaves invariant the axis through the fixed basevector for $\HH^2$ used in the group/orbifold correspondence. So, we have our nullhomotopy.

Finally, $\mathrm{Mod}(T)$ is isomorphic to the subgroup of  $\mathrm{Mod}(T')$ consisting of elements that permute only cusps and cone points with the same order. So by results of Dehn and Lickorish \cites{dehn, lickorish} on mapping class groups of surfaces, $\mathrm{Mod}(T)$ is generated by Dehn twists around essential, nonperipheral simple closed curves on $T'$, and half twists around essential simple closed curves on $T'$ that bound a disk in $T$ with either two punctures or two cone points of the same order. Fixing again a curve $\gamma$ on $T'$, we now describe a loop in $\Sub_T(G)$ that projects to the Dehn twist around $\gamma$. Fix a hyperbolic orbifold $X \cong T, $ regard $\gamma$ as a geodesic on $X$, and create a path $\Gamma_t:=\Gamma(X_t,v_t)$, where $X_t$ as follows. Cut $X$ along $\gamma$ to create an orbifold $X \dsm \gamma$ with two boundary components, fix a vector $v$ tangent to one of the two boundary components, let $X_t$ be the orbifold obtained by gluing the two components with a twist by $t \cdot \length(\gamma)$, so that $X_1$ is isometric to $ X_0$, and let $v_t$ be the projection of $v$. This loop in $\Sub_T(G)$ can be contracted to a point in $M_\cplx$ by pinching $\gamma$, just as in the previous case. The case of a half Dehn twist is similar.

\medskip

We now turn to part (b). Note that $$M_\cplx =  M_{\cplx-1} \cup \bigcup_T \Sub_{T}(G) ,$$ where $T$ ranges over the finitely many orbifolds with complexity $\cplx$ that embed in $S$ as the interior of an orbifold with geodesic boundary. Fixing some small $\epsilon>0$, consider now the open cover
 $$\left \{ \ M_{\cplx-1} \cup \bigcup_{T \text{ as above }} \Sub_T^\epsilon(G) \right \} \cup \Big \{ \Sub_T(G) \ | \ T \text{ as above } \Big \} \ $$ of $M_\cplx$. The $\Sub_T(G)$ are all path connected and disjoint, and each intersects the first element of the cover in $\Sub_T^\epsilon(G)$, which is path connected by Proposition \ref{prop: oneend}. So it follows from Van Kampen's Theorem that $\pi_1 M_\cplx$ is generated by the fundamental groups of the elements of the cover. In light of part (a), this implies that the fundamental group of the first element of the cover surjects onto $\pi_1 M_\cplx$. But Theorem~\ref{thm: deformation retract thm} says that for each $T$ there is a deformation retraction $$\Sub_T^\epsilon(G) \cup \partial \Sub_T(G) \longrightarrow \partial \Sub_T(G) \subset M_{\cplx-1},$$ so combining all these deformation retractions, we get that every loop in the first element of the cover is homotopic into $M_{\cplx-1}$. Hence, $\pi_1 M_{\cplx-1}$ normally generates $\pi_1 M_\cplx$ as required. This proves (b).
 
 \medskip

We now prove part (c). First, note that $M_0$ is the union of $\mathcal A, \mathcal N, $ possibly $\mathcal A'$, and the closures $\overline {\Sub_T(G)}$, where $T$ is a sphere with 3 total cone points and cusps, and at least one cusp, that embeds in $S$. By Lemma \ref{prop: lensprop}, all these spaces are lens spaces glued along $\mathcal N$, which is connected, so we can make a generating set for $\pi_1 M_0$ by picking a generator for the (cyclic) fundamental group of each space. 

First, $\pi_1 \mathcal A$ is generated by the loop $A(\alpha_t), \ t\in [0,\pi]$, where $\alpha_t \subset \HH^2$ is the geodesic through our basepoint for $\HH^2$ that makes a CCW angle of $t$ with our basevector. By our assumption on $S$, either a four punctured sphere or a punctured torus embeds in $S$. Let $T$ be a four-punctured sphere, and assume $T$ embeds, so that $\Sub_T(G) \subset M$. Equip $T$ with the hyperbolic structure obtained by doubling a regular ideal quadrilateral. Let $\gamma$ be a simple closed geodesic obtained by doubling the common perpendicular to a pair of opposite sides of the quadrilateral. Rotating the quadrilateral by $\pi$ gives an orientation preserving involutive isometry $f$ of $T$ that has two fixed points on $\gamma$. Let $p$ be one of these fixed points and let $v_t$ be a shortest path in $T^1 T_p$ from some $v_0$ to $v_1 = -v_0$.  Then the path $\Gamma_t:=\Gamma(T,v_t), \ \ t\in [0,1]$ is actually a loop in $\Sub_T(G)$, since $df(v_0)=v_1$, and is homotopic in $M$ to the loop $A(\alpha_t)$ above, just by pinching the geodesic $\gamma$ while keeping the basevectors on the geodesic. The case where a punctured torus embeds in $S$ is similar; we just take a hyperbolic structure corresponding to a opposite-side gluing of a regular ideal quadrilateral, rotate the quadrilateral by $\pi$ around its center to get an involution, and look at the loop corresponding to half the tangent space to the center.

Similarly, $\pi_1 \mathcal A'$ is generated by the loop $A'(\alpha_t),$ with $\alpha_t $ as in the previous paragraph. If $\mathcal A' \subset M$, then $S$ has two order two cone points $p,q$. Pick a hyperbolic structure on $S$ and join these cone points by a geodesic segment $\gamma$. The double $2\gamma$ of $\gamma$ is a degenerate admissible geodesic on $S$. By pinching $2\gamma$, we get a homotopy in $M$ from the loop $\Gamma_t := \Gamma(S,v), \ \ v\in T^1S_p$ (essentially, a copy of the singular fiber over the order two cone point $p$) to our loop $A'(\alpha_t)$. 

Next, let $T$ be a sphere with two cone points of orders $n, k$, and one cusp, that embeds in $S$. (It is fine here if $n=k$.) By Proposition \ref{prop: lensprop}, $\overline{\Sub_T(G)} \cong T^1 S \cup \mathcal N$ is a lens space, and its fundamental group is generated by either of the singular fibers in $T^1 S$, say the order $k$ cone point. But by degenerating $S$ to $T$ via pinching an appropriate curve, we get a homotopy from a loop $\Gamma_t := \Gamma(S,v), \ \ v \in TS_p$, where $S$ is hyperbolized and $p$ is a cone point of order $k$, to our generator for $\overline{\Sub_T(G)} $.

Finally, let $T$ be a sphere with a cone point of order $k$ and two cusps, and assume that $T$ embeds in $ S $. Here, Proposition \ref{prop: lensprop} implies that $\pi_1 \overline{\Sub_T(G)}$ is generated by half the singular fiber. The assumptions on $S $ imply that there is some orbifold $R$ with $\Sub_R(G) \subset M$ such that $R$ is either
\begin{itemize}
    \item a sphere with four cusps and an order $k$ cone point,
    \item a torus with one cusp and one order $k$ cone point.
\end{itemize}
Let's focus on the former case for a moment. Hyperbolize $R$ as the double of a hyperbolic pentagon $P$ with 4 ideal vertices and one vertex of angle $\pi/k$, in such a way that $P$ is symmetric about the bisector of the $\pi/k$ angle. There is 1 degree of freedom in this construction; for instance, if the non-ideal vertex is at the origin in the disk model and the angle bisector ray is the positive vertical axis, the pentagon is specified once either of the two vertices that are not adjacent to the non-ideal vertex is chosen on $\partial \HH^2$. The reflection through this angle bisector induces an orientation preserving involution of $R$ that exchanges the two copies of $P$, and fixes the order $k$ cone point. Varying the parameter in the construction of the hyperbolic metric on $R$ pinches a pair of geodesics that are exchanged by the involution, so that the resulting orbifolds converge smoothly to $T$ when the basepoints are chosen to be the singular points. So, a loop in $\Sub_R(G)$ consisting of half the singular fiber will then be homotopic in $M$ to our generator for $\Sub_T(G) \subset M_0$. The case of the torus with one puncture and a cone point is similar.
\end{proof}

\section{Grafting}
\label{sec: graftingsec}
In this section we discuss \emph{(conformal) grafting} of  hyperbolic orbifolds, see e.g.\ \cite{mcmullen}, \cite{hensel} and \cite{bourque} in the case of surfaces, and study how it interacts with the smooth topology. 

Let $X$ be a hyperbolic orbifold and let $\mathcal S$ be a  proper collection of pairwise disjoint simple closed geodesics, possibly degenerate. Here, `proper' means that the inclusion  into $X$ of the disjoint union  of all the geodesics is a proper map, or  equivalently, that only finitely many of the geodesics intersect a given compact subset of $X$. In later sections, we will often replace properness by the stricter assumption that the lengths of all the curves in $\mathcal S$ are universally bounded. Pick a function
$$L: \mathcal S \longrightarrow [0,\infty].$$ From this data, we construct a new orbifold $X_{L}$ equipped with a piecewise hyperbolic/Euclidean metric, which can be uniformized  using Theorem \ref{thm: uniformization} to create a hyperbolic orbifold $X_{L}^{hyp}$. Below, we are going to careful specify all the gluing maps in our construction, in order to set notation for \S \ref{sec: retractsec}.

For each geodesic $c\in \mathcal S$, choose first a fixed isometric parametrization $$p_c : \bbS^1_{\ell(c)} \longrightarrow X,$$ where $\ell(c)$ is the hyperbolic length of $c$ and $\bbS^1_{\ell(c)}$ is a circle with length $\ell(c)$.  If $c$ is  \emph{regular}, then $p_c$ is an embedding. If $c$ is \emph{degenerate} then there is a reflection $$i : \bbS^1_{\ell(c)} \longrightarrow \bbS^1_{\ell(c)}$$ such that $p_c(i(x))=p_c(x)$ for all $x\in \bbS^1_{\ell(c)}$, and where $p_c$ takes the two fixed points of $i$ to distinct order two singular points of $X$.

Next, it will be convenient below to supplement $\mathcal S$ with  the set $\mathcal S_\pm $ of all pairs $(c,\sigma)$, where $c\in \mathcal S$ and $\sigma$ is a \emph{side} of $c$, i.e.\  a component of $\nhd(c) \setminus c$, where  $\nhd(c)$  is a regular neighborhood of $c$. Instead of using the notation of pairs, though, we will write elements of $\mathcal S_\pm $ as $c_+$ or $c_-$, where $c$ is the underlying geodesic. When $c$ is a regular geodesic and ${c}_+$ is a choice of side, the other side will be denoted as $c_-$, but in general $c_+$ can refer to an arbitrary side, as there is no fixed sign assignment. The function $L$ induces a function $ \mathcal S_\pm  \longrightarrow [0,\infty]$, which we also call $L$. This perspective will allow us to treat degenerate and regular geodesics more similarly, shortening the proofs in \S \ref{sec: retractsec}.

For each $c_+\in \mathcal S_\pm $, we now define  a Euclidean orbifold $C_{c_+}$ as follows. 
\begin{itemize}
\item If 	$L(c)=\infty$, set $C_{c_+}$ to be the half-infinite Euclidean annulus $$C_{c_+} := \bbS^1_{\ell(c)} \times  [0,\infty).$$
\item  If $L(c)<\infty$, set $C_{c_+}$ to be either $$\bbS^1_{\ell(c)} \times [0,L(c)] \ \ \text{ or } \ \ \ \bbS^1_{\ell(c)} \times  [0,L(c)] \ \Big / \ (x,0) \sim (i(x),0),$$
 depending on whether $c$  is regular or degenerate.  So, in the regular case,   $C_{c_+} $  is an annulus, and in the degenerate case, $C_{c_+}$  is a disk with two order $2$  singular points. \end{itemize}

Let $X \dsm \mathcal S $ be the hyperbolic $2$-orbifold with geodesic boundary obtained by cutting $X$ along each  element of $\mathcal S $.  For each $c_+\in \mathcal S^\pm  $, we abusively write $c_+ \subset \partial (X \dsm  \mathcal S) $ to denote the boundary component that is adjacent to the preferred side of $c_+$, and we lift our parametrization $p_c$ to give a parametrization of this $c_+$ as well. 

\begin {definition}[The piecewise grafted surface]\label{def: graftingdef}
 Create a piecewise hyperbolic/Euclidean orbifold $X_{L}$ by gluing together $X \dsm \mathcal S $ and the generalized Euclidean annuli $C_{c_+}$ as follows.
\begin {itemize}
\item If $L(c) < \infty$, glue $C_{c_+}$ to  the  corresponding boundary component $c_+ \subset \partial (X \dsm  \mathcal S) $ via the map $$(x , L(c)) \in \partial C_{c} \longleftrightarrow p_c(x) \in  c_+ \subset \partial (X \dsm \mathcal S).$$
 Moreover, if $c$ is regular and $c_+,c_- \in \mathcal S_\pm$ are $c$ equipped with its two side preference,  we also glue  $C_{c_+}$  to $C_{c_-}$   via the map $$(x , 0) \in \partial C_{c_+} \longleftrightarrow  (x,0) \in \partial C_{ c_-}.$$
\item If $L(c)=\infty$, glue  $C_{c_+}$ to  the  corresponding boundary component $c_+ \subset \partial (X \dsm  \mathcal S) $ via
$$(x,0) \in \partial C_{c_+} \longleftrightarrow p_c(x) \in c_+ \subset \partial (X \dsm \mathcal S).$$
\end {itemize}
Finally, if $c\in \mathcal S$ is regular and $c_+,c_-$ are the two side preferences, we set
$$C_c := C_{c_+} \cup C_{c_-}$$
and we set $C_c := C_{c_+}$ when $c$ is degenerate and $c_+$ is the unique side preference. So, when $L(c)<\infty$ and $c$ is regular, $C_c$ is the entire grafted in annulus around $c$. When $L(c)<\infty$ and $c$ is degenerate, $C_c$ is a  generalized annulus. And when $L(c)=\infty$, then $C_c$ is a pair of half-open annuli. Let $$\iota : X \dsm \mathcal S  \longrightarrow X_{L}$$ be the obvious inclusion\footnote{If some $L(c)=0$, then $\iota_t$ is not   injective on the boundary of $\mathcal S_\pm $}.
\end {definition}

Below, it is important to regard $X_L$ as having a tangent space, so let us be a little more careful with our gluings here. Recall that the \emph{strip model} for the hyperbolic plane is the open strip in $\CC$ defined by $|\mathfrak{Im}(z)|<\pi/2$, which we endow with the hyperbolic metric $ds = \sec(y) \sqrt{dx^2+dy^2}$.  In this model, the real axis is a geodesic and the natural parametrization has unit speed. So, given a (say, nonsingular) point in $X_L$ that lies on the gluing locus,  we can create a chart around $p$ by mapping a neighborhood of the corresponding point in $X \dsm \mathcal S$ isometrically to a half-disk in the strip model of the form $$\{ x+iy \ | \ y\geq 0, \ x^2+y^2<\epsilon\},$$
and mapping a neighborhood of the associated point in $C_{c}$ isometrically to the lower half of that disk
$$\{ x+iy \ | \ y\leq 0, \ x^2+y^2<\epsilon\},$$
 which we consider equipped with Euclidean metric. These charts together with charts for the pieces we are gluing form a smooth atlas for $X_L$.  In coordinates near the gluing locus, our piecewise metric is of the form
 $$ds = \begin{cases}
 \sec(y) \sqrt{dx^2+dy^2} & y\geq 0\\ 
 \sqrt{dx^2+dy^2} & y \leq 0,
 \end{cases}
 $$
 and therefore is $C^1$ since $\frac d{dy}\sec(y)=0$ at $y=0.$  
 
 By Theorem \ref{thm: uniformization}, there is a $C^0$-Riemannian metric $d^{hyp}$ in the conformal class of $d$ that is hyperbolic, in the sense that there is an $(Isom(\HH^2), \HH^2)$-structure on $X_L$ that is $C^1$-compatible with the original structure and where the charts are $d^{hyp}$-isometries.  Note that since $X$ is hyperbolic, it is not covered by the $2$-sphere, so Theorem \ref{thm: uniformization} applies. A priori, it could output a Euclidean metric, but this is obstructed topologically unless $X$ is a convex-cocompact generalized annulus, in which case it is obvious that the grafted surface is conformal to a hyperbolic orbifold rather than a Euclidean one.
 
 \begin{definition}[Hyperbolizing $X_L$]
 We write the orbifold $X_{L}$ as $X_{L}^{hyp}$ when it is equipped with the metric $d^{hyp}$ and the associated hyperbolic structure.  Note that since all  structures are $C^1$-compatible, the tangent space $TX_L$ is canonically identified with $TX_L^{hyp}$. We call both $X_L$ and $X_L^{hyp}$ the \emph{grafted orbifolds}.
 \end {definition}

 Here is a first result about the geometry of $X_L^{hyp}$.

\begin{claim}\label{claim: geodesic claim}
Let $c\in \mathcal S$. If $L(c)<\infty$, there is a geodesic $\hat c \subset X_{L}^{hyp}$  in the homotopy class of $c$. If $L(c)=\infty$,  the homotopy class of $c$ consists of loops that are homotopic out a cusp of $X_{L}^{hyp}$. 
\end{claim}
\begin{proof}
	On a hyperbolic orbifold, a simple closed curve  is homotopic to a geodesic unless it bounds an annulus that is conformally equivalent to a punctured disk. Since the original curve $c$ is a geodesic, the latter happens if and only if $L(c)=\infty$. \end{proof}

 More generally, however, uniformization can be quite mysterious! In the rest of this section, we  try to further understand the geometry of $X_{L}^{hyp}$, keeping our arguments as elementary and general as possible.

\begin{remark}
Note that  in this section, we are grafting  generalized annuli of length $L(c)$ along \emph{each side} of a regular geodesic $c$. The reader familiar with grafting on hyperbolic surfaces may initially think that some of our estimates below are off by a factor of two, and this is why.
\end{remark}

\subsection{Generalized Conformal Annuli} \label{sec: conformalannuli}
A $2$-orbifold $A$ is a \emph{generalized annulus} if it is either homeomorphic to an annulus $\bbS^1 \times I$, where $I $ is some (closed, open, or half-open) interval, or homeomorphic to an open or closed disk with two order $2$ singular points. In the latter case, $A$ is  homeomorphic to the quotient $$\bbS^1 \times I / (x,t) \sim (i(x),R(t)) ,$$
where $i : \bbS^1 \longrightarrow \bbS^1$ is a reflection of $\bbS^1$ and $R : I \longrightarrow I$ is a reflection through the midpoint of $I$. Note that $A$ is then double covered by the annulus $\bbS^1 \times [-1,1]$.  Often, we will write generalized conformal annuli as $$A:=\bbS^1 \times I \modsim,$$ where  the equivalence relation $\sim$ is trivial if $A$ is an annulus, and is the equivalence relation above otherwise.

As before, we let $\bbS^1_\ell$ be a circle of length $\ell$. Given an interval $I \subset \RR$, we consider the  generalized annulus $\bbS^1_\ell \times I \modsim$ with the conformal structure associated to the product metric. Its \emph{modulus} is
$$mod(\bbS^1_\ell \times I \modsim) := \length(I)/\ell.$$ By uniformization, any conformal generalized annulus $A$ is conformally equivalent to a generalized annulus $\bbS^1_\ell \times I\modsim$  as above,  and the latter is well-defined up to scale. So, we can extend the definition of modulus to all generalized conformal annuli by requiring it to be a conformal invariant. Modulus is monotonic in the sense that if $f: A \longrightarrow B$ is a conformal embedding then $mod(A) \leq mod(B)$.

Every generalized conformal annulus admits a complete hyperbolic or Euclidean metric with geodesic boundary, and this metric is unique up to isometry. On the annuli $\bbS^1_\ell \times I\modsim$,  the metric can be  described explicitly.  For example, if $\ell>0$ the annulus $\bbS^1_\ell \times (-\pi/2,\pi/2)$   admits the conformal metric
$$ds = \sec \left (  y\right ) \sqrt{dx^2+dy^2}$$
One can easily check that this metric is complete and hyperbolic, and that $\bbS^1_\ell \times 0$ is a geodesic with length $\ell$. This metric is also invariant under the  involution $(i,R)$ above, and hence gives a metric on the  quotient generalized conformal annulus $\bbS^1_\ell \times (-\pi/2,\pi/2) \modsim$. Similarly, the annulus $\bbS^1_\ell \times (0,\infty)$ admits  the conformal metric
$$ds = 1/y \sqrt{dx^2+dy^2},$$
which is complete and hyperbolic, and where the $\infty$-end is a cusp.

Note that in the examples above,
\begin{equation}
	mod(\bbS^1_\ell \times (-\pi/2,\pi/2) )  = \pi/\ell.
\label{eq: modeq}
\end{equation}
and the modulus of a degenerate annulus it covers is also $\pi/\ell$. It follows that \emph{the modulus of a complete generalized hyperbolic annulus with a core geodesic of length $\ell$ is $\pi/\ell$.}

If $A$ is a  generalized conformal annulus and $$\phi : A \longrightarrow \bbS^1_\ell \times I\modsim$$  is a uniformizing map, a subannulus $B \subset A$ is called \emph{straight} if it has the form $B = \phi^{-1}(\bbS^1_\ell \times J)$ for some subinterval $J\subset I$.  The following lemma from conformal analysis is well known. 

\begin{lemma}[see Lemma 4.4 in \cite{bourque}] \label{lem: straight} If $A$ is a conformal (generalized) annulus and $B \subset A$ is a subannulus with modulus $m \geq 1$,  there is a straight subannulus $C \subset A$ with $mod(C)=m-1$ such that $C \subset B$.
\end{lemma}

 Finally, if $A$ is a conformal annulus,  the \emph{$D$-truncation} of $A$ is obtained by deleting neighborhoods of all its ends  that are straight subannuli of modulus $D$.  For instance, the $2$-truncation of the half open annulus $\bbS^1_2 \times [0,9)$ is the subannulus $\bbS^1_2 \times [0,5)$, and the $2$-truncation of $\bbS^1_1 \times (-3,3)$ is $\bbS^1_1\times (-1,1)$.

 As a corollary of the lemma above, we have  the following two results under truncations. 

\begin{lemma}\label{lem: truncation contained}
	 Suppose that $A$ is an open annulus and $B \subset A$  is a subannulus  such that each component of $A \setminus B$ has  modulus at most $D$. Then the $(D+1)$-truncation of $A$ is  contained in $B$.
\end{lemma}
\begin{proof}
 We can assume that $mod(A) \geq 2(D+1)$, since otherwise the claim is vacuously true, so $mod(B)\geq 2$. By Lemma \ref{lem: straight}, there is a  straight subannulus $|B|_A \subset A$  that is contained in $B$ and where $$mod(|B|_A) \geq mod(B) -1.$$
 The moduli of both components of $A\setminus |B|_A$ are at most $D+1$, so the $D$-truncation of $A$  is contained in $|B|_A \subset B$.
\end{proof}

\begin {lemma}[Quasi-monotonicity of truncation]
 Suppose that $D>0$, that $A$ is an open annulus and $B \subset A$ is an open subannulus. Then the $(D+2)$-truncation of $B$ is contained in the $D$-truncation of $A$.\label{lem: quasimonotone}
\end {lemma}
\begin {proof} Let $B^{D+2}$ and $A^D$ be the respective truncations of $B$ and $A$. Note that if the modulus of $B$ is less than $2(D+2)$, the lemma is true vacuously. So, applying Lemma \ref{lem: straight}, there is a straight subannulus $|B|_A \subset A$ with $|B|_A \subset B$ and $mod(|B|_A) \geq mod(B)-1$.  The $D$-truncation $(|B|_A)^D$ of $|B|_A$ is contained in $A^D$. On the other hand, let $|(|B|_A)^D|_B \subset B$ be a straight subannulus that is contained in $(|B|_A)^D$ with modulus at least $mod((|B|_A)^D)-1$.   Each component of the complement of $|(|B|_A)^D|_B$ in $B$  is an annulus with modulus at most $D+2$. So, we have $|(|B|_A)^D|_B \supset B^{D+2}$.  It follows that
$$ B^{D+2} \subset |(|B|_A)^D|_B \subset (|B|_A)^D \subset A^D.\qedhere$$\end {proof}

\subsection{Standard Collars}

\label{sec: collarsec}

Let $c$ be a simple  closed geodesic on a hyperbolic surface $X$ of length $\ell: =\ell(c)$. The Collar Lemma \cite[Lemma 13.6]{primer} says that the  open metric neighborhood $A$ of $c$ with radius 
\begin{equation}\label{eq: def of Mc}
M_{\ell}:=\sinh^{-1}\left ( \frac 1{\sinh(\ell/2)}\right )%\frac 12 \log \frac {\cosh(\ell/2)+1}{\cosh(\ell/2)-1}
\end{equation}
is an open generalized annulus. We call this neighborhood $A$ the \emph{standard collar} of $c$. The proof of \cite[Lemma 13.6]{primer} also shows that disjoint geodesics have disjoint standard collars. A simple estimate shows:

\begin{lemma}[see Proposition 6.1 in \cite{douady1993proof}]\label{lem: modestimate}
We have $\pi/\ell - 1 \leq mod(A) \leq \pi/\ell$.
\end{lemma}

By a standard Jacobi field estimate, 
$\length(\partial A) = \ell \cosh(M_{\ell})$.
One can then compute that as $\ell\to 0$, we have $\length(\partial A) \to 2.  $
And if we define the \emph{standard collar} of a cusp to be the open horoball neighborhood whose frontier has length $2$, then this standard collar is an embedded annulus and is disjoint from the standard collar around any other cusp or simple closed geodesic. This can be proved similarly to the usual Collar Lemma, but one can also see that it is true via a limiting argument, by  approximating cusps by short geodesics.  

\paragraph*{The conformal parametrization.} The standard collars can be conformally parametrized as follows. If $c_+$ is a simple closed geodesic on $X$, equipped with a preferred side, let $p_c : \bbS^1_{\ell} \longrightarrow c$ be an  isometric parametrization. We can now parametrize the standard collar $A$ of $c$ by the conformal map $$\bbS^1_{\ell} \times (-\rho_{\ell},\rho_{\ell})\modsim \longrightarrow A$$ that restricts to  $p_c$ on $\bbS^1_{\ell} \times 0$ and takes $\bbS^1_{\ell} \times [0,\rho)\modsim$ to the preferred side of $c_+$. The number $ 0\leq \rho_{\ell} <\pi/2$ can be given explicitly by $$ \rho_{\ell} := \Sec^{-1}(M_{\ell})$$ where $$\Sec(y) := \int_0^y\sec(t)dt.$$
To make this parametrization an isometry its domain is endowed with the  hyperbolic metric $$ds = \sec(y) \sqrt{dx^2+dy^2}$$ discussed in \S \ref{sec: conformalannuli}.  More concretely, if $c$ is regular, this map takes a point $$(x,t) \in \bbS^1_{\ell} \times (-\rho,\rho)$$
to the point $p\in A$ that orthogonally projects to $p_c(x) \in c$, where $d(p, p_c(x)) = \Sec(|t|)$, and where $p$ lies on the preferred side of $c_+$ if $t>0$. 

 On the other hand, the \emph{conformal parametrization} of standard collar $A$ of a cusp is given by the isometry 
$$\bbS^1 \times (1/2,\infty) \longrightarrow A,$$
where the domain is considered with the hyperbolic metric $ds = 1/y \sqrt{dx^2+dy^2}$. Note that the circle $\bbS^1 \times 1/2 \subset \bbS^1 \times (0,\infty)$ has length $2$ in this metric, as required.  We can write this parametrization explicitly if we have some other intrinsic way of thinking of the cusp, e.g.\  we can measure depth in the cusp using injectivity radius, but it is not so important to do this here.

 Finally, it will be convenient below to have a one-sided version of standard collars that works for geodesics with a  preferred side choice. As in the beginning of \S \ref{sec: graftingsec}, if $c_+$ is a geodesic in $X$ equipped with a preferred side, there is a natural boundary component, which we just call $c_+ \subset X \dsm c$, that projects to $c$ and is adjacent to the left of the preferred side of $c$. Define the \emph{standard collar} $A \subset X \dsm c$ of the boundary component $c_+ \subset \partial(X \dsm c)$ to be its $M_{\ell}$-neighborhood in $X \dsm c$. Informally, $A$ is just the half of the standard collar of $c$ that lies on the preferred side.  After choosing a parametrization $p_c$ of $c$, the one-sided standard collar can be parameterized similarly to the two-sided one, via  the isometry 
$$\bbS^1_{\ell} \times [0,\rho)\modsim \longrightarrow A, \ \ \ 0\leq \rho <\pi/2$$
that restricts to  $p_c$ on $\bbS^1_{\ell} \times 0$.  Here, if $c$ is  regular then $\sim$ is a trivial equivalence relation, and if $c$ is degenerate then $(x,0) \sim (i(x),0)$, where $i$ is a reflection.

\paragraph*{The semi-hyperbolic parametrization.} Later on it will be useful to consider a parametrization of the standard collars of the form $\bbS^1 \times I$ in which 
the fibers $\{x\}\times I$ are unit-length parametrizations of geodesics.

If $A$ is the standard collar of a curve $c$ of length $\ell$, the \emph{semi-hyperbolic parametrization} of $A$ is given by $$ \bbS^1_{\ell} \times (-M_{\ell},M_{\ell}) \to A$$
where the map sends $(x,y)$ to the point in $A$ whose nearest point projection to the geodesic $c$ is $p_c(x)$, and whose distance from $c$ is $y$.

Since we have explicitly expressed the hyperbolic metric in the conformal parametrization, we can now deduce the transition map between them to be $$\Id\times \Sec : \bbS^1_{\ell} \times (-\rho_\ell,\rho_\ell) \to \bbS^1_{\ell}\times (-M_{\ell},M_{\ell}).$$

Similarly, if $A$ is a standard collar of a cusp, then the \emph{semi-hyperbolic parametrization} of $A$ is given by $\bbS^1_{\ell} \times (0,\infty)$ where the second coordinate is now the distance to the boundary curve $\partial A$.
The transition map between the conformal parametrization and the semi-hyperbolic is given by $$\Id\times \ln(2\cdot):\bbS^1_{\ell} \times (1/2,\infty) \to \bbS^1_{\ell}\times (0,\infty). $$

\subsection{The distance into extended standard collars}\label{sec: distance in extended collars}

If $c\in \mathcal S$, the \emph{extended standard collar} of $c$  is defined to be  $$ A^{ext} = \iota(A\dsm c) \cup C_c,$$
where $A \subset X$ is  the standard collar of $c$ and $C_c$ is the  generalized annulus grafted along $c$, as defined in \S \ref{sec: collarsec}. 
We will also find it convenient to have a one-sided version of the extended standard collar. If $c_+ \in \mathcal S_\pm$ is one of our grafting curves equipped with a preferred side, the \emph{extended standard collar} of $c_+$ is 
$$A^{ext} := \iota(A) \cup C_{c_+} \subset X_L,$$
where $A \subset X \dsm \mathcal S$ is the (one-sided) standard collar of $c_+$, see \S \ref{sec: collarsec}, and $C_{c_+}$ is the grafted-in generalized annulus from the beginning of \S \ref{sec: graftingsec}.

\begin{proposition}[Distance into the truncation]\label{prop: modelprop}
Suppose that $\ell(c) \leq \mathcal L$ for all $c\in \mathcal S$. Then there is some constant $M=M(\mathcal L)$ as follows. Let $A^{ext}\subset X_L^{hyp}$ be either 
\begin{enumerate}
\item the extended standard collar of some $c\in \mathcal S$, or
\item  the inclusion $A^{ext} = \iota(A)$ of the standard collar of a cusp of $X$.
\end{enumerate} Then with respect to the hyperbolic metric $d^{hyp}$ on $X_L^{hyp}$, the  distance from the $D$-truncation of $A^{ext}$ to the $X_L^{hyp} \setminus A^{ext}$ is at least $\ln(D-2) - M(\mathcal L)$.
\end{proposition}

This gives a stronger, uniform version of \S 4 of Bourque \cite{bourque}, which says that the hyperbolic distance across the entire grafting annulus goes to infinity when longer and longer annuli are grafted along a fixed curve in a fixed surface. The constant $M(\mathcal L)$ can  be determined explicitly by tracking through the proofs. Also, note that Proposition \ref{prop: modelprop}  can be applied to the inclusion of the standard collar around some simple closed geodesic $c$ disjoint from all the grafting curves. Indeed, such a curve can be formally added to $\mathcal S$, setting $L(c)=0$, and then the statement of the proposition applies. 

We briefly record the following corollary, out of interest to the reader. A version of this result that is more convenient for our purposes is given later in Proposition \ref{prop: limits in collars}.

\begin{corollary}[Modeling the uniformized metric]\label{cor: metric comparison cor}
  Using the notation of Proposition \ref{prop: modelprop}, given $\epsilon>0$  there is some $D=D(\mathcal L,\epsilon)>0 $ such that inside the $D$-truncation of $A^{ext}$, we have
 $(1-\epsilon)\cdot \hat d \leq d^{hyp} \leq (1+\epsilon)\cdot \hat d$, where $\hat d$  is the complete hyperbolic metric on $A^{ext}$.
\end{corollary}

Notice $\hat d$ is explicit, while the metric $d^{hyp}$ is obtained through the black box of uniformization. 
An explicit expression of a closely related hyperbolic metric $\check{d}$ is presented in \S\ref{sec: parametrization of extended collar neighborhoods}.
The proof of the corollary uses the fact that if a point $p$ lies in some domain $V$ and there are two conformal hyperbolic metrics on $V$ such that $d(p,\partial V)$ is large with respect to \emph{both} metrics, then the two metrics are almost identical near $p$.  This in turn is a consequence of compactness theorems for conformal maps and the fact that every conformal automorphisms of the disk is a hyperbolic isometry. We omit a careful proof of the corollary, since we  never use it in our work and we employ  essentially  the same argument to prove Proposition \ref{prop: limits in collars} below.

\medskip

 We will need the following result for the proof of Proposition \ref{prop: modelprop}. 
\begin{lemma}[Collar comparison]\label{lem: collarlem}
Let $A^{ext}\subset X_L^{hyp}$ be the extended standard collar around some curve of $\mathcal S$, or the inclusion of the standard collar around some cusp of $X$, just as in Proposition \ref{prop: modelprop}. Let $\hat A$ be the standard collar of the corresponding geodesic or cusp in $X_L^{hyp}$, where we get a geodesic in the first case if $L(c)<\infty$, and a cusp otherwise. Then $\hat A$  contains the $3$-truncation of $A^{ext}$.

 Moreover, if  all the curves in $\mathcal S$ have length at most $\mathcal L$,  there is an explicit constant $N(\mathcal L)$  such that $A^{ext}$  contains the $N(\mathcal L)$-truncation of $\hat A$.
\end{lemma}

In case 1, note that if $\ell(c)$ is short, its standard collar $A \subset X$ has large modulus, so the lemma says that the $3$-truncation of  the extended standard collar $A^{ext}$ contains the entire grafting cylinder $C_{c_+}$. One can compare this with Hensel's proof of  \cite[Lemma 4.1 (iii)]{hensel}, the main point of which is to prove that if $c$ is short, its grafting cylinder is contained in the hyperbolic standard collar $\hat A$. The proof below, using quasi-monotonicity of truncation, is  much shorter than Hensel's proof, though. 

 These kinds of results are intimately tied to estimates of  geodesic lengths under grafting.  For instance, in case 1 when $L(c)<\infty$, Lemma~\ref{lem: modestimate} implies that $mod(A^{ext})$ and $mod(\hat A)$ are  each within $1$ of $(\pi + 2L(c))/ \ell(c)$ and $\pi / \ell(\hat c)$, respectively, where $\hat c$  is the geodesic in $X_L^{hyp}$ in the homotopy class of $c$, and the lemma above  estimates the difference between $mod(A^{ext})$ and $mod(\hat A)$. The  interested reader can compare  the resulting estimates on $\ell(\hat c)$ to those given in Proposition 3.4 of Diaz-Kim \cite{diaz2012asymptotic}, and Hensel's Lemma 4.1 \cite{hensel}.

\begin{proof}
Let us assume we are in case 1 and that  $L(c)<\infty$. Case 2 and the $L(c)=\infty$ case are similar.
Let $$Y := \bbS^1_{\ell( \gamma)} \times (-\pi/2,\pi/2) \longrightarrow X, \ \ \ \ \hat Y := \bbS^1_{\ell(\hat c)} \times (-\pi/2,\pi/2)\longrightarrow X_L^{hyp}$$
  be the annular covers  corresponding to $c$ and $\hat c$. Then $A^{ext}$ and $\hat A $ lift  homeomorphically to subannuli $$\mathrm{Lift}( A^{ext},\hat Y), \qquad \mathrm{Lift}(\hat A,\hat Y) \subset \hat Y.$$ By Lemma \ref{lem: modestimate}, we have $\mathrm{Lift}(\hat A,\hat Y) \supset \hat Y^1$, the $1$-truncation of $\hat Y$. By quasi-monotonicity of truncation (Lemma \ref{lem: quasimonotone}), it follows that the $3$-truncation of $\mathrm{Lift}( A^{ext},\hat Y)$ lies inside $\hat Y^1$, and hence in $\mathrm{Lift}(\hat A,\hat Y)$. So, the $3$-truncation $A^{ext,3}\subset A^{ext}$ lies in $\hat A$.

Now let $B$  be one of the components of $\hat A \setminus A^{ext,3}$. We want to bound $mod(B)$  from above. Apply  the first part of the proof to all the other curves in $\mathcal S$,  so that all the $3$-truncations of their extended standard collars are contained in the  corresponding standard collars in $X_L^{hyp}$. Since standard collars of disjoint geodesics on a hyperbolic surface are disjoint,  it follows that $B$ cannot intersect the $3$-truncation of the extended standard collar around any other grafting curve.  As a consequence, if we let $Z$ be the orbifold obtained by grafting modulus $3$ generalized annuli along each side of every curve in $ \mathcal S$, then there is a conformal embedding $$i: B \longrightarrow Z.$$
%\tinytodo{So, I'm not completely convinced that you don't get a uniform bound  instead of this $N(\mathcal L)$.  I mean,  the fact that $B$ can  can intersect  the other grafting cylinders a bit shouldn't really multiply its modulus.  But I'm not sure how to control the modulus of $B$ without using the quasi-conformal maps.  In any case, this point doesn't really matter for our argument.}
The standard collar of any curve $\gamma \in \mathcal S$ has modulus $m_\gamma:=\Sec^{-1}(M_\gamma)/\ell(\gamma)$, where $M_\gamma$  is as defined in \S \ref{sec: collarsec}. If we stretch each such annulus over its union with the modulus $6$ annulus that is grafted in to form $Z$, we get a $K$-quasiconformal homeomorphism $$\phi : X \longrightarrow Z, \ \text{ where } \  K=(6+m)/m, \  \text{ and } \ m = \sup_{\gamma \in \mathcal S} m_\gamma.$$
Since $K$-quasiconformal maps  between annuli distort moduli by a factor of at most $K$, we get that $$mod(B)=mod(i(B)) \leq K \cdot mod(\phi^{-1}\circ i(B)) \subset X.$$
But if we lift $\phi^{-1}\circ i(B)$ to the annular cover $Y \longrightarrow X$, we have an annulus $\mathrm{Lift}(\phi^{-1}\circ i(B),Y)$ that is contained in  a component of $Y \setminus \mathrm{Lift}(A^{ext,3},Y)$. Each such component has modulus at most $4$, so $mod(\phi^{-1}\circ i(B)) \leq 4$,  and hence $mod(B) \leq K\cdot 4$.   It then follows that the $(4K+1)$-truncation of $\hat A$ is contained in $A^{ext,3}$. So, unraveling all the estimates above, the lemma follows with  the constant $$N(\mathcal L) := 4 \cdot \frac{6+m}{m} +1, \qquad m =\sup_{\gamma \in \mathcal S}\ \Sec^{-1}\left (\sinh^{-1}\left ( \frac 1{\sinh(\ell(\gamma)/2)} \right )\right )/\ell(\gamma) .\qedhere$$
\end{proof}

We now are ready to prove Proposition \ref{prop: modelprop}.

\begin{proof}[Proof of Proposition \ref{prop: modelprop}]
As in the statement, assume that every curve in $\mathcal S$  has length at most $\mathcal L$. 

First, assume that $c\in \mathcal S$ and $L(c)<\infty$.  Let $\hat c$ be the geodesic on $X_L^{hyp}$  in the homotopy class of $c$ and let $$\hat Y := \bbS^1_{\ell(\hat c)} \times (-\pi/2,\pi/2)\longrightarrow X_L^{hyp}$$
 be the  annular cover corresponding to $\hat c$.  Lift $A^{ext}$  homeomorphically to a subannulus $\mathrm{Lift}(A^{ext}, \hat Y) \subset  \hat Y.$ By Lemma \ref{lem: collarlem}, $A^{ext}$  contains the $N(\mathcal L)$-truncation of $\hat A$. So, by Lemma \ref{lem: modestimate}, $\mathrm{Lift}(A^{ext}, \hat Y)$  contains the $(N(\mathcal L)+1)$-truncation of $\hat Y$, which is the subannulus 
$$\bbS^1_{\ell(\hat c)} \times \big (-\pi/2 + (N(\mathcal L)+1)\ell(\hat c), \ \pi/2 - (N(\mathcal L)+1)\ell(\hat c)\big ) \ \ \subset \ \ \hat Y.$$
 Now given some large $D>0$, let $A^{ext,D}$ be the $D$-truncation of $A^{ext}$, and let $\mathrm{Lift}(A^{ext,D}, \hat Y)$ be its lift to $\hat Y$. By Lemma \ref{lem: quasimonotone}, $\mathrm{Lift}(A^{ext,D}, \hat Y)$ is contained in the $(D-2)$-truncation of $\hat Y$, which is the subannulus 
$$\bbS^1_{\ell(\hat c)} \times \big (-\pi/2 + (D-2)\ell(\hat c), \ \pi/2 - (D-2)\ell(\hat c)\big ) \ \ \subset \ \ \hat Y.$$
 It follows that the  hyperbolic distance between $A^{ext,D}$ and the frontier of $A^{ext}$ is at least 
$$R:= \int_{\pi/2-(D-2)\ell(\hat c)}^{\pi/2 - (N(\mathcal L)+1)\ell(\hat c)} \sec(t) \, dt \geq \int_{(D-2)\ell(\hat c)}^{(N(\mathcal L)+1)\ell(\hat c)} 1/t \, dt = \ln(D-2) - \ln(N(\mathcal L)+1), $$
where the first  inequality uses that $\sec(\pi/2-t) \geq 1/t$. Setting $M(\mathcal L)=\ln(N(\mathcal L)+1)$,  we are done.
\end{proof}

\subsection{Parametrization of cusps, standard collars and extended collar neighborhoods}\label{sec: parametrization of extended collar neighborhoods}

Let $c \in \mathcal S$, let $\ell=\ell(c)$, and let $A$ be its standard collar neighborhood. 
Let $\check{A}$ be the cover of $X$ corresponding to $\pi_1(A)$. $\check{A}$ is a generalized annulus with a complete hyperbolic metric.
Thus, $\check{A}$ can be parametrized conformally by $$\bbS^1_{\ell(c)} \times (-\pi/2,\pi/2)\modsim.$$ 
The lift of $A\subset \check{A}$ is parametrized by \[\bbS^1_{\ell(c)} \times (-\rho_{\ell},\rho_{\ell})\modsim.\] 
This is exactly the \emph{conformal parametrization} discussed in \S\ref{sec: collarsec}.

% Recall that the complete hyperbolic metric on $\check{A}$ is given by $ds = \sec(y) \sqrt{dx^2+dy^2}$ in the conformal parametrization. Thus, the middle curve $c= S^1_{\ell}  \times \{0\}$ is a geodesic, and the distance of a point $(x,s)$ from $c$ is given by $\Sec(s)$ where $\Sec$ is the antiderivative of $\sec$. Therefore, it will also be convenient to use the \emph{semi-hyperbolic parametrization} given by applying $\Id \times \Sec$ to the conformal parametrization. In the semi-hyperbolic parametrization $$A \simeq S^1_{\ell}\times (-M_{\ell},M_{\ell}) \subset S^1_{\ell}\times (-\infty,\infty) \simeq  \check{A}.$$

After grafting a cylinder onto $c$, the extended collar neighborhood $A^{\ext}$ can be seen as a subannulus of the extension of $\check{A}$ defined by \[\check{A}^{\ext} = \iota(\check{A}\dsm c) \cup C_c.\]
The extension $\check{A}^{\ext}$ is a generalized conformal annulus and so by Theorem \ref{thm: uniformization} admits a unique complete hyperbolic metric $\check{d}^{\ext}$.
The restriction of the metric $\check{d}^{\ext}$ to $A^{\ext}$ will be used in the next subsection to analyze geometric limits of grafted surfaces when the basevectors are taken deep inside the grafting cylinders. As preparation for this, here we write down explicit parametrizations of $\check{A}^{\ext}$ and  express $\check d$ in them.

We divide into two cases, depending on whether $L(c)<\infty $ or $L(c)=\infty$. For convenience, we will let $c_+ \in \mathcal S_\pm$ be $c$ equipped with a choice of preferred side, and work with the one-sided collar $A^{ext}_+$ of  $c_+$.

\paragraph*{Case 1. $L(c) < \infty$.} 

% {\bf *** Copied from Section 8, we should decide if we keep it here, move it to Subsection 7.1 or keep it in Section 8.}
The generalized annuli $A^{\ext}\subseteq \check{A}^{\ext}$ have 3 natural parametrizations --  conformal, rescaled conformal, and semi-hyperbolic -- defined as follows. \begin{itemize}
    \item The \emph{conformal parametrization} is the one given by joining the (reversed) Euclidean parametrization $C_{c_+} \longrightarrow \bbS^1_{\ell} \times [-L(c),0) \modsim$ and the conformal parametrization of $A$ or $\check{A}$. I.e,
$$A^{\ext} \longrightarrow  \bbS^1_{\ell}\times [-L(c) , \rho_\ell)\modsim \qquad \textrm{and} \qquad  \check{A}^{\ext} \longrightarrow \bbS^1_\ell \times [-L(c), \pi/2)\modsim.$$
    \item The \emph{rescaled conformal parametrization} is obtained from the conformal parametrization by shifting the second coordinate by $L(c)$ and rescaling both coordinates by a factor of $$\omega = \omega_{c} := \frac{\pi/2}{\pi/2+L(c)}=\frac{\pi}{\pi+2L(c)}.$$ This gives 
$$A^{\ext}\longrightarrow \bbS^1_{\ell'} \times [0,\rho')\modsim \qquad \textrm{and} \qquad  \check{A}^{\ext} \longrightarrow \bbS^1_{\ell'} \times [0,\pi/2)\modsim,$$ where
$\ell' = \ell \omega$ and $\rho' = (L(c)+\rho_\ell) \omega$. 
The complete hyperbolic metric $\check{d}$ on $\check{A}^{\ext}$ is given by 
\[ds = \sec(y) \sqrt{dx^2+dy^2}\] in the rescaled conformal parametrization.
    \item The \emph{semi-hyperbolic parametrization} is obtained from the rescaled conformal parametrization by postcomposing with $\Id \times \Sec$, which gives
\[ A^{\ext} \longrightarrow \bbS^1_{\ell'} \times [0,\Sec(\rho')) \modsim.\]
In the semi-hyperbolic parametrization the second coordinate is the hyperbolic distance of a point to the geodesic $\bbS^1_{\ell'}\times\{0\}$ with respect to the complete hyperbolic metric on $\check{A}^{ext}$.
\end{itemize}

% {\bf ****}

% In this case, $\check{A}_+^{\ext}$ is an annulus of finite modulus which is the union of the two annuli $C_{c_+}$ and 
% $\iota(\check{A}_+\dsm c)$.  These are parametrized respectively  by 
% \[S^1_{\ell}\times [0,L(c) \modsim\text{\: and \:}S^1_{\ell}\times [0,\pi/2)\]
% and so their union, $\check{A}_+^{\ext}$ can be parametrized by 
% \[S^1_{\ell} \times [0,L(c)+\pi/2)\modsim.\]
% After scaling by $\omega = \frac{\pi/2}{L(c)+\pi/2} = \frac{\pi}{2L(c)+\pi}$ we get the parametrization
% \[\Omega:\check{A}_+^{\ext} \to S^1_{\check{\ell}} \times [0,\pi/2)\modsim.\]
% where $\check{\ell} =  \ell\omega$.
% The complete hyperbolic metric $\check{d}$ is given as usual by 
% \[ds = \sec(y) \sqrt{dx^2+dy^2}.\]

% The subannulus $A^{\ext}_+$ is therefore parametrized by
% \[A^{\ext}_+\to S^1_{\check{\ell}} \times [0,\check{\rho})\modsim,\]
% where $\check{\rho} =  (L(c)+\rho)\omega$.

\paragraph*{Case 2. $L(c) = \infty$.}
% {\bf **** copied from the retraction section}
In this case we will only use the conformal parametrization which is obtained by joining together the parametrization of the infinite cylinder $C_{c_+} \to \bbS^1_{\ell}\times(-\infty,0]$ and the conformal parametrization of $A$. The \emph{conformal parametrization} of $A^{\ext}$ is thus given by
$$A^{\ext} \longrightarrow \bbS^1_\ell \times (-\infty,\rho_\ell] \qquad \textrm{and} \qquad \check{A}^{\ext} \longrightarrow \bbS^1_\ell \times (-\infty,\pi/2),$$ and the hyperbolic metric $\check{d}$ on $\check{A}^{\ext}$ is given by $ds=\frac{1}{\pi/2 - y}\sqrt{dx^2+dy^2}$.

We will also need to measure angles in extended standard collars. For that purpose, in all the parametrizations above, we can consider the foliation by equidistant circles of the form $\bbS^1\times\{y\}$. The \emph{inward direction} at a point $p \in A^{\ext}$ is the direction perpendicular to the equidistant foliation that points towards the geodesic curve $\bbS^1\times \{0\}$ or cusp. Note that for points on the geodesic the inward direction is only defined up to an angle of $\pi$.

\subsection{Limits when the basepoint is deep inside a collar}\label{sec: limits deep in collars}

Suppose that $X^i$ is a  sequence of hyperbolic $2$-orbifolds  that for each $i$ we have a finite set $\mathcal S^i$  of simple closed geodesics in $X^i$ and a function $L^i : \mathcal S^i \longrightarrow [0,\infty]$. As in  the previous section, let $E^i$ be either
\begin{enumerate}
	\item the extended standard collar of some $c^i\in \mathcal S^i$, 
\item  the inclusion $E^i = \iota(A^i)$ of the standard collar $A^i$ of a cusp of $X^i$.
\end{enumerate}

\begin{proposition}[Chabauty limits deep in extended standard collars]\label{prop: limits in collars}
Take tangent vectors $v^i \in TE^i$, with basepoints $p^i \in E^i$, suppose that $p^i$ is contained in the $D^i$-trunctation of $E^i$, where $D^i\to \infty$. 

Let $s^i$ be the second coordinate of $p^i$ in the semi-hyperbolic parametrization of $E^i$ if $E^i$ is not conformally equivalent to a cusp, and take  $s^i= \infty$ if $E^i$ is conformally equivalent to a cusp. 
Assume $\theta^i \to \theta \in [0,2\pi],$ where $\theta^i$ is the CCW angle the vector $v^i$ forms with the inward direction
% \footnote{The inward direction is the direction at $p^i$ of the shortest geodesic connecting $p^i$ to the unique admissible geodesic or cusp in $E^i$. When $v^i$ is on the unique geodesic of $E^i$ then the angle $\theta^i$ is only defined modulo $\pi$.} 
of $E^i$. Then setting $$\Gamma^i := \Gamma\big ((X^i_{L^i})^{hyp}, v^i\big ),$$ we have:
\begin{enumerate}
\item If $s^i \to s<\infty$ and all $c^i$ are regular geodesics, then $(\Gamma^i)$  converges in the Chabauty topology to the one-parameter group $A(\alpha)$ of all hyperbolic-type isometries with axis $\alpha \subset\HH^2$, where $d_{\HH^2}(\alpha,\pbase)=s$ and the shortest path from $\pbase$ to $\alpha$ starts out in the direction that is an angle of $\theta$ CCW from $v_{base}.$
\item If $s^i \to s$ and all $c^i$ are degenerate geodesics, then $(\Gamma^i)$  converges to the group $A'(\alpha)=A(\alpha) \rtimes \ZZ/2$ of \emph{all} orientation preserving isometries that leave $\alpha$ invariant, where $\alpha$ is as in part 1.
\item If $s^i \to \infty$
then $(\Gamma^i)$  converges in the Chabauty topology to the one-parameter group $N(\xi)$ of parabolic isometries fixing the  endpoint $\xi \in \partial \HH^2$ of the geodesic ray emanating out from $\pbase$ in the direction that is an angle of $\theta$ CCW from $v_{base}$. 
\end{enumerate}
\end{proposition}

\begin{proof}[Proof of Proposition \ref{prop: limits in collars}]
In the proof, it will be useful for us to consider the complete hyperbolic metric $\hat d^i$ on $E^i$ in the given conformal class. 
If $E^i$ is not conformally equivalent to a cusp, $E^i$ is the extended standard collar of a grafting curve $c^i$ with $L^i(c^i)<\infty$. For such $E^i$ there is a unique $\hat d^i$-geodesic $m^i$ in $E^i$ which coincides with the geodesic $\bbS^1_{\ell'}\times \{0\}$ in the semi-hyperbolic parametrization. 

Let $\hat{s}^i$ be the distance of $p^i$ to the geodesic $m^i$ in the metric $\hat d^i$. 
By conformally parametrizing $E^i$ by $\bbS^1_{\hat\ell^i} \times (-\pi/2,\pi/2)$ where $\hat{\ell}^i=\frac{\pi\ell(c^i)}{2\rho_{\ell(c^i)}+2L^i(c^i)}$ and expressing its complete hyperbolic metric $\hat{d}^i$ as we did in \S\ref{sec: parametrization of extended collar neighborhoods}, we can relate $\hat{s}^i$ to $s^i$ by 
\begin{equation}\label{eq: relation between s and s hat} \hat{s}^i = \Sec\left(  \frac{\pi + 2L^i(c^i)}{2 \rho_{\ell(c^i)}+2L^i(c^i)}\cdot \Sec ^{-1} (s^i) \right).\end{equation}

The modulus of $E^i$ is given by $$mod(E^i) = \begin{cases}
\frac{2\rho_{\ell(c^i)}+L^i(c^i)}{\ell(c^i)} & \textrm{if it is not conformal to a cusp} \\
\infty & \textrm{otherwise.}
\end{cases}
$$ %if it is not conformal to a cusp, and $\infty$ otherwise. 
The hypothesis that $p^i$ lies in the $D^i$-truncation of $E^i$, where $D^i\to \infty$, implies that the modulus of $E^i$ tends to $\infty$. This happens when  $L_t \to \infty$ or $\ell(c^i) \to 0$, and in particular  $\frac{\pi + L_t}{2 \rho_{\ell(c^i)}+L_t}\to 1$.
Therefore, by \eqref{eq: relation between s and s hat} if $s^i \to s$ then $\hat{s}^i\to s$ as well.

% can be interpreted more concretely in cases. For instance, if each $E^i$ is the extended standard collar around some $c^i\in  \mathcal S^i$, then writing $$E^i=\iota(A^i) \cup C_{c^i},$$ where $A^i$ is the standard collar of $c^i \subset X^i$, it suffices that any of the following three conditions hold for all $i$:
% \begin{itemize}
%     \item $p^i \in \iota(A^i)$, and when regarding $p^i\in A^i$ we have $d(p^i,\partial A^i)\to \infty$,
%     \item $p^i \in C_{c^i}$, and either $\ell(c^i)\to 0$ or $d(p^i,\partial C_{c^i})/ \ell(c^i) \to \infty.$ 
% \end{itemize}
% \todo{rewrite this, and deduce the following:}
% In particular, $\frac{\pi + L_t}{2 \rho_{\ell(c^i)}+L_t}\to 1$. So by \eqref{eq: relation between s and s hat} if $s^i \to s$ then $\hat{s}^i\to s$.

Let $\Delta^i := \Gamma(E^i,v^i) < G$ be the groups corresponding to the vectored orbifolds $(E^i,v^i)$, where we consider $E^i$ as equipped with the complete hyperbolic metric $\hat d$. We first show that the groups $\Delta^i $ have the Chabauty limits in 1--3 above. We then show that $\Delta^i$ and $\Gamma^i$  have the same limit.

 Since each $E^i$ is a  generalized annulus, each group $\Delta^i$  is of one of the following two forms:
 \begin{itemize}
     \item[(a)] $\Delta^i$ is isomorphic to either $\mathbf{a}(\alpha^i,t^i)$ or $\mathbf{a}'(\alpha^i,t^i,q^i)$ where $\alpha^i$ is the geodesic in $\mathbb{H}^2$ at distance $\hat{s}^i$ from $\tilde{p}$ and the CCW angle from $v_{base}$ to the shortest path from $\pbase$ to $\alpha^i$  is $\theta^i$.
     
    %  $\ZZ$ or $\ZZ/2\ZZ$ and stabilizes a geodesic axis $\alpha^i$ in $\HH^2$, where the  distance from $\tilde p$ to $\alpha^i$ is $\hat{s}^i$, and the angle from $v_{base}$ to the shortest path from $\tilde p$ to $\alpha^i$  is $\theta^i$.
     \item[(b)] $\Delta^i$ is a discrete cyclic subgroup of $N(\xi^i)$ of parabolic isometries fixing some point $\xi^i \in \partial \HH^2$, where the geodesic from $\pbase$ to $\xi^i$ is an angle of $\theta^i$ counterclockwise from $v_{base}.$
 \end{itemize} 

The condition that $p^i$ is contained in the $D^i$-truncation of $E^i$, where $D^i\to \infty$, implies that the  minimal amount that $\pbase$ is translated by $\Delta^i$ goes to zero as $i\to \infty$. Hence, the groups $\Delta^i$ converge to the desired limits by Lemma \ref{prop: subspace of elem subgroups}.

To relate the $\Gamma^i$ to the $\Delta^i$, choose universal covers  $$\tilde E^i \cong \HH^2 \longrightarrow E^i, \ \ \ \ \HH^2 \longrightarrow X_{L^i}^{hyp},$$
taking our standard fixed base vector $v_{base} \in T\HH^2$ to $v^i$. Then the inclusion $E^i \hookrightarrow X_{L^i}^{hyp}$ lifts to a conformal map between the universal covers $$I^i : \tilde E^i \cong \HH^2 \hookrightarrow \HH^2 \cong \tilde X_{L^i}^{hyp}.$$
By assumption, $p^i$ is contained in the $D^i$-truncation of $E^i$, where $D^i\to \infty$. So, Proposition \ref{prop: modelprop}  says that the $d^{hyp}$-distance from $p^i$ to $\partial E^i$ goes to infinity with $i$. Hence, the images of the maps $I^i$ exhaust $\HH^2$. Since $DI^i(v_{base})=v_{base}$ for all $i$, by standard compactness theorems for conformal maps (or at worst, Lemma \ref{lem: compactquasi}), after passing to a subsequence we can assume that the maps $I^i$ converge uniformly on compact sets to a conformal automorphism of $\HH^2$. By the Cauchy integral formula the derivatives $DI^i$  also converge, so the limit automorphism fixes the vector $v_{base},$  and therefore is the identity.   The $I^i$ are $(\Gamma^i,\Delta^i)$-equivariant, Lemma~\ref{lem: same limit} then implies that $\Gamma^i$ converges to the same limit as $\Delta^i.$
\end{proof}

\subsection{Global distance estimates}
\label {sec: global distance estimates}
In this section, we establish the following proposition, which relates the hyperbolic metric on a grafted orbifold with the original hyperbolic metric, at least outside of collars of any short grafting curves.

\begin{proposition}\label{prop: globalesstimates}
Given $\LLL>0$, there is a homeomorphism
$$\eta : [0,\infty] \longrightarrow [0,\infty], $$
 with $\eta(\infty)=\infty$ as follows. Suppose that $X$  is a hyperbolic $2$-orbifold, $\SSS$  is a collection of admissible geodesics on $X$ which satisfy $$\ell(c) \le \LLL  \quad \forall c\in \SSS.$$ Then, for all $L: \SSS \longrightarrow [0,\infty]$ and for all $x,y\in X$ that lie outside the standard collars of all curves in $\SSS$, 
$$d_{(X_L)^{hyp}}\big (\iota(x),\iota(y)\big ) \geq \eta\big (d_X(x,y)\big ).$$
\end{proposition}

%Here, a quick computation involving a side identity for Saccheri quadrilaterals, see Figure \ref{fig:saccheri} in the appendix, shows that the $\epsilon$-thin collar of a short geodesic $\gamma$ has radius
%$$R_{\delta,\gamma} := \cosh^{-1}\left ( \frac{\sinh(\delta)}{\sinh(\ell(\gamma)/2)} \right ).$$
%For small $\delta $, this radius is less than the radius $M_\gamma = \sinh^{-1}(1/\sinh(\ell(\gamma)/2))$ of the standard collar, implying that the $\delta$-thin collar is contained in the standard collar. So, one could use standard collars in the statement of Proposition \ref{prop: globalesstimates} if desired.

The idea of the proof is as follows. First, we reduce to the case that $X $ is a surface by passing to a suitable cover. If $X$ is a surface, we construct in Lemma \ref{lem: pinching} a explicit hyperbolic surface $Y$ marked by a homeomorphism $h:X \longrightarrow Y$, by shrinking the lengths of curves in $\mathcal S$, using Fenchel-Nielsen coordinates. In Claim \ref{claim: lipschitz bounds}, we describe how to bound distances in $Y$ using distances in $X$, at least outside collars of elements of $\mathcal S$. We then show that if the pinching above is done appropriately, $Y$ will be quasiconformal to $(X_L)^{hyp}$. Proposition \ref{prop: globalesstimates} follows from the fact that quasiconformal maps distort distances in a controlled way, see for instance \cite[ Theorem 4.4.1]{hubbard}.

In the proof, it will be useful to replace standard collars with smaller `reduced' collars.

\begin{definition}[Reduced Collars]
If $c$ is an admissible geodesic in a hyperbolic $2$-orbifold $X$, the \emph{$\delta$-reduced collar} $\RC_\delta(\gamma)$ is defined to be the metric neighborhood of $\gamma$ with radius $$R_{\delta,\gamma} := \sinh^{-1} \left ( \frac \delta{ \sinh(\ell(\gamma)/2)}\right ).$$
\end{definition}

Note that the standard collar of $\gamma$ is the $1$-reduced collar, and for $\delta < 1$, the $\delta$-reduced collar is contained in the standard collar. Below, what we actually will prove is that if $\delta>0$ is sufficiently small, there is some $\eta$ depending on $\mathcal L,\delta$ such that the conclusion of Proposition \ref{prop: globalesstimates} holds outside of all $\delta$-reduced collars of elements of $\mathcal S$. If we fix such an $\delta<1$, this impiles Proposition \ref{prop: globalesstimates}.

\subsubsection{Reducing to the surface case}

 It will reduce the number of cases we have to consider if we are allowed to assume $X$ is a surface rather than an orbifold. So, let us first reduce the proposition to the case of surfaces, while noting that the same proof would work for orbifolds with a few more pages of effort. 

\begin{lemma}\label{lem: smooth cover with double covers}
Let $X$ be an orbifold, and let $\SSS$ be a collection of disjoint admissible geodesics. Then there exists a (possibly infinite) surface cover $\hat X$ of $X$ such that each lift of $c\in \SSS$ double covers $c$.
\end{lemma}

\begin{proof}
As the conclusion of the lemma is completely topological and $X$ is homeomorphic to the interior $$Y:=int(CC(X))$$ of its convex core, it suffices to construct a surface cover of $Y$ with the indicated property. Extend $\SSS$ to a generalized pair of pants decomposition of $Y$ using Theorem \ref{thm: pantsdecomps}. So, $Y$ is a union of generalized pairs of pants (possibly with missing boundary components) glued along their boundaries. If we show that each pair of pants $P$ has a surface cover $\hat P$ in which each lift of $c\subseteq \partial P$ double covers $c$, then we get the desired cover $\hat X$ as the ``tree of spaces'' formed by gluing together the covers $\hat P$ of each $P$ along their matching boundaries.
Thus it suffices to show that such a cover $\hat P$ of $P$ exists.

The fundamental group of a generalized pair of pants $P$ is given by 
\[ \pi_1(P) = \left<x_1,x_2,x_3 \mid x_1x_2x_3,\; {x_1}^{n_1},\; {x_2}^{n_1},\; {x_3}^{n_3}\right>\]
where $n_1,n_2,n_3\in \NN \cup \{\infty\}$ and at least one of them is $\infty$.
The boundary curves of $P$ correspond to (conjugacy classes of) the generators $x_i$ with $n_i=\infty$.
Thus, $\pi_1(P)$ has the quotient
\[ H= \left<x_1,x_2,x_3 \mid x_1x_2x_3,\; {x_1}^{n_1'},\; {x_2}^{n_1'},\; {x_3}^{n_3'}\right>\]
where $n_i' = n_i$ if $n_i<\infty$ and $n_i'=2$ otherwise.
Since $H$ is a triangle group, we know that $x_1,x_2,x_3$ are elements of order $n_i'$ in $H$. 
This means that in the regular cover $\hat P$ of $P$ that corresponds to $\pi_1(P)\to H$ each lift of a cone point is smooth, and each lift of a boundary component is a double cover.
\end{proof}

So, given $X, \mathcal S,\delta$ as in the statement of the proposition, let $q:\hat X\to X$ be the surface cover given by Lemma~\ref{lem: smooth cover with double covers} and let $\hat\SSS$ be all the lifts to $\hat X$ of all curves in $\SSS$. Then
\begin{enumerate}
    \item If $\hat c\in \hat \SSS$ then $\ell_{\hat X}(\hat c)< 2\mathcal L$. 
    \item If $\hat{c}$ is a lift of $c$ then $q$ maps the $\delta$-reduced collar $\RC_\delta(\hat c)$ inside of $\RC_\delta( c)$.
    \item If we define $\hat L : \hat \SSS \longrightarrow [0,\infty]$ by $\hat L ( \hat c) = L (q(\hat c))$ for every $\hat c\in \hat \SSS$, then $q$ induces a covering map $q_L : \hat X_{\hat L} \longrightarrow X_L$ with $q_L \circ \iota = \iota \circ q.$
    \item For all $x,y\in X$, we have $d_X(x,y) = \min d_{\hat{X}}(\hat{x},\hat{y})$ where the minimum is taken over all lifts $\hat{x}, \hat{y} \in \hat X$ of $x,y$ respectively. Similarly, $d_{(X_L)^{\hyp}} (x,y) =\min d_{(\hat X_{\hat L})^{\hyp} } (\hat x,\hat y)$.
\end{enumerate} 
So, if $x,y\in X$ lie outside the $\delta$-reduced collars of all short elements of $\mathcal S$, choose lifts $\hat x,\hat y \in \hat X$ such that 
$$d_{(X_L)^{\hyp}} (\iota(x),\iota(y)) = d_{(\hat X_{\hat L})^{\hyp} } (\iota(\hat x),\iota(\hat x)).$$
By property 2 above, $\hat x,\hat y$ lie outside all $\delta$-reduced collars of short curves in $\hat {\mathcal S}$. So if we prove the proposition for $\hat X$, then we will have some $\eta=\eta(2\mathcal L,\delta)$ such that 
$$d_{(X_L)^{\hyp}} (\iota(x),\iota(y))=d_{(\hat X_{\hat L})^{\hyp} } (\iota(\hat x),\iota(\hat x)) \geq \eta(d_{\hat X}(\hat x, \hat y)) \geq \eta( d_X(x,y) ),$$
and the proposition follows for $X$.

\subsubsection{The surface case}

Let $X$ be a surface. Let $\LLL,\SSS,L$ be as in the proposition. By Theorem \ref{thm: pantsdecomps}, $X$ has a geodesic pants decomposition that includes $\SSS$. That is, there is a collection of geodesics $\PPP \supset \SSS$ such that $X\dblsetminus \PPP$ consists of pair of pants, flares and halfspaces with geodesic boundaries. To avoid painful notation, let us just assume everywhere below that $L(c) < \infty$ for all $c\in \mathcal S$. The proof in general is almost exactly the same.

As mentioned above, the idea behind our proof is to construct a quasiconformal model for $X_L$ by pinching $\mathcal S$, using Fenchel-Nielsen coordinates. Now in \S \ref{sec: fenchelnielsen} we only discussed Fenchel-Nielsen coordinates for finite type surfaces. This was for good reason, since there are some subtleties in the infinite type case. For instance, while one can always glue together pairs of pants as dictated by length and twist parameters, a gluing of infinitely many pants may not produce a \emph{complete} hyperbolic surface. However, we show in the following lemma that if we start with a collection $\mathcal S$ of curves on $X$ with lengths at most some fixed $\mathcal L$, and vary their lengths within $[0,\mathcal L]$ while keeping all other length and twist parameters constant, then one constructs a complete hyperbolic surface $Y$ that is naturally identified with $X$.

\begin{lemma}[Pinching bounded length curves]\label{lem: pinching}
Fix a sufficiently small $\delta >0$, let $X$ be a hyperbolic surface with geodesic pants decomposition $\mathcal P$ and a subset $\mathcal S \subset \mathcal P$ such that $\ell_X(c)\leq \mathcal L$ for all $c\in \mathcal S$, and let $$l : \mathcal P \longrightarrow (0,\mathcal L]$$
be a function such that $l \equiv \ell_X$ on $\mathcal P \setminus \mathcal S$. Then there is a complete hyperbolic surface $Y$ and a homeomorphism $h : X \longrightarrow Y$
such that the following properties hold.
\begin{enumerate}
    \item $h$ takes elements of $\mathcal P$ to simple closed geodesics in $Y$, and $\ell_Y(h(c)) = l(c)$ for all $c\in \mathcal P$.
    \item After a system of $\mathcal P$-tranversals is chosen, and the pants decomposition and transversals for $X$ are pushed forward to $Y$ using $h$, the twist parameter of each $c \in \mathcal P$ is the same as that of $h(c)$.
\end{enumerate}
Moreover, we can assume that our map $h$ has the following metric properties.
\begin{enumerate}
    \item[3.] For each $c\in \mathcal S$, we have $h(\RC_\delta(c)) = \RC_\delta(h(c))$, and $h$ sends the two orthogonal foliations of $\RC_\delta(c)$ by geodesic segments and equidistant circles to the corresponding foliations of $\RC_\delta(h(c))$.
    \item[4.] $h$ is locally $K$-bilipschitz outside the $\delta$-reduced collars of curves $c\in \mathcal S$, where $K=K(\mathcal L,\delta)$. 
\end{enumerate} 
\end{lemma}

Before we start the proof, we record some useful elementary facts about reduced collars. Below, we often denote the frontier of $\RC_\delta(c) \subset X$ by $\partial \RC_\delta(c)$. We feel obliged to say this for precision, since $\RC_\delta(c)$ is an open annulus, not a closed annulus.

\begin{fact} \label{fact: collarfact} Suppose $c$ is a curve in $X$ with length $\ell_X(c) \in (0,\mathcal L]$. \begin{enumerate} 
\item[(a)] The length of each boundary component of $\RC_\delta(c) $ is contained in some interval $[\ell_{min},\ell_{max}] \subset (0,\infty).$ Here, both constants depend only on $\mathcal L,\delta$. 
\item[(b)] The radius $R_{\delta,c}$ of the $\delta$-reduced collar of $c$ decreases monotonically as $\ell_X(c)$ increases. Hence, it is bounded below by the radius $R_{ \delta,\mathcal L}$ one gets when $\ell(c) = \mathcal L$.
\item[(c)] If $\alpha \subset \RC_{\delta}(c)$ is an arc that has both endpoints on the same component of $\partial \RC_\delta(c)$, then $\alpha$ is homotopic rel endpoints to a path $\alpha' \subset \partial  {\RC_{\delta}(c)}$ with length $\ell(\alpha')\leq C \cdot \ell(\alpha)$, where $C=C(\mathcal L,\delta)$.
\end{enumerate}
\end{fact}
\begin{proof}
For (a), note that the length of each boundary component is $\ell_X(c) \cdot \cosh R_{\delta,c}$, which is always positive, varies continuously with $\ell_X(c) \in (0,\mathcal L]$, and converges to $\delta/2$ as $\ell_X(c)\to 0$. (b) is trivial. For (c), suppose that $\alpha$ has both endpoints on the boundary component $\gamma$ of $\RC_\delta(c)$, let $r = \min\{1,R_{\delta,\mathcal L}\}$ and let $\mathcal N_r(\gamma)$ be the $r$-neighborhood of $\gamma$ inside $\RC_\delta(c)$. Consider the projection $\pi : \mathcal N_r(\gamma) \longrightarrow \gamma$ obtained by projecting along geodesic rays orthogonal to $c$. The lipschitz constant of $\pi$ is bounded above by the ratio of the length of $\gamma$ to the length of the other boundary component $\gamma'$ of $\mathcal N_r(\gamma)$. But this ratio is 
$$\frac {\ell(\gamma)} {\ell(\gamma')} = \frac{\ell_X(c) \cdot \cosh( R_{\delta,c} )}{\ell_X(c) \cdot \cosh( R_{\delta,c}-r )} \leq 4 \cosh(r) < 10, $$ 
so $\pi$ is $10$-lipschitz. If $\alpha$ is contained in $\mathcal N_r(\gamma)$, then we are done with $\alpha'=\pi \circ \alpha$ and $C=10$. Otherwise, $\ell(\alpha)\geq r$. If we let $\alpha'$ be one of the two segments of $\gamma$ that has the same endpoints as $\alpha$, then $\ell(\alpha') \leq \ell_{max}$, the constant from part (a). So, we are done with $C=\ell_{max} / r$.
\end{proof}
 
We are now ready to prove the lemma.

\begin{proof}[Proof of Lemma \ref{lem: pinching}] For each pants component $P$ of $X \dsm \PPP$, construct a pair of pants $P'$ marked by a homeomorphism $h: P\longrightarrow P'$ such that $\ell_{P'}(h(c)) = l(c)$ for each component $c \subset \partial P$. By Proposition \ref{prop: alteration}, we can take $h$ to be affine on boundary components, to map seams to seams, to map $\delta$-reduced collars of boundary components of $P$ to those of $P'$ in a way that preserves the natural foliation by orthogonal geodesic segments, and to be locally $K$-bilipschitz outside of these reduced collars. Let $Z$ be the surface obtained by gluing all the pants $P'$ together in the pattern dictated by the topology of $X$, making sure that when $P_1,P_2$ are pants that share a boundary component $c\in \mathcal P$, the corresponding pants $P_1',P_2'$ are glued isometrically along $h_1(c)$ and $h_2(c)$ in such a way that the distance between the seam endpoints matches the distance between the seam endpoints of $P_1,P_2$ on $c$. Then the maps $h$ will all agree on pants boundaries, giving a homeomorphism $$h:  CC(X) \setminus \mathcal B  \longrightarrow Z,$$ where $\mathcal B $ is the union of all boundary components of $CC(X)$ that are biinfinite geodesics.  Note that this $h$ satisfies properties 1--4 in the lemma\footnote{Property 3 is currently only true if one considers one-sided open $\delta$-reduced collars of curves $c \in \mathcal S$ that are boundary components of $CC(X)$, since $h$ is not defined on the part of such a collar that lies in the adjacent flare.}, simply because Proposition \ref{prop: alteration} says that its restriction to each pair of pants does. By Theorem 4.5 and Proposition 4.6 in \cite{alessandrini2011fenchel}, we can identify $Z$ with $CC(Y) \setminus \mathcal B_Y$, where $Y$ is some complete hyperbolic surface, $CC(Y)$ is its convex core\footnote{When reading \cite{alessandrini2011fenchel}, note that they take the convex core to be the smallest convex subset of $Y$, while we take it to be the smallest closed, convex subset of $Y$, so in their language $Z$ is the union of $CC(Y)$ with any closed geodesics that are boundary components of $CC(Y)$.}, and $\mathcal B_Y$ is the union of all boundary components of $CC(Y)$ that are biinfinite geodesics. We want to extend $h$ continuously to a homeomorphism $X\longrightarrow Y$ that still satisfies properties 1--4. 

First, take a flare $A_X \subset X$ that is bounded by a curve $c\in \mathcal P$. The image $h(c)$ is a component of $\partial CC(Y)$, so bounds a flare $A_Y \subset Y$. Parameterize these flares via
$$A_X\cong c \times [0,\infty), \ \ \ A_Y \cong h(c) \times [0,\infty)$$
where if $x\in c$, then $\{x\}\times [0,\infty)$ is a unit speed parametrization of the geodesic ray orthogonal to $c$ at the point $x$, and the second parametrization is similar. With respect to these parametrizations, we have
$$\RC_\delta(c) \cap A_X \cong c \times [R_{\delta,c},\infty), \ \ \ \RC_\delta(h(c)) \cap A_Y \cong h(c) \times [R_{\delta,h(c)},\infty).$$
Using the coordinates above, we now extend our map $h$ to a map $$H : A_X \longrightarrow A_Y, \ \ H(x,t) = \begin{cases} \left (h(x),\frac {R_{\delta,h(c)}}{R_{\delta,c}} \cdot t\right ) & t \leq R_{\delta,c} \\
\left (h(x),t - {R_{\delta,c}}+ R_{\delta,h(c)}\right ) & t \geq R_{\delta,c} 
\end{cases}.$$ This map sends $\RC_\delta(c) \cap A_X $ to $\RC_\delta(h(c))\cap A_Y$ in a way that satisfies property 3 of the lemma. It also is locally $K$-bilipschitz on $A_X\setminus \RC_\delta(c)$, for some increased $K=K(\mathcal L,\delta)$. To see this, note that outside of the reduced collar, $H$ is isometric in the direction of the second coordinate, and is affine in the direction of the first coordinate. Moreover, we have the following length estimate:
\begin{align*}
    \ell \big (H(c\times \{t\}) \big ) &=\ell \big (h(c) \times \{t - R_{\delta,c} + R_{\delta,h(c)}\} \big ) \\
    &= \ell_Y(h(c)) \cdot \cosh (t - R_{\delta,c} + R_{\delta,h(c)}) \\
    &\approx \ell_Y(h(c))  \cdot \frac{\cosh (R_{\delta,h(c)}))}{\cosh(R_{\delta,c})} \cdot \cosh t \\
    &\approx \ell_X(c) \cosh(t)\\
    &= \ell(c \times \{t\}).
\end{align*} 
Here, $x \approx y$ if $x,y$ differ by a  multiplicative factor depending on $\mathcal L,\delta$. The first $\approx$ uses the exponential estimate $e^x/2\leq \cosh x \leq e^x$ to coarsely distribute $\cosh$ over addition, so a multiplicative factor of $10$ works, say. The second uses the fact that $\ell_Y(h(c)) \cosh (R_{\delta,h(c)})$ is the length of the boundary components of $\RC_{\delta}(h(c))$, which by Fact \ref{fact: collarfact} (a) is within a multiplicative factor of the length $\ell_X(c) \cosh (R_{\delta,c})$ of the boundary components of $\RC_{\delta}(c)$. The result of these estimates is that $H$ distorts distances in the direction of the first coordinate by a factor that is bounded in terms of $\mathcal L,\delta$, implying that $H$ is locally $K$-bilipschitz for some new increased $K=K(\mathcal L,\delta)$. So, we have extended $h$ to a map of flares as desired. From now on, we will just use the notation $h$ to denote the entire extended map.

Next, take a component $\beta \subset \mathcal B$, i.e.\ a boundary component of $CC(X)$ that is a biinfinite geodesic. This $\beta$ has a neighborhood that is disjoint from all reduced collars of elements $c\in \mathcal S$. For $\beta$ is disjoint from the closure of each individual collar; the collars cannot accumulate since each point in the collar of $c$ lies on a curve homotopic to $c$ that has length less than some universal $D=D(\mathcal L,\delta)$, and bounded length curves in distinct homotopy classes cannot accumulate in $X$. Since $h$ is locally $K$-bilipschitz outside reduced collars, $h$ extends continuously to $\beta$. Moreover, $\beta$ is the limit of a sequence of pants curves $\gamma_i \in \mathcal P$. Discarding finitely many $i$, we can assume all $\gamma_i $ are long, and hence are in $\mathcal P \setminus\mathcal S$. As $h$ maps each curve of $\mathcal P \setminus\mathcal S$ isometrically, it follows that it also maps $\beta$ isometrically, onto a boundary component $h(\beta)$ of the convex core of $Y$. Let $H_X,H_Y$ be the hyperbolic half planes in $X,Y$ bounded by $\beta,h(\beta)$. Then we can extend $h$ to a unique isometric $H_X\longrightarrow H_Y$ that agrees with the previous definition on the boundary. So defined, we have a map $$h: X \longrightarrow Y$$ satisfying 1--4 of the lemma.

It is immediate from the construction that $h$ is continuous, open and injective, so we only have to show it is surjective. First, we know that the image of $h$ contains the interior of the convex core of $Y$, and also contains all flares of $Y$, so the only possibility is that $h$ misses some closed half-plane component $H \subset Y \setminus int(CC(Y))$. Choose a finite length path $\gamma : [0,1)\longrightarrow Y$ such that $\gamma(t) \in int(CC(Y))$ for $t<1$ and $\gamma(t)$ converges to a point in $ \partial H$ as $t\to 1$. We may also assume that $\gamma$ intersects only finitely many reduced collars $\RC_\delta(h(c))\subset Y$, where $c\in \mathcal S $. For the radius of every such collar is bounded below by Fact \ref{fact: collarfact} (b), implying that $\gamma$ can intersect only finitely many collars essentially. And if $\gamma$ intersects some $\RC_\delta(h(c))$ in a boundary-parallel arc, Fact \ref{fact: collarfact} (c) says that this arc can be homotoped rel endpoints to an arc on $\partial \RC_\delta(h(c))$ while only increasing its length by a uniform multiplicative factor. So, if we perform all such homotopies, we will still have a finite length arc with the same endpoint.

Pulling $\gamma$ back to $X$ gives a path $h^{-1} \circ \gamma $, which goes to infinity in $X$. Indeed, if $h^{-1} \circ \gamma $ accumulates onto any point $p\in CC(X)$ as $t\to 1$, then continuity of $h$ means that $\gamma(t)$ accumulates onto $h(p)$ as $t\to 1$, but by assumption $\gamma(t)$ converges to a point outside the image of $h$. Since $h$ maps collars to collars, the path $h^{-1}\circ \gamma$ can intersect only finitely many collars $\RC_\delta(c)$, where $c\in \mathcal S$. As $h$ is locally $K$-bilipschitz outside of reduced collars, and is smooth (hence bilipschitz) on each individual reduced collar, $h$ is bilipschitz on the path $h^{-1} \circ \gamma$, which contradicts the fact that $h^{-1} \circ \gamma$ has infinite length while $\gamma$ has finite length.
\end{proof}

Below, we will apply Lemma \ref{lem: pinching} only when the curves of $\mathcal S$ are being pinched, i.e.\ when $l\leq \ell_X$. In this case, the local bilipschitz criterion in part 4 can be promoted into a global metric distortion inequality, at least outside of reduced collars.

\begin{claim}[A global distortion inequality]\label{claim: lipschitz bounds}
With the notation of Lemma \ref{lem: pinching}, there is some $K'=K'(\mathcal L,\delta)$ such that if $x_1,x_2 \in X$ lie outside all  $\delta$-reduced collars of curves in $\mathcal S$, then $$d_{Y}(h(x_1),h(x_2)) \geq \frac 1{K'} \cdot d_X(x_1,x_2).$$
\end{claim}
 \begin{proof}
 Let $\gamma$ be the geodesic in $Y$ connecting $h(x_1)$ and $h(x_2)$. Then $h^{-1}\circ \gamma$ is a path from $x_1$ to $x_2$ in $X$. The geodesic $\gamma$ can be broken along the boundaries of reduced collars in $Y$, so that it becomes a concatenation $$\gamma = \alpha_1 \cdots \alpha_n,$$ where each $\alpha_i$ either has interior contained in $\RC_\delta(h(c))$ for some $c\in \mathcal S$, or lies outside of all such reduced collars. It suffices to replace each arc $h^{-1} \circ \alpha_i \subset X$ with an arc $\beta_i $ that has the same endpoints, such that \begin{equation}
     \label{eq: lengtheq}\ell_Y(\alpha_i) \ge \frac{1}{K'}\ell_X(\beta_i),
 \end{equation} for some $K'=K'(\mathcal L,\delta)$. For if we can do this, the concatenation $\beta = \beta_1 \cdots \beta_n$ is a path from $x_1$ to $x_2$ and $$d_X(x_1,x_2) \leq \ell_X(\beta) \leq K' \ell_Y(\alpha) = K'd_Y(h(x_1),h(x_2)).$$

If $\alpha_i$ lies outside all reduced collars, let $\beta_i=h^{-1} \circ \alpha_i$. By part 4 of Lemma \ref{lem: pinching}, $h$ is locally $K$-bilipschitz on $\alpha_i$, so \eqref{eq: lengtheq} follows with $K'=K$. Therefore, we may assume that the interior of $\alpha_i$ is contained in some $\delta$-reduced collar $\RC_\delta(h(c))$, where $c\in \mathcal S$. Since $x_1,x_2$ lie outside of all reduced collars, the index $i$ cannot be $1$ or $n$, so the endpoints of $\alpha_i$ lie on the boundary of $\RC_\delta(h(c))$. There are then two cases to consider.

\medskip

\noindent \textit{Case 1.} Suppose that the endpoints of $\alpha_i$ lie on different boundary components of $ {\RC_\delta(h(c))}$. Then
$$\ell_Y(\alpha_i) \geq 2R_{\delta,h(c)} \geq 2R_{\delta,c},$$ where the last inequality follows from Fact~\ref{fact: collarfact}~(b) and the fact that $\ell_Y(h(c)) \leq \ell_X(c)$. By Fact~\ref{fact: collarfact}~(a), the boundary components of $ {\RC_\delta(c) }$ have length at most some $\ell_{max}=\ell_{max}(\mathcal L,\delta)$. So, if $p,q$ are the endpoints of $\alpha_i$, we can join the preimages $h^{-1}(p),h^{-1}(q)$ with a path $\beta_i$ in $ {\RC_\delta(c) }$ of length
$$\ell_X(\beta_i) = 2\ell_{max}+R_{\delta,c} \leq 2\ell_{max}+\ell_Y(\alpha_i),$$
 by concatenating a shortest path between the boundary components of $ {\RC_\delta(c) }$ with segments of each boundary component. By Fact \ref{fact: collarfact} (b), we have $\ell_Y(\alpha_i) \geq R$ for some $R=R_{\delta,\mathcal L,}$  so
$$\ell_X(\beta_i) \leq (1+2\ell_{max}/R)\ell_Y(\alpha_i).$$

\medskip

\noindent \textit{Case 2.} Suppose the endpoints of $\alpha_i$ lie on the same  component $\gamma \subset \partial  { \RC_\delta(h(c))}$. Fact \ref{fact: collarfact} says that there is a path $\alpha_i' \subset \partial \RC_\delta(h(c))$ connecting the endpoints of $ \alpha_i$ such that $$\ell_Y(\alpha_i')\leq C \ell_Y(\alpha_i)$$
for some $C=C(\mathcal L,\delta)$. If we let $\beta_i = h^{-1} \circ \alpha_i'$, then since $h$ is locally $K$-bilipschitz outside the reduced collars of curves in $\mathcal S$, we have $\ell_X(\beta_i) \leq K\ell_Y(\alpha_i') \leq KC \ell_Y(\alpha_i),$ so we are done with $K'=KC$.
\end{proof}

If $c\in \mathcal S$, define the \emph{extended $\delta$-reduced collar} of $c$ in $X_L$ to be the union $$\RC_\delta^{ext}(c):=\iota(\RC_\delta(c) \dsm c) \cup C_c \subset X_L.$$ 
This is the natural reduced analogue of the extended standard collars discussed in \S \ref{sec: distance in extended collars}. 

\begin{claim}[The quasiconformal model] \label{claim: quasiconformal map from XL to Y} 
For each $l : \mathcal P \longrightarrow (0,\mathcal L]$, let $Y$ and $h: X \longrightarrow Y$ be as in Lemma~\ref{lem: pinching}. Then for some choice of  $l(\cdot) \leq \ell_X(\cdot)$, there is a homeomorphism $f : Y \longrightarrow X_L$ such that
\begin{enumerate}
    \item for each $c\in \mathcal S$, the map $f$ restricts to a conformal homeomorphism 
    $\RC_\delta(h(c)) \longrightarrow \RC_\delta^{ext}(c) ,$ and 
    \item on the complement of all the extended $\delta$-reduced collars $\RC_\delta^{ext}(c)$, $c\in \mathcal S$, we have $f \circ h = \iota$.
\end{enumerate}
In particular, by part 4 of Lemma \ref{lem: pinching}, $f$ is $K$-quasiconformal for some $K=K(\mathcal L,\delta)$.
\end{claim}

\begin{proof}
Let us leave the choice of $l$ unspecified for a moment. Suppose that $c\in \mathcal S$ and let $c'= h(c) \subset Y$. The extended $\delta$-reduced collar $\RC_\delta^{ext}(c)$ has modulus at least the modulus $m(c)$ of $\RC_\delta(c)$, and as $l(c)$ varies in the interval $(0,\ell_X(c))$, the modulus of $\RC_\delta(c')$ varies between $m$ and $\infty$. So, if we choose $l(c)$ appropriately, the two moduli will be equal, in which case there is a conformal map
$$f: \RC_\delta(c') \longrightarrow \RC_\delta^{ext}(c) .$$
Moreover, we can take $f$ to be affine in the natural conformal parametrizations of these two annuli (compare with \S \ref{sec: parametrization of extended collar neighborhoods}), and then rotating the map appropriately, we can make it agree with $\iota \circ h^{-1}$ on the boundary of $\RC_\delta^{ext}(c)$. Here, we use part 3 of Lemma \ref{lem: pinching} to say that once we rotate our map to agree with $\iota \circ h ^{-1}$ on one boundary component, it also agrees with $\iota \circ h^{-1}$ on the other.

So, choosing every $l(c)$ as in the previous paragraph, we construct a function $l$ such that the map $ \iota \circ h^{-1}$ can be extended conformally across all the $\RC_\delta(c')$ to give a map $f : Y \longrightarrow X_L$ as desired. Since $\iota \circ h^{-1}$ is locally $K$-bilipschitz outside of reduced collars and its extension is conformal on each collar, we have that $f$ is $K$-quasiformal as required.
\end{proof}

We now finish the proof of Proposition \ref{prop: globalesstimates}. Fix all the notation given by Claim \ref{claim: quasiconformal map from XL to Y}, and let $x_1,x_2 \in X$ be outside the $\delta$-reduced collar neighborhoods of all curves in $\SSS$. By Claim \ref{claim: lipschitz bounds}, we have 
$$d_Y(h(x_1),h(x_2)) \geq \frac 1K d_X(x_1,x_2).$$
But $f$ is $K$-quasiconformal, so  \cite[Theorem 4.4.1]{hubbard} gives a homeomorphism $\mu_K:\RR_+ \to \RR_+$ such that 
\[ d_{(X_L)^{\hyp}}(f\circ h(x_1),f\circ h(x_2))\ge \mu_K \big ( d_Y(h(x_1),h(x_2))\big ) \geq \mu_K\left (\frac 1K d_X(x_1,x_2)\right ).\]
Since $f\circ h(x_i)=\iota(x_i)$, the proof is complete if we take $\eta (s)= \mu_K( \frac 1K s)$.

\subsection{Grafting and smooth convergence} \label{sec: grafting and smooth convergence}

 The grafting construction takes in a  hyperbolic $2$-orbifold $X$, a  finite collection of disjoint simple closed geodesics $\mathcal S$ on $X$, and a function $L : \mathcal S \longrightarrow [0,\infty] $, and outputs a new  hyperbolic orbifold $X_L^{hyp}$.  In this section, we show that when  we consider vectored orbifolds varying in the smooth topology, then grafting is continuous in both $X$ and $L$. The statement is necessarily a bit complicated; we set it up as follows.

Suppose that $(X^i,v^i) \to (X^\infty,v^\infty)$ is a  smoothly convergent sequence of oriented hyperbolic $2$-orbifolds. For each $i=1,\ldots, \infty$, let $p^i$  be the base point of $v^i$ and let $$\psi^i : B_{X^\infty}(p^\infty,R^i) \longrightarrow X^i$$ be a sequence of almost isometric maps witnessing this convergence on balls, where $R^i \to \infty$. So, each $\psi^i$ is orientation preserving, $d\psi^i(v^\infty) = v^i\in TX^i$ for all $i$, and if $g^i$ is the Riemannian metric on $X^i$, we have $(\psi^i)^* g^i \to g^\infty$ in the $C^\infty$ topology on  $X^\infty$.

Let $\mathcal S^\infty$ be a finite\footnote{Finiteness is not really necessary here.  All that one really needs is that $\sup_{c\in \mathcal S^\infty} \sup_i  \ell(\psi^i(c)) < \infty$.  However, the arguments are little bit simpler if we assume  at least that $\mathcal S^\infty$ is finite.} collection of disjoint, simple closed geodesics on $X^\infty$.  Discarding finitely many $i$, we can assume that  every geodesic $c\in \mathcal S^\infty$ is contained in the domain of $\psi^i$,  and then using Lemma \ref{lem: geodesicspreserved}, we can  further assume that $\psi^i$  takes constant speed parametrizations of each $c$ to constant speed parametrizations of geodesics in $X^i$. Next, pick  collections $\mathcal S^i$ of disjoint, simple closed geodesics on each $X^i$, $i<\infty$, and functions $L^i: \mathcal S^i\longrightarrow [0,\infty]$, where $i=1,2,\ldots,\infty$,  satisfying the following properties.
\begin{enumerate}
\item $\sup_i \sup_{c^i \in \mathcal S^i} \ell(c^i) < \infty$,
	\item for each $i$, we have $\psi^i(\mathcal S^{\infty}) := \{ \psi^i(c) \ | \ c \in \mathcal S^{\infty} \} \subset \mathcal S^{i}$,
\item for each $c\in \mathcal S^\infty$, we have $L^i(\psi^i(c)) \to L^\infty(c),$
\item $\min_{c^i \in  \mathcal S^i \setminus \psi^i(\mathcal S^\infty)} d\big (p^i, c^i) \to \infty$ as $i\to \infty$.
\end{enumerate} 
 For convenience,  define $\mathcal S^i_\pm$ as in  the beginning of \S \ref{sec: graftingsec}, so that an element of $\mathcal S^i_\pm$  is an element of $\mathcal S^i$ equipped with a preferred side.  Note that $\psi^i$ induces a natural map $\mathcal S^\infty_\pm \longrightarrow \mathcal S^i_\pm$,  which we also will  abusively call $\psi^i$.

We now graft along the curve collections $\mathcal S^i$ to produce new piecewise hyperbolic/Euclidean orbifolds $X^i_{L_i}$, where $i=1,\ldots,\infty $, equipped with smooth embeddings $$\iota : X^i \dsm \mathcal S^i \longrightarrow X^i_{L^i},$$ where $X^i \dsm \mathcal S^i$ is the orbifold with boundary obtained by cutting $X^i$ along $\mathcal S^i$. Note that the maps $\psi^i$ extend to $$\psi^i : B_{X^\infty}(p^\infty,R^i) \dsm \mathcal S^\infty \longrightarrow X^i \dsm \mathcal S^i,$$ smooth maps that we denote abusively by $\psi^i$ as well. Informally, 2 -- 4 above mean that we graft in $X^i$ along the same curves as in $X^\infty,$ using annuli of approximately the same lengths, but we also allow additional grafting in $X^i$ along curves far from the basepoints. This latter allowance is necessary in \S \ref{sec: proof of retraction}, when proving continuity of a certain deformation retract.

We now prove that grafting is continuous  with respect to smooth convergence. The statement is separated into two cases, depending on whether the base vectors we are interested in lie outside the grafting cylinders, or inside the grafting cylinders. 

% \begin{figure}

% 	\caption{This will be an awesome picture illustrating Theorem \ref{thm: contgraft}, part 2.}
% \label{fig: contgraftfig}
% \end{figure}
\begin {theorem}[Continuity of grafting] \label{thm: contgraft} \

\begin{enumerate}
	\item For each $i$, suppose that $w^i$ is a tangent vector to $X^i \dsm \mathcal S^i$, and that we have $d (\psi^i)^{-1}(w^i) \to w^\infty \in T(X^\infty \dsm S^\infty)$ as $i\to \infty$. Then in the smooth topology on vectored hyperbolic $2$-orbifolds, we have $$((X^i_{L^i})^{hyp}, d\iota(w^i)) \to ((X^\infty_{L^\infty})^{hyp}, d\iota(w^\infty)). $$
\item Fix $c^\infty\in \mathcal S^\infty$ and a preferred side $c^\infty_+\in \mathcal S^\infty_\pm$ of $c^\infty$, and let $c^i:=\psi^i(c^\infty)\in \mathcal S^i$ and $c^i_+ = \psi^i(c^\infty_+)$. For $i=1,\ldots,\infty$, let $C_{c^i} \subset X^i_{L^i}$ be the grafting cylinder, let $w^i \in TC_{c^i_+}$ and let $q^i$ be the basepoint of $w^i$. Identify $c^i_+$ with the boundary component of $C_{c^i_+}$ to which it is glued, suppose that $\pi(q^i) \in c^i_+$ is a closest point to $q^i$ within $C_{c^i_+}$, and let $d^i$ be the Euclidean distance in $C_{c^i_+}$ from $q^i$ to $\pi(q^i)$. Let $\theta^i$ be the angle that the vector $w^i$ makes with the inward direction. If $$(\psi^i)^{-1}(\pi(q^i)) \to \pi(q^\infty), \ d^i\to d^\infty, \ \theta^i\to \theta^\infty,$$
then in the smooth topology on vectored hyperbolic $2$-orbifolds, we have
$$((X^i_{L^i})^{hyp}, w^i) \to ((X^\infty_{L^\infty})^{hyp}, w^\infty). $$
\end{enumerate} 
\end {theorem}

Note that in the second part of the theorem, the base vectors $w^i$ lie at a bounded distance away from the boundary of the grafting cylinder.  In Proposition \ref{prop: limits in collars}, we  identified the continuous groups that we obtain as Chabauty limits when when we take base vectors deeper and deeper into the cylinders.  
 
 For closed orbifolds, one can easily show that grafting is continuous in both the surface and the cylinder lengths by constructing  explicit quasi-conformal maps between the \emph{entire} relevant grafted surfaces. In the situation of more general smooth convergence, this is impossible since the orbifolds in question may not even be homeomorphic. What we do is construct a quasi-conformal map from a large subset of $(X^\infty_{L^\infty})^{hyp}$ to a large subset of $(X^i_{L^i})^{hyp}$, using Proposition \ref{prop: globalesstimates} to say that the image is also large with respect to the uniformized metric. Some care is also required in order to get the relevant vectors to converge, essentially because a sequence of uniformly convergent quasiconformal maps need not have convergent derivatives.

\begin {proof}
Set $Y^i=(X^i_{L^i})^{hyp},$ where $i=1,2,\ldots\infty.$  By a continuity argument, it suffices to assume that the vectors $w^i,w^\infty$  do not lie  on the boundaries of the  orbifolds $X^i\dsm \mathcal S^i$, or the boundaries of the grafting annuli $C_{c^i}$.  This assumption is not really necessary for the argument below, but it will make some of the technical parts of it simpler.
In both cases 1 and 2, our goal is to construct quasi-conformal maps between large subsets of $Y^\infty$ and $Y^i$, taking $w^\infty$ to $w^i$. We currently have maps $\psi^i$ between large subsets of $X^\infty$ and $X^i$ that are almost isometric, and we first show how to extend these maps across the grafting cylinders. 
 
  Recall that the domains of our maps $\psi^i$  contain all the geodesics in $ \mathcal S^\infty$,  and that $\psi^i$ maps each such  geodesic  with constant speed to some  geodesic in $ \mathcal S^i$. For each $c_+ \in \mathcal S^\infty_\pm$, we can then extend $\psi^i$ across (a large piece of) the annulus $C_{c^+} \subset Y^\infty$.  Namely, we define maps
\begin{equation}
	\bar \psi^i : \iota(B_{X^\infty}(p^\infty,R^i)) \ \cup \bigcup_{\substack{c_+ \in \mathcal S^{\infty}_\pm \\ L(c_+)<\infty} } C_{c_+} \cup \bigcup_{\substack{c_+ \in \mathcal S^{\infty}_\pm \\ L(c_+)=\infty} } C^i_{c_+} \longrightarrow Y^i\label{eq: psii}
\end{equation}
piecewise as follows, where here $C^i_{c_+} \subset C_{c_+}$ is a compact subannulus that is defined below.
\begin{enumerate}
	\item On $\iota(B_{X^\infty}(p^\infty,R^i))$, we just set $\bar \psi^i = \iota \circ \psi^i \circ \iota^{-1}$.
\item If ${c_+}\in \mathcal S^{\infty}_\pm$, we define $\bar \psi^i |_{C_{c_+}}$ to  be the unique map that agrees with $\iota \circ \psi^i \circ \iota^{-1}$ on  the boundary component $\bbS^1_{\ell(c)} \times \{L^\infty(c)\} \subset \bbS^1_{\ell(c)} \times [0,L^\infty(c)]\modsim =: C_{c_+}$, which is glued to $c_+ \subset X^\infty \setminus \mathcal S^\infty$, and where
\begin{enumerate}
\item if $L^{\infty}({c_+}) < \infty$, then $\bar \psi^i$ stretches each vertical line segment $$\{x \} \times I \subset \bbS^1_{\ell(c)} \times [0,L^\infty(c)]\modsim \cong C_c$$ linearly over the  corresponding vertical line segment in $C_{\psi^i(c_+)} \subset X^i_{L^i}$, while
\item if $L^{\infty}(c)=\infty$, then we set 
$$C^i_{c_+} : = \bbS^1_{\ell(c)} \times [0,L^i(\psi^i(c_+))] \subset \bbS^1_{\ell(c)} \times [0,\infty) \cong C_{c_+}$$ and define $\bar \psi^i$ on $C^i_{c_+}$ so that it sends  each vertical subsegment  $$\{x \} \times [0,L^i(\psi^i(c_+))] \subset \bbS^1_{\ell(c)} \times [0,\infty) \cong C_{c_+}$$  isometrically to the corresponding vertical line segment in $C_{\psi^i(c_+)} \subset X^i_{L^i}$.
\end{enumerate}
\end{enumerate}
 Note that for 2 (a) to really make sense, we need that $L^{i}(\psi^i(c)) < \infty$, but this is true after excluding finitely many $i$. While the notation above may be intimidating, all we are really doing here is extending the map $\psi^i$ by stretching the grafting annuli in $Y^\infty$ over the annuli in $Y^i$ in the obvious way, except that if $L^\infty(c) = \infty$  we cannot map the entire infinite annulus over a finite approximating annulus, so we just map in a compact piece of it in a way that is isometric in the second coordinate.  Also, note that the map $\bar \psi^i$ above may not actually be well-defined if the image of $\psi^i$  contains curves of $\mathcal S^i \setminus \psi^i(\mathcal S^\infty)$. But since $$\min_{c^i \in  \mathcal S^i \setminus \psi^i(\mathcal S^\infty)} d\big (p^i, c^i) \to \infty, \ \ \text{ as } \ \ i\to \infty$$
 we can assume  after shrinking the domains of  the $\psi^i$ that  this is not the case, in which case the $\bar \psi^i$ are well-defined topological embeddings that are smooth except on  the boundaries of the grafting annuli. 
 
 By construction, have $d(\bar \psi^i)^{-1}(w^i)\to w^\infty$ in both cases 1 and 2  of the theorem. Then using  Lemma \ref{lem: straightening}, we can modify $\bar \psi^i$  in a small neighborhood of $w^\infty$ so that we still have all of the above properties, but also that $\bar \psi^i$  is an isometry in a small neighborhood of $p^\infty$, with respect to the piecewise hyperbolic/Euclidean metrics on $Y^\infty$ and $Y^i$. (This uses either the  hyperbolic or Euclidean case of Lemma \ref{lem: straightening}, depending on whether $p^\infty$  is outside or inside the grafting annuli.)  In terms of the uniformized hyperbolic metrics on $Y^\infty$ and $Y^i$,  we then get that $\bar \psi^i$ is conformal in a small neighborhood of $p^\infty$. 
 
 Next, except on the boundaries of the grafting annuli, the maps $\bar \psi^i$  are almost isometric  with respect to the piecewise hyperbolic/Euclidean metrics, in the sense that they pull back the piecewise Riemannian metrics on the $Y^i$ to a sequence of piecewise Riemannian metrics on $Y^\infty$  that converge smoothly to the given piecewise metric. It follows that with respect to the uniformized hyperbolic metrics, the $\bar \psi^i$ are $K^i$-quasiconformal embeddings, with $K^i \to 1.$
 
 \begin{claim}
The image of $\bar \psi^i$ contains a radius $R^i$ ball around $p^i$, where $R^i\to \infty$, and where radius  is computed in the uniformized hyperbolic metric on $Y^i$.
 \end{claim}
 \begin{proof}
 Let $d_{hyp}$ be the hyperbolic metric on $Y^i$, let $d_{he}$ be the piecewise hyperbolic/Euclidean metric on $Y^i$, and let $d_X$ be the push forward to $Y^i$ of the hyperbolic distance function on $X^i$ under the inclusion $\iota$. Then we have $d_{he} \geq d_X$ on $\iota(X^i\setminus \mathcal S^i)$, and  we also have $d_{he} \geq d_{hyp}$, as $d_{he}$ is the `projective', or  `Thurston' metric on $X^i$, which is always at least as big as the Poincar\'e metric, see e.g.\ \S 3 of \cite{mcmullen}.
 
 If the claim is false, then after possibly passing to a subsequence, we can find points $x^i$ in the boundaries of the images of the maps $\bar \psi^i$  such that $$S:=\sup_i d_{hyp}(p^i,x^i) < \infty.$$ 
 Proposition \ref{prop: modelprop} then implies that  there is some universal $D>0$ such that the $x^i$ lie outside the $D$-truncation of all extended standard collars of the curves of $\mathcal S^i$. So, there are points $y^i\in  Y^i$ that lie at bounded $d_{he}$-distance from the $x^i$, and which lie outside all extended standard collars of grafting curves. Since $d_{he}\geq d_{hyp}$, the hyperbolic distances $d_{hyp}(x^i,y^i)$  are bounded as well. Similarly, $(\bar \phi^i)^{-1}(p^i)\to p^\infty,$ the $p^i$ do not travel arbitrarily deep into the grafting annuli, or   deeper and deeper into standard collars, so there are points $q^i\in  Y^i$ that lie outside  all the extended standard collars of the grafting curves, and where $d_{hyp}(p^i,q^i)$ is bounded above independently of $i$. 
 
Since the original maps $\psi^i$ in the statement of the theorem are almost isometric, and their domains exhaust $X^i$, it follows from construction of $\bar \psi^i$ that $d_{X}(q^i,y^i)\to \infty$. But by Proposition \ref{prop: globalesstimates}, 
$$d_{hyp}(q^i,y^i) \geq \eta(d_{X}(q^i,y^i))$$
 for some function $\eta$ with $\eta(t)\to \infty$ as $t\to \infty.$ So, $d_{hyp}(q^i,y^i)\to \infty$, implying that $d(p^i,x^i)\to \infty$  as well, contradicting  our assumption.
 \end{proof}
 
 In summary, we have constructed maps 
 $\bar \psi^i : B^i \longrightarrow Y^i,$
 where $B^i$ is a nested sequence of precompact subsets of $Y^\infty$ with $\cup_i B^i = Y^\infty$, such that
 \begin{itemize}
     \item $d(\bar \psi^i)^{-1}(w^i)\to w^\infty$,
     \item $\bar \psi^i$ is  conformal  in a fixed small neighborhood of $p^\infty$  that is independent of $i$,
     \item $\bar \psi^i$ is a $K^i$-quasiconformal embedding, with $K^i \to 1,$
     \item the image $\bar \psi^i(B^i)$ contains a  radius-$R^i$ (hyperbolic) ball around $p^i \in Y^i$, where $R^i\to \infty.$
 \end{itemize}
 It follows immediately from Proposition \ref{prop: conv via quasiconformal} that $(Y^i,w^i)\to (Y^\infty,w^\infty)$  smoothly.
\end {proof}

\section{Retracting to the boundary}
\label{sec: retractsec}
Recall the following definitions from \S \ref{sec: finitevolumesec}. If $S$ is a finite type orientable $2$-orbifold, then $\Sub_S(G)$ is the space of all discrete subgroups $\Gamma< G$  such that the orbifold $\Gamma \backslash \HH^2$ has finite volume and is homeomorphic to $S$. If $\epsilon>0$ is smaller than the Margulis constant, the \emph{$\epsilon$-end  neighborhood} of $\Sub_S(G)$ is the subset $\Sub_S^\epsilon(G)$ consisting of all groups $\Gamma=\Gamma(X,v) \in \Sub_S(G)$ such that either
\begin{enumerate}
	\item the \emph{systole} of $X$ is less than $\epsilon$, or 
\item the vector $v$ lies in a component of the $\epsilon$-thin part of $X$ that is a horoball neighborhood of a cusp.
\end{enumerate}
Finally, $\partial \Sub_S(G) := \overline{\Sub_S(G)} \setminus \Sub_S(G)$ was described in Proposition \ref{prop: boundarybewhat}.
Our goal in this section is to prove Theorem \ref{thm: deformation retract thm}.

% \begin{theorem}
% 	Suppose $S$ has finite topological type. Then  for small $\epsilon$, there  is a deformation retraction 
%     $$\Sub_S^\epsilon(G) \cup \partial \Sub_S(G) \longrightarrow \partial \Sub_S(G).$$ 
% \end{theorem}

\begin{named}{Theorem \ref{thm: deformation retract thm}}
	Suppose $S$ has finite topological type. Then  for small $\epsilon$, there  is a deformation retraction in $\overline{\Sub_S(G)}$
    $$\Sub_S^\epsilon(G) \cup \partial \Sub_S(G) \longrightarrow \partial \Sub_S(G).$$ 
\end{named}

It would suffice here to just take $\epsilon$ smaller than the Margulis constant. However, the constants that come up in the proof are simpler if we allow ourselves to say `as long as $\epsilon$ is small' a few times.

\medskip

% Recal that the boundary $\partial \Sub_S(G)$  can be described explicitly.  In particular, $\partial \Sub_S(G)$ contains many discrete groups $\Gamma=\Gamma(Y,w)$ where $Y$ has smaller topological complexity than $S$, as well as continuous groups like $A,N$ and their conjugates. Indeed, if we form a path $X_t$ of hyperbolic structures on $S$ by pinching some curve $\gamma$ to a cusp, then letting $\gamma_t$  be the corresponding geodesic on $X_t$, we have:
% \begin{itemize}
% 	\item if  the injectivity radius at $v_t \in TX_t$ stays  above some  fixed positive constant, then $\Gamma(X_t,v_t)$ accumulates onto groups $\Gamma(Y,w)$, where $Y$ is homeomorphic to a component of $S\setminus \gamma$; 
% 	\item if the injectivity radius at $v_t \in TX_t$ goes to zero, but the distance from $v_t$ to $\gamma_t$ goes to infinity, then $\Gamma(X_t,v_t)$ accumulates onto conjugates of $N$ in $\partial \Sub_S(G)$; while 
%     \item  if the distance from $v_t$ to $\gamma_t$ is bounded,  then $\Gamma(X_t,v_t)$ accumulates onto conjugates of $A$ in $\partial \Sub_S(G)$.
% \end{itemize}
%  The reader is encouraged to think through proofs of these facts, if desired. However, we will explicitly prove variants of all these statements below, as we work towards Theorem \ref{thm: deformation retract thm}.

\medskip

%  The rest of \S \ref{sec: retractsec} is devoted to a proof of Theorem \ref{thm: deformation retract thm}.  
 To prove Theoren \ref{thm: deformation retract thm}, given $\Gamma={\Gamma}(X,v)\in \mathcal \Sub_S^\epsilon(G)$, we define a path ${\Gamma_t}= \Gamma(X_t^{hyp},v_t^{hyp}) \in \Sub_S^\epsilon(G)$, where $t \in [0,1)$, with \begin{itemize}
 	\item $(X_0^{hyp},f_0^{hyp})=(X,v)$, and
\item $\lim_{t \to 1 } {\Gamma} (X_t^{hyp},v_t^{hyp}) \in \partial \Sub_S(G)$.
 \end{itemize}  
 To construct $X_t^{hyp}$ and the vector $v_t^{hyp}$, the idea is as follows. 
 
 If the systole of $X$ is less than $\epsilon$, then as $t\to 1$, we pinch all minimal-length geodesics on $X$ to cusps. 
 For the deformation retract to be continuous, we must do this pinching in a canonical way that does not depend on any marking of $X$, e.g. we cannot just shrink the length parameters in a choice of Fenchel--Nielsen coordinates. 
 This is accomplished by grafting.
%  So, we glue in `generalized' Euclidean annuli with length going to infinity along all minimal length geodesics in $X$. The resulting piecewise hyperbolic-Euclidean orbifolds $X_t$ can then be uniformized to give hyperbolic orbifolds $X_t^{hyp}$ in which the lengths of the relevant curves go to zero as $t\to 1$.  This operation, called \emph{grafting}, has been well studied for surfaces, see e.g.\ \cite{mcmullen}, \cite{hensel} and \cite{bourque}, and generalizes somewhat naturally to orbifolds, as we explain in more detail below.
While if the vector $v$ is in the $\epsilon$-thin part of $X$ that is a horoball of a cusp, we wish to flow $v_t^{hyp}$ out the cusp, so that $\Gamma(X_t^{hyp},v_t^{hyp})$ will converge to a conjugate of $N$ as $t\to 1$.
%
 %After reading the above paragraph, the forward-looking reader may wonder why \S \ref{sec: retractsec} is so long.  Mostly, the reason is that it requires quite a bit of work to 
 The main work in this section is to define the vectors $v_t^{hyp} \in TX_t^{hyp}$ in such a way that the resulting deformation retract is continuous.  

\subsection{The grafting parameters and the orbifold $X_t$}
\label{sec: xtsec}
In this section we define the grafting parameters that are used to produce the hyperbolic orbifolds $X_t^{hyp}$ in the path $\Gamma_t$.

Let $\epsilon_0$ be the Margulis constant from Theorem \ref{thm: thickthin}, let $\epsilon<\epsilon_0/2$, pick a continuous function $\phi : [0,\infty) \longrightarrow [0,1]$ such that $\phi(x)=0$ for $x \geq 2$ and $\phi(x)=1$ for $x\leq 1$, and fix an orientation preserving homeomorphism $\mathfrak h : [0,1]\to [0,\infty]$ that is smooth on $[0,1)$. 
Set $\minsystole(X) := \min\{\systole(X),\epsilon\}$.

\begin {definition}
	Given $(X,v)\in \mathcal \Sub_S^\epsilon(G)$, let $\mathcal S$  be the set of all admissible geodesics $c \subset X$ with length $\ell(c)<2\epsilon$, equipped with a preferred {side}. 
\end {definition}

% Here, a \emph{side} of a geodesic $c$ is a component of $\mathcal N(c) \setminus c$, where $\mathcal N(c)$ is a regular neighborhood of $c$.  We will  consider each $c$  as coming with a fixed isometric parametrization $$p_c : S^1_{\ell(c)} \longrightarrow X.$$ Then each $c$ is  either \emph{regular}, i.e.\ a simple closed geodesic contained in the regular part of $X$, or is \emph{degenerate}, i.e.\  there is a reflection $i : S^1_{\ell(c)} \longrightarrow S^1_{\ell(c)}$ such that if $p_c(i(x))=p_c(x)$ for all $x\in S^1_{\ell(c)}$,  such that $p_c$ takes the two fixed points of $i$ to distinct order two singular points of $X$. See \S \ref{2orbifolds}. A  degenerate geodesic has only one side, while a regular geodesic has two.  Note that  elements of $\mathcal S$  are really pairs $(\text {\it geodesic},\text{\it side})$, but we will suppress this in notation and just refer to the `preferred side' of  an admissible geodesic $c\in \mathcal S$.  Finally, for convenience we will assume that the  parametrization of  a given  regular geodesic $c \in \mathcal S$  only depends on the  underlying geodesic, and not on the preferred side. 

For each $c\in \mathcal S,$ set $$L_t(X,c) := \mathfrak h( t \cdot \phi( \ell(c) /\minsystole(X) ) ).$$ Note that $L_t(X,c)$ is a continuous function of $\ell(c),t,\systole(X)$, and has the following properties:

\begin{enumerate}
\item If $\ell(c) \geq 2\epsilon$ then $L_t(X,c)=0$, so in particular, $\systole(X) \geq 2\epsilon \implies L_t(X,c)=0$.
%\item If $sys \geq \epsilon$, then $L$ monotonically increases from $0$ to $h(t)$ as $\ell$ decreases from $2\epsilon$ to $\epsilon$. 
\item If $\systole(X) < \epsilon$, then $L_t(X,c)$ increases from $0$ to $\mathfrak h(t)$ as $\ell(c)$ decreases from $2\systole(X)$ \\ to $\systole(X)$. 
\item $L_t(X,c)=\infty$ if and only if $t=1$ and $\ell(c)=\systole(X)\leq \epsilon$.
\end{enumerate}

Let $X_t:=X_{L_t(X,\cdot)}$. That is, $X_t$ is the hyperbolic/Euclidean orbifold obtained from $X$ by grafting along the collection $\mathcal{S}$ using the function $L_t(X,\cdot):\mathcal{S} \to [0,\infty]$ as explained in \S\ref{sec: graftingsec}. Let $(X_t)^{hyp}$ be $X_t$ with its associated complete hyperbolic metric.

Motivated by property 3,  we make the following definition.

\begin {definition} 
When $t=1$ and $\systole(X)\leq \epsilon$, define $\mathcal S_\infty \subset\mathcal S$  to be the subset consisting of all $c$  such that $\ell(c)=\systole(X)$. When either $t<1$ or $\systole (X)>\epsilon $, we set $\mathcal S_\infty =\emptyset$. 
\end{definition} 

So $\mathcal S_\infty$ \emph{depends on $t$}, although we suppress it in the notation, and $L_t(X,c) =\infty \iff c \in \mathcal S_{\infty}.$   

\subsection{The marking $h_t$ on $X_t$}

Our first goal is to define a marking $h_t$ on $X_t$. The marking $h_t$ will be defined to be a homeomorphism from (a subset of) $X$ to $X_t$.
On most of $X$, particularly on the thick part, the marking will be the inclusion map $\iota$. Thin parts come in two forms, and the marking is defined accordingly:
\begin{itemize}
    \item In collar neighborhoods of curves the marking will stretch certain subannuli of the collar neighborhood of a curve in $\mathcal{S}$ over the corresponding grafted-in cylinders.
    \item In cusp neighborhoods the marking will push vectors into the cusps, by stretching certain subannuli of cusps. 
\end{itemize}
In both cases we will need to define and parametrize the aforementioned subannuli, which we do in the next section.

\subsubsection{The subannuli of standard and extended collar neighborhoods.}\label{sec: ATsec}
Given $c_+\in \mathcal{S}_{\pm}$, let $\ell:=\ell(c)$, and let $A\subset X \dsm S$ be its (one-sided) standard collar neighborhood (see \S\ref{sec: collarsec}).
Let $d$ be the restriction of the hyperbolic metric of $X$ on $A$. 

In \S\ref{sec: collarsec} we saw that the annulus $A$ has two natural parametrizations -- the conformal and the semi-hyperbolic -- with domains and transition map 
\[\bbS^1_{\ell} \times [0,\rho_\ell)  \modsim \xrightarrow{\Id \times \Sec} \bbS^1_{\ell} \times [0,M_\ell). \]
% The \emph{conformal parametrization} of $A$ is given by \[A \longrightarrow S^1_{\ell} \times [0,\rho_\ell)  \modsim, \; 0\le \rho_\ell < \pi/2,\] in these coordinates the metric $d$ can be defined by $ds=sec(y)\sqrt{dx^2+dy^2}$.
% The \emph{semi-hyperbolic} parametrization is obtained from the conformal parametrization by postcomposing with $\Id \times \Sec$, and maps 
% \[A \longrightarrow S^1_{\ell} \times [0,M_\ell)\modsim.\]
% We call it the \emph{semi-hyperbolic} parametrization, to reflect the fact that the first coordinate is the closest point projection on $c$ and the second coordinate is the $d$-distance from $c$.

{The reader may wish to refer to Figure \ref{fig: annuli before and after grafting} to keep track of the coordinates we will define for the different parametrizations.} To begin with, given $l \in [0,2\epsilon]$, $t\in [0,1]$, define 
\begin{equation}
\label{eq: rt}	R_{\I}=R_{\I}(l)\in [0,\infty],\ \  R_{\II} = R_{\II}(l,t) \in [0,1], 
\end{equation}
 such that the following properties hold:
\begin{enumerate}
	\item $R_{\I}(l)$ is continuous in $l$, and $R_{\I} \le R_{\II} = R_{\II}(l,t)$  is  continuous in both $l$ and $t$, 
\item $R_{\I}(l) = R_{\II}(l,t)$ if and only if $L_t=0$ (i.e, either $l=2\sigma$ or $t= 0$), 
\item $R_{\I}(l)=0$ if and only if $l =2\epsilon$, $R_{\I}(l)=\infty$ if and only if $l=0$,
\item \label{item: property R+r} $R_{\II}(l,t) \le  \min\{l \cdot M_l, M_l\}$,  where $M_l$ is defined as in \eqref{eq: def of Mc}.
% , and  $\delta_l = \min\{l,1\}$, say. (All that is important about the function $l \mapsto \delta_l$ is that it maps into $[0,1]$, is continuous, and vanishes exactly when $l=0$.)
\end{enumerate}
Now if $\ell=\ell(c)$ is the length of our fixed $c \in \mathcal S$, we set $$R_{\I} :=R_{\I}(\ell), \ \ \ \ R_{\II}:= R_{\II}(\ell,t). $$
We also set $$\rho_{\I} :=\Sec^{-1} (R_{\I}), \ \ \ \ \rho_{\II}:= \Sec^{-1} (R_{\II}). $$

Let $A_{\I},A_{\II} \subset A$ be the (generalized) sub-annuli of $A$ of all points whose $d$-distance from $c$ is in the interval $[0,R_{\I}],(R_{\I},R_{\II}]$ respectively. 
In the conformal parametrization they correspond to 
\[  \bbS^1_\ell\times [0,\rho_{\I}]\modsim \longrightarrow A_{\I}\;\text{ and }\;   \bbS^1_\ell\times (\rho_{\I},\rho_{\II}]\modsim \longrightarrow A_{\II}.\]
In the semi-hyperbolic parametrization they correspond to \[  \bbS^1_\ell\times [0,R_{\I}]\modsim \longrightarrow A_{\I}\;\text{ and }\;  \bbS^1_\ell\times (R_{\I},R_{\II}]\modsim \longrightarrow A_{\II}.\]
Note that in the notation we suppress the dependence of $A_{\I},A_{\II}$ on $c$ and $t$. We will denote them by $A_{\I,c,t},A_{\II,c,t}$ when we want to specify $c$ and $t$. Also note that by property \ref{item: property R+r} above, $A_{\I},A_{\II}$ are indeed subannuli of $A$.

\medskip

We now turn to the grafted orbifold $X_t$.  Let $C_{c_+}$ be the grafted-in generalized annulus, and let $A^{\ext}=\iota(A)\cup C_{c_+}$ be the extended standard collar, as defined in \S\ref{sec: distance in extended collars}.
There are two cases to consider:

\medskip

\noindent \bf The case $c\notin \mathcal S_{\infty}$. \rm
% In \S\ref{hatdsec}, we conformally parametrized $A^{\ext}$ by $S^1_{\ell'} \times [0,\pi/2) / \sim$, and endowed it with the complete hyperbolic metric $\hat{d}$ defined by $ds=\sec(y)\sqrt{dx^2+dy^2}$. 
The generalized annulus $A^{\ext}$ has 3 natural parametrizations: conformal, rescaled conformal, and semi-hyperbolic. 
The domains of these parametrizations and the transition maps between them are given by
\begin{center} 
\begin{tikzcd} [column sep=large, row sep=tiny]
% \mathbb{S}^1_\ell \times [ - L(c) , \rho_\ell) \modsim  \arrow{r}{ \cdot \frac{\pi}{\pi +2 L(c)} } & \mathbb{S}^1_{\ell'} \times [0, \rho') \modsim \arrow{r}{ \textrm{Id} \times \Sec} & \mathbb{S}^1_{\ell'} \times [0, \Sec(\rho)) \modsim \\
 \mathbb{S}^1_\ell \times [ - L(c) , \rho_\ell) \modsim  \arrow{r}{ \cdot \frac{\pi}{\pi +2 L(c)} } &  \mathbb{S}^1_{\ell'} \times [0, \rho') \modsim \arrow{r}{ \textrm{Id} \times \Sec}   & \mathbb{S}^1_{\ell'} \times [0, \Sec(\rho')) \modsim.\\
 \textrm{conformal} & \textrm{rescaled conformal} & \textrm{semi-hyperbolic} 
%  \\
%  \mathbb{S}^1_\ell \times [ - L(c) , \frac{\pi}{2}) \modsim  \arrow{r}{ \cdot \frac{\pi}{\pi +2 L(c)} } & \mathbb{S}^1_{\ell'} \times [0, \frac{\pi}{2}) \modsim \arrow{r}{ \textrm{Id} \times \Sec} & \mathbb{S}^1_{\ell'} \times [0, \infty) \modsim
 \end{tikzcd} 
\end{center} 

%\begin{itemize}
    %\item The \emph{conformal parametrization} is the one given by joining the Euclidean parametrization of the cylinder $C_{c_+} \longrightarrow S^1_{\ell} \times [-L_t/2,0) \modsim$ and the conformal parametrization of $A$. I.e,
%$$A^{\ext} \longrightarrow S^1_{\ell} \times [-L_t/2 , \rho_\ell)\modsim \subseteq [-L_t/2, \pi/2)\modsim.$$
   % \item The \emph{rescaled conformal parametrization} is obtained from the conformal parametrization by shifting the $y$ coordinate by $+L_t/2$ and rescaling both coordinates by a factor of $\omega = \omega_{c,t} = \frac{\pi/2}{\pi/2+L_t/2}=\frac{\pi}{\pi+L_t}$. This gives
% \todo{N: note that it does not coincide with the parametrization in \S\ref{hatdsec}}
%$$A^{\ext}\longrightarrow S^1_{\ell'} \times [0,\rho')\modsim \subset S^1_{\ell'} \times [0,\pi/2)\modsim,$$ where
%$\ell' = \ell \omega$ and $\rho' = (L_t+\rho_\ell) \omega$. 
  %  \item The \emph{semi-hyperbolic parametrization} is obtained from the rescaled conformal parametrization by postcomposing with $\Id \times \Sec$, which gives
%\[ A^{\ext} \longrightarrow S^1_{\ell'} \times [0,\Sec(\rho')) \modsim.\]
%\end{itemize}

% \medskip

Using the semi-hyperbolic parametrization, let us define the sub-annuli $A_{\I}^{\ext}$ and $A_{\II}^{\ext}$ as follows. 
The curve $\partial \big ( \iota_t(A_{\I} \cup A_{\II} ) \cup C_{c_+})$ is given by  $\bbS^1_{\ell'} \times \{\Delta_{\II}\}$ in the semi-hyperbolic parametrization, where $\Delta_{\II}$ can be given explicitly by
$$\Delta_{\II} := \Sec((L_t/2 + \rho_{\II} ) \omega).$$
% \noindent \bf The case when $c\not \in \mathcal S_{\infty}$. \rm  
% Let $\Delta$ be the $\hat d$-distance in $A^{\ext}$ from $c \times 0$ to $\partial \big ( \iota_t(A_{\I} \cup A_{\II} ) \cup C_{c_+})$. Explicitly,
% $$\Delta : = \Sec\left ( \frac {L_t + \Sec^{-1}(R+r) } { \frac 2\pi \big (L_t +\rho_\ell\big )}      \right ), $$
% $$\Delta  = \Sec ( (L_t + \Sec^{-1}(R+r) ) \omega), $$
%  which can be easily verified by tracing through the definition of $\hat d$. However, the only reason we present the formula here is to convince the reader that it varies continuously in $\ell$ and $t$. Let 
Let $$\Delta_{\I} := (1-\mathfrak{h}^{-1}(L_t)/4) \Delta_{\II},$$ and let 
$A_{\I}^{\ext}$ and $A_{\II}^{\ext} $ be the subsets of $A^{\ext}$ which are given by
\begin{equation}\label{eq: subW semi-conformal}
 \bbS^1_{\ell'}\times [0,\Delta_{\I}]\modsim \longrightarrow A_{\I}^{\ext} \;\text{ and }\;   \bbS^1_{\ell'}\times (\Delta_{\I},\Delta_{\II}]\modsim \longrightarrow A_{\II}^{\ext}.
\end{equation}
in the semi-hyperbolic parametrization. In the conformal coordinates these are given by 
\begin{equation}\label{eq: subW conformal}
 \bbS^1_{\ell}\times [-L_t/2,\delta_{\I}]\modsim \longrightarrow A_{\I}^{\ext} \;\text{ and }\;   \bbS^1_{\ell}\times (\delta_{\I},\rho_{\II}]\modsim \longrightarrow A_{\II}^{\ext}.
\end{equation}
where $\delta_{\I} := \Sec^{-1} ( \Delta_{\I}) / \omega - L_t/2$.

\begin{figure}
    \centering
    \includegraphics[scale=.9]{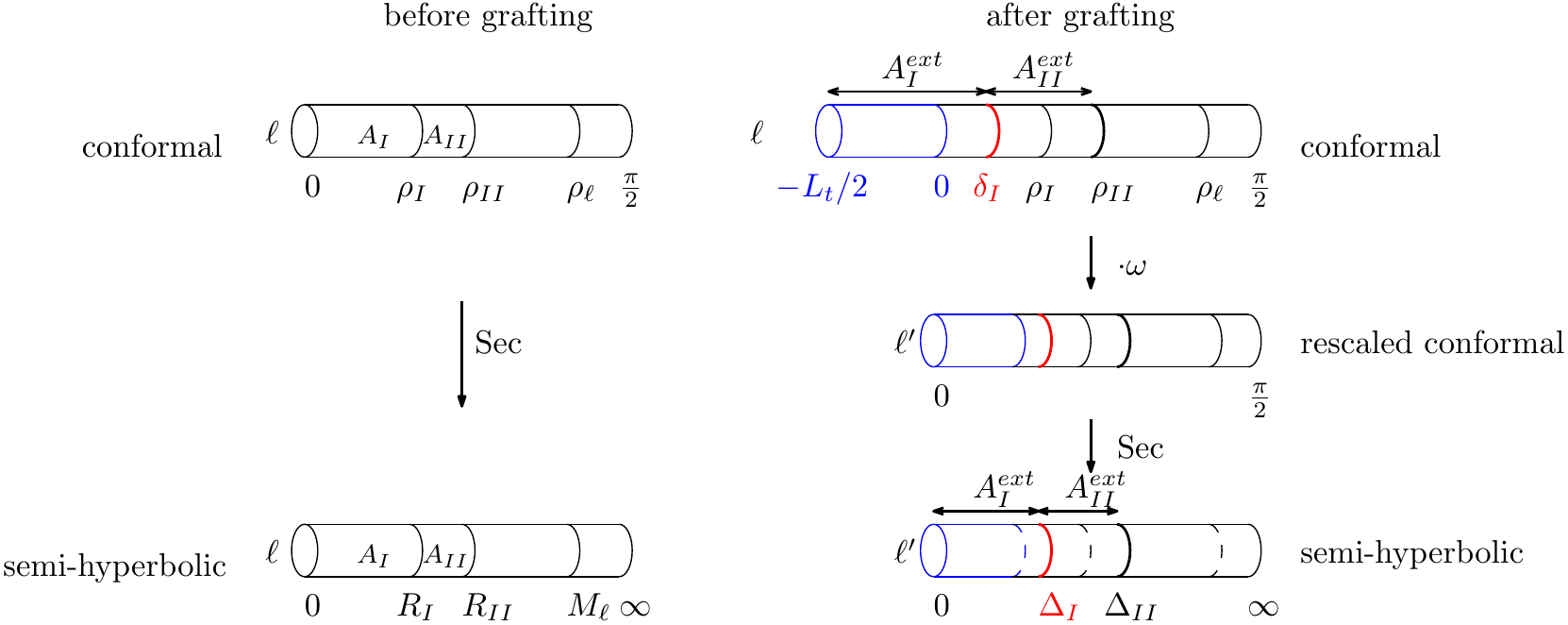}
    \caption{Annuli before and after grafting in different parametrizations.}
    \label{fig: annuli before and after grafting}
\end{figure}

The reason we write $A^{\ext}_{\I},A^{\ext}_{\II}$ in both the semi-hyperbolic and the conformal parametrization is that the map $h_t$ defined in \S\ref{sec: htsec} is defined using the semi-hyperbolic parametrization on $A_{\I} \to A^{\ext}_{\I}$ and using the conformal parametrization for $A_{\II} \to A^{\ext}_{\II}$.

\noindent \bf The case  $c\in \mathcal S_{\infty}$. \rm
In this case $t=1$, and we graft an infinite cylinder on both sides of $c$. 
Let $A^{\ext}_{\I} = \emptyset$ and $A^{\ext}_{\II} := \iota_1(A_{\I} \cup A_{\II} ) \cup C_{c_+}\subset A^{\ext}$.
We will only use the \emph{conformal parametrization} of $A^{\ext}$ in the definition of $h_t$. Recall that the conformal parametrization is given by
$$ \bbS^1_\ell \times (-\infty,\rho_\ell]\longrightarrow A^{\ext},$$ and the hyperbolic metric $\hat{d}$ is given by $ds=\frac{1}{\pi/2 - y}\sqrt{dx^2+dy^2}$.
The subannulus $A^{\ext}_{\II}$ is given by 
% Using this parametrization $A^{\ext}_{\II}$ is parametrized by
$$ \bbS^1_\ell \times (-\infty, \rho_{\II}] \longrightarrow A^{\ext}_{\II}.$$
For consistency we define $\delta_{\I} = -\infty$.

% The $(c,\hat{d})$-parametrization of $A_{\II}^{\ext}$ is obtained by postcomposing with $\Id \times (-\ln(\cdot/\Delta''))$. In this parametrization  
% $$A^{\ext}_{\II} \longrightarrow S^1_\ell \times [0,\infty),$$ where the second coordinate is  the $\hat{d}$-distance from $\partial A^{\ext}_{\II}$. In other words, it is the Busemann function of the cusp $\zeta$, normalized to be $0$ on $\partial A_{\II}^{\ext}$.
% In this case, there is no annulus $A^{\ext}_{\I}$.

\subsubsection{Subannuli of cusp neighborhoods}\label{sec: cusp neighborhoods}

%We will also need to parameterize cusp neighborhoods of $X$. Given a cusp $\zeta $ of $X$, let $K_\epsilon := K_{\zeta,\epsilon}$ be the component of the $\epsilon$-thin part of $X$ that is a neighborhood of $\zeta$. If $\epsilon_0$ is the Margulis constant, we also  define $K_{\epsilon_0} :=K_{\zeta,\epsilon_0} $  similarly.  Note that $K_\epsilon \subset K_{\epsilon_0}$, and that  $$d(K_{\epsilon},\partial K_{\epsilon_0}) = \delta,$$
% for some explicit continuous function $\delta=\delta(\epsilon,\epsilon_0).$ The reader is encouraged to calculate $\delta$ if desired.  Then we can parameterize $K_\epsilon$ and $ K_{\epsilon_0}$ as $$K_\epsilon \cong S^1 \times [\delta,\infty)  \subset S^1 \times [0,\infty) \cong K_{\epsilon_0},$$
% in such a way that each level circle $S^1 \times \{y\} \hookrightarrow X$  is a horocyclic curve parameterized with constant hyperbolic speed, while every ray $\{x\} \times [0,\infty)$ is a unit speed geodesic exiting the cusp $\zeta$.
Let $\mathcal K$ be the set of cusps of $X$.
We begin by parameterizing cusp neighborhoods of $X$. Let $\zeta\in \mathcal K$ be a cusp of $X$.
% For $\sigma \le \epsilon_0$ where $\epsilon_0$ is the Margulis constant, let $K_{\sigma}:=K_{\xi,\sigma}$ be the component of the $\sigma$-thin part of $X$ that is a neighborhood of $\xi$. 
Let $K=K_\zeta$ be the standard collar neighborhood of the cusp $\zeta$. In \S\ref{sec: collarsec}, we saw that the standard collar has two  parametrizations -- the conformal and the semi-hyperbolic -- whose domains and transition map are 
\[\bbS^1 \times [1/2,\infty) \xrightarrow{\ln(2 \cdot)} \bbS^1 \times [0,\infty)\]
respectively.

% of $K$ is given by
% \[K \longrightarrow S^1 \times [1/2,\infty),\]
% with the hyperbolic metric given by $ds=1/y \sqrt{dx^2+dy^2}$ on $S^1 \times [1/2,\infty)$.
% However, as before, it would be more convenient to work with the \emph{semi-hyperbolic} parametrization obtained by postcomposing with $\Id \times \ln(2 \cdot)$. Thus the semi-hyperbolic parametrization is a map
% \[K \longrightarrow S^1 \times [0,\infty),\]
% in which the level circles $S^1\times\{y\}$ are horocyclic curves parametrized with constant hyperbolic speed, while the second coordinate is the distance from $\partial K$ in $X$.

Let $\sigma = \sigma(X)$. In the semi-hyperbolic parametrization the $\sigma$-thin part $K_{\I}$ of $K$ corresponds to the subannulus \[\bbS^1\times [\mu_{\I},\infty)   \longrightarrow K_{\I},\]
where $\mu_I=\mu_I(\sigma)=d(\partial K, K_{\I})$. The $\mu_{\I}(\sigma)$ is a continuous decreasing function satisfying $\mu_I(\sigma)\to \infty$ as $\sigma \to 0$.

Let $K_{\II}$ be the subannulus of $K$ given in the semi-hyperbolic parametrization by \[ \bbS^1 \times [\mu_{\II},\mu_{\I}) \longrightarrow K_{\II} ,\] where $\mu_{\II} = \mu_{\I} -1$.
% For reasons that will become clear in Definition \ref{htdef}, let us denote by $\sigma_0:=\sigma_0 (X) := \delta ^{-1} (\delta (\sigma )-1 )$. Since $\delta(\epsilon_0/2) \ge 1$ we have $K_{\sigma_0} \subseteq K_{\epsilon_0}$, and $K_\sigma$ is at distance 1 deeper into the cusp from the boundary  $\partial K_{\sigma_0}$.

\subsubsection{The map $h_t : X^{\disc} \longrightarrow X_t$}
\label{sec: htsec}
At time $t=1$, the deformation retract we define will map some subgroups $\Gamma(X,v)\in \Sub_S^\epsilon(G)$ to non-discrete subgroups of $G$ while others will be mapped to discrete surface groups of lower complexity. For a given surface $X$ the set $X^{\nd}$  (resp. $X^{\disc}$) defined below is the set of points of $X$ that will be mapped to non-discrete (resp. discrete) subgroups of $G$.

When $t=1$, set 
\[ X^{\nd}  := 	\big(\cup_{c\in \mathcal S_{\infty}} \, A_{\I,c,t}\big) \cup \big( \cup_{\zeta\in\mathcal K} K_{\I,\zeta} \big) \subset X, \ \ \ \ X^{\disc} := X \setminus X^{\nd},\]
where the generalized annuli $A_{\I,c,t}$ and the cusp neighborhoods $K_{\I,\zeta}$ are as defined in \S\ref{sec: ATsec} and \S \ref{sec: cusp neighborhoods}, respectively. 
%So $X^{\nd} \subset X$ is a union of neighborhoods of  cusps, and neighborhoods of curves that are fully pinched in $X_1$.  
Also, when $t<1$,  just set $$X^{\nd} = \emptyset, \ \ X^{\disc}=X.$$
Note that for any $t\in [0,1]$, the surface  $X^{\disc} $ is homeomorphic to $ X_t$. Indeed, when $t=1$ then $X_t$ is split along each $c\in \mathcal S_\infty$; deleting $A_{\I,c,1}$ from $X$ does the same thing topologically, while removing cusp neighborhoods does not change the topological type.   We call $X^{\disc}$ the \emph{finite part} of $X$, and we now show how to construct an explicit homeomorphism
$$h_t : X^{\disc} \longrightarrow X_t.$$
 that is necessary for the deformation retract referenced in Theorem \ref{thm: deformation retract thm}.

 We first need  to define the following \emph {stretch functions}. Given two intervals $[a,b] \subset \RR$ and $[a',b']\subset \RR\cup \{+\infty\}$, define an orientation preserving  homeomorphism 
% $$\mathfrak s := \mathfrak s_{a,b,a',b'} : [a,b] \longrightarrow [a',b'], \ \ \mathfrak s(y) =  a' + \mathfrak h \left ( \frac {t-a}{b-a} \cdot \mathfrak h^{-1}(b'-a')  \right ),$$
$$\mathfrak s := \mathfrak s_{a,b,a',b'} : [a,b] \longrightarrow [a',b'], \ \ \mathfrak s(y) =  a' + \mathfrak h \left ( \frac {\mathfrak{h}^{-1}(y-a)}{\mathfrak{h}^{-1}(b-a)} \cdot \mathfrak h^{-1}(b'-a')  \right ),$$
where as above, $\mathfrak h : [0,1] \longrightarrow [0,\infty]$ is a fixed orientation preserving homeomorphism. Note that $\mathfrak s$  is continuous, not just in $y$, but in the parameters $a,b,a',b'$ too, \emph{even when $b'=\infty$.}  The fact that $\mathfrak s$ is defined when $b'=\infty $ is the reason we use this function rather than a linear one. Another important property of $\mathfrak s$ is that if $b-a=b'-a'$ then $\mathfrak{s}$ is the translation mapping one interval to the other. 

For convenience, we also define, for intervals $[a,b] \subset \RR$ and $[a',b'] \subset \RR \cup \{-\infty\}$, a similar orientation preserving homeomorphism 
% $$\mathfrak s^* :=\mathfrak s^*_{a,b,a',b'} : [a,b] \longrightarrow [a',b'], \ \ \mathfrak s^* = b' - \mathfrak h \left ( \frac {b-t}{b-a} \cdot \mathfrak h^{-1}(b'-a')  \right ).$$
$$\mathfrak s^*(y) :=\mathfrak s^*_{a,b,a',b'} : [a,b] \longrightarrow [a',b'], \ \ \mathfrak s^* = b' - \mathfrak h \left ( \frac {\mathfrak{h}^{-1}(b-y)}{\mathfrak{h}^{-1}(b-a)} \cdot \mathfrak h^{-1}(b'-a')  \right ).$$
 The difference between $\mathfrak s$ and $\mathfrak s^*$ is as follows. If the interval $[a',b']$  is much larger than $[a,b]$, then any homeomorphism between them must stretch $[a,b]$ a lot. In $\mathfrak s$, the stretching happens near the right endpoints of the intervals, while in $\mathfrak s^*$ the stretching happens near the left endpoints. Note that these stretch maps commute with translations of the relevant intervals, e.g.\ 
$$\mathfrak s_{a,b,0,b'-a'}(x) = \mathfrak s_{a,b,a',b'}(x) - a', \ \ \ \mathfrak s_{0,b-a,a',b'}(x) = \mathfrak s_{a,b,a',b'}(x+a).$$

\medskip

Next, we wish to define a homeomorphism $$h_t : X^{\disc} \longrightarrow X_t$$ that stretches the cylinders $A_{\I}, A_{\II}, K_{\I},K_{\II}$ for every short curve and every cusp. 
\begin {definition} \label{def: htdef}  
% Let $\epsilon_0$  be the Margulis constant and $\epsilon < \frac 12 \epsilon_0$ be our fixed constant. 
On the complement of these cylinders, which we denote by 
\[X^{\iota} = X \setminus \left( \cup_{c\in\mathcal{S}} (A_{\I,c,t}\cup A_{\II,c,t}) \cup \big(\cup_{\zeta\in \mathcal K} K_{\zeta,\I} \cup K_{\zeta,\II}\big) \right),\] 
the map $h_t$ is defined using the map $\iota$. 
That is, if $p\in X^{\iota}$, set $$h_t(x) := \iota_t(x).$$
We now define $h_t$ on each of $A_{\I,c,t}, A_{\II,c,t}$, $K_{\zeta,\I}$ and $K_{\zeta,\II}$ whenever they are contained in $X^{\disc}$. 
	
%If $x \in X^{\fin}$ lies outside all the generalized annuli $A_{c,t},T_{c,t}$ associated to $c\in \mathcal S$, and also outside all cuspidal components of the $\epsilon_0$-thin part of $X$, set 
%$$h_t(x) := \iota_t(x).$$
%We now define $h_t$ on each of these pieces,  using the parametrizations discussed in  \S \ref{ATsec} and \S \ref{cusp neighborhoods}. 
\begin{enumerate}
\item[($A_{\I}$)] A generalized annulus $A_{\I}:=A_{\I,c,t}$ is contained in $X^{\disc}$ if and only if $c \not \in \mathcal S_\infty$.  In that case, using the semi-hyperbolic parametrizations from \S \ref{sec: ATsec}, we define
\[h_t : A_{\I} \simeq \bbS^1_{\ell} \times [0,R_{\I}] \longrightarrow  \bbS^1_{\ell'} \times [0,\Delta_{\I}] \simeq A^{\ext}_{\I}\] to be the map which is linear in the first coordinate, and the homeomorphism $$D_{\I,c,t} : [0,R_{\I}] \longrightarrow [0, \Delta_{\I}]$$
in the second coordinate, defined to be the concatenation of
\begin{align*}
	\id : [0, R_{\I}/2] \longrightarrow [0,R_{\I}/2] \\ 
\mathfrak s : [R_{\I}/2,R_{\I}] \longrightarrow [R_{\I}/2,\Delta_{\I}].
\end{align*} 
Note that indeed $R_{\I}/2 < \Delta_{\I}$.
\item[($A_{\II}$)] Using the conformal parametrizations we define 
\[h_t : A_{\II} \simeq \bbS^1_{\ell} \times (\rho_{\I},\rho_{\II}] \longrightarrow \bbS^1_{\ell} \times (\delta_{\I},\rho_{\II}] \simeq A^{\ext}_{\II}\] 
to be the identity in the first coordinate and to be the stretch homeomorphism $$\mathfrak s^* : (\rho_{\I},\rho_{\II}] \longrightarrow (\delta_{\I},\rho_{\II}] $$ in the second. 

Note that this makes sense both when  $c \in \mathcal S_\infty$ and when $c \not \in \mathcal S_\infty$.
% , , if, then $t=1$ and we define $h_1 : A_{\I} \simeq \bbS^1_{\ell} \times (R,R+r] \longrightarrow S^1_{\ell} \times (-\infty,0] \simeq A^{\ext}_{\II}$ to be the identity in the first coordinate and to be $$\mathfrak s^* : (R,R+r] \longrightarrow (-\infty,0] $$ in the second.
%\item[(K)] Let $\zeta$ be a cusp of $X$ and let $K_\epsilon \subset K_{\epsilon_0}$  be the cusp neighborhoods defined in \S \ref{cusp neighborhoods},  together with their given parametrizations $ K_\epsilon \cong S^1 \times [\delta,\infty)  \subset S^1 \times [0,\infty) \cong K_{\epsilon_0}$. Let 
%\begin{equation}
% 	d_t(y) := \mathfrak h \left  ( t \cdot \frac{\min\{y,\delta\}}{\delta} \right ) \in [0,\infty], \label{ddef}
%\end{equation}
%and note  that  the following properties hold:
%\begin{itemize}
%\item $d$ is continuous in both $t$ and $y$,
%\item $d_t(0)=0$ for all $t$,
%\item $d_t(y) = \mathfrak h(t)$ if $y \geq \delta$, so that in particular $d_1(y)=\infty$ when $y\geq \delta$.
%\end{itemize}
%Given $p\in K_{\epsilon_0}$, write it in coordinates as $(x,y) \in S^1 \times [0,\infty)$. If $t<1$, let $p_t \in K_{\epsilon_0}$ be the point with coordinates $(x,y+d_t(y))$, and define $h_t(p) = \iota_t(p_t)$. If $t=1$, we do the same thing, except only when $p \in  K_{\epsilon_0} \cap X^{\fin} = K_{\epsilon_0} \setminus K_\epsilon$,  in which case $d_t(y)<\infty$.

\item[($K$)] Let $K=K_\zeta$ be the standard collar neighborhood of $\zeta$.
% Define
% % Let $\zeta$ be a cusp of $X$ define the map $h_t$ on 
% % let $\delta_0 = \delta (\sigma_0)$ and let $\delta=\delta(\sigma)$. Recall from the definition of $\sigma_0$ that $\delta_0+1 = \delta$. 
% % % Define $d^+_t : [0,\infty) \to [0,\infty]$ by
% \begin{equation}
% d^+_t(y) := \mathfrak h \left  (t \cdot \nu(y-\mu_{\II})  \right ) \in [0,\infty], \qquad \textrm{where} \qquad    \nu(s) :=\begin{cases} 0 & s<0\\
% s & 0\le s<1\\
% 1 &  1 \le s
% \end{cases}. \label{ddef}
% \end{equation}
% Notice that  the following properties hold for $d^+_t(y)$:
% \begin{itemize}
% 	\item $d^+_t(y)$ is continuous in both $t$ and $y$,
% 	\item $d^+_t(0)=0$ for all $t$ and for $y<\mu_{\II}$,
% 	\item $d^+_t(y) = \mathfrak h(t)$ if $y \geq \mu_I$, so that in particular $d^+_1(y)=\infty$ exactly when $y\geq \mu_{\I}$ and $t=1$.
% \end{itemize}
Define $h_t$ on $K_{\I}$ in the semi-hyperbolic coordinates as the map $$\bbS^1 \times [\mu_I ,\infty) \to \bbS^1 \times [\mu_I  + \mathfrak h (t),\infty)$$ which is the identity on the first coordinate and the translation $y \mapsto y+\mathfrak h(t)$ on the second. 

On $K_{\II}$ we define $h_t$ as the map $$\id \times \mathfrak{s} : \bbS^1 \times [\mu_{\II},\mu_{\I}) \to \bbS^1 \times [\mu_{\II},\mu_{\I}+\mathfrak h (t)).$$
% Let $K_{\epsilon_0} \cong S^1 \times [0,\infty)$ be the $(\zeta,d)$-parametrization defined in \S\ref{cusp neighborhoods}. 
% Given $p\in K_{\epsilon _0}$, write $p=(x,y) \in S^1 \times [0,\infty)$. Let $(x,y)$ be the coordinates of a point $p \in K_{\epsilon_0}$. Note that $p\in X^{\disc}$ if and only if $d^+(y)<\infty$. Let $p_t$ be the point of $K_\epsilon$ given by the coordinates $(x,y+d^+_t(p))$ in the $(\zeta,d)$ parametrization. 
% That is, $d^+_t$ is the amount by which $p$ is pushed into the cusp.
% Set
% \[h_t(p) = \iota_t (p_t).\]
% Note that there is no conflict of definition with $h_t | _{X^\iota}$ since $h_t(p)=\iota(p)$ for all $p\in K_\epsilon \cap X^{\iota}$.
\end{enumerate}
\end{definition}

 In the definition above, note that the stretch maps $\mathfrak s$ and $\mathfrak s^*$ are  chosen  so that when $t$ is very large, almost all the stretching that $h_t$ does when it maps $A$ over the much larger extended annulus $A^{\ext}$ happens near the curve $\partial A_{\I}$.

\begin{figure*}[t]
\centering
	\includegraphics{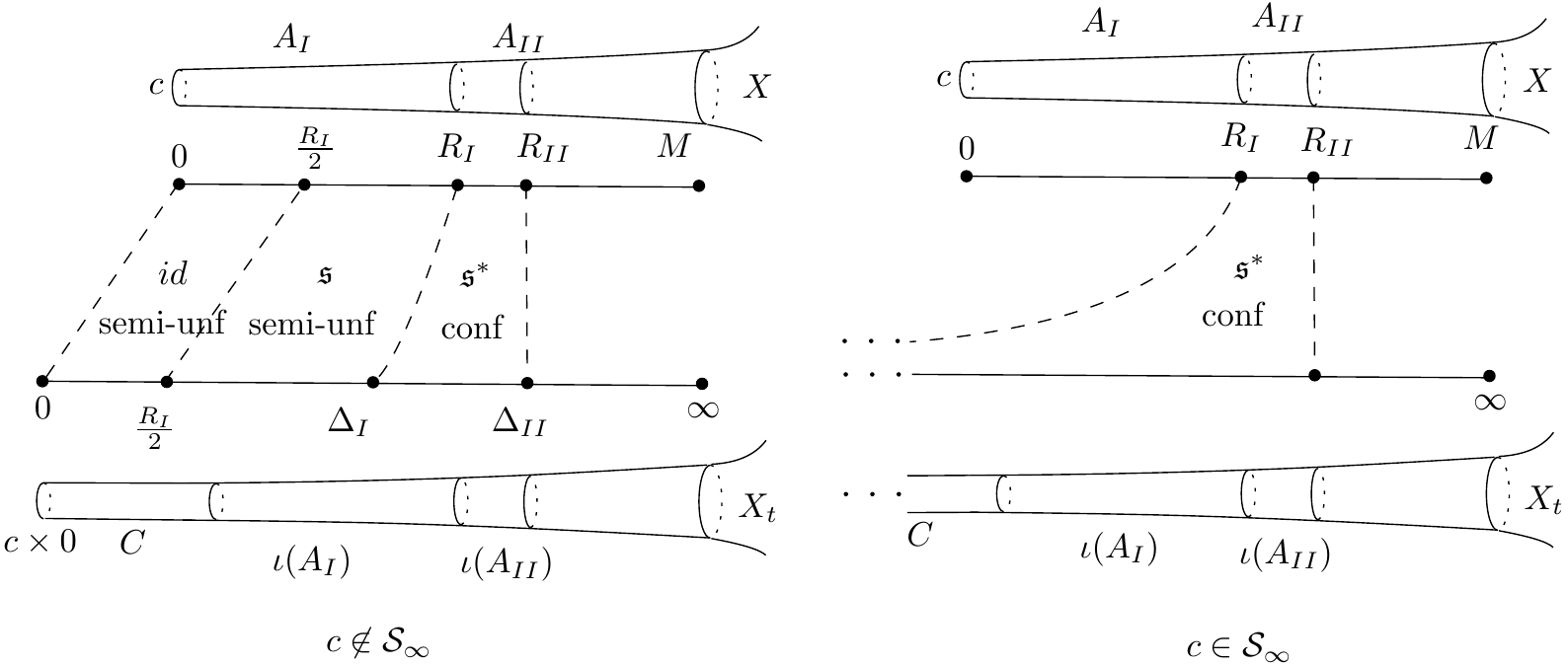}
\caption{ (\emph{The definition of $h_t$  on $A_{\I}$ and $A_{\II}$.}) 
On the left, the cylinders $A$ (top) and $A^{ext}$ (bottom) for $c\notin S_\infty$ are shown with their semi-hyperbolic parametrizations. The map $h_t$ stretches the sub-annuli $A_{\I}$ and $A_{\II}$ by the maps $\id $, $\mathfrak{s}$ and $\mathfrak{s}^*$ in the appropriate parametrizations.
On the right, the same is shown for $c\in S_\infty$.
}
% {\bf***** change figure, change caption ****}  Here, $c\in\mathcal S$ is a regular geodesic, although the picture in the degenerate case differs only in that $c$ and $c\times 0$ are replaced by order $2$ quotients.   In both pictures, the preferred side $V:=V_c$ of the Margulis tube around $c$  is shown at the top, and the bottom picture is of $W$. Collapsing level circles to points, we obtain schematic representations of $V$ and $W$ as line segments. We parametrize the top line segments using hyperbolic distance to $c$ on $V$. The bottom line segments are parameters using the $\hat d$-distance to $c\times 0$ on $W$, and the $\hat d$-distance to $\partial( C \cup \iota(A \cup T))$, respectively, where $\hat d$ is the complete (infinite volume) hyperbolic metric on $W$ in the conformal class of the given piecewise hyperbolic/Euclidean metric. Using the parameter on the line segments as the second coordinate in parametrizations of $V$ and $W$, the map $h_t$ is given by the indicated maps on subintervals.}
\label{fig: annulusmap}
\end{figure*}

\medskip

\begin{remark}\label{rem: continuity of definition of h_t}
We emphasize the following important properties of $h_t$: 
\begin{enumerate}
    \item \label{rem: continuity of definition of h_t - conformal coordinates} When written in the conformal parametrizations, the map $h_t$ on on $X^{\disc}-X^\iota$  is a map $\bbS^1_{\ell} \times [0,\rho_{\II})\modsim \to \bbS^1_{\ell} \times [-L_t/2,\rho_{\II})\modsim $ that is continuous in all of the involved parameters, namely $t$, $\ell=\ell(c)$, $\systole(X)$ and the coordinates of the base point of $v$. 
    The map $\mathfrak{s}^*$ is chosen so that the continuity extends also when the base point of $v$ is in $ \bbS^1_{\ell} \times \{ \rho_{\I}\}$ and $c\in \mathcal{S}_\infty$ if one allows the second coordinate to take the value $-\infty$.
    \item \label{rem: continuity of definition of h_t - semi-unif coordinates} When written in the semi-hyperbolic parametrizations, the second coordinate of $h_t$ on $X^{\disc} - X^{\iota}$ is a map $[0,R_{\II}) \to  [0,\Delta_{\II})$ that is continuous in all of the involved parameters, namely $t$, $\ell=\ell(c)$, $\systole(X)$ and the coordinates of the base point of $v$. 
    The map $\mathfrak{s}$ is chosen so that the continuity extends also when $v \in X^{\nd}$, as long as it is allowed to take the value $\infty$.
    \item \label{rem: continuity of definition of h_t - 0 grafting} If $L_{t}(c)=0$ then $h_t |_{A} = \iota$, as in this case $R_{\II} =R_{\I} = \Delta_{\I}=\Delta_{\II}$, and all stretch maps are thus the identity maps.
\end{enumerate}
\end{remark}

\medskip

	 There is one more property of the homeomorphism $h_t$  that is worth highlighting.  First, note that $h_t$ is not smooth. It is smooth on the interiors of $X^{\iota}$, the $A_{\I}$-annuli, the cusp annuli $K_{\I}$ and $K_{\II}$, and the subsets $\bbS^1_\ell \times [0,R_{\I}/2)$ and $\bbS^1_\ell \times (R_{\I/2},R_{\I})$  of the $A_{\I}$-annuli, but these smooth maps do not patch together smoothly along the boundaries of the annuli.  However, since $h_t$ is obtained by pasting together smooth functions defined on smooth submanifolds that are closed subsets of $X$, $h_t$ is \emph{tangentiable}, meaning that it has well-defined one-sided directional derivatives at every point. So in the following, given $v \in TX_p$,  we will freely speak of the one-sided directional derivative
$dh_t(v) \in T(X_t)_{h_t(p)}$.  Note also that the map $$dh_t : TX^{\disc} \longrightarrow TX_t$$ 
 is a homeomorphism. See e.g.\ Bonahon \cite{bonahon1997variations} for more information on tangentiable maps.

\subsection{The vectored orbifolds $(X_t,v_t)$}
\label{sec: vtsec}

\medskip

Let $(X,v)$  be a   vectored orbifold with $\Gamma(X,v) \in \Sub_S^\epsilon(G)$, let $t\in [0,1]$, and  suppose that  the base point $p$ of $v$ lies in $X^{\disc}$.  In this section, we define the vectors $v_t \in TX_t$. 

 Recall that the piecewise hyperbolic/Euclidean orbifold $X_t$  was created by grafting on $X$ generalized Euclidean annuli along curves $c\in \mathcal S$, see Definition \ref{def: graftingdef}, and that we defined a  tangentiable homeomorphism $$h_t : X^{\disc} \longrightarrow X_t$$   in Definition \ref{def: htdef}. We will define $v_t \in TX_t$  to be a certain vector based at the point $h_t(p)$.  A natural choice would be the directional derivative $dh_t(v)$, but  this would cause problems later\footnote{Defining $v_t:=dh_t(v)$ would be problematic because $h_t$ stretches a lot in the second coordinate near the intersection $ A_{c,t,\I} \cap A_{c,t,\II}$, and hence if $v$ is a vector based at a point very near this intersection, then $dh_t(v)$ has a huge second coordinate.  The resulting deformation will not be continuous when $t\to 1$.}. Instead, first note that for every  $c\in \mathcal S$, there is a \emph{natural foliation} of its collar neighborhood $A_c$ by simple closed curves equidistant from $c$, and its image under $h_t$ is the \emph{natural  foliation} on $A_c^{\ext} \subset X_t$.  Similarly, every cusp neighborhood $K_{\zeta}$ has a \emph{natural foliation} by horocyclic curves, and its image under $h_t$ is the corresponding \emph{natural foliation} of $h_t(K_{\zeta})$, which is a cusp neighborhood in $X_t$. Then we define $v_t$ as follows. 

\begin{definition}[The vector $v_t \in TX_t$]
\label{def: vtdef}
As above, let $p \in X^{\disc}$ be the base point of $v$.	When $p\in A_c$ for some $c\in \mathcal S$, or $p\in K_{\zeta}$ for some cusp $\zeta\in\mathcal K$, define $v_t \in TX_t$  to be the unit vector based at $h_t(p) \in X_t$ that makes  the same counter-clockwise angle with the  natural foliation as $v\in TX$ did with its natural foliation, such that $v_t$ and  $dh_t(v)$ lie on the same side of the leaf through $h_t(p)$. Otherwise, we just set $v_t = d\iota_t(v)$.
\end{definition}

So defined, the map $v \longmapsto v_t$ is a homeomorphism $TX^{\disc} \longrightarrow TX_t$.

\subsection{The deformation retract $\mathcal D$}

\label{sec: deformation retract section}
We now show how to use the orbifolds $X_t$ from Definition \ref{def: graftingdef} and the vectors $v_t$ from Definition \ref{def: vtdef} to construct a deformation retract $$\Sub_S^\epsilon(G) \cup \partial \Sub_S(G) \longrightarrow \partial \Sub_S(G),$$ 
 as required in Theorem \ref{thm: deformation retract thm}. First, note  that by the geometric uniformization theorem, see Theorem \ref{thm: uniformization}, there is a unique hyperbolic metric on $X_t$ in the conformal class of the piecewise hyperbolic-Euclidean metric. We denote $X_t$ equipped with this metric by $X_t^{hyp}$, and let $v_t^{hyp}$ be the vector with unit hyperbolic length in the direction of $v_t$.  In this way, from  a group ${\Gamma}(X,v) \in \Sub_S^\epsilon(G)$ and $t\in [0,1]$,  we get a  hyperbolic orbifold
$${\Gamma}(X_t^{hyp},v_t^{hyp}) \in \Sub_S(G)$$ as long as the base point of $v$ is in $X^{\disc}$.

 We now use this to define the promised deformation  retract.  It will be convenient in the description below to  think of $\HH^2$ in the disc model, and to take the preferred vector $v_{base} \in T\HH^2$ that determines the correspondence between groups and vectored orbifolds to be the vector $1 \in \CC \cong (T\HH^2)_0$ based at the origin.

\begin{definition}\label{def: retr} 
Define a map	$\mathcal D : \overline{\mathcal \Sub_S^\epsilon(G)} \times [0,1] \longrightarrow \Sub(G)$ in cases as follows.
\begin{enumerate}
	\item If $H \in \partial  \Sub_S^\epsilon(G), $  set $\mathcal D(H,t)=H$ for all $t\in [0,1]$.
\end{enumerate}
Now assume that ${\Gamma}(X,v) \in  \Sub_S^\epsilon(G)$, and let $p\in X$ be the base point of $ v$.
\begin{enumerate}
\item[2.] If  $p\in X^{\disc}$,  set $\mathcal D (\Gamma(X,v),t) = {\Gamma} (X_t^{hyp},v_t^{hyp})$.
\end{enumerate}
Now assume $p \in X^{\nd}$, so in particular $t=1$.  
\begin{enumerate}
\item[3.] If  $p \in K_{\zeta,\I}$  for some cusp $\zeta$ of $X$, define $\mathcal D(\Gamma(X,v),1) \in \Sub(G)$ to be the one-parameter group $N(e^{i\theta})$ of parabolic isometries fixing the point $e^{i\theta}\in \partial \HH^2$ in the disk model, where here $\theta$ is the counterclockwise angle from $v$ to the  tangent vector of the unique geodesic from $p$ to $\zeta$.
\end{enumerate} 
 Now assume  further that $p$ lies in the annulus $A_{\I}=A_{c,\I}$ associated to some $c \in \mathcal S_{\infty}$. Equip $A_{\I}$ with the semi-hyperbolic parametrization $A \longrightarrow \bbS^1_{\ell(c)} \times [0,R_{\I}]$. Let $\theta$ be the counterclockwise angle from $v$ to the  tangent vector of the shortest path from $p$ to $c$, and let $(x,s) \in \bbS^1_{\ell(c)} \times [0,R_{\I}]$  be the coordinates of $p \in A$.
\begin{enumerate}
\item[4.] If $s=R_{\I}$,  set $\mathcal D(\Gamma(X,v),1)  \in \Sub(G)$  to be the one-parameter group $N(\xi)$ of parabolic isometries fixing the point $e^{i\xi}\in \partial \HH^2$  in the disk model.
\item[5.] If $s<R_{\I}$ and $c$ is regular, define $\mathcal D (\Gamma (X,v),1)  \in \Sub(G)$  to be the one-parameter group $A(\alpha)$ of all hyperbolic-type isometries  with axis $\alpha$, 
where $\alpha \subset\HH^2$ is the geodesic line such that the shortest geodesic from $\pbase$ to $\alpha$ forms a CCW angle of $\theta$ with $v_{base}$ and has length $D_{A_{\I},1}(s)$ 
where $D_{A_{\I},1} : [0,R_{\I}) \longrightarrow [0,\infty)$  is the concatenation of $id: [0,R_{\I}/2] \longrightarrow [0,R_{\I}/2]$ and $\mathfrak s: [R_{\I}/2,R_{\I}) \longrightarrow [R_{\I}/2,\infty)$ (cf. Definition \ref{def: htdef}).
\item[6.] If $s<R_{\I}$ and $c$ is degenerate, define $\mathcal D(\Gamma(X,v),1)  \in \Sub(G)$  to be the group $A'(\alpha) \cong \RR \rtimes \ZZ/2\ZZ$ of \emph{all}  orientation preserving isometries leaving $\alpha \subset\HH^2$ invariant, where $\alpha$ is as in part 5.
\end{enumerate}
\end{definition}

% \begin{figure*}
% 	\centering
% 	\def\svgwidth{0.7\textwidth}
% %\input{./figures/deform.pdf_tex}
% \caption{A sketch to help understand Definition \ref{def: retr}. }
% \label{fig: deformation}
% \end{figure*}

\medskip

\subsection{Proof of Theorem \ref{thm: deformation retract thm}}\label{sec: proof of retraction}

We want to prove that $\mathcal D$ is a deformation retraction from $\overline{\mathcal \Sub_S^\epsilon(G)}$ onto its boundary in $\Sub(G)$.   It suffices to show that the following properties are satisfied.
\begin {enumerate}[label=(\alph*)]
\item $\mathcal D(H ,0) = H$  for all $H \in \overline{\Sub_S^\epsilon(G)}.$
\item $\mathcal D(H ,1) \not \in \Sub_S(G)$ for all $H$.
\item $\mathcal D(H ,t) \in \overline{\Sub_S(G)}$ for all $H$ and $t<1$,
\item $\mathcal D$ is continuous,
\end {enumerate}
\noindent Note that it follows from properties (b), (c), and (d) together that $\mathcal D(H,1) \in \partial \Sub_S^\epsilon(G)$ for all $H$.

\medskip

Properties (a) and (b) obviously hold when $H\in \partial \Sub_S^\epsilon(G)$, so assume $(X,v) \in \Sub_S^\epsilon(G)$. 

\paragraph*{Proof of (a):} It obviously holds when $H\in \partial \Sub_S^\epsilon(G)$, so assume $(X,v) \in \Sub_S^\epsilon(G)$. Note that when $t=0$ the lengths of all the Euclidean annuli and orbifolds we are grafting in are zero, so $X_0^{hyp} = X_0 = X$.  Moreover, the homeomorphism $h_0=id$ by Remark \ref{rem: continuity of definition of h_t}(\ref{rem: continuity of definition of h_t - 0 grafting}). So $v_0^{hyp}=v_0=v$. Hence, (a) follows. 

\paragraph*{Proof of (b):} Again, (b) clearly holds for $H\in \partial \Sub_S^\epsilon(G)$, so assume $(X,v) \in \Sub_S^\epsilon(G)$. 

First assume that the base point $p$ of $v$ lies in $X^{\disc}$. Since $v\in \Sub_S^\epsilon(G)$,  this implies that $\systole(X)\leq \epsilon$. 
If $\systole(X)\leq \epsilon$, then $\mathcal S_\infty \neq \emptyset$, so some curve is fully pinched in $(X_1,v_1)$, meaning that $X_1$ is not homeomorphic to $X$. So, $\mathcal D(\Gamma(X,v),1) = {\Gamma}(X_t^{hyp},v_t^{hyp})$ which is not in $\Sub_S(G)$.
Otherwise, if $p\in X^{\nd}$, then $D(\Gamma(X,v),1)$ is a non-discrete subgroup, and hence not in $\Sub_S(G)$.

\paragraph*{Proof of (c):} Property (c) follows from the definition of $\DDD$. 

\paragraph*{Proof of (d):} 
Let $(H^i,t^i) \to (H,t)\in \overline{\Sub_S^\epsilon(G)}$, we wish to show that $\DDD(H^i,t^i) \to \DDD(H,t)$.
We divide the proof into cases according to $(H,t)$.

\paragraph*{Case 1. $H\in \Sub_S^\epsilon(G).$} 
\paragraph*{Case 1.a. $v\in X^{\disc}$.}
Write $H=\Gamma(X,v)$. Without loss of generality $H^i$ are all discrete, write them as $H^i=\Gamma(X^i,v^i)$.
$H^i \to H$, and therefore $(X^i,v^i) \to (X,v)$ in smooth convergence. Let $\psi^i$ be a sequence of almost isometric maps witnessing this convergence, as discussed in \S\ref{sec: grafting and smooth convergence}. The length of curves and the systole vary continuously, and therefore also the grafting parameters $L_{t^i}(\psi^i(c)) \to L_t (\psi^i(c))$ for all $c\in \mathcal{S}$.
If $v_t \in \iota(X \dsm \mathcal{S})$ then without loss of generality $v^i_{t^i} \in \iota(X^i\dsm \mathcal{S}^i)$, and it follows from the definition of $h_t$ that $d(\psi^i)^{-1}(\iota^{-1} (v^i_{t^i})) \to \iota^{-1}(v_t)$.
It then follows from part 1 of Theorem \ref{thm: contgraft} that $\DDD(H^i,t^i)  \to \DDD(H,t)$ by the definition of $\DDD$.

If $v_t \in C_{c,t}$ for some curve $c\in \mathcal{S}$, then without loss of generality, we may assume that $v^i_{t^i} \in C_{\psi^i(c),t^i}$. By Remark \ref{rem: continuity of definition of h_t}(\ref{rem: continuity of definition of h_t - conformal coordinates}) the map $h_t$ depends continuously on all parameters, the assumptions of part 2 of Theorem \ref{thm: contgraft} are satisfied, and it follows that $\DDD(H^i,t^i) \to \DDD(H,t)$.

\paragraph*{Case 1.b. $v\in X^{\nd}$.} In particular $t=1$.
This case follows from Proposition \ref{prop: limits in collars} by Remark \ref{rem: continuity of definition of h_t}.

\paragraph*{Case 2. $H\in \partial \Sub_S^\epsilon(G).$} Since $\DDD$ is the identity on $\partial \Sub_S^\epsilon (G)$ it suffices to show that if $(H^i,t^i) \to (H,t)$ then $\DDD(H^i,t^i) \to \DDD(H,t)=H$ for $H^i \in \Sub_S^\epsilon(G)$. Write $H^i = \Gamma(X^i,v^i)$.

\paragraph*{Case 2.a. $H=A(\alpha)$ for some geodesic line $\alpha\subset \HH^2$.} 
Since $H^i=\Gamma(X^i,v^i) \to H=A(\alpha)$, necessarily $\injrad(v^i,X^i) \to 0$, moreover, $v^i$ is in a standard collar neighborhood around a non-degenerate simple curve $c^i$ and $\ell(c^i) \to 0$. If $d^i$ is the distance of the base point of $v^i$ from $c^i$ and $\theta^i$ is the angle that the vector $v^i$ makes with the natural foliation of the standard collar neighborhood around $c^i$, then $d^i \to d^\infty$ and $\theta^i \to \theta^\infty$ where $d^\infty$ and $\theta^\infty$ are the distance and angle of the geodesic from the origin to the closest point in $\alpha$ in  $\HH^2$.

In the semi-hyperbolic parametrization of $A_{c^i,\I}$, the second coordinate of the base point of $v^i$ is, by definition, $d^i$.
The vector $(v^i)^{t^i}$ is a vector $w^i$ in the extended standard collar $A_{c^i,t^i,\I}^{\ext}$.
Since $\ell(c^i) \to 0$, $R_{\ell(c^i),\I} \to \infty$ and so at some point $d^i < R_{\ell(c^i),\I}/2$. By the definition of $h^{t^i}$, it follows that in semi-hyperbolic parametrization of $A_{c^i,t^i,\I}^{\ext}$ the second coordinate of $w^i$ is also $d^i$.
Moreover, the vector $w^i$ forms an angle of $\theta^i$ with the natural foliation of $A_{c^i,t^i}^{\ext}$. 

It follows from Proposition \ref{prop: limits in collars}(1) and the definition of $\DDD$, that $\DDD(H^i) \to A(\alpha) = \DDD(A(\alpha))$. 

\paragraph*{Case 2.a'. $H=A'(\alpha)$ for some geodesic line $\alpha\subset \HH^2$.} This case is done similarly to the previous using Proposition \ref{prop: limits in collars}(2).
\paragraph*{Case 2.b. $H=N(e^{i\theta})$ for some $\xi =e^{i\theta}\in\partial\HH^2$.} Again, we may assume without loss of generality that $H^i \in \Sub_S^\epsilon(G)$, and write $H^i=\Gamma(X^i,v^i) \to N(\xi)$. Necessarily, $\injrad(v^i,X^i)\to 0$. It follows that $v^i$ is either in a cusp neighborhood $K_{\zeta^i}$ or in a collar neighborhood $A_{c^i}$ of a curve $c^i\in X^i$. In both cases the distance of $v^i$ from the boundary of the neighborhood tends to $\infty$. Or in other words, it is in the $D^i$ truncation of $A_c$ (or the cusp $K_\zeta$) for some $D^i \to \infty$. Moreover, in the second case, the distance $d^i$ of $v^i$ from the curve $c^i$ tends to $\infty$ as well.

By definition of $h_t$ in the conformal parametrization, the base points of the vectors $(v^i)^{t^i}$ can only move dipper into the extended collar neighborhood $A^{\ext}_{c^i,t^i}$ (or the cusps $K_\zeta$) and are thus in its $D^i$ truncation. Again, the desired limit follows from  Proposition \ref{prop: limits in collars}(3).

\paragraph*{Case 2.c. $H\in \partial\Sub_S(G)$ is discrete.} In this case denote $H=\Gamma(X,v)$ for some orbifold $X$. It follows that $\lim_{i\to\infty} \injrad(X^i,v^i)=\injrad(X,v)>0$, and $\systole(X^i) \to 0$ as otherwise the limit will be in $\Sub_S (G)$.
It follows that eventually $v^i \in (X^i)^\iota$. Thus, $(v^i)^{t^i}= \iota (v^i)$ and the desired limit follows from Theorem \ref{thm: contgraft}(1).\qed

\section{Appendix}

The purpose of this appendix is to prove the following proposition. Below, recall from \S \ref{sec: global distance estimates} that the \emph {$\delta $-reduced  collar} of a curve $c$ is the metric neighborhood $\RC_\delta(\gamma)$ of radius $R_{\delta,\gamma} := \sinh^{-1}(\delta/\sinh(\ell(\gamma)/2)).$

\begin{proposition}[Bounded alteration of boundary lengths is bilipschitz outside collars]\label{prop: alteration}
Given $\mathcal L$ and any sufficiently small $\delta>0$, there is some $K$ as follows. Suppose that $P,P'$ are pairs of pants with geodesic boundary components $\gamma_1,\gamma_2,\gamma_3$ and $\gamma_1',\gamma_2',\gamma_3'$, respectively, and that $\mathcal S \subset \{1,2,3\}$. Suppose that \begin{itemize}
    \item if $i \in \mathcal S$, then $\ell(\gamma_i),\ell(\gamma_i') \leq \mathcal L,$
    \item if $i \not \in \mathcal S$, then $\ell(\gamma_i)=\ell(\gamma_i')$.
\end{itemize}
For each $i\in \mathcal S$, let $\RC_\delta(\gamma_i),\RC_\delta(\gamma_i')$ be the $\delta$-reduced collars of $\gamma_i,\gamma_i'$ in $P,P'$. Then there is a homeomorphism $f : P\longrightarrow P'$ with the following properties.
\begin{enumerate}
    \item $f$ is constant speed on each component of $\partial P$, and maps seams of $P$ to seams of $P'$. (Here, a \emph{seam} of a pair of pants is the shortest path joining two boundary components.)
    \item $f(\RC_\delta(\gamma_i))=\RC_\delta(\gamma_i')$ for all $i\in \mathcal S$, and sends the two orthogonal foliations of $\RC_\delta(\gamma_i)$ by geodesic segments and equidistant circles to the corresponding foliations of $\RC_\delta(\gamma_i')$. Moreover, $f$ is constant speed on each leaf of the foliations. (The speed may depend on the leaf, in the equidistant circle case.)
    \item $f$ is locally $K$-bilipschitz outside of the reduced collars $\RC_\delta(\gamma_i)$.
\end{enumerate}
\end{proposition}

The confident reader might assume that such a result  follows from some sort of compactness argument: if we are ignoring collars around short curves, then shouldn't everything else be bilipschitz? However, an important quality of the proposition is that the boundary components that are not indexed by $\mathcal S$ can be unboundedly long, while the bilipschitz constant $K$ just depends on $\mathcal L$ and $\delta$.

The proof of Proposition \ref{prop: alteration} is the subject of Section \ref{sec: prop proof}. Here is the basic idea. Any pair of pants is the double of a right angled hexagon along three nonadjacent sides, where boundary components are the doubles of the remaining `free' sides. Right angled hexagons have a thick-thin decomposition (see Proposition~\ref{prop: thick thin hex}) into thin rectangles called \emph{collars} and bounded geometry complementary pieces. We show in Lemma \ref{lem: controlled distortion of necks} that altering a bounded free side length to another bounded value does not significantly affect the sizes of any other collars. Using this, we can patch together a map between the two hexagons by explicitly mapping collars to collars and using that the complementary pieces all have bounded geometry, so are all bilipschitz to each other. Doubling this map proves the proposition above.

We should mention that if $\mathcal L < 1/2$, it is possible to prove the proposition above using techniques of Buser--Makover--Muetzel--Silhol \cite{buser2014quasiconformal}. While probably the two approaches have similar length, we think ours is a bit more conceptual rather than computational\footnote{Perhaps this is just because we haven't intuited their approach well enough, though.}. There are lots of computations below, but they are not very subtle (our only trick is that $\cosh(x)$ and $\sinh(x)$ are basically just $x\mapsto e^x/2$ for large $x$) and mostly they verify qualitative facts about hexagons that will be familiar to many hyperbolic geometers. 

\subsection{Right angled hexagons, necks, collars, and the thick-thin decomposition}

A \emph{right-angled hexagon} is a convex hexagon $P \subset \HH^2$ with geodesic sides and $\pi/2$ interior angles. Here, we develop a `thick-thin decomposition for right angled polygons $P$, where $P$ is divided into rectangles called `collars' and complementary polygons with bounded geometry. 

For technical reasons that will become apparent later, the thick thin decomposition that we will need here needs to have a lot more flexibility than the usual decomposition for hyperbolic surfaces. Namely, ours will not be determined by a single choice of some small $\epsilon$. Rather, there is some freedom about which thin parts to include, and one is even allowed to add in some regions disjoint from thin parts that are not themselves thin. Hopefully, though, the reader is courageous and can power through the following definitions.

\begin{definition}[Necks, coorientations and reduced collars]\label{neckdef}A \emph {neck} of a right angled hexagon $P$ is a common perpendicular to two sides of $P$. Since $P$ is right-angled, any side of $P$ is a neck of $P$. However, $P$ also has \emph{non-side necks} that are common perpendiculars to opposite sides of $P$. 

A \emph{coorientation}\footnote{Colloquially, one might call such a coorientation a `side' of the neck, but that is too confusing given that we talk a lot about sides of polygons.} of a neck $\gamma$ is a choice of orientation for the normal bundle of $\gamma$ consisting of vectors that point into $P$. So, necks that are sides of $P$ have one possible coorientation, and non-side necks have two. When $\gamma$ is equipped with a coorientation, we call it a \emph{cooriented neck}, and usually we will denote it by $\gamma_+$, using $\gamma_-$ to refer to the opposite choice of side if there is one. We will also sometimes refer to the \emph{part of P} determined by a cooriented neck; if the neck is a side of $P$, this is just $P$, while otherwise it is the half of $P$ that lies in the direction of the given coorientation.

For each cooriented neck $\gamma_+$ of $P$, define the \emph{$\delta$-reduced collar} of $\gamma_+$, written  $\RC_\delta(\gamma_+)$, to be the set of all points that lie at a distance less than $R_{\delta,\gamma}$ from $\gamma$ in the part of $P$ determined by the coorientation, where
$$R_{\delta,\gamma} := \sinh^{-1}\left (\frac {\delta}{\sinh(\ell(\gamma))}\right ).$$ For an (uncooriented) neck $\gamma$, we set either
$$\RC_\delta(\gamma) = \RC_\delta(\gamma_+), \ \ \text{ or } \ \ \RC_\delta(\gamma) = \RC_\delta(\gamma_+) \cup \RC_\delta(\gamma_-)$$
depending on whether $\gamma$ has one or two coorientations. 
\end{definition}

\begin{figure}
    \centering
    \includegraphics{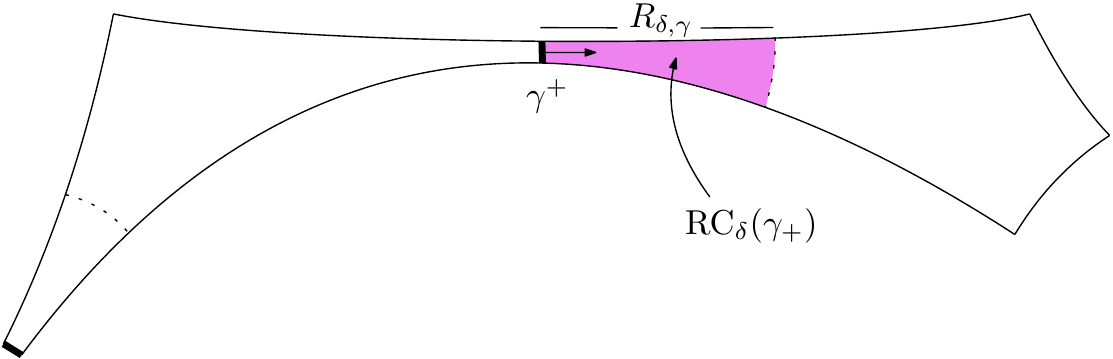}
    \caption{The $\delta$-reduced collar of a cooriented neck $\gamma_+$ is illustrated above. The $\delta$-reduced collar of $\gamma_-$ is a similar region on the left of $\gamma$, and the reduced collar of $\gamma$ includes both regions. In the bottom left corner, there is also (unlabeled) a reduced collar of a neck that is a side of $P$.}
    \label{fig: reduced collar}
\end{figure}
\begin{lemma}[Collars are rectangles]\label{lem: collars are rects}
For functions $\delta\leq  1$, the following holds. Let $P$ be a right angled hexagon. When $\gamma_+ \in \mathfrak N^\pm$ is a cooriented neck, its $\delta$-reduced collar $\RC_\delta(\gamma)$ is a rectangle bounded by four arcs. These arcs are $\gamma,$ an equidistant arc at distance $R_{\delta,\gamma}$ from $\gamma$, and two subsegments of the incident sides of $P$ that have length $R_{\delta,\gamma}$ and go in the direction of the given coorientation.
\end{lemma}
\begin{proof}
This is essentially proven for $\delta=1$ in the course of the usual proof of the Collar Lemma for curves on hyperbolic surfaces, see \cite[Figure 13.3, pg 382]{primer}. Briefly, divide $P$ into two right angled pentagons $A,B$ in such a way that $\gamma$ is a side of $A$. Here, if $\gamma$ is a side of $P$, we cut $P$ along some non-side neck that doesn't intersect $\gamma$, while if $\gamma$ is a non-side neck, we cut along $\gamma$ itself. By the right angled pentagon formula, see Figure \ref{fig: rapr}, and using the fact that $\cosh \geq 1$, whenever $\alpha$ is a side of $A$ adjacent to $\gamma$, we have $$\sinh(\ell(\gamma)) \sinh(\ell(\alpha)) \geq 1 \implies \ell(\alpha) \geq \sinh^{-1}\left ( \frac 1{\sinh(\ell(\gamma))} \right ).$$
As the sides adjacent to $\gamma$ are the shortest paths from $\gamma$ to the sides nonadjacent to $\gamma$, this implies that for $\delta<1$ the  $R_{\delta,\gamma}$-neighborhood of $\gamma \subset A$ is a rectangle bounded by $\gamma$, portions of the two adjacent sides, and an equidistant arc.
\end{proof}

\begin{figure}
    \centering
     \begin{minipage}[c]{0.3\textwidth}
     \vspace{0pt}
    \includegraphics{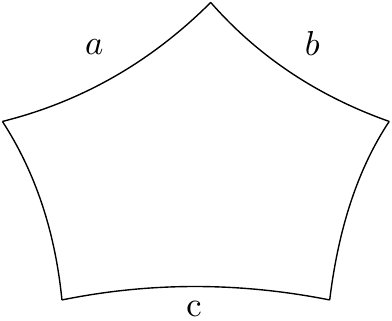}
     \end{minipage}\qquad
  \begin{minipage}[c]{0.3\textwidth}
  \vspace{0pt}
$$\cosh(c)=\sinh(a)\sinh(b)$$
   \end{minipage}
    \caption{The \emph{right angled pentagon rule} for any hyperbolic right angled pentagon, see Theorem 3.5.11 of \cite{ratcliffe1994foundations}.}
    \label{fig: rapr}
\end{figure}

Below, we will only consider necks $\gamma$ that have bounded length. In that case, the $\delta$-reduced collar shrinks by a uniform distance amount when $\delta$ is decreased. For instance:

\begin{claim}\label{claim: collar shrinking}
Fix $\mathcal L>0$ and assume that $\gamma$ is a neck of $P$ with length at most $\mathcal L$. Then for $\delta \leq \frac 14$, we have $$R_{\delta,\gamma} \leq R_{1,\gamma} - c$$ for some positive $c=c(\mathcal L,\delta)$. Moreover, for fixed $\delta\leq \frac 14$ and sufficiently short $\gamma$ we have $$R_{\delta,\gamma} \leq R_{1,\gamma} - \ln(4).$$
\end{claim}
\begin{proof}
The point is to show that there is such a $c$ that does not depend on the length of $\gamma$. Since there is an upper bound on $\mathcal L$, and $R_{\delta,\gamma}$ is continuous in $\ell(\gamma)$, it suffices to show that there is such a $c$ when $\ell(\gamma)$ is small, for then the claim follows by a compactness argument.  But for small $x$,  $\sinh x \approx x$ up to small multiplicative error, and for large $x$, $\sinh x \approx e^x/2$ up to small additive error. In particular, it follows that $R_{\delta,\gamma} \approx \ln (2\delta /\ell(\gamma))$ up to an additive error that goes to zero when $\delta$ is fixed and $\ell(\gamma)\to 0$.  From this, one can deduce that if $\delta \leq \frac 14$ and $\ell(\gamma)$ is small, then $R_{\delta,\gamma} \leq R_{1,\gamma} - \ln(4).$
\end{proof}

We now describe our flexible thick-thin decomposition of $P$.

\begin{proposition}[A thick-thin decomposition for right angled hexagons]\label{prop: thick thin hex} For every $\mathcal L>0,$ the following holds for all sufficiently small $0<\delta_{min} < \delta_{max}$, and all $0<\epsilon_{min} <\epsilon_{max}$ much smaller than $\delta_{min}$. Suppose that $P$ is a right angled hexagon, and that we have:
\begin{enumerate}
    \item a collection of necks $\mathfrak N_{\epsilon}$ in $P$ that includes all necks of length less than $\epsilon_{min}$, and no necks of length at least $\epsilon_{max}$, 
    \item a set $\mathfrak N_{\mathcal L}$ of pairwise nonadjacent sides of $P$ that all have length in $[\epsilon_{max},\mathcal L]$,
    \item a function $\delta : \mathfrak N^\pm \longrightarrow [\delta_{min},1/4]$, where $\mathfrak N= \mathfrak N_\epsilon \cup \mathfrak N_{\mathcal L}$.
\end{enumerate} 
Writing $\RC_{\delta}(\gamma)$ for the $\delta(\gamma)$-reduced collar of $\gamma$, the following hold.
\begin{enumerate}
    \item The distance between any two reduced collars $\RC_\delta(\gamma), \gamma \in \mathfrak N$, is at least some positive constant $c$ depending on $\mathcal L,\delta_{min},\epsilon_{min},\epsilon_{max}.$
    \item The complement $T \subset P$ of all the collars $\RC_\delta(\gamma), \gamma \in \mathfrak N$ is either a topological hexagon or the union of two topological pentagons. Each component of $T$ has `bounded geometry', in the sense that it is the $K$-bilipschitz image of a regular hyperbolic pentagon or hexagon with unit side lengths, where the bilipschitz map is constant speed on sides, and where $K$ depends only on $\mathcal L,\delta_{min},\epsilon_{min}$.
\end{enumerate}
\end{proposition}

\begin{figure}
    \centering
    \includegraphics[width=6in]{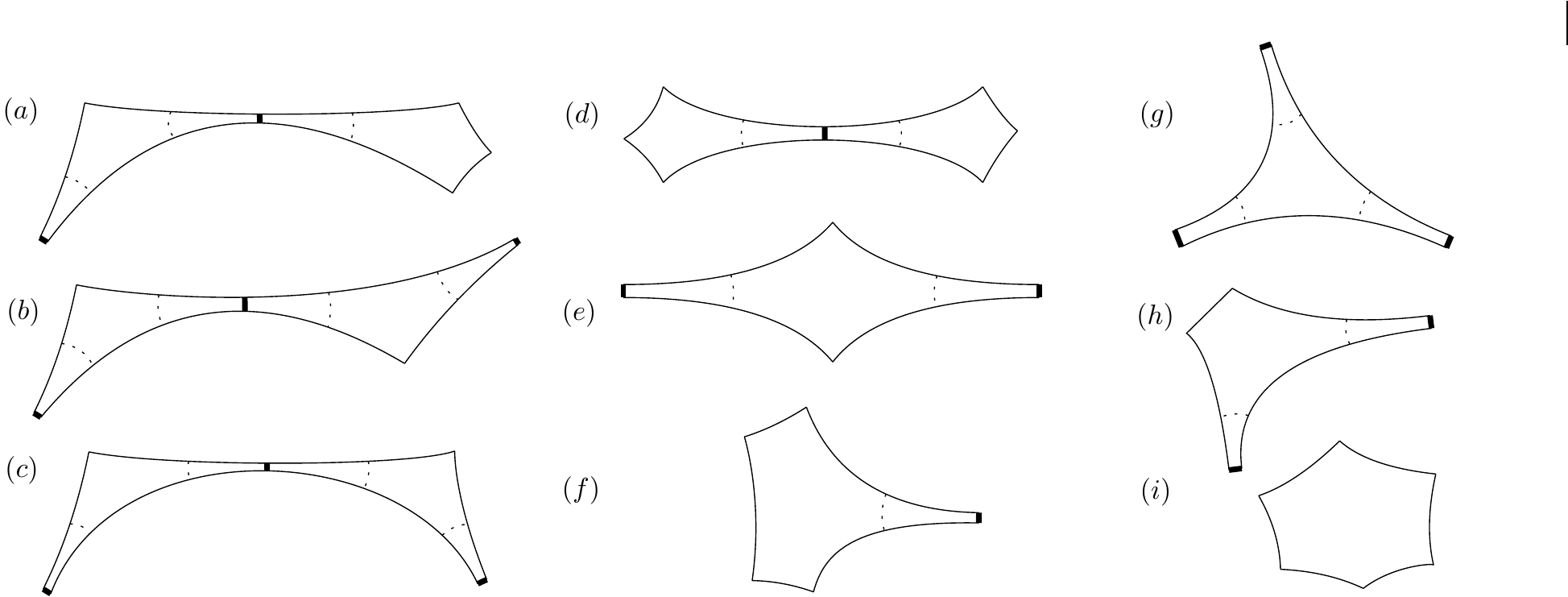}
    \caption{Up to isometry, all the possible configurations of the necks $\mathfrak N$ from Proposition \ref{prop: thick thin hex} are depicted above. The necks are bold, and the boundaries of their reduced collars are dotted. In the picture, the necks are all drawn short for clarity, but note that in Proposition \ref{prop: thick thin hex} there may be some necks with moderate length.}
    \label{fig: hexagons}
\end{figure}

Intuitively, the collars around necks from $\mathfrak N_\epsilon$ are the ones that would correspond to thin parts in the usual thick-thin decomposition. However, there is some flexibility, where we allow ourselves to choose whether to include collars around necks with length in  $[\epsilon_{min},\epsilon_{max}]$, and we also allow some collars around additional non-adjacent bounded length sides of $P$.

Below, we give an elementary proof using only the right angled pentagon formula and exponential estimates for hyperbolic trig functions. A more high-level proof could also be given by relating the thick thin decomposition above to that of the $2$-orbifold obtained by doubling $P$ along its entire boundary. Results similar to Proposition \ref{prop: thick thin hex} also hold for all right angled polygons, but we have stuck to hexagons here since the statement is already awful enough.

\begin{proof}
First, we describe the possible configurations of $\mathfrak N$. Let us call the necks in $\mathfrak N_\epsilon$ \emph{short necks}, and the ones in $\mathfrak N_{\mathcal L}$ \emph{exceptional necks}. First, note that exceptional necks cannot be adjacent sides of $P$, by assumption. If $\gamma $ is a short neck that is a side of $P$, then Lemma \ref{lem: collars are rects} says that the adjacent sides of $P$ both have length at least $\sinh^{-1}(\delta/\sinh(\epsilon))$, which is larger than $\mathcal L$ if $\epsilon$ is small, and therefore those sides cannot lie in $\mathfrak N$. Similarly, if $\gamma $ is a short neck that is not a side of $P$, then neither of the incident sides of $P$ can lie in $\mathfrak N$. Also similarly, there is at most one $\gamma \in \mathfrak N$ that is not a side of $P$, since if there were two, each would cross the $\delta$-reduced collar of the other. More concretely, the necks $\mathfrak N$ are configured in $P$ as in one of the hexagons in Figure \ref{fig: hexagons}, where necks in $\mathfrak N$ are drawn in bold.

\medskip

\textbf{Step 1: Separation of collars. }We want to show that the $\delta$-reduced collars of all necks of $\mathfrak N$ are disjoint, and are separated by at least some positive distance. Since there is at most one non-side neck in $\mathfrak N$, there are two cases to consider.

\medskip

\noindent \emph{Case 1.} If $\alpha,\beta \in \mathfrak N$ are necks that are sides of $P$, let $\gamma$ be the common perpendicular to any pair of opposite sides of $P$ that does not include one of the curves $\alpha,\beta$. Then $\gamma$ divides $P$ into two right angled pentagons, where $\alpha,\beta$ are on opposite sides of $\gamma$. As in the proof of Lemma \ref{lem: collars are rects}, as long as $\delta\leq 1$ the $\delta$-reduced collars of $\alpha,\beta$ are contained in the corresponding pentagons, so they are disjoint. Moreover, if $\delta \leq 1/4$ and $\epsilon$ is small with respect to $\delta$, then Claim \ref{claim: collar shrinking} shows that $R_{1,\alpha} \geq R_{\delta,\alpha} + c$, where $c=c(\mathcal L,\delta)$, and similarly for $\beta$ so the two collars are at least $2c$-separated.

\medskip

\begin{figure}[t]
    \centering
    \includegraphics{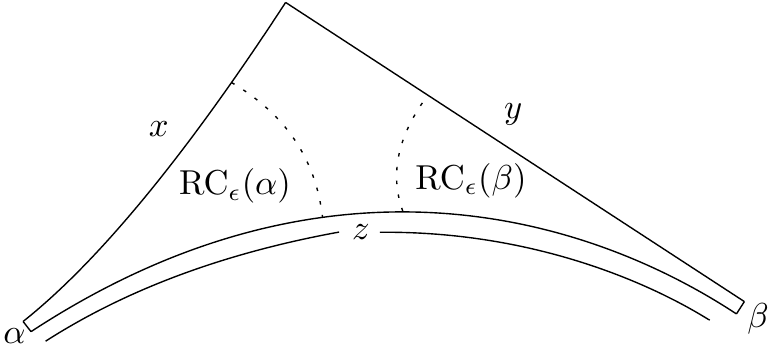}
    \caption{In Case 2, we have $z \geq R_{\delta,\beta} + R_{\delta,\alpha} + c$, where $c=c(\delta)>0$. So, the collars $\RC_\delta(\alpha)$ and $\RC_\delta(\beta)$ are separated by some definite distance.}
    \label{fig: case2}
\end{figure}

\noindent \emph{Case 2.}  Suppose $\alpha \in \mathfrak N$ is a side of $P$ and $\beta \in \mathfrak N$ is a non-side neck of $P$. Cut $P$ along $\beta$ to get a right angled pentagon $A$ containing $\alpha$, and let $x,y,z$ be the lengths of the remaining sides, where $z$ is the length of the side joining $\alpha,\beta$. See Figure \ref{fig: case2}. Then by the right angled pentagon rule, we have $$\sinh(x)\sinh(y) = \cosh(z), \ \ \sinh(\ell(\beta)) \sinh y \geq 1, \ \ \sinh(\ell(\alpha)) \sinh x \geq 1,$$
so it follows that $$\cosh(z) \geq 1/\sinh(\ell(\beta)) \cdot 1/\sinh(\ell(\alpha) = \frac 1 {\delta(\alpha)\delta(\beta)} \cdot \sinh(R_{\delta,\beta}) \cdot \sinh( R_{\delta,\alpha}).$$ We want to say that $z \geq R_{\delta,\beta} +  R_{\delta,\alpha} + c$ for some uniform $c$, since then the collars $\RC_{\delta}(\beta)$ and $\RC_\delta(\alpha)$ will be $c$-separated in $P$. Note that $\beta$ is always a short neck, but $\alpha$ may be either short or exceptional. First, suppose that $\alpha$ is exceptional. By choosing $\delta$ small in terms of $\mathcal L$, we can ensure that $R_{\delta,\alpha} < \ln(2)$, say. Lemma \ref{lem: collars are rects} says that the collar $z \geq R_{1,\beta}$, and then Claim \ref{claim: collar shrinking} says that as long as $\epsilon_{max}$ is small enough (making $\beta$ short), we have $R_{1,\beta} \geq R_{\delta,\beta} + \ln(4).$ So, $$z \geq R_{\delta,\beta} + \ln(4) \geq R_{\delta,\beta} + R_{\delta,\alpha} + \ln(2).$$
Now assume that $\alpha $ is a short neck. A long as $\epsilon_{max}$ is small relative to $\delta_{min}$, the inputs $z,R_{\delta,\beta},R_{\delta,\alpha}$ are large, in which case their $\cosh$'s and $\sinh$'s are approximated up to a multiplicative factor of $2$ by the function $t\mapsto e^t/2$. Hence, we have $$e^z \geq \frac 1{\delta(\alpha) \cdot \delta(\beta)} \cdot \frac { e^{R_{\delta,\beta}}  }{2 } \cdot \frac {  e^{R_{\delta,\alpha}}}{2 }.$$
So, as long as $\delta_{max} < \frac 14$, say, we have $z \geq R_{\delta,\beta} + R_{\delta,\alpha} + \ln 4$. As $z$ is the distance from $\alpha$ to $\beta$, this shows that the neighborhoods $\RC_\delta(\beta)$ and $\RC_\delta(\alpha)$ are at least $\ln 4$-separated in $P$.

\medskip

Finally, let $T$ be the complement in $P $ of all the reduced collars of elements of $\mathfrak N$. From Figure \ref{fig: hexagons}, it is clear that $T$ is either a topological hexagon or the disjoint union of two topological pentagons. Also, all sides $s \subset T$ are either segments of sides of $P$, or \emph{collar boundaries}, i.e. equidistant curves to necks $\gamma \in \mathfrak N$ that are boundary arcs of the reduced collar $\RC_\delta(\gamma)$. Eventually, we want to show that each component of $T$ is the $K$-bilipschitz image of a regular hexagon. As a warm-up, let us prove that there are upper bounds and lower bounds depending on $\epsilon_{min},\epsilon_{max},\delta_{min},\delta_{max}$ for the length of any side of $T$. 

\medskip

\textbf{Step 2: bounding side lengths of $T$.} If $s$ is a collar boundary associated to $\gamma \in \mathfrak N$, then we have $$\ell(s) = \ell(\gamma) \cdot \cosh \left( \sinh^{-1} \left ( \frac {\delta(\gamma)}{\sinh(\ell(\gamma)/2)}\right ) \right ).$$ If $\gamma$ is a short neck, then as long as our upper bound $\epsilon_{max}$ for short neck lengths is small, the outer $\cosh$ and $\sinh$ are both approximated by $x\mapsto e^x/2$ up to uniform multiplicative error (so $\cosh \sinh^{-1} $ cancels out) and $\sinh(\ell(\gamma))/2\approx \ell(\gamma)/2$, so up to a uniform multiplicative error we have $\ell(s) \approx \delta(\gamma)$. See also Fact \ref{fact: collarfact} (a). If $\gamma$ is an exceptional neck, the upper and lower bounds on $\ell(s)$ follow from a compactness argument.

Suppose next that $s$ is a segment of a geodesic side $\alpha \subset P$. We first describe how to get a lower bound for $\ell(s)$. If the endpoints of $s$ lie on collar boundaries, then $\ell(s)\geq c=c(\mathcal L,\delta)$, since we proved above that all the reduced collars around elements of $\mathfrak N$ are $c$-separated. If $s$ is a side of $P$, it has length at least $\epsilon_{min}$, since it is not an element of $\mathfrak N$. The case where one endpoint of $s$ lies on a reduced collar and the other is a vertex of $P$ follows from Lemma \ref{lem: collars are rects} and Claim \ref{claim: collar shrinking}.

We now describe how to find an upper bound for $\ell(s)$. Given $r>0$, let $A_r(s)$ be the union of all geodesic segments of length $r$ emanating out orthogonally from $s$ in the direction of $P$. The hyperbolic area of $A_r(s)$ is at least $r\cdot \ell(s)$. We claim that for  small $r=r(\delta_{max})$, we have $A_r(s) \subset P$. Assuming this, we have  $r\cdot \ell(s) \leq Area(P)=\pi$, so $\ell(s) \leq \pi/r$, and is bounded as desired.

So, let $\beta$ be a side of $P$ that is not adjacent to $\alpha$. We want to show that if $r$ is small, there cannot be a segment of length at most $r$ that comes out orthogonally from $s$ and intersects $\beta$. So, assume there is such a segment that starts at $p\in s$, as in Figure \ref{fig: quad}. Then using the identity in Figure \ref{fig: quad}, when $\ell(\gamma)$ is small
$$e^{d(p,\gamma)}/2 \leq \cosh d(p,\gamma) \leq \frac {\tanh r}{\tanh(\ell(\gamma))} \leq \frac{\tanh r}{\ell(\gamma)/2}, \ \ \implies d(p,\gamma) \leq \ln \left (\frac{4 \tanh(r)}{\ell(\gamma)}\right ).$$
Applying similar estimates, we get  $R_{\delta,\gamma} \geq \ln ( 4/\ell(\gamma))$, say. So as long as $r<1$, we get that $p\in \RC_\delta(\gamma)$, a contradiction. If $\ell(\gamma) $ is not small, then neither is $r$, in which case we are also done.

\begin{figure}
    \centering
    \begin{minipage}[c]{0.4\textwidth}
     \vspace{0pt}
    \includegraphics{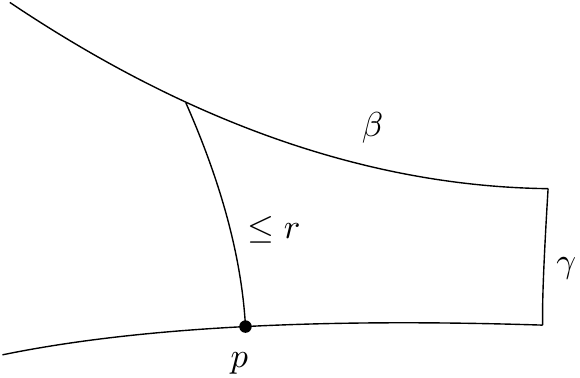}
    \end{minipage}\qquad
    \begin{minipage}[c]{0.3\textwidth}
  \vspace{0pt}
 $$\cosh d(p,\gamma) \cdot \tanh(\ell(\gamma)) \leq \tanh r$$
   \end{minipage}
    \caption{An identity for Lambert quadrilaterals, see  \cite{martin2012foundations}.}
    \label{fig: quad}
\end{figure}

\medskip

\textbf{Step 3: Bounding the geometry of $T$ up to bi-Lipschitz error.} We now show that each component $T_0 \subset T$ has bounded geometry. Currently, we know that the side lengths of $T_0$ are bounded above and below. So, parameterizing each side of $T_0$ via a constant speed map from the unit interval, we can construct a map 
\begin{equation}\label{eq: peqn}
    p : \partial X \longrightarrow \partial T_0 ,
\end{equation} where $X$ is a regular hyperbolic hexagon or pentagon with unit side lengths, and where we consider $\partial T_0,\partial X$ with their intrinsic length metrics. Now, the length metric on $\partial X$ is uniformly bilipschitz to the metric induced from $X$. \emph{We claim that the length metric on $\partial T_0$ is bilipschitz to the  metric induced from the hyperbolic metric on $T_0\subset \HH^2$, where the bilipschitz constant depends only on $\mathcal L,\delta_{min}, \delta_{max}, \epsilon_{min},\epsilon_{max}$.} Assuming this, $p$ will be $K$-bilipschitz for some (new) $K$ depending only on these constants.

Restricting to a particular side of $T_0$, the fact that the arclength metric is bilipschitz to the hyperbolic metric follows since sides of $T_0$ are either geodesics or are bounded length curves with geodesic curvature at most $1$. So, it suffices to take a pair of points $x,y\in \partial T_0$ on distinct sides $s,t$ of $T_0$, and show that 
\begin{equation}
    \label{eq: distortttttion}
    d_{\HH^2}(x,y) \geq c \cdot d_{\partial T_0} (x,y)
\end{equation}
for some $c>0$ depending only on the constants in the proposition.
If $s,t$ are adjacent on $T_0$, then \eqref{eq: distortttttion} follows from the fact that $s,t$ intersect at right angles. If $s,t$ are not adjacent, it suffices to show that $d_{\HH^2}(x,y)$ is bounded below, since the distance between $x,y$ on $\partial T_0$ is at least the length of some in-between side. If $s,t$ are collar boundaries, then $d_{\HH^2}(x,y)\geq 1$, so we are done. There are then two more cases to consider.

\smallskip

\noindent \emph{Case 1.} Assume that $s,t$ are nonadjacent geodesic sides of $T_0$. If $d_{\HH^2}(x,y) < \epsilon_{min}, $ then $s,t$ are segments of sides $\alpha,\beta$ of $P$ that are both incident to a neck $\gamma \in \mathfrak N$. Imagine $\gamma,\beta$ as pictured in Figure~\ref{fig: quad}, where the bottom geodesic is $\alpha$, the top left corner of the quadrilateral is $y$, the left side of the quadrilateral is the shortest path from $y$ to $\alpha$, and where $p \in \alpha$ is the endpoint of this path.  By another identity from \cite{martin2012foundations}, 
$$\sinh d_{\HH^2}(x,y) \geq \sinh  d(y,p) = \sinh \ell(\gamma) \cdot \cosh d(y,\gamma) \geq \sinh \ell(\gamma) \cdot \sinh d(y,\gamma) \geq R_{\delta,\gamma} \cdot \sinh \ell(\gamma),$$
where the last step uses that $y$ is outside $\RC_\delta(\gamma)$. The usual exponential estimates and compactness arguments then show that this is bounded below by some function of $\delta_{min},\mathcal L$, so the same is true of $d_{\HH^2}(x,y)$. 

\smallskip
\noindent \emph{Case 2.} Suppose that $s,t$ are nonadjacent sides of $T_0$, and that $s$ is a collar boundary associated to some neck $\gamma$ while $t$ is a geodesic. By Lemma \ref{lem: collars are rects} and Claim \ref{claim: collar shrinking}, the $\delta$-reduced collar of $\gamma$ is at least some $c$ away from any side of $P$ that is not adjacent to $\gamma $. So, $d_{\HH^2}(x,y)\geq c$.

\medskip

We now know that the map $p : \partial X \longrightarrow \partial T_0$ from \eqref{eq: peqn} is $K$-bilipschitz for some $K$ depending only on $\delta,\epsilon$, where we consider the domain and range with their induced metrics from $\HH^2$. A result of Tukia~\cite{tukia1980planar} says that any $K$-bilipschitz map from the unit circle into $\RR^2$ can be extended to a $K'$-bilipschitz map $\RR^2\to \RR^2$, where $K'=K'(K)$. Since the unit disc is uniformly bilipschitz to $X$ and any compact region in $\HH^2$ is bilipschitz\footnote{Here, the bilipschitz constant depends only on the diameter of the region. In our case, the region of interest is $T_0$, which has diameter bounded by $\ell(\partial T_0)$, which is founded in terms of $\delta,\epsilon$.} to a region in $\RR^2$, Tukia's theorem implies that $p$ extends to a bilipschitz map $X \longrightarrow T_0$. Moreover, by construction our map $p$ is constant speed on sides, as desired.
\end{proof}

\subsection{Bilipschitz maps between right angled hexagons}

It is well known that a right angled hexagon $P$ is determined up to isometry by the lengths of three nonadjacent sides, and that the lengths of the sides can be prescribed arbitrarily, see e.g.\ Ratcliffe \cite[Theorem 3.5.14]{ratcliffe1994foundations}. We will often fix such a triple of sides, calling them `free sides' of $P$, and calling the other three `determined sides'. Doubling $P$ along its determined sides gives a pair of pants whose boundary lengths are twice the length of the free sides; this is a crucial observation in the development of Fenchel Nielsen coordinates. The main point of this section is to show that when a free side length $a$ of $P$ is changed to $a'$ while the other free side lengths are held constant, then as long as $a,a'$ are bounded, the resulting hexagon $P'$ is bilipschitz to $P$ outside the reduced collar of that free side. Intuitively, the reduced collar around a bounded length free side just lengthens if that side is pinched, and nothing else dramatic happens to the hexagon.

We begin by showing that under such an alteration, the other neck lengths of $P$ only change by a uniformly bounded factor. Moreover, the location of any non-side neck of $P$ only shifts a bounded distance.

\begin{lemma}[Bounded changes in free side lengths don't affect other necks much]\label{lem: controlled distortion of necks}
For every $\LLL$ there exists $\epsilon,M$ such that the following holds.
Let $P$ (resp. $P'$) be a (marked) right angled hyperbolic hexagon with free side lengths $a,b,c$ (resp. $a',b',c'$) and assume that $$a,a'\le \LLL, \ \ \  b=b', \ \ \ c=c',$$
If $\gamma \neq a$ is a neck of $P$ with length less than $\epsilon$, and $\gamma'$ is the corresponding neck in $P'$, then
\begin{equation}\label{eq: neck distortion}
\frac{1}{M} \ell (\gamma) <\ell (\gamma') < M \cdot \ell (\gamma).
\end{equation}
Moreover, suppose that the neck $\gamma$ above joins the $b$-side to its opposite side in $P$, dividing  the $b$-side into two arcs of length $b_+,b_-$. If $\gamma',b',b_+',b_-'$ are the corresponding data in $P'$, then   
\begin{equation}\label{eq: neck displacement}
|b_+-b_+'|\le M, \ \ |b_--b_-'|\le M.
\end{equation}
\end{lemma}

\begin{proof} We first prove \eqref{eq: neck distortion} in three cases. 

\medskip

\noindent \textit{Case 1: $\gamma,\gamma'$ are free sides of $P$.} This is trivial since $\gamma$ is one of the sides of length $b$ or $c$, and so is $\gamma'$. 

\medskip
\noindent \textit{Case 2: $\gamma,\gamma'$ are determined sides of $P$.} Since $a\le \LLL$ and $\ell(\gamma) < \epsilon$, by choosing $\epsilon$ small enough, we may assume that $\gamma$ is not adjacent to $a$.
That is, $\gamma$ is the side opposite to $a$. 
By \cite[Theorem 3.5.14]{ratcliffe1994foundations}, the length of $\gamma$ is given by
\[\cosh(\ell(\gamma)) = \frac{\cosh(b)\cosh(c) + \cosh(a)}{\sinh(b)\sinh(c)}.\]
Using the hyperbolic trigonometric identity $\cosh(x-y)=\cosh(x)\cosh(y)-\sinh(x)\sinh(y)$ we get
\begin{equation}
    \cosh(\ell(\gamma))-1 
    = \frac{\cosh(b-c) + \cosh(a)}{\sinh(b)\sinh(c)}.
\end{equation}
Thus, 
\begin{equation}
    \frac{\cosh(\ell(\gamma))-1}{\cosh(\ell(\gamma'))-1} = \frac{\cosh(b-c) + \cosh(a)}{\cosh(b-c) + \cosh(a')} \le 1+\cosh{\LLL} =: M_1'(\LLL)
\end{equation}
where the last inequality follows since $\cosh(b-c)\ge 1$ and $1\le \cosh{a} \le \cosh(\LLL)$.

By symmetry,
\[1/M_1'(\LLL)  \le \frac{\cosh(\ell(\gamma))-1}{\cosh(\ell(\gamma'))-1} \le M_1'(\LLL) \]
Since $\ell(\gamma)$ is small, we get that $\ell(\gamma')$ is small as well, and we can use the approximation $\cosh(x)  \approx 1 + \frac{1}{2}x^2$ to get 
\[1/M_1 \le \frac{\ell(\gamma)}{\ell(\gamma')} \le M_1\] for some $M_1=M_1(\LLL)$.

\medskip
\noindent \textit{Case 3: $\gamma$ is an arc connecting two opposite sides.} Again, by choosing $\epsilon$ small enough we may assume that $\gamma$ does not meet the side with length $a$. Say $\gamma$ meets the side of length $b$ and its opposite side. Then $\gamma$ divides the hexagon $P$ into two right angled hexagons, and divides the side of length $b$ to two subarcs of length $b_+,b_-$ in the pentagons containing the side of lengths $a,c$ respectively. 

By \cite[Equation 2.6.17]{Thurston}, we have
\begin{equation}\label{eq: length of neck in two pentagon1}
    \sinh(\ell(\gamma)) \sinh(b_+) = \cosh(a) \quad\text{ and }\quad\sinh(\ell(\gamma)) \sinh(b_-) = \cosh(c)
\end{equation}
Thus, 
\begin{equation}\label{eq: length of neck in two pentagon}
    \sinh(\ell(\gamma))^2= \frac{ \cosh(a) \cosh(c)}{\cosh(b_+)\cosh(b_-)}
\end{equation}
Using the inequality
\[ \cosh(b_+)\cosh(b_-) \le \cosh(b_++b_-) = \cosh(b) \le 2 \cosh(b_+)\cosh(b_-)\]
we can, up to a multiplicative factor of $2$, replace the denominator of \eqref{eq: length of neck in two pentagon} by $\cosh(b)$.
Similar estimates for $\gamma'$ give, 
\begin{equation}
    \frac{\sinh(\ell(\gamma))^2}{\sinh(\ell(\gamma'))^2} \le 2\frac{ \cosh(a)}{\cosh(a')} \le 2\cosh(\LLL) =: M_2'(\LLL)
\end{equation}
By symmetry,
\begin{equation}\label{eq: sinh neck distortion}
     1/M_2'(\LLL) \le \frac{\sinh(\ell(\gamma))^2}{\sinh(\ell(\gamma'))^2} \le M_2'(\LLL)
\end{equation}
Finally, since $\ell(\gamma)$ is small, so must be $\ell(\gamma')$. We can use the approximation $\sinh(x)\approx x$ to get the desired inequality
\[1/M_2 \le \frac{\ell(\gamma)}{\ell(\gamma')} \le M_2.\]
for some $M_2=M_2(\LLL)$. We have now shown that the desired inequality \eqref{eq: neck distortion} holds for any $M\ge \max\{M_1,M_2\}$. 

\medskip

To prove \eqref{eq: neck displacement}, since $b_++b_-=b=b'=b_+'+b_-'$, it suffices to show either $|b_+-b_+'|\leq M$ or $|b_--b_-'| \leq M$. So, we can work with $b_+,b_+'$ and assume that they lie in the same component of $P \dsm \gamma$ as the $a$-side. By \ref{eq: length of neck in two pentagon1}, $$ \sinh(b_+) = \frac{\cosh(a)}{\sinh \ell(\gamma)}, \ \ \ \  \sinh(b_+') = \frac{\cosh(a')}{\sinh \ell(\gamma')}.$$
Hence, $$\frac{\sinh(b_+)}{\sinh(b_+')}  =  \frac{\cosh(a)\sinh \ell(\gamma')}{\cosh(a')\sinh \ell(\gamma)}.$$
By \eqref{eq: sinh neck distortion}, 
\begin{equation}\label{eq: sinh neck displacement}
    1/M_3'\le \frac{\sinh(b_+)}{\sinh(b_+')} \le M_3'
\end{equation} 
where $M_3'=\sqrt{M_2'(\LLL)}\cosh(\LLL)$.
Applying $\ln(\cdot)$ to \eqref{eq: sinh neck displacement} and recalling that $\ln \sinh (x) \approx x - \ln(2)$ we get that $$ |b_+ - b_+'|\le M_3,$$ for some $M_3$ depending only on $\mathcal L$, as desired. 

\medskip

Taking $M=\max\{M_1,M_2,M_3\}$, completes the proof.
\end{proof}

We now use the thick thin decomposition for right angled hexagons and the previous lemma to show that altering a bounded side length of $P$ gives a polygon $P'$ that is uniformly bilipschitz to $P$ outside of the reduced collar of that side.

\begin{lemma}[Bilipschitz maps between hexagons]\label{lem: bilipschitz maps hex}
For every $\LLL>0$, the following hold for sufficiently small constants $\delta$, and sufficiently large $K=K(\mathcal L,\delta)$. As in Lemma \ref{lem: controlled distortion of necks}, let $P$ (resp. $P'$) be a (marked) right angled hyperbolic hexagon with free side lengths $a,b,c$ (resp. $a',b',c'$) and assume that $$a,a'\le \LLL, \ \ \  b=b', \ \ \ c=c'.$$
Then there exists a homeomorphism $f:P \longrightarrow P'$ with the following properties.
\begin{enumerate}
    \item $f$ is constant speed on each free side of $P$.
    \item $f$ takes the $\delta$-reduced collar of any free side of $P$ that has length at most $\mathcal L$ to the $\delta$-reduced collar of the corresponding free side of $P'$. (This applies to the $a$-side, but perhaps also to the $b$ and $c$ sides.) On each such collar, $f$ is the unique map that sends the two orthogonal foliations by geodesic segments and equidistant arcs to the corresponding foliations of the image collar and is constant speed on every leaf of these foliations. (The speed depends on the leaf in the equidistant arc case.)
    \item $f$ is locally $K$-bilipschitz outside of the $\delta$-reduced collar of the $a$-side of $P$. 
\end{enumerate}
\end{lemma}
\begin{proof}
We want to apply Proposition \ref{prop: thick thin hex} to both $P$ and $P'$. Fix some $\delta$ that is small in terms of $\mathcal L$. Define $\mathfrak N$ to be the collection of all necks of $P$ with length less than some $\epsilon<<\delta$, together with all free sides of $P$ with length at most $\mathcal L$. Let $\mathfrak N'$ be the corresponding set of necks in $P'$. Then as long as $\epsilon$ is much smaller than $\delta$, Proposition \ref{prop: thick thin hex} applies to the set of collars $\RC_{\delta_0}(\gamma)$, where $\gamma \in \mathfrak N$. 

We now want to apply Proposition~\ref{prop: thick thin hex} to the collection $\mathfrak N'$, but here it is important to let the corresponding $\delta'$ vary in a certain way with the (cooriented) neck. So, given a neck $\gamma \in \mathfrak N$, let $\gamma'$ be the corresponding neck in $\mathfrak N'$. 
\begin{enumerate}
    \item[(a)] If $\gamma$ is a free side of $P$, set $\delta'(\gamma') = \delta$.
    \item[(b)] If $\gamma$ is a determined side of $P$, set $\delta'(\gamma')$ to be the unique value such that $R_{\delta',\gamma'}=R_{\delta,\gamma}$.
    \item[(c)] If $\gamma$ is a non-side neck, then as long as $\epsilon$ is small $\gamma$ cannot be incident to the $a$-side, so we can assume without loss of generality that it is incident to the $b$-side, and so divides the $b$-side into arcs of length $b_+,b_-$. Label the cooresponding cooriented necks as  $\gamma_+,\gamma_-$ so that the preferred coorientation of $\gamma_+$ points along the $b_+$-arc, and let $\gamma_+',\gamma_-',b_+',b_-'$ be the corresponding data in $P'$. Then set $\delta'(\gamma_+')$ to be the unique value such that $b_+' - R_{\delta',\gamma_+'}  =  - b_+ - R_{\delta,\gamma_+} $.
\end{enumerate}
The reason for the complicated definition in (c) is that the location along the $b'$-side of the non-side neck $\gamma'$ may be slightly different than that of $\gamma$ along the $b$-side. So, if we want to construct a map $f$ that is constant speed along the $b$ side, we cannot just map a symmetric collar around $\gamma$ to a symmetric collar around $\gamma'$. Instead, we use a collar around $\gamma'$ where some of the radius is shifted from one coorientation to the opposite coorientation. Really, this is the entire reason why we developed our thick-thin decomposition with enough flexibility to allow for varying $\delta$.

Let us check that the conditions of Proposition \ref{prop: thick thin hex} are satisfied for $P',\mathfrak N', \delta'$. By Lemma \ref{lem: controlled distortion of necks}, there exists $M'=M'(\LLL)$ so that the collection $\mathfrak N'$ includes all $\epsilon/M'$-necks and does not include any neck of length $\ge M'\epsilon$. So, as long as $\epsilon$ is small enough, $\mathfrak N'$ is an allowable set of necks in Proposition \ref{prop: thick thin hex} associated to the choice $[\epsilon_{min},\epsilon_{max}]=[\epsilon/M',M'\epsilon]$. It remains to define $\delta_{min}, \delta_{max}$.

We claim that the function $\delta'$ has values in an interval $[\delta_0/N,N\delta_0]$, where $N=N(\mathcal L)$. Assuming this, as long as $\epsilon$ was chosen much smaller than $\delta_0$ to begin with, Proposition \ref{prop: thick thin hex} will apply with $\mathcal L$, the intervals $[\delta_{min},\delta_{max}] = [\delta_0/N,N\delta]$ and $[\epsilon_{min},\epsilon_{max}]=[\epsilon/M',M'\epsilon]$, the collection $\mathfrak N'$, and the function $\delta'$ defined above. It of course suffices to bound $\delta'$ when $\gamma_+$ is as in (b) or (c) above, as otherwise $\delta(\gamma)=\delta$. When $\gamma$ is as in (b), as long as $\epsilon$ is tiny compared to $\delta$, exponential estimates apply to the radii in question, and the fact that $\ell(\gamma)/\ell(\gamma')$ is bounded implies that $\delta/\delta'(\gamma')$ is bounded too. Similarly, indicated in the notation of (c), as long as $\epsilon$ is tiny compared to $\delta$, the radius $R_{\delta,\gamma}$ is huge, must larger than the difference $b_+'-b_+$, which is bounded by Lemma \ref{lem: controlled distortion of necks}. So, exponential estimates apply to both $R_{\delta,\gamma}$ and $R_{\delta',\gamma_+'}$, giving that 
$$\ln(4\delta'/\ell(\gamma)) \approx \ln(4\delta/\ell(\gamma')) + b_+'-b_+$$
up to uniform additive error. But $\ell(\gamma') \approx \ell(\gamma)$ up to a multiplicative error of $M=M(\mathcal L)$, and $b_+'-b_+$ is also bounded by $M$, so this implies that $\delta' \in [\delta/N,N\delta]$ for some $N=N(\mathcal L)$. 

By Proposition \ref{prop: thick thin hex}, we now know that all the $\delta$-reduced collars in $P$ are disjoint, and similarly for $\delta',P'$. Moreover, if $T,T'$ are the complements in $P,P'$ of all these collars, there is a locally $K$-bilipschitz homeomorphism $f: T \longrightarrow T'$
that is constant speed on the sides of $T$ and that respects the markings of $P,P'$, in the sense that on $\partial T$, segments of sides of $P$ are sent to segments of sides of the corresponding sides of $P'$, and collar boundaries in $P$ are sent to the corresponding collar boundaries in $P'$.  Extend $f$ to a map $$f : P \longrightarrow P'$$ by defining $f$ on each reduced collar to be the unique map $$f: \RC_{\delta}(\gamma) \longrightarrow \RC_{\delta'}(\gamma')$$ satisfying property 2 of the lemma. Note that property 2 requires that within the collar, $f$ is constant speed on all equidistant curves to $\gamma$, so defining the map like this does indeed give a continuous extension of the earlier map $T \longrightarrow T'$ given by Proposition \ref{prop: thick thin hex}. It only remains to check properties 1 and 3 of the lemma.

For property 1, suppose that $\beta$ is a free side of $P$, and $\beta'$ is the corresponding side of $P'$. If $\beta \in \mathfrak N$, then by definition of $f$ on $\RC_\delta(\beta)$, it is constant speed on $\beta$. Otherwise, consider all the necks $\gamma$ that are incident to $\beta$. If $\gamma$ is a determined side of $P $ that is adjacent to $\beta $, then the collar radii $R_{\delta',\gamma'},R_{\delta,\gamma}$ were defined to be equal, so $f$ is an isometry on the part of $\beta$ that is contained in $\RC_{\delta}(\gamma)$. If $\gamma$ is a non-side neck incident to $\beta$, then we are in case (c) above, so assuming without loss of generality that $\beta$ is the $b$-side, we have \begin{equation}\label{eq: same radius}
    R_{\delta',\gamma_+'}+ R_{\delta',\gamma_-'} = R_{\delta,\gamma_+} + b_+'-b_+ +  R_{\delta,\gamma_-} + b_-'-b_- = 2R_{\delta,\gamma} +b'-b = 2R_{\delta,\gamma}.
\end{equation}
In other words, the length of each leaf of the foliation of $\RC_{\delta}(\gamma)$ by segments orthogonal to $\gamma$ is the same as the length of its image in $\RC_{\delta'}(\gamma')$. So in particular, $f$ is isometric on the part of $\beta$ that lies in $\RC_\delta(\gamma)$. (Note also that the two halves of $\RC_{\delta,\gamma}$ have the same radius, while the two halves of $\RC_{\delta',\gamma'}$ may not, so the map $f$ may not take $\gamma$ to $\gamma'$, but that is fine.) Finally, consider the parts of $\beta,\beta'$ that remain when we delete the intersections with all reduced collars. We know that $f$ is an isometry on the (at most three) segments of $\beta$ that lie in reduced collars, and that $\ell(\beta)=\ell(\beta')$, so the remainder in $P$ has the same length as it does in $P'$.  If there are no non-side necks of $\mathfrak N$ incident to $\beta$, then the remainder is a single segment that is mapped isometrically by $f$, since $f$ has constant speed on sides. If there is a non-side neck $\gamma$ as above, then the part of the remainder that lies in the direction of the coorientation $\gamma_+$ has length either 
$$b_+ - R_{\delta,\gamma} - R_{\delta,\alpha} \ \ \ \ \text{ or } \ \ \ \ b_+ - R_{\delta,\gamma},$$ depending on whether the determined side $\alpha$ adjacent to $\beta$ in the direction of the given coorientation on $\gamma_+$ lies in $\mathfrak N$ or not. Similarly, the part of the remainder of $\beta'$ the lies in the given direction from $\gamma$ has length $$b_+' - R_{\delta',\gamma'} - R_{\delta',\alpha'} \ \ \ \ \text{ or } \ \ \ \ b_+' - R_{\delta',\gamma'}.$$  But in (b) we set $R_{\delta',\alpha'}=R_{\delta,\alpha}$, and in (c) we defined our radii so that $b_+ - R_{\delta,\gamma}=b_+' - R_{\delta',\gamma'}$. So in either case, each segment of the remainder of $\beta$ has the same length as the corresponding segment of $\beta'$. It follows that $f$ is an isometry on $\beta$.

Finally, we show that $f$ is locally $K$-bilipschitz outside of the $\delta$-reduced collar of the $a$-side of $P$. It suffices to check this separately on each reduced collar $\RC_{\delta}(\gamma)$, where $\gamma \in \mathfrak N$ is not the $a$-side. If $\gamma$ is a free side, then $f$ is an isometry from $\RC_\delta(\gamma)$ to $\RC_\delta'(\gamma')$, since $\ell(\gamma)=\ell(\gamma')$ and in (a) above we defined $\delta'=\delta$. There are then two more cases to consider.

\medskip

\noindent \emph{Case 1.} If $\gamma$ is a determined side, then the radii $R_{\delta,\gamma}$ and $ R_{\delta',\gamma'}$ were defined to be equal, so $f$ is isometric on each leaf of the foliation of $R_{\delta,\gamma}$ by geodesics orthogonal to $\gamma$. Moreover, the length of the equidistant arc to $\gamma$ at distance $t$ is $\ell(\gamma) \cosh t$, and similarly for $\gamma'$, so each such arc in $\RC_\delta(\gamma)$ is mapped at speed $\ell(\gamma')/\ell(\gamma)$ onto the corresponding arc of $\RC_{\delta'}(\gamma')$. Lemma \ref{lem: controlled distortion of necks} implies that $1/M \leq \ell(\gamma)/\ell(\gamma') \leq M$, so $f$ is $M$-bilipschitz on each such arc. As the two foliations are orthogonal in both the domain and the image collar, this implies $f$ is $M$-bilipschitz, where $M=M(\mathcal L)$.

\medskip

\noindent \emph{Case 2.} Suppose that $\gamma$ is a non-side neck. Then as long as $\epsilon$ is small, $\gamma$ cannot be incident to the $a$-side of $P$, which has length at most $\mathcal L$. So, we can assume that it is incident to the $b$-side, and we can use the notation $\gamma_\pm,\gamma_\pm',b_\pm,b_\pm'$ of case (c) above. Because of \eqref{eq: same radius}, the the geodesic leaves of $\RC_\delta(\gamma) $ that are orthogonal to $\gamma$ are mapped isometrically onto those of $\RC_\delta(\gamma)$, so it suffices to bound the speed of $f$ on each equidistant arc to $\gamma$. For $t\in [-R_{\delta,\gamma},R_{\delta,\gamma}]$, consider the equidistant arc $\gamma_t \subset \RC_{\delta}(\gamma)$ at distance $|t|$ from $\gamma$ where the sign of $t$ determines the component of $\RC_\delta(\gamma) \setminus \gamma$ in which $\gamma_t$ lies. Similarly, for each $t'\in [-R_{\delta',\gamma_-'},R_{\delta,\gamma_+'}]$ we have an arc $\gamma_t'$. Then $f(\gamma_t) = \gamma'_{t+b_+'-b_+}.$ Now
$$\ell(\gamma_t) = \cosh t \cdot \ell(\gamma), \ \ \ \  \ell(\gamma_{|t|+b_+'-b_+}) =  \cosh(|t+b_+'-b_+|) \cdot \ell(\gamma'),$$
and Lemma \ref{lem: controlled distortion of necks} says that $1/M \leq \ell(\gamma')/\ell(\gamma) \leq M$ and $|b_+'-b_+|\leq M$. So, it suffices then to show $\cosh(|t+x|)$ is within a multiplicative factor of $\cosh(|t|)$ if $|x| \leq M$. But this follows from a compactness argument for $t$ in some closed interval around $0$, while for large positive $t$ we have $\cosh(t+x)\approx \frac 12 \cosh(t)\cosh(x)$, so the ratio with $\cosh(t)$ is approximately $\cosh(x)$, which is bounded. The case of large negative $t$ is similar.
\end{proof}

\subsection{The proof of Proposition \ref{prop: alteration}}

\label{sec: prop proof}

Suppose we are in the setting of Proposition \ref{prop: alteration}, so that $P,P'$ are pairs of pants with boundary components $\gamma_i,\gamma_i'$, $i=1,2,3$, and where for some $\mathcal S \subset \{1,2,3\}$ we have that for $i\in S$, the lengths $\ell(\gamma_i),\ell(\gamma_i') \leq \mathcal L$, while for $i\not \in S$, the lengths $\ell(\gamma_i)=\ell(\gamma_i')$. Cuting along all seams in $P,P'$, we see that $P,P'$ are the doubles of right angled hexagons $H,H'$. Applying Lemma \ref{lem: bilipschitz maps hex} iteratively at most three times, we can change the lengths $\ell(\gamma_i)$, $i\in \mathcal S$, one by one into the lengths $\ell(\gamma_i')$. Composing the at most three maps one obtains, we get a map $f : H \longrightarrow H'$ satisfying properties 1 and 2 of Lemma \ref{lem: bilipschitz maps hex}, and where $f$ is $K$-bilipschitz outside the reduced collars of all elements of $\mathcal S$. Doubling $f$ proves the proposition.

\bibliographystyle{plain}
\bibliography{biblio}
\end{document}